\documentclass[letterpaper, 12pt, oneside]{book}

\usepackage[margin = 1in, includehead, footskip=0.25in]{geometry}

\usepackage{setspace}
\usepackage{caption}
\usepackage{subcaption}

\usepackage{tocbibind}

\usepackage{url}
\usepackage{amsmath}
\usepackage{amsthm}
\usepackage{amssymb}
\usepackage{multirow}
\usepackage{booktabs}
\usepackage{tabularx}
\usepackage{longtable}
\usepackage{lscape}
\usepackage{graphicx}
\usepackage{import}
\usepackage[multiple]{footmisc}
\usepackage{multicol}
\usepackage{epigraph}

\usepackage{tikz}
\usetikzlibrary{arrows}
\usetikzlibrary{shapes.misc}

\usepackage[pdfauthor={Kameryn J Williams},
    pdftitle={The Structure of Models of Second-order Set Theories},
    hidelinks
]{hyperref}

\usepackage[backend=bibtex,style=alphabetic,maxbibnames=15,maxcitenames=6,dateabbrev=false]{biblatex}
\renewcommand{\UrlFont}{\sffamily\small} 
\renewbibmacro{in:}{\ifentrytype{article}{}{\printtext{\bibstring{in}\intitlepunct}}} 
\DeclareFieldFormat{url}{{\UrlFont\url{#1}}} 
\DeclareFieldFormat{urldate}{
  (version \thefield{urlday}\addspace%
  \mkbibmonth{\thefield{urlmonth}}\addspace%
  \thefield{urlyear}\isdot)}
\DeclareFieldFormat{eprint:arxiv}{
  \ifhyperref
    {\href{http://arxiv.org/abs/#1}{%
        arXiv\addcolon\nolinkurl{#1}}\iffieldundef{eprintclass}{}{\UrlFont{\mkbibbrackets{\thefield{eprintclass}}}}}
    {arXiv\addcolon\nolinkurl{#1}\iffieldundef{eprintclass}{}{\UrlFont{\mkbibbrackets{\thefield{eprintclass}}}}}}

\bibliography{../../ref}

\ifdefined\theorem \else \newtheorem{theorem}{Theorem}[chapter] \fi
\theoremstyle{plain}
\ifdefined\theorem \else \newtheorem{theorem}{Theorem} \fi
\ifdefined\maintheorem \else \newtheorem{maintheorem}[theorem]{Main Theorem} \fi
\ifdefined\lemma \else \newtheorem{lemma}[theorem]{Lemma} \fi
\ifdefined\lemmaschema \else \newtheorem{lemmaschema}[theorem]{Lemma Schema} \fi
\ifdefined\proposition \else \newtheorem{proposition}[theorem]{Proposition} \fi
\ifdefined\corollary \else \newtheorem{corollary}[theorem]{Corollary} \fi
\ifdefined\fact \else \newtheorem{fact}[theorem]{Fact} \fi
\ifdefined\problem \else \newtheorem{problem}[theorem]{Problem} \fi
\ifdefined\conjecture \else \newtheorem{conjecture}[theorem]{Conjecture} \fi
\ifdefined\question \else \newtheorem{question}[theorem]{Question} \fi
\ifdefined\observation \else \newtheorem{observation}[theorem]{Observation} \fi
\newtheorem*{theorem*}{Theorem}
\newtheorem*{lemma*}{Lemma}
\newtheorem*{proposition*}{Proposition}
\newtheorem{sublemma}{Lemma}[theorem]
\theoremstyle{definition}
\ifdefined\definition \else \newtheorem{definition}[theorem]{Definition} \fi
\theoremstyle{remark}
\ifdefined\remark \else \newtheorem{remark}[theorem]{Remark} \fi
\ifdefined\remarks \else \newtheorem{remarks}[theorem]{Remarks} \fi
\ifdefined\example \else \newtheorem{example}[theorem]{Example} \fi
\ifdefined\claim \else \newtheorem{claim}[theorem]{Claim} \fi
\ifdefined\acknowledgment \else \newtheorem*{acknowledgment}{Acknowledgment} \fi
\ifdefined\dedication \else \newtheorem*{dedication}{Dedicated} \fi
\ifdefined\case \else \newtheorem*{case}{Case} \fi

\newcounter{my_enumerate_counter}

\newcommand\comment[1]{}

\newcommand\Trm{\mathrm{T}}

\newcommand\Zsf{\mathsf{Z}}

\newcommand\Bfrak{\mathfrak{B}}
\newcommand\Cfrak{\mathfrak{C}}

\newcommand\Lfrak{\mathfrak{L}}
\newcommand\Mfrak{\mathfrak{M}}

\newcommand\Rfrak{\mathfrak{R}}

\newcommand\Ufrak{\mathfrak{U}}

\newcommand\Acal{\mathcal{A}}
\newcommand\Bcal{\mathcal{B}}

\newcommand\Dcal{\mathcal{D}}
\newcommand\Ecal{\mathcal{E}}

\newcommand\Lcal{\mathcal{L}}
\newcommand\Mcal{\mathcal{M}}

\newcommand\Pcal{\mathcal{P}}

\newcommand\Ucal{\mathcal{U}}

\newcommand\Wcal{\mathcal{W}}
\newcommand\Xcal{\mathcal{X}}
\newcommand\Ycal{\mathcal{Y}}
\newcommand\Zcal{\mathcal{Z}}

\renewcommand\Bbb{\mathbb{B}}

\newcommand\Pbb{\mathbb{P}}
\newcommand\Qbb{\mathbb{Q}}

\newcommand\abf{\mathbf{a}}
\newcommand\bbf{\mathbf{b}}
\newcommand\cbf{\mathbf{c}}
\newcommand\dbf{\mathbf{d}}

\newcommand\cof{\operatorname{cof}}

\newcommand{\dom}{\operatorname{dom}}
\newcommand{\ran}{\operatorname{ran}}

\newcommand{\Cov}{\operatorname{Cov}}

\ifdefined\alt \else
  \newcommand{\alt}{\operatorname{alt}}
\fi

\newcommand{\one}{\mathbf{1}}
\newcommand{\zero}{\mathbf{0}}

\newcommand{\forces}{\Vdash}

\newcommand{\compat}{\parallel}
\newcommand{\incompat}{\mathbin\bot}

\newcommand\cat{{}^\smallfrown}

\renewcommand{\diamond}{\diamondsuit}

\newcommand\axiom{\mathsf}

\newcommand\AC{\axiom{AC}}

\newcommand\KP{\axiom{KP}}
\newcommand\KPU{\axiom{KPU}}
\newcommand\ZF{\axiom{ZF}}
\newcommand\ZFC{\axiom{ZFC}}
\newcommand\ZFm{\axiom{ZF}^-}
\newcommand\ZFCm{\axiom{ZFC}^-}

\newcommand\GB{\axiom{GB}}
\newcommand\GBC{\axiom{GBC}}
\newcommand\GBc{\axiom{GBc}}

\newcommand\GBCm{\GBC^-}
\newcommand\GBcm{\GBc^-}
\newcommand\ETR{\axiom{ETR}}

\newcommand\PCA{\Pi_1^1\text{-}\axiom{CA}}
\newcommand\PnCA[1]{\Pi_{#1}^1\text{-}\axiom{CA}}
\newcommand\SnCC[1]{\Sigma_{#1}^1\text{-}\axiom{CC}}
\newcommand\ECC{\axiom{ECC}}

\newcommand\PnCAp[1]{\PnCA{#1} + \SnCC{#1}}
\newcommand\PnCAm[1]{\PnCA{#1}^-}

\newcommand\SkTR[1]{\Sigma_{#1}^1\text{-}\axiom{TR}}
\newcommand\KM{\axiom{KM}}
\newcommand\KMp{\KM^+} 
\newcommand\KMCC{\axiom{KMCC}}
\newcommand\KMm{\KM^-}
\newcommand\KMpm{(\KMp)^-} 
\newcommand\KMCCm{\KMCC^-}
\newcommand\ZFmi{\ZFm_{\mathrm I}}
\newcommand\ZFCmi{\ZFCm_{\mathrm I}}
\newcommand\ZFCmr{\ZFCm_{\mathrm R}}

\newcommand\AD{\axiom{AD}}

\newcommand\PA{\axiom{PA}}
\newcommand\TA{\axiom{TA}}

\newcommand\RCA{\axiom{RCA}}
\newcommand\WKL{\axiom{WKL}}
\newcommand\ACA{\axiom{ACA}}
\newcommand\ATR{\axiom{ATR}}

\newcommand\class{\mathrm}

\newcommand\HOD{\class{HOD}}
\newcommand\Ord{\class{Ord}}

\newcommand\Add{\class{Add}}

\newcommand\wfp{\operatorname{wfp}}

\newcommand{\seq}[1]{\left\langle #1 \right\rangle}

\renewcommand{\epsilon}{\varepsilon}

\newcommand\mand{\textrm{ and }}
\newcommand\mor{\textrm{ or }}

\newcommand\card[1]{\left\lvert #1 \right\rvert}
\newcommand\length{\operatorname{len}}
\newcommand\rest{\upharpoonright}

\newcommand\powerset{\Pcal}
\newcommand\rank{\operatorname{rank}}
\newcommand\tc{\operatorname{TC}}

\newcommand\omegaoneck{{\omega_1^{\mathrm{CK}}}}
\newcommand\Hyp{\operatorname{Hyp}}

\newcommand\supsetend{\supseteq_{\mathsf{end}}}

\newcommand\inv{^{-1}}

\newcommand\prtlfn{ \mathbin{{\vbox{\baselineskip=3pt\lineskiplimit=0pt\hbox{.}\hbox{.}\hbox{.}}}}}

\newcommand{\nats}{\mathbb N}

\newcommand{\reals}{\mathbb R}

\newcommand{\proves}{\vdash}

\newcommand\Th{\operatorname{Th}}

\newcommand\Con{\operatorname{Con}}

\newcommand\Def{\operatorname{Def}}
\ifdefined\Form \else
  \newcommand\Form{\axiom{Form}}
\fi

\newcommand{\impl}{\Rightarrow}
\renewcommand{\iff}{\Leftrightarrow}
\renewcommand{\phi}{\varphi}

\newcommand\rlzn[2]{#1\text{-}\mathfrak{Re}(#2)}
\newcommand\rlznc[2]{#1\text{-}\mathfrak{Re}_{<\omega_1}(#2)}
\newcommand\rlznp[2]{#1\text{-}\mathfrak{Re}_{\mathrm{Pr}}(#2)}
\newcommand\GC{\GBC\text{-}\mathfrak{Re}_{< \omega_1}}
\newcommand\GP{\GBC\text{-}\mathfrak{Re}_{\mathrm{Pr}}}

\newcommand\vin{\mathbin\varepsilon}
\newcommand\oin{\mathord\in}
\newcommand\lol{\multimap}
\newcommand\pen{\operatorname{elts}}

\newcommand\wZFC{\axiom{wZFC}}
\newcommand\wZFCmi{\wZFC^-_{\mathrm I}}
\newcommand\wZFCmr{\wZFC^-_{\mathrm R}}

\newcommand\Tr{\mathrm{Tr}}

\newcommand\ITR{\axiom{ITR}}
\newcommand\itr{\axiom{itr}}

\theoremstyle{plain}
\ifdefined\masterlemma \else \newtheorem{masterlemma}[theorem]{Master Lemma} \fi

\begin{document}

\frontmatter

\begin{titlepage}

\begin{center}

~\vspace{2in}

\textsc{The Structure of Models of Second-order Set Theories} \\[0.5in]
by \\[0.5in]
\textsc{Kameryn J Williams}

\vspace{\fill}
A dissertation submitted to the Graduate Faculty in Mathematics in partial fulfillment of the requirements for the degree of Doctor of Philosophy, The City University of New York. \\[0.25in]
2018

\end{center}

\end{titlepage}

\setcounter{page}{2}

\phantom{}\vspace{\fill}
\begin{center}
\copyright~2018\\
\textsc{Kameryn J Williams}\\
All Rights Reserved\\
\end{center}

\begin{center}
  
\textsc{The Structure of Models of Second-order Set Theories} \\[0.25in]
by \\[0.25in]
\textsc{Kameryn J Williams}

\vspace{0.75in}

This manuscript has been read and accepted by the Graduate Faculty in Mathematics in satisfaction of the dissertation requirement for the degree of Doctor of Philosophy.
\end{center}

\vspace{0.55in}

\begin{tabular}{p{1.75in}p{0.5in}p{3.5in}}
~                                   & & \textbf{Professor Joel David Hamkins}\\
~                                   & & \\
\hrulefill                          & &\hrulefill \\
Date                                & & Chair of Examining Committee\\
~                                   & & \\
~                                   & & \textbf{Professor Ara Basmajian}\\
~                                   & & \\
\hrulefill                          & &\hrulefill \\
Date                                & & Executive Officer\\
\end{tabular}

\vspace{0.5in}

\begin{tabular}{l}
\textbf{Professor Joel David Hamkins} \\
\textbf{Professor Arthur Apter} \\
\textbf{Professor Gunter Fuchs} \\
Supervisory Committee \\
\end{tabular}

\vspace{\fill}
\begin{center}
\textsc{The City University of New York}
\end{center}

\begin{center}
Abstract \\
\textsc{The Structure of Models of Second-order Set Theories} \\
by \\
\textsc{Kameryn J Williams} \\[0.25in]
\end{center}

\vspace{0.25in}

\noindent Advisor: Professor Joel David Hamkins

\vspace{0.25in}

\noindent
This dissertation is a contribution to the project of second-order set theory, which has seen a revival in recent years. The approach is to understand second-order set theory by studying the structure of models of second-order set theories. The main results are the following, organized by chapter. First, I investigate the poset of $T$-realizations of a fixed countable model of $\ZFC$, where $T$ is a reasonable second-order set theory such as $\GBC$ or $\KM$, showing that it has a rich structure. In particular, every countable partial order embeds into this structure. Moreover, we can arrange so that these embedding preserve the existence/nonexistence of upper bounds, at least for finite partial orders. Second I generalize some constructions of Marek and Mostowski from $\KM$ to weaker theories. They showed that every model of $\KM$ plus the Class Collection schema ``unrolls'' to a model of $\ZFCm$ with a largest cardinal. I calculate the theories of the unrolling for a variety of second-order set theories, going as weak as $\GBC + \ETR$. I also show that being $T$-realizable goes down to submodels for a broad selection of second-order set theories $T$. Third, I show that there is a hierarchy of transfinite recursion principles ranging in strength from $\GBC$ to $\KM$. This hierarchy is ordered first by the complexity of the properties allowed in the recursions and second by the allowed heights of the recursions. Fourth, I investigate the question of which second-order set theories have least models. I show that strong theories---such as $\KM$ or $\PCA$---do not have least transitive models while weaker theories---from $\GBC$ to $\GBC + \ETR_\Ord$---do have least transitive models.

\chapter*{Dedication}

\begin{center}

\vspace{2in}

For Margaret and Martin.

\end{center}

\chapter*{Acknowledgments}

First and foremost I must thank Joel. You have been a fantastic advisor throughout my studies and there is no doubt that my growth as a mathematician owes much to your support and guidance. Thank you for the many fruitful conversations, for thoroughly reading through this text, and for the myriad helpful suggestions along the way.

Thank you Arthur and Gunter for serving on my dissertation committee and for reading and giving me comments on this document. Thank you Roman for being on my committee in spirit and for organizing the models of arithmetic group at CUNY. Thank you Kaethe, Miha, Corey, Alex, Ryan, and Eoin. It was great to have other students interested in set theory and to have a student-run seminar in the subject for all of my years at CUNY. Thank you Vika for many conversations and for your hard work in making things run smoothly for the set theory group. And thank you to all the others who make CUNY's logic community so vibrant and such a productive place to be.

Thank you Herbie. You are my best friend and I love you. Thank you for moving across the country with me and for supporting and putting up with me through the joys and trials of graduate school. And finally, thank you mom, dad, Chris, Kalen, Christian, and Lexi.

\tableofcontents
\listoffigures

\mainmatter

\setcounter{chapter}{-1}
\chapter{Introduction}
\chaptermark{Introduction}

\epigraph{\singlespacing Eine Vielheit kann n\"amlich so beschaffen sein, da\ss{} die Annahme eines ``Zusammenseins'' {\em aller} iherer Elemente auf einen Widerspruch f\"uhrt, so da\ss{} es unm\"oglich ist, die Vielheit als eine Einheit, als ``ein fertiges Ding'' aufzufassen. Solche Vielheiten nenne ich {\em absolut unendliche} oder {\em inconsistente Vielheiten}.}{Georg Cantor}

The distinction between set and class can be traced back to Cantor. He distinguished sets from those multiplicities he termed {\em absolutely infinite} or {\em inconsistent}. As the name suggests, inconsistent multiplicities are those which lead to contradiction if taken as a set. For example, if we assume that the collection of ordinals is a set then we can derive the Burali--Forti paradox. 

From a more modern perspective based upon the iterative conception of set we have a clear distinction between sets and classes. A collection is a set if it appears at some stage $\alpha$ in the cumulative hierarchy, while a (proper) class consists of elements unbounded in rank. So sets can be elements of other collections, while classes can never be elements.

In ordinary set theoretic practice in the early twenty-first century classes are treated as mere syntactic sugar; for example, $x \in \Ord$ is an abbreviation for the formula expressing that $x$ is a transitive set linearly ordered by $\in$. But set theory can also be formalized with classes as actual objects, rather than relegating them to a metatheoretic role.

The first axiomatization of second-order set theory---set theory with both sets and classes, sometimes called class theory---is due to von Neumann \cite{neumann1925}. 
His system survives into the modern day as $\GBC$, allowing only predicative definitions of classes.\footnote{But note that the contemporary $\GBC$ is quite different from von Neumann's original axiomatization. Most strikingly, von Neumann's system did not use sets and classes but rather what he termed I-objects and II-objects. I-objects are the sets whereas II-objects are not classes but rather functions (possibly class-sized), with I-II-objects being the set-sized functions. The modern formulation in terms of sets and classes, originally due to Bernays, is much more convenient to work with.}
But $\GBC$ is not the only well-studied axiomatization of second-order set theory. The other major axiomatization $\KM$, which allows impredicative comprehension, was independently proposed by multiple logicians---among them Morse, Quine, and Tarski. See the appendix to \cite{kelley1955} for a popularization of this system. 

In the decades following von Neumann's axiomatization, second-order axiomatizations of set theory saw significant use among set theorists. Perhaps most notably, G\"odel's original presentation \cite{godel1938} of his relative consistency proof for the axiom of choice was in terms of von Neumann's system. But over time the use of second-order systems waned, with the first-order system $\ZFC$ becoming the de facto standard.

In recent years, however, second-order set theory has enjoyed a revived interest. Several mathematicians independently arrived at second-order set theory as the natural arena in which to pursue certain projects. It has seen use in work on the foundations of class forcing \cite{antos2015, HKLNS2016, HKS2016a, HKS2016b, GHHSW2017}, hyperclass forcing \cite{antos-friedman2015}, formalizing the inner model hypothesis \cite{antos-barton-friedman2018}, determinacy for class games \cite{gitman-hamkins2016, hachtman2016}, and in truth theoretic work \cite{fujimoto2012}.

One fact that has emerged is that $\GBC$ and $\KM$ are not the only interesting second-order set theories. For some applications, $\KM$ is not quite strong enough so we need to extend to a stronger system. And for other applications $\GBC$ is too weak while $\KM$ is overkill, so we want to study natural intermediate theories.

This dissertation is a contribution towards this project of second-order set theory. Rather than apply the tools of second-order set theory to some domain, the aim is to study second-order set theories themselves. My approach will be model theoretic, aiming to understand these theories by understanding their models. A better knowledge of the foundations of second-order set theory will then facilitate applications thereof.

The layout of this dissertation is as follows.

Chapter 1 begins with the main theories we will consider, as well as several important classes of models. This is followed by a discussion of how to check whether a collection of classes gives a model of such and such theory, and when the axioms are preserved by forcing. To conclude the chapter, I show that the partial order consisting of $\GBC$-realizations of a fixed countable model of $\ZFC$ has a rich structure. Much of the material is this chapter is already known, but I include it for the sake of giving a complete presentation.

Chapter 2 is dedicated to three constructions, which were originally studied in the context of models of $\KM$ by Marek and Mostowski \cite{marek1973,marek-mostowski1975}. The first of these constructions, which I call the unrolling construction, takes a model of $\KM$ (plus Class Collection) and gives a model of $\ZFCm$ with a largest cardinal, which is inaccessible.  The second construction, I call it the cutting-off construction, takes a model of $\ZFCm$ with a largest cardinal and gives a model of second-order set theory. Together, these two constructions show that $\KM$ (plus Class Collection) is bi-interpretable with a first-order set theory without powerset. The third construction is a version of G\"odel's constructible universe in the classes. Given a model of $\KM$ this gives a smaller model of $\KM$ plus Class Collection with the same ordinals.

I investigate these constructions over a weaker base theory than $\KM$ (plus Class Collection), generalizing Marek and Mostowski's results to weaker theories. In particular, this shows that for the second-order set theories $T$ in which we are interested that being $T$-realizable is closed under taking inner models. I close the chapter with an application of the constructions, showing that the least height of a transitive model of $\GBC + \PnCA k$ is less than the least height of a $\beta$-model of $\GBC + \PnCA k$. This generalizes an analogous result---due to Marek and Mostowski, as the reader may have guessed---about transitive and $\beta$-models of $\KM$.

Chapter 3 is about transfinite recursion principles in second-order set theory. The main result is that there is a hierarchy of theories, ranging in strength from $\GBC$ to $\KM$, given by transfinite recursion principles. This hierarchy is ordered first by the complexity of the properties we can do recursion for and second by the lengths of recursion that can be done.

Chapter 4 investigates the phenomenon of minimal models of second-order set theories. The main result is that strong second-order set theories---e.g. $\KM$ or $\GBC + \PCA$---do not have least transitive models whereas weaker second-order set theories---e.g. $\GBC$ or $\GBC + \ETR_\Ord$---do. Indeed, the results of that chapter show that no countable model of $\ZFC$ can have a least $\KM$-realization (and similarly for other strong theories). Left open is the question of whether $\GBC + \ETR$ has a least transitive model. I show that there is a basis of minimal $(\GBC+\ETR)$-realizations for any $(\GBC+\ETR)$-realizable model $M$ so that if $(M,\Xcal) \models \GBC + \ETR$ then $\Xcal$ sits above precisely one of these basis realizations. I also show that the second-order set theories considered in this dissertation have least $\beta$-models. 

Some of the material in this dissertation also appeared in a paper of mine \cite{williams-min-km} which, at time of writing, is under review. In particular, that paper contains most of the results of chapter 4, parts of chapter 3, and a little bit from the end of chapter 1.

\chapter{A first look at models of second-order set theories}
\chaptermark{A first look}

\epigraph{\singlespacing Wesentlich aber ist, daß auch "zu gro\ss{}e" Mengen Gegenstand dieser Mengenlehre sind, n\"amlich diejenigen II. Dinge, die keine I.II. Dinge sind. Anstatt sie g\"anzlich zu verbieten, werden sie nur f\"ur unf\"ahig erkl\"art Argumente zu sein (sie sind keine I. Dinge!). Zum Vermeiden der Antinomien reicht das aus und ihre Existenz ist f\"ur gewisse Schlußweisen notwendig.}{John von Neumann}

The purpose of this chapter is to introduce the important objects of study for this dissertation and lay out some basic properties thereof. 

I begin by introducing the major second-order set theories under study, followed by some important classes of models for those theories. Next comes a discussion of means for checking whether a structure satisfies the axioms of these theories. In particular, I look at when these axioms are preserved by class forcing. This transitions into some well-known constructions for producing models of weak second-order set theories.

I end the chapter with an in-depth look at the collection of $\GBC$-realizations for a fixed countable model of $\ZFC$. Some of those results can be generalized to stronger theories, and I discuss to what extent this can be done. But a fuller look at the topic for stronger theories is delayed until chapter 4, after we have built up more tools.

Most of the work in this chapter is not new. I have strived to indicate clearly where each theorem originates. At times, however, I have resorted to labeling a result as folklore when its origin is lost in the literature to me.

\section{Dramatis Personae}

In this section I present the main players in the drama. I introduce the main second-order set theories of interest and the important classes of models of these theories. First, let me set up some framework.

The reader may know that there are two main approaches to formalizing second-order set theory. The first is to use a one-sorted theory, where the only objects are classes and sets are those classes which are elements of another class. I will not take that approach. Instead, I will take the two-sorted approach, where there are two types of objects: sets and classes.
The domain of a model $(M,\Xcal)$ has two parts, where $M$ is the {\em first-order part of the model}, with elements of $M$ being the {\em sets}, and $\Xcal$ is the {\em second-order part of the model}, with elements of $\Xcal$ being the {\em classes}. I will suppress writing the membership relation for the model, referring simply to $(M,\Xcal)$. In case we need to refer to the membership relation I will write $\in^{(M,\Xcal)}$.

In this dissertation we will often be interested in different models which have the same sets but different classes. This easily fits to the two-sorted approach, where we can easily talk about models $(M,\Xcal)$ and $(M,\Ycal)$ with the same first-order part. In contrast, the one-sorted approach is awkward here.

I will use $\Lcal_\in$ to refer to the language of set theory, whether second-order or first-order. The only non-logical symbols in the language are for membership. Extensions of this language by adding new symbols will be denoted e.g.\ $\Lcal_\in(A)$.

It should be emphasized that the ``second-order'' in second-order set theory refers to the use of classes, not to the logic. The theories we consider here will all be formalized in first-order logic.\footnote{One could also think of them as formalized in second-order logic with Henkin semantics, but officially we will use first-order logic.} Consider the analogous situation of second-order arithmetic, where two-sorted theories of arithmetic---with numbers and sets of numbers as objects---are formulated in first-order logic.

In the formal language for these theories I will distinguish between variables for sets and variables for classes by using lowercase letters for the former and uppercase letters for the latter. For instance, the formula $\forall X \exists y\ y \in X$ gives the (false) proposition that every class has some set as a member. Say that a formula in the language of set theory is {\em first-order} if it has no class quantifiers (though classes may appear as variables). Those formulae which do have class quantifiers are called {\em second-order}. The first-order formulae are stratified according to the usual L\'evy hierarchy, which will be denoted by $\Sigma^0_n$ and $\Pi^0_n$. There is also a stratification of the second-order formulae. A formula is $\Sigma^1_1$ if it is of the form $\exists X\ \phi(X)$ or is $\Pi^1_1$ if it is of the form $\forall X\ \phi(X)$, where $\phi$ is first-order. This extends upward to $\Sigma^1_n$ and $\Pi^1_n$ in the obvious manner. It will sometimes be convenient to have a name for the first-order and second-order formulae. I will use $\Sigma^0_\omega$, $\Pi^0_\omega$, $\Sigma^1_0$, or $\Pi^1_0$ for the first-order formulae and $\Sigma^1_\omega$ or $\Pi^1_\omega$ for the second-order formulae.

All second-order set theories I consider will include Class Extensionality as an axiom. It will often be convenient to assume that our models are of the form $(M,\Xcal)$ where $\Xcal \subseteq \powerset(M)$ and the set-membership relation is the true $\in$. However, one must be a little careful here. It may be that some of the elements of $M$ are subsets of $M$. If $M$ is transitive and its membership relation is $\mathord{\in} \rest M$, then this is no problem as in this case the set-set and set-class membership relations will cohere for any $\Xcal \subseteq \powerset(M)$ we choose. But it could be that the membership relation of $M$ is some weird thing. Nevertheless, it is always true that our models $(\bar M,\bar \Xcal)$ are isomorphic to a model of the form $(M,\Xcal)$ where $\Xcal \subseteq \powerset(M)$ and the set-membership relation is the true $\in$. We first find $M \cong \bar M$ so that $M \cap \powerset(M) = \emptyset$. Next, using that all our models will satisfy Extensionality for classes, we can realize the classes of (the isomorphic copy of) our model as literal subsets of the model.\footnote{There is a small point here which needs to be addressed. Under this approach the sets and classes are disjoint. So, for example, if we are being really careful we must distinguish between the set of finite ordinals and the class of finite ordinals, as they are different objects. Nevertheless, it will follow from the axioms we use that every set has the same elements as some class. In practice I will not always be careful to distinguish a set from the class it is co-extensive with.}

\subsection{Second-order set theories}

The theories can be roughly grouped into three groups: weak, strong, and medium. I will present them in that order.

\begin{definition}
{\bf G}\"odel--{\bf B}ernays set theory with Global {\bf C}hoice $\GBC$ is axiomatized with the following.\footnote{In the literature one also sees this axiom system called $\mathsf{NBG}$.}
\begin{itemize}
\item $\ZFC$ for sets.
\item Extensionality for classes.
\item Class Replacement---if $F$ is a class function and $a$ is a set then $F'' a$ is a set.
\item Global Choice---there is a class bijection $\Ord \to V$.
\item Elementary Comprehension---if $\phi(x)$ is a first-order formula, possibly with set or class parameters, then $\{ x : \phi(x) \}$ is a class.
\end{itemize}
Dropping Global Choice from the axiomatization gives the theory $\GBc$. The $\mathsf{c}$ reminds one that while Global Choice is lacking, there are still {\bf c}hoice functions for sets. Though it will not be used in this dissertation, $\GB$ is used to refer to $\ZF$ plus Class Extensionality, Class Replacement, and Elementary Comprehension.
\end{definition}

This axiomatization is not parsimonious. In particular, it has infinitely many axioms whereas $\GBC$ is known to be finitely axiomatizable. An advantage of this axiomatization is that it makes immediately apparent the distinction between classes and sets and how $\GBC$ relates to $\ZFC$. It is obvious from this axiomatization that we can obtain a model of $\GBC$ by taking a model of $\ZFC$ (perhaps we require more from the model\footnote{Though we will see later in this chapter that we do not need to require more, at least for countable models. Any countable model of $\ZFC$ can be expanded to a model of $\GBC$. Consequently, $\GBC$ is conservative over $\ZFC$.})
and adding certain classes. Later in this chapter we will see how to verify whether a collection of classes for a model of $\ZFC$ gives a model of $\GBC$.

Note that by Elementary Comprehension for every set $x$ there is a class $X$ which has the same elements. But not all classes are co-extensive with sets. For instance, By Elementary Comprehension there is a class of all sets. But there can be no such set, by a well-known argument of Russell's. A class which is not co-extensive with a set is called a {\em proper class}.

Also note that $\GBc$ proves Separation for classes, i.e.\ that $A \cap b$ is a set for every class $A$ and every set $b$. To see this, let $F$ be the class function which is the identity on $A$ and sends every set not in $A$ to some designated element, say $\emptyset$. By Class Replacement $a = F''b$ is a set. Then either $a$ or $a \setminus \{\emptyset\}$ will be $A \cap b$, depending upon whether $b \subseteq A$ and $\emptyset \in A \cap b$.

Next we look at much stronger theories. The difference in axiomatization may appear slight---allowing impredicative definitions in Comprehension---but the effects are profound.

\begin{definition}
{\bf K}elley--{\bf M}orse set theory $\KM$ is axiomatized with the axioms of $\GBC$ plus the full Comprehension schema. Instances of this schema assert that $\{ x : \phi(x) \}$ is a set for any formula $\phi$, possibly with class quantifiers and set or class parameters.\footnote{In the literature $\KM$ has many other names---I have seen $\mathsf{MK}$ for Morse--Kelley (e.g.\ \cite{antos-friedman2015}), $\mathsf{MT}$ for Morse--Tarski (e.g.\ \cite{chuaqui1980}), $\mathsf{MKT}$ for Morse--Kelley--Tarski (e.g.\ \cite{chuaqui1981}), and $\mathsf{QM}$ for Quine--Morse (e.g.\ \cite{diener1983}). If Monty Python did sketches about set theory instead of breakfast \cite{python:spam} no doubt we would also have Morse--Kelley--Morse--Morse--Tarski--Morse, or $\mathsf{MKMMTM}$.}
\end{definition}

We can strengthen $\KM$ by adding the Class Collection schema. For some purposes, $\KM$ is not quite enough and we need the extra strength of this schema.

\begin{definition}
The theory $\KMCC$ is obtained from $\KM$ by adding the {\bf C}lass {\bf C}ollection axiom schema.\footnote{Continuing a theme of previous footnotes, both $\KMCC$ and Class Collection have different names in the literature. Antos and Friedman \cite{antos-friedman2015} call them $\mathsf{MK}^*$ and Class Bounding while Gitman and Hamkins \cite{gitman-hamkins2015} call them $\KMp$ and Class Choice. I myself previously have used $\KMp$ \cite{williams-min-km}, but in this dissertation I will consider second-order set theories formulated without the axiom of Powerset. Following the standard of referring to $\ZFC - \text{Powerset}$ as $\ZFCm$ I will call these theories $\GBCm$, $\KMCCm$, and so forth. Using $\KMp$ would lead to the infelicitous $\KMpm$. So $\KMCC$ it is.}
 Informally, this schema asserts that if for every set there is a class satisfying some property, then there is a coded hyperclass\footnote{A {\em hyperclass} is a collection of classes. A hyperclass $\Acal$ is {\em coded} if there is a class $C$ so that $\Acal = \{ (C)_x : x \in V \}$ where $(C)_x = \{ y : (x,y) \in C \}$ is the $x$-th slice of $C$. Officially, of course, hyperclasses are not objects in the models and any talk of such is a paraphrase, similar to the usage of classes in first-order set theory.}
consisting of witnesses for each set. Formally, let $\phi(x,Y)$ be a formula, possibly with parameters. The instance of Class Collection for $\phi$ asserts
\[
[\forall x \exists Y \ \phi(x,Y)] \impl [\exists C \forall x\ \exists y\ \phi(x,(C)_y)]
\]
where $(C)_y = \{ z : (y,z) \in C \}$ is the $y$-th slice of $C$.
\end{definition}

Observe that under Global Choice, Class Collection is equivalent to the schema with instances
\[
[\forall x \exists Y \ \phi(x,Y)] \impl [\exists C \forall x\ \phi(x,(C)_x)],
\]
that is where $x$ is the the index of the slice in $C$ witnessing the property for $x$. This version of the schema has the flavor of a choice principle, hence it sometimes being called Class Choice.

The set theorist who does not work with second-order set theories may wonder why we would want to work with something even stronger than $\KM$. To her I have two responses. First, $\KM$ behaves badly with some constructions. For instance, set theorists like to take ultrapowers of the universe using some measure. In order for \L{}o\'s's theorem to be satisfied for the full second-order language, we need Class Collection. Gitman and Hamkins showed that $\KM$ alone does not suffice \cite{gitman-hamkins2015}.
Second, the natural models of $\KM$ are actually models of $\KMCC$. If $\kappa$ is inaccessible then $(V_\kappa,V_{\kappa+1})$ is a model of $\KMCC$. This may not satisfy the skeptic who is worried about a jump in consistency strength, but we will see in chapter 2 that the skeptic need not worry, as $\KMCC$ does not exceed $\KM$ in consistency strength.

It is immediate that $\KM$ is stronger than $\GBC$. Indeed, $\KM$ proves the existence of $\Sigma^1_k$ truth predicates for every (standard) $k$. Therefore, $\KM$ proves $\Con(\ZFC)$ so once we see that $\GBC$ and $\ZFC$ are equiconsistent we will see that the separation is also in terms of consistency strength. However, there is a significant gap between the two theories. We can weaken Comprehension to get intermediate theories, though the following are still grouped among the strong theories.

\begin{definition}
Let $k$ be a (standard) natural number. The  $\Pi^1_k$-Comprehension Schema $\PnCA k$ is the restriction of the Comprehension schema to $\Pi^1_k$-formulae. Note that, over $\GBcm$, this is equivalent to restricting Comprehension to $\Sigma^1_k$-formulae.

Recall that $\Pi^1_\omega$ refers to the second-order formulae, of any complexity. It will sometimes be convenient to use $\PnCA{\omega}$ or $\Pi^1_\omega$-Comprehension as a synonym for the full second-order Comprehension schema.
\end{definition}

Observe that $\PnCA 0$ is Elementary Comprehension. So we are really only interested in the case where $k > 0$.

We get that $\GBC + \PCA$ proves $\Con(\GBC)$ and $\GBC + \PnCA{k}$ proves $\Con(\GBC + \PnCA n)$ for $n < k$, as the $\Sigma^1_n$ truth predicate can be defined via a $\Sigma^1_{n+1}$-formula. So there is a hierarchy of theories between $\GBC$ and $\KM$, increasing in consistency strength.

It is also useful to consider fragments of Class Collection.

\begin{definition}
Let $k$ be a (standard) natural number. The {\em $\Sigma^1_k$-Class Collection axiom schema}, denoted by $\SnCC k$, is the restriction of the Class Collection schema to $\Sigma^1_k$-formulae. Elementary Class Collection $\ECC$ is another name for $\SnCC 0$. 
\end{definition}

In chapter 2 we will see that $\GBC + \PnCAp k$ does not exceed $\GBC + \PnCA k$ in consistency strength. See corollary \ref{cor2:get-a-plus}.

\begin{observation}
Over $\GBCm$ we have that $\SnCC k$ implies $\Pi^1_k$-Comprehension.
\end{observation}

\begin{proof}
Let $\phi(x)$ be a $\Sigma^1_k$-formula, possibly with (suppressed) parameters. Apply the instance of Class Collection to the formula
\[
(\phi(x) \land Y = \{x\}) \lor (\neg \phi(x) \land Y = \emptyset)
\]
to get a class $C$ so that $(C)_x = \{x\}$ if $\phi(x)$ and $(C)_x = \emptyset$ otherwise. Then, $\{ x : \phi(x) \}  = \{ x : (C)_x \ne \emptyset \} \in \Xcal$, as desired.
\end{proof}

On the other hand, Gitman and Hamkins \cite{gitman-hamkins2015} produced a model of $\KM$ which does not satisfy even $\Sigma^1_0$-Class Collection. 

Between the weak $\GBC$ and the strong $\PnCA k$ we have the medium theories.

\begin{definition}
We define the {\bf E}lementary {\bf T}ransfinite {\bf R}ecursion schema $\ETR$. This schema asserts that recursions of first-order properties along well-founded relations have solutions. Formally, let $\phi(x,Y,A)$ be a first-order formula, possibly with a class parameter $A$ and let $R$ be a well-founded class relation. Denote by $<_R$ the transitive closure of $R$. The instance of $\ETR$ for $\phi$ and $R$ asserts that there is a class $S \subseteq \dom R \times V$ which satisfies
\[
(S)_r = \{ x : \phi(x,S \rest r, A) \}
\]
for all $r \in \dom R$. Here, $(S)_r = \{ x : (r,x) \in S \}$ denotes the $r$-th slice of $S$ and
\[
S \rest r = S \cap \left(\{ r' \in \dom R : r' <_R r\} \times V\right)
\]
is the partial solution below $r$.
\end{definition}

One example of an elementary recursion is the Tarskian definition of a (first-order) truth predicate. Thus, $\GBC + \ETR$ proves $\Con(\ZFC)$ and thereby exceeds $\GBC$ in consistency strength. On the other hand, $\GBC + \PCA$ proves $\Con(\GBC + \ETR)$ (see \cite{sato2014}) so $\ETR$ sits below the strong second-order set theories.

It is equivalent, over $\GBC$, to formulate $\ETR$ for recursions over well-founded relations, well-founded partial orders, or well-founded tree orders. See \cite[lemma 7]{gitman-hamkins2016}.

We get fragments of $\ETR$ by restricting the length of recursions.

\begin{definition}
Let $\Gamma$ be a class well-order. Then $\ETR_\Gamma$ is the Elementary Transfinite Recursion schema restricted to well-orders of length $\le \Gamma$.
\end{definition}

There is a subtlety here. Namely, the issue is whether $\ETR_\Gamma$ can be expressed as a theory in the language $\Lcal_\in$ of set theory. If $\Gamma$ is a definable well-order, say $\Gamma = \omega$ or $\Gamma = \Ord$, then this can be done in the obvious way. Different models of set theory may disagree on what $\Ord$ is, but it is sensible to ask whether they satisfy $\ETR_\Ord$. 

But we will also be interested in the case where $\Gamma$ is a specific well-order, possibly undefinable. To be more precise, consider a model $(M,\Xcal)$ of second-order set theory with $\Gamma \in \Xcal$ a well-order. We can then ask whether $(M,\Xcal) \models \ETR_\Gamma$. This may not expressible as a theory in the language $\Lcal_\in$ of set theory, but because $\Gamma \in \Xcal$ we can use it as a parameter to define $\ETR_\Gamma$ in the expanded language $\Lcal_\in(\Gamma)$. 

It will be clear from context which of the two meanings is had in mind, so I will refer to both as simply $\ETR_\Gamma$. 

Let me give an example to illustrate where the distinction matters. Take countable $(M,\Xcal) \models \GBC + \ETR_\gamma$ where $\gamma = \omega_1^M$. Assume that $(M,\Xcal) \not \models \ETR_{\gamma \cdot \omega}$. We will see in chapter 3 that this assumption can be made without loss---if $(M,\Xcal)$ does satisfy $\ETR_{\gamma \cdot \omega}$ then we can throw out classes to get $\bar \Xcal$ so that $(M,\bar \Xcal)$ satisfies $\GBC + \ETR_\gamma$ but does not satisfy $\ETR_{\gamma \cdot \omega}$. Let $g \subseteq M$ be generic over $(M,\Xcal)$ for the forcing to collapse $\omega_1$ to be countable. It is not difficult to check that $(M,\Xcal)[g] \models \GBC + \ETR_\gamma$ but $(M,\Xcal)[g] \not \models \ETR_{\gamma \cdot \omega}$. So $(M,\Xcal)[g]$ will not be a model of the $\Lcal_\in$-theory $\ETR_{\omega_1}$, even though it is a (set) forcing extension of a model of $\ETR_{\omega_1}$. 

Another issue with expressing $\ETR_\Gamma$ as an $\Lcal_\in$-theory is that we can have definitions for a well-order $\Gamma$ which are highly non-absolute. For instance, suppose $\Gamma$ is defined as ``if $V = L$ then $\Gamma = \Ord$ and otherwise $\Gamma = \omega_1$''. Then there is a model of $\GBC + \ETR_\Gamma$ whose $L$ is not a model of $\GBC + \ETR_\Gamma$. On the other hand, as we will see in chapter 3, $\GBC + \ETR_\Gamma$ as an $\Lcal_\in(\Gamma)$-theory does go down to inner models.\footnote{There is a technical caveat here. Namely, $\Gamma$ must be sufficiently nice over the inner model to avoid pathologies such as $\Gamma \subseteq L$ which codes $0^\sharp$. See theorem \ref{thm3:etr-gamma-inner-model} for details.}

The reader who is familiar with reverse mathematics may see an analogy to second-order arithmetic. Namely, three of these theories line up with the strongest three of the ``big five'' subsystems of second-order arithmetic: $\GBC$ is analogous to $\ACA_0$, $\GBC + \ETR$ is analogous to $\ATR_0$ and $\GBC + \PCA$ is analogous to the subsystem of second-order arithmetic which is referred to as $\PCA_0$. At the highest level, $\KM$ is analogous to $\Zsf_2$, full second-order arithmetic. This analogy can be useful to keep in mind. However, the reader should beware that results from arithmetic do not always generalize to set theory. For example, Simpson proved that there is no smallest $\beta$-model of $\ATR_0$---see \cite{simpson:book} for a proof. But there is a smallest $\beta$-model of $\GBC + \ETR$, as we will see in chapter 4. (For the reader who does not know what a $\beta$-model is, we will get to that later in this section). 

We will briefly return to this analogy at the end of chapter 4, after we have seen enough theorems about models of second-order set theories to satisfactorily explore its limits.

\subsection{Doing without Powerset} \label{subsec1:sans-powerset}

All of the second-order set theories considered so far include that the sets satisfy $\ZFC$. But we can ask for less out of the first-order part. We get variants on all of the above theories by dropping the requirement that the first-order part satisfy Powerset.

\begin{definition}
The first-order set theory $\ZFCm$ is axiomatized by Extensionality, Pairing, Union, Infinity, Foundation, Choice, Separation, and Collection. 
\end{definition}

The reader should be warned that in the absence of Powerset that Collection is stronger than Replacement \cite{zarach1996} and thus we do not want to use Replacement to axiomatize $\ZFCm$, as the resulting theory is badly behaved; see \cite{gitman-hamkins-johnstone2016} for some discussion of how this theory misbehaves.

We get ``minus versions'' of all the above-defined second-order set theories by dropping the requirement that the first-order part satisfy Powerset. I will give one definition in full and leave it to the reader to fill in the pattern for the others. The way to think of them is that, for example, $\KMm$ is $\KM - \text{Powerset}$. (But one should keep in mind the Collection versus Replacement issue.)

\begin{definition}
The second-order set theory $\GBCm$ is axiomatized by the following.
\begin{itemize}
\item $\ZFCm$ for sets.
\item Extensionality for classes.
\item Class Replacement.
\item Global Choice, in the form ``there is a bijection $\Ord \to V$''.
\item Elementary Comprehension.
\end{itemize}
\end{definition}

In the absence of Powerset, the various equivalent forms of Global Choice are no longer equivalent. (See section \ref{sec1:preserving-axioms} for a proof.) The strongest is the assertion that there is a bijection from $\Ord$ to $V$, or equivalently, that there is a global well-order of ordertype $\Ord$. I adopt this strongest version as the official form of Global Choice for $\GBCm$, though at times we could get away with less. 

Models of, say $\KMCCm$ are not hard to come by. Indeed, $\KMCCm$ is much weaker than $\ZFC$ in consistency strength. In a model of $\ZFC$ if $\kappa$ is a regular uncountable cardinal then $(H_\kappa,\powerset(H_{\kappa})) \models \KMCCm$. A special case of particular interest is that of the hereditarily countable sets: $(H_{\omega_1},\powerset(H_{\omega_1})$ is a model of $\KMCCm$ $+$ every set is countable.

\subsection{Models of second-order set theory}

An important theme of this work is the following: given a fixed model $M$ of first-order set theory what can be said about possible second-order parts that can be put on $M$ to make a model of some second-order set theory? It will be convenient to have a name for these possible collections of classes.

\begin{definition}
Let $M$ be a model of first-order set theory and $T$ be some second-order set theory. A {\em $T$-realization for $M$} is a set $\Xcal \subseteq \powerset(M)$ so that $(M,\Xcal) \models T$. If $M$ has a $T$ realization then we say $M$ is {\em $T$-realizable}.
\end{definition}

Many properties of first-order models can also be had by second-order models, via the exact same definition. For instance, $(M,\Xcal)$ is {\em $\omega$-standard} (synonymously, is an {\em $\omega$-model}) if $\omega^M$ is well-founded. One important property is transitivity.

\begin{definition}
A model $(M,\Xcal)$ of second-order set theory is {\em transitive} if its membership relations are the true $\in$. This is equivalent to requiring that $M$ is transitive, due to our convention of only considering models so that $\Xcal \subseteq \powerset(M)$. 
\end{definition}

It is well-known that transitive models of $\ZFC$ are correct about well-foundedness: if transitive $M \models \ZFC$ thinks that $R \in M$ is a well-founded relation then $R$ really is well-founded. (Indeed, the same is true for much weaker theories, e.g.\ $\ZFCm$.) This does not hold for transitive models of second-order set theories. While they will be correct about whether set relations are well-founded they can be wrong about whether a class relation is well-founded.\footnote{For the reader who was previously unaware of this folklore result, in chapter 4 we will construct, as a by-product toward other goals, transitive models which are wrong about well-foundedness.}
Models of second-order set theory which are correct about which of their class relations are well-founded are of special interest.

\begin{definition}
A model $(M,\Xcal)$ of second-order set theory is a {\em $\beta$-model} if its membership relations are well-founded and it is correct about well-foundedness. That is, if $R \in \Xcal$ is a relation which $(M,\Xcal)$ thinks is well-founded then $R$ really is well-founded. (The inverse direction, that if $R$ is well-founded then $(M,\Xcal)$ thinks $R$ is well-founded is always true for well-founded models by downward absoluteness.)
\end{definition}

Observe that every $\beta$-model is isomorphic to a transitive model, so we can usually assume without loss that a $\beta$-model is transitive.

The following observation shows that the distinction between $\beta$-model and transitive model does not arise a certain class of models, which includes many natural models considered by set theorists.

\begin{observation} \label{obs1:unc-cof-beta}
Suppose $V_\alpha \models \ZFCm$ is transitive with $\cof \alpha > \omega$. Then $(V_\alpha,\Xcal)$, equipped with the true membership relation, is a $\beta$-model for any $\Xcal \subseteq \powerset(V_\alpha)$.
\end{observation}

\begin{proof}
Suppose $R \in \Xcal$ is ill-founded (i.e. from the perspective of $V$). But witnesses to ill-foundedness are countable sequences and $V_\alpha$ is closed under countable sequences. So $(V_\alpha,\Xcal)$ thinks that $R$ is ill-founded. Since $R$ was arbitrary, $(V_\alpha,\Xcal)$ is correct about well-foundedness.
\end{proof}

\begin{figure}
\begin{center}
\begin{tikzpicture}

\draw[rounded corners] (-.5,-.5) rectangle (14,4.5)
                       (0,0) rectangle (11,4)
                       (.5,.5) rectangle (8,3.5)
                       (1,1) rectangle (5,3);
\draw (2.5,2) node {$\beta$-models}
      (6.25,2) node[text width = 1.8cm] {transitive models}
      (9.75,2) node {$\omega$-models}
      (12.5,2) node {all models};
\end{tikzpicture}
\end{center}

\caption{Some classes of models of second-order set theory.}
\label{fig1:hier-of-models}
\end{figure}
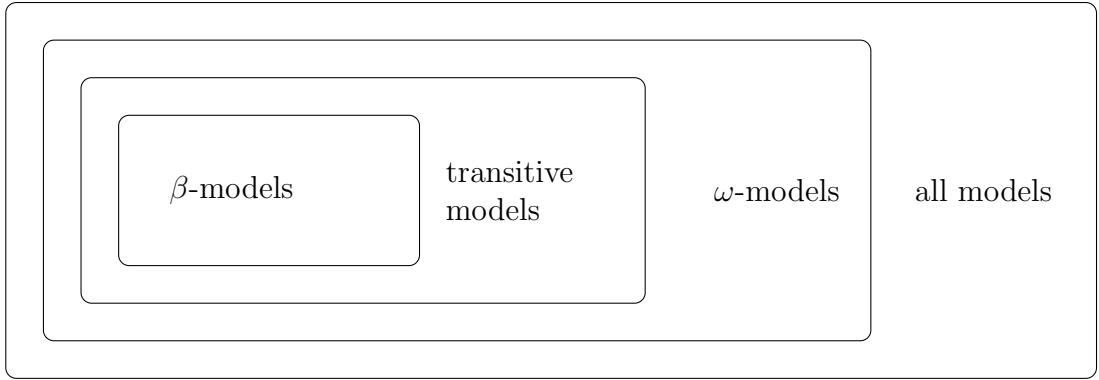

We will also be interested in various ways models may be contained within each other. The basic notion is that of a submodel. This is a familiar concept, but let me give a definition specialized to the context of second-order set theory.

\begin{definition}
Let $(M,\Xcal)$ and $(N,\Ycal)$ be models of second-order set theory. Say that $(M,\Xcal)$ is a {\em submodel} of $(N,\Ycal)$, written $(M,\Xcal) \subseteq (N,\Ycal)$, if $M \subseteq N$, $\Xcal \subseteq \Ycal$, and the membership relation for $(M,\Xcal)$ is the restriction of $\in^{(N,\Ycal)}$ to $(M,\Xcal)$. 
\end{definition}

This definition can be strengthened in various ways. 

\begin{definition}
Let $(M,\Xcal)$ and $(N,\Ycal)$ be models of second-order set theory.
\begin{itemize}
\item Say that $(M,\Xcal)$ is a {\em $V$-submodel} of $(N,\Ycal)$ if $M = N$ and $\Xcal \subseteq \Ycal$. The name is because in this case $V^{(M,\Xcal)} = V^{(N,\Ycal)}$.
\item Say that $(M,\Xcal)$ is an {\em $\Ord$-submodel} of $(N,\Ycal)$ if $(M,\Xcal) \subseteq (N,\Ycal)$ and $\Ord^M = \Ord^N$. The name is because in this case $\Ord^{(M,\Xcal)} = \Ord^{(N,\Ycal)}$. 
\item Say that $(M,\Xcal)$ is an {\em inner model} of $(N,\Ycal)$ if $(M,\Xcal)$ is an $\Ord$-submodel of $(N,\Ycal)$ and $\Xcal$ is definable over $(N,\Ycal)$, possibly via a second-order formula using parameters. Note that it is automatic that $M \in \Ycal$, since $M \in \Xcal \subseteq \Ycal$.
\item Let $(M,\Xcal)$ be a $V$-submodel of $(M,\Ycal)$. Say that $(M,\Xcal)$ is a {\em coded $V$-submodel} of $(M,\Ycal)$ if $\Xcal$ is coded in $\Ycal$. That is, there is a single class $C \in \Ycal$ so that $\Xcal = \{ (C)_x : x \in M \}$.
\end{itemize}
\end{definition}

The reader may find examples to be helpful. Suppose $\kappa$ is inaccessible and that $L_\kappa \ne V_\kappa$. Then $(V_\kappa, \Def(V_\kappa)) \models \GBc$ is a $V$-submodel of $(V_\kappa, \powerset(V_\kappa)) \models \KMCC$ and $(L_\kappa, \Def(L_\kappa)) \models \GBC$ is an $\Ord$-submodel of $(V_\kappa, \powerset(V_\kappa))$. In fact, they are both inner models because $\powerset(V_\kappa)$ contains truth predicates for all $A \subseteq V_\kappa$ and can thus uniformly pick out which classes are in $\Def(A)$.

\section{Verifying the axioms} \label{sec1:preserving-axioms}

Many times in this dissertation we will find ourselves in the following situation. We have some first-order model $M$ of set theory and some $\Xcal \subseteq \powerset(M)$ a collection of classes over $M$. We will want to be able to say something about the theory of $(M,\Xcal)$. In this section I present some basic tools one can use in this situation, focusing here on the axioms of $\GBC$.

Let us begin with the most obvious of observations.

\begin{observation}
A second-order model of set theory $(M,\Xcal)$ with $\Xcal \subseteq \powerset(M)$ and the true $\in$ for its set-class membership relation always satisfies Extensionality for classes. \qed
\end{observation}

To see that $(M,\Xcal) \models \GBc$ satisfies Global Choice one just has to see that $\Xcal$ contains a bijection $\Ord \to V$. As with the ordinary axiom of choice there are several equivalent forms.

\begin{fact}
Let $(M,\Xcal) \models \GBc$ be a second-order model of set theory. The following are equivalent.
\begin{enumerate}
\item $\Xcal$ contains a bijection $\Ord \to V$.
\item $\Xcal$ contains a global choice function, that is a class function $F$ whose domain is the class of nonempty sets so that $F(x) \in x$ for all $x$.
\item $\Xcal$ contains a global well-order, that is a well-order whose domain is the entire universe of sets.
\item $\Xcal$ contains a global well-order of ordertype $\Ord$.
\end{enumerate}
\end{fact}

\begin{proof}
Both $(1 \impl 2)$ and $(4 \impl 1)$ are obvious. That leaves only $(2 \impl 3)$ and $(3 \impl 4)$ to check.

$(2 \impl 3)$ Using the global choice function $F$ we can well-order the $V_\alpha$'s in a coherent fashion. Suppose we have already defined a well-order $<_\alpha$ of $V_\alpha$. Then define a well-order $<_{\alpha+1}$ of $V_{\alpha+1}$ which extends $<_\alpha$ by using the global choice function: $<_{\alpha+1}$ is simply $F(w)$, where $w$ is the set of well-orders of $V_{\alpha+1}$ which extend $<_\alpha$. Then $\bigcup_\alpha \mathord{<_\alpha}$ is a global well-order (of ordertype $\Ord$, in fact).

$(3 \impl 4)$ Let $<^*$ be a global well-order. We define a new global well-order $<^\dagger$ as:
\[
x <^\dagger y \quad \text{iff} \quad \rank x < \rank y \mor (\rank x = \rank y \mand x <^* y).
\]
Then $<^\dagger$ has ordertype $\Ord$.
\end{proof}

As was mentioned in subsection \ref{subsec1:sans-powerset}, these are not all equivalent in the absence of Powerset. We officially adopted the strongest form of Global Choice for the powerset-free context, whose equivalent forms are ``there is a bijection $\Ord \to V$'' and ``there is a global well-order of ordertype $\Ord$''. Let us quickly see that the other forms of Global Choice are weaker in this context.

\begin{fact}\ 
\begin{enumerate}
\item Over $\KMCCm \mathbin-$ Global Choice,\footnote{Note that this theory includes the well-ordering theorem for sets, which is necessary to have a global well-order at all.}
the existence of a global well-order does not imply the existence of a global well-order of ordertype $\Ord$.
\item (Shapiro \cite[theorem 5.4]{shapiro:book}) Assume $\AD_\reals$, asserting the determinacy of every two-player, perfect information game where the two players play reals, is consistent with $\ZF$.\footnote{See \cite{gassner1994} for a proof which does not need this consistency assumption.}
Then over $\KMCCm \mathbin-$ Global Choice, the existence of a global choice function does not imply the existence of a global well-order.\footnote{The context for Shapiro's result here is second-order logic, as is the context for Ga\ss{}ner's paper cited in the previous footnote. This is a reformulation of his result into the context of second-order set theory.}
\end{enumerate}
\end{fact}

\begin{proof}[Proof Sketch]
$(1)$ Force if necessary to get the continuum to have size $\aleph_2$. Then one can check that $(H_{\omega_1},\powerset(H_{\omega_1}))$ is a model of all the axioms of $\KMCCm$ except Global Choice. Easily, it has a global well-order (externally seen to have ordertype $\omega_2$) but has no global well-order of ordertype $\omega_1 = \Ord^{H_{\omega_1}}$. 

$(2)$ Assume $\ZF + \AD_\reals$. Consider the model $(M,\Xcal) = (H_{\omega_1}, \powerset(H_{\omega_1}))$. One can check it satisfies all the axioms of $\KMCCm \mathbin{-}$ except Global Choice. Observe that $\Xcal$ cannot contain a global well-order, as that would imply there is a well-order of $\reals$ in the ambient universe, contradicting $\AD_\reals$. But $\Xcal$ does have a global choice function, which arises from a winning strategy for the following game: Player I plays a real which codes a nonempty hereditarily countable set. Player II responds by playing a real, then the game ends. Player II wins if her real codes an element of the set coded by player I's real and otherwise player I wins. This game is determined by $\AD_\reals$ and it is clear that player I could not possibly have a winning strategy. So player II has a winning strategy from which we can extract a global choice function for $M$.
\end{proof}

Note that $(2)$ requires far from the full strength of $\AD_\reals$, only needing that one can have clopen determinacy for games played with reals while not having a well-order of the reals. The reader who wishes to know the exact strength needed for $(2)$ is welcome to do that work herself.

Next let us see how to check whether our classes satisfy Elementary Comprehension.

\begin{definition}
Let $M$ be a model of set theory with $A \subseteq M$. Then $\Def(M;A)$ is the collection of classes of $M$ definable from $A$, possibly with set parameters. Formally,
\[
X \in \Def(M;A) \iff X = \{ x \in M : (M,A) \models \phi(x,p) \} \text{ for some {\em first-order} } \phi \mand p \in M.
\]
More generally, let $\Xcal$ be a collection of classes from $M$ and $A_i$, for $i$ in some index set $I$, be classes of $M$. Then $\Def(M;\Xcal, A_i : i \in I)$ is the collection of classes of $M$ definable from (finitely many) classes from $\Xcal \cup \{A_i : i \in I\}$.\footnote{Two remarks are in order. First, this definition is ambiguous, as a set can be both a subset of $M$ and also a collection of subsets of $M$. This happens if, for instance, $M$ is transitive and $\Xcal \subseteq M$. But this will not arise in practice and we will sacrifice complete unambiguity in favor of readable notation.
Second, while we could generalize the definition further and allow multiple collections $\Xcal_i$ of classes of $M$, this is not needed for this work. We will be interested in $\Def(M;\Xcal,A)$ when $(M,\Xcal)$ is a model of second-order set theory.}

I will write $\Def(M)$ to refer to $\Def(M;\emptyset)$.
\end{definition}

\begin{observation}
Let $(M,\Xcal)$ be a second-order model of set theory. Then $(M,\Xcal)$ satisfies Elementary Comprehension if and only if $\Xcal$ is closed under first-order definability---i.e.\ for any $A_0, A_1, \ldots, A_n \in \Xcal$ we have $\Def(M;A_0, A_1, \ldots, A_n) \subseteq \Xcal$.
\end{observation}

\begin{proof}
$(\Rightarrow)$ If $B \in \Def(M;A_0, A_1, \ldots, A_n)$ then $B$ was defined from the $A_i$'s by some first-order formula. But then Elementary Comprehension yields that $B$ is a class.

$(\Leftarrow)$ Fix class $A_0, A_1, \ldots, A_n$ and a formula $\phi(x,A_0, A_1, \ldots A_n)$. Then
\[
\{ x : \phi(x,A_0, A_1, \ldots, A_n) \}
\]
is a class because it is definable from the $A_i$'s.
\end{proof}

In general Class Replacement can be tricky to check. Nevertheless, there are some circumstances where it is trivial.

\begin{observation}
Suppose $(M,\Xcal) \models \GBcm$. Let $(N,\Ycal)$ be an $\Ord$-submodel of $(M,\Xcal)$ which satisfies all the axioms of $\GBCm$ except possibly Class Replacement. Then in fact $(N,\Ycal)$ satisfies Class Replacement.
\end{observation}

\begin{proof}
Suppose $F \in \Ycal$ is a class function and $a \in N$ is a set so that $F''a \not \in N$. By Global Choice in $(N,\Ycal)$, there is a bijection in $\Ycal$ between $F''a \in \Ycal$ and $\Ord^N$. So the failure of Class Replacement for $F$ and $a$ gives a map in $\Ycal$ from $\alpha \in \Ord^N$ to $\Ord^N$. But this same map must be in $\Xcal$, contradicting that $(M,\Xcal)$ satisfies Class Replacement.
\end{proof}

This argument does not need Global Choice. It is enough that $\Ycal$ contains a $\subseteq$-increasing sequence $\seq{n_\alpha : \alpha \in \Ord^N}$ of sets from $N$ so that $\bigcup_\alpha n_\alpha = N$. This allows $(N,\Ycal)$ to define a ranking function relative to this $n_\alpha$-hierarchy, and from that get a map from $\alpha$ to $\Ord^N$. In particular, the argument goes through if $N$ satisfies Powerset, since then it has the $V_\alpha$-hierarchy.

\begin{observation}
Suppose $(M,\Xcal) \models \GBcm$. Let $(N,\Ycal)$  be an $\Ord$-submodel of $(M,\Xcal)$ satisfying all the axioms of $\GBc$ except possibly Class Replacement. Then in fact $(N,\Ycal)$ satisfies Class Replacement. \qed
\end{observation}

Together these observations give us the tools to check whether $\Xcal \subseteq \powerset(M)$ is a $\GBC$-realization (or $\GBCm$-realization) for $M$. We do not have such nice tools for stronger theories. Nevertheless, something can be said. For the medium theories, in chapter 3 we will see that Elementary Transfinite Recursion is equivalent to the existence of certain classes, namely iterated truth predicates. A similar result will hold for $\ETR_\Gamma$.

For stronger forms of Comprehension, it is true that $(M,\Xcal)$ satisfying $\Pi^1_k$-Comprehension is equivalent to $\Xcal$ being closed under $\Pi^1_k$-definability. But this means, of course, definability over $(M,\Xcal)$. So in practice this characterization is not useful and we want other tools. 

One specific situation of interest is when our model arises as a forcing extension of a structure already satisfying a strong form of Comprehension. In this case we can say something about whether our model also satisfies Comprehension. More generally, we can ask about the preservation of the axioms under class forcing, to which we now turn.

\section{Preserving the axioms}

There is more than one approach to formalize class forcing. I will take the following. We work over a model $(M,\Xcal)$. A {\em forcing notion} $\Pbb \in \Xcal$ is a separative partial order with a maximum element $\one$. If $p \le q$ we say that $p$ is {\em stronger than} $q$. Two conditions $p$ and $q$ are {\em compatible}, denoted $p \compat q$, if there is $r \le p,q$. Otherwise, $p$ and $q$ are {\em incompatible}, denoted $p \incompat q$.

Given $\Pbb$ we can define the collection of $\Pbb$-names. These are sets or classes whose elements are of the form $(\tau, p)$ where $\tau$ is a $\Pbb$-name and $p \in \Pbb$. This {\it prima facie} circular definition is actually a recursion on ranks. The convention here will be to use capital letters such as $\Sigma$ for proper class $\Pbb$-names and lowercase letters such as $\sigma$ for set $\Pbb$-names. In case I want to refer to either, I will use lowercase letters.

The forcing relation $\forces$ is defined recursively, via the following schema.

\begin{definition}
Let $p \in \Pbb$ and $\sigma,\tau,\ldots$ be $\Pbb$-names. Unless otherwise indicated, they may be either set names or class names. A forcing relation $\mathord\forces = \mathord\forces_\Pbb$ for $\Pbb$ is a relation between $p \in \Pbb$ and formulae in the forcing language which satisfies the following recursive schema on its domain.
\begin{itemize}
\item $p \forces \sigma \in \tau$ if and only if there are densely many $q \le p$ so that there is $(\rho,r) \in \tau$ with $q \le r$ and $q \forces \sigma = \rho$;
\item $p \forces \sigma \subseteq \tau$ if and only if for all $(\rho,r) \in \sigma$ and all $q \le p,r$ we have $q \forces p \in \tau$;
\item $p \forces \sigma = \tau$ if and only if $p \forces \sigma \subseteq \tau$ and $p \forces \tau \subseteq \sigma$;
\item $p \forces \phi \land \psi$ if and only if $p \forces \phi$ and $p \forces \psi$;
\item $p \forces \neg \phi$ if and only if no $q \le p$ forces $\phi$;
\item $p \forces \forall x \phi(x)$ if and only if $p \forces \phi(\sigma)$ for every set $\Pbb$-name $\sigma$; and
\item $p \forces \forall X \phi(X)$ if and only if $p \forces \phi(\dot A)$ for every class $\Pbb$-name $\dot A$.
\end{itemize}

If $\Phi$ is a collection of formulae say that {\em $\Pbb$ admits a forcing relation for $\Phi$} (or, synonymously, {\em $\forces_\Pbb$ exists for $\Phi$}) if there is a class $\forces$ which satisfies the above schema which covers all $\phi \in \Phi$.
For $\phi$ a formula, {\em $\Pbb$ admits a forcing relation for $\phi$} if $\Pbb$ admits a forcing relation for the collection of all instances of subformulae of $\phi$. Note that if $\forces_\Pbb$ exists for the atomic formulae then, by an induction in the metatheory, $\forces_\Pbb$ exists for all $\phi$.
\end{definition}

Observe that each step in this recursion, except the last, is done in a first-order way. As such, it is immediate that $\GBC + \ETR$ proves the forcing relation $\forces_\Pbb$ exists for first-order formulae. Then $\forces_\Pbb$ restricted to subformulae of a second-order formula is a definable hyperclass, via an induction in the metatheory.

Indeed, the existence of $\forces_\Pbb$ for every $\Pbb$ is equivalent, over $\GBC$, to a fragment of $\ETR$.

\begin{theorem}[\cite{GHHSW2017}] \label{thm1:str-class-forcing}
Over $\GBC$ the following are equivalent.
\begin{itemize}
\item $\ETR_\Ord$, Elementary Transfinite Recursion restricted to well-orders of length $\le \Ord$.
\item The class forcing theorem, asserting that for every class forcing $\Pbb$ and every formula $\phi$ in the forcing language for $\Pbb$ admits a forcing relation $\forces_\Pbb$ for subformulae of $\phi$.
\item The uniform first-order class forcing theorem, asserting that every class forcing $\Pbb$ admits a forcing relation $\forces_\Pbb$ for all first-order formulae in the forcing language.
\item The atomic class forcing theorem, asserting that every class forcing $\Pbb$ admits a forcing relation for atomic formulae.
\end{itemize}
\end{theorem}

Nevertheless, $\GBCm$ suffices to prove that $\forces_\Pbb$ exists for a nice collection of class forcings. 

\begin{definition}[S.\ Friedman \cite{friedman:book}]
Let $\Pbb$ be a forcing notion. 
\begin{itemize}
\item $D \subseteq \Pbb$ is {\em predense below $p$} (or {\em predense $\le p$}) if for every $q \le p$ is compatible with an element of $D$.
\item $\Pbb$ is {\em pretame} if given any set-indexed sequence $\seq{D_i : i \in a}$ of dense subclasses of $\Pbb$ and any $p \in \Pbb$ there is $q \le p$ and a sequence $\seq{d_i : i \in a}$ of predense $\le q$ sub{\em sets} of $\Pbb$ with $d_i \subseteq D_i$ for all $i \in a$.
\item $(D,D')$ is a {\em predense below $p$ partition} (or a {\em predense $\le p$ partition}) if $D \cup D'$ is predense $\le p$ and $p \in D$ and $p' \in D'$ implies $p \incompat p'$.
\item Two sequences $\Dcal = \seq{(D_i,D_i') : i \in a}$ and $\Ecal = \seq{(E_i,E_i') : i \in a}$ of predense $\le p$ partitions are {\em equivalent below $q$} (or {\em equivalent $\le q$}) if for each $i$ the collection of $r \in \Pbb$ so that $r$ meets $D_i$ if and only if $r$ meets $E_i$ is dense below $q$.\footnote{To clarify, $r$ {\em meets} predense $D$ means that $r \le q$ for some $q \in D$.}
\item $\Pbb$ is {\em tame} if it is pretame and additionally for every $p \in \Pbb$ and every set $a$ there is $q \le p$ and ordinal $\alpha$ so that if $\Dcal = \seq{(D_i, D_i') : i \in a}$ is a sequence of predense $\le q$ partitions then
\[
\{ r \in \Pbb : \Dcal \text{ is equivalent $\le r$ to some sequence $\Ecal \in V_\alpha$ of predense $\le q$ partitions} \}
\]
is dense below $q$.
\end{itemize}
\end{definition}

\begin{theorem}[S.\ Friedman, Stanley] \label{thm1:pretame-def-lemma}
$\GBCm$ proves that $\forces_\Pbb$ exists for $\phi$ for every pretame forcing $\Pbb$ and every formula $\phi$ in the forcing language.\footnote{Friedman \cite{friedman:book} proved the theorem with the assumption of Powerset, while Stanley \cite{stanley1984} independently gave a proof which did not need that assumption.}
\end{theorem}

Besides allowing for the forcing language to be definable in a weak theory, the other role of pretameness is in the preservation of the axioms. First, let us make precise how forcing extensions are built. The case for transitive models is well-known, but we can also work over non-transitive models. We work in $(M,\Xcal) \models \GBCm$ and assume that for our forcing notion $\Pbb \in \Xcal$ we have that for every first-order formula $\phi$ of the forcing language the forcing relation $\forces_\Pbb$ restricted to instances of subformulae of $\phi$ exists as a class in $\Xcal$. In particular, $\forces_\Pbb$ for atomic formulae is a single class in $\Xcal$. Given a generic $G \subseteq \Pbb$ and this class we can define the forcing extension as follows.

Work externally to $(M,\Xcal)$, as is necessary if we have a generic. Define the following relations on set $\Pbb$-names:
\begin{align*}
\sigma \in_G \tau \qquad&\iff\qquad \exists p \in G\ p \forces \sigma \in \tau \\
\sigma =_G \tau \qquad&\iff\qquad \exists p \in G\ p \forces \sigma = \tau
\end{align*}
and the similar relations $\in_G$ between set $\Pbb$-names and class $\Pbb$-names and $=_G$ between class $\Pbb$-names. It can be straightforwardly checked that $=_G$ is an equivalence relation\footnote{To be clear, both the set-set and class-class relations $=_G$ are equivalence relations on, respectively, the set $\Pbb$-names and the class $\Pbb$-names.}
and $\in_G$ is a congruence modulo $=_G$, meaning that if $\sigma \in_G \tau =_G \rho$ then $\sigma \in_G \rho$. Given a $\Pbb$-name $\sigma$ let $[\sigma]_G$ denote the equivalence class of $\sigma$ modulo $=_G$. Let $[M]_G$ denote the collection of equivalence classes of set $\Pbb$-names and $[\Xcal]_G$ denote the collection of equivalence classes of class $\Pbb$-names. 

\begin{definition}
Let $(M,\Xcal)$, $\Pbb$, and $G$ be as above. Then the {\em forcing extension of $(M,\Xcal)$}, denoted $(M,\Xcal)[G]$, is the structure $([M]_G, [\Xcal]_G)$ with membership relation $\in_G$. To refer to the sets of the extension I will use $M[G]$ and to refer to the classes of the extension I will use $\Xcal[G]$.
\end{definition}

\begin{proposition}[The truth lemma]
Consider a model $(M,\Xcal)$ of set theory with forcing notion $\Pbb \in \Xcal$ and generic $G \subseteq \Pbb$. Suppose that for each $\phi$ in the forcing language for $\Pbb$ that $\Pbb$ admits a forcing relation for $\phi$ in $\Xcal$. Let $\phi(x_0, \ldots, x_n, Y_0, \ldots, Y_m)$ be a formula, $\tau_0, \ldots \tau_n$ be set $\Pbb$-names, and $\Sigma_0, \ldots, \Sigma_m$ be class $\Pbb$-names. Then,
\begin{align*}
(M,\Xcal)[G] &\models \phi([\tau_0]_G, \ldots, [\tau_n]_G, [\Sigma_0]_G, \ldots, [\Sigma_m]_G) \\
&\text{ if and only if } \\
\exists p\in G\ p &\forces \phi(\tau_0, \ldots, \tau_n, \Sigma_0, \ldots, \Sigma_m).
\end{align*}
\end{proposition}

\begin{proof}[Proof sketch]
By induction on formulae. See \cite{GHHSW2017} for more detail.
\end{proof}

In particular, if $(M,\Xcal)$ is countable then we can always find a generic $G \subseteq \Pbb$. Externally to $(M,\Xcal)$ line up the countably many dense subclasses of $\Pbb$ in $\Xcal$ in ordertype $\omega$ and then inductively meet each of them. As such, for countable models the only possible impediment to having forcing extensions is having forcing relations. But so long as we only look at pretame forcings this is no impediment.

Let us turn now to the preservation of the axioms under class forcings. First, let us see the importance of pretameness to this question.

\begin{theorem}[\cite{stanley1984} for $(1 \iff 2)$, \cite{HKS2016b} for the rest]
Consider $(M,\Xcal) \models \GBcm$ and let $\Pbb \in \Xcal$ be a forcing notion.\footnote{Holy, Krapf, and Schlicht formulate their result in terms of countable {\em transitive} models, but it is not hard to see that their result holds more generally.}
The following are equivalent.
\begin{enumerate}
\item $\Pbb$ is pretame.
\item $\Pbb$ preserves $\GBcm$.
\item $\Pbb$ preserves Collection.
\item $\Pbb$ preserves Replacement.
\item $\Pbb$ preserves Separation and $\Pbb$ satisfies the class forcing theorem.
\end{enumerate}
\end{theorem}

Tameness has a similar importance for preserving $\GBc$.

\begin{theorem}[Friedman \cite{friedman:book}]
Consider $(M,\Xcal) \models \GBc$ and let $\Pbb \in \Xcal$ be a forcing notion. The following are equivalent.
\begin{enumerate}
\item $\Pbb$ is tame;
\item $\Pbb$ is pretame and preserves Powerset; and
\item $\Pbb$ preserves $\GBc$.
\end{enumerate}
\end{theorem}

I will not prove these results in full. But as a warmup towards showing that pretame forcings preserve strong forms of Comprehension, let us see that pretame forcings preserve Elementary Comprehension.

\begin{proposition}
Let $G \subseteq \Pbb$ be generic for $(M,\Xcal) \models \GBC$ with $\Pbb \in \Xcal$ admitting a forcing relation for each $\phi$ in the forcing language. Then $(M,\Xcal)[G]$ satisfies Elementary Comprehension.
\end{proposition}

In particular, pretame forcings always satisfy the forcing theorem so this proposition works for all pretame forcings over some model.

\begin{proof}
Consider an instance of Elementary Comprehension. That is, we have a first-order formula $\phi(x,P)$ with possible parameter $P$ and we want to see that the class $\{ x : (M,\Xcal)[G] \models \phi(x,P) \} \in \Xcal[G]$. Towards this end, let $\dot P \in \Xcal$ be a name for $P$. Now consider the name $\dot B = \{ (\sigma, p) : p \forces \phi(\sigma, \dot P) \}$. Then, because $\Pbb$ satisfies the forcing theorem, $\dot B \in \Xcal$. So
\[
\dot B^G = \{ \sigma^G : \sigma \in M \land \exists p \in G\ p \forces \phi(\sigma,\dot P) \} = \{ x \in M[G] : (M,\Xcal)[G] \models \phi(x,P) \}
\]
is in $\Xcal$.
\end{proof}

In fact, the same argument applies higher up.

\begin{corollary}
Class forcing preserves $\Pi^1_k$-Comprehension, for $1 \le k \le \omega$.\footnote{Recall that $\Pi^1_\omega$-Comprehension is another name for the full Comprehension schema.} That is, if $(M,\Xcal) \models \GBC + \PnCA k$ and $G \subseteq \Pbb \in \Xcal$ is generic over $(M,\Xcal)$ then $(M,\Xcal)[G] \models \PnCA k$.
\end{corollary}

\begin{proof}
Because $(M,\Xcal) \models \PnCA k$ it in particular satisfies $\ETR_\Ord$, so it satisfies the uniform first-order class forcing theorem. Now run the same argument as before, but use that $\forces_\Pbb$ restricted to subformulae of a $\Sigma^1_n$-formula $\phi$, for $1 \le n < \omega$, is $\Sigma^1_n$-definable from $\forces_\Pbb$ restricted to first-order formulae.
\end{proof}

We are also interested in the preservation of (fragments) of Class Collection.

\begin{theorem}
Let $G \subseteq \Pbb$ be generic for $(M,\Xcal) \models \KMCC$ with $\Pbb \in \Xcal$ a pretame forcing. Then $(M,\Xcal)[G]$ satisfies Class Collection.
\end{theorem}

\begin{proof}
Suppose that $(M,\Xcal)[G] \models \forall \alpha \exists Y\ \phi(\alpha,Y,A)$, for some class $A$. We want to find a class $C \in \Xcal[G]$ so that $(M,\Xcal)[G] \models \forall \alpha\ \phi(x,(C)_x,A)$. Take $p \in G$ forcing $\forall \alpha \exists Y\ \phi(x,Y,\dot A)$, where $\dot A$ is a $\Pbb$-name for $A$. Fix an ordinal $\alpha$ Then $p \forces \exists Y\ \phi(\check \alpha,Y,\dot A)$. 

I claim there is a class name $\dot Y_\alpha$ so that $p \forces \phi(\alpha, \dot Y_\alpha, \dot A)$. To see this: let $D$ be the dense below $p$ class of conditions $q \le p$ so that $q \forces \phi(\alpha, \dot Y_q, \dot A)$ for some class name $\dot Y_q$. This $D$ is in $\Xcal$ by an instance of Comprehension. Now construct a maximal antichain $N \subseteq D$ by recursion using a bijection $b : \Ord \to D$. First, put $b(0)$ into $N$. Continuing upward, we include $b(\xi)$ in $N$ if and only if $b(\xi)$ is incompatible with all the $b(\zeta)$ for $\zeta < \xi$ we have already committed to being in $N$. We can now use this antichain $N$ to build the desired $\dot Y_\alpha$ via a mixing argument. Namely,
\[
\dot Y_\alpha = \{ (\sigma,r) : \exists q,q' \in N\ \exists p\ (\sigma,p) \in \dot Y_q \mand r \le q',p \}).
\]
Note that this definition uses Class Collection to pick the $\dot Y_q$ corresponding to each $q \in N$. Now given these names $\dot Y_\alpha$ for each $\alpha$ we can put them together to get a name for a class $C$ so that the $\alpha$th slice of $C$ is the interpretation of $\dot Y_\alpha$. 
\end{proof}

\begin{corollary}
Let $G \subseteq \Pbb$ be generic for $(M,\Xcal) \models \GBC + \PnCAp k$, for $k \ge 1$, with $\Pbb \in \Xcal$ a pretame forcing. Then $(M,\Xcal)[G]$ satisfies $\Sigma^1_k$-Class Collection.
\end{corollary}

\begin{proof}
In the above argument we used Comprehension to get $D$ and Class Collection to choose the $\dot Y_q$. This instance of Comprehension used $p \forces \phi$, which is $\Sigma^1_k$-definable for a $\Sigma^1_k$ formula $\phi$. Similarly, the instance of Class Collection is also $\Sigma^1_k$. So both go through in this context.
\end{proof}

I end this section with an open question. We have seen that strong second-order set theories are preserved by tame forcing, as is $\GBC$. What about intermediate theories?

\begin{question}
Is $\ETR$ preserved by tame forcing?
\end{question}

\section{Some basic constructions}

In this section I survey some basic constructions for models of second-order set theories. I will focus on models of the weak theories, as the constructions for stronger theories are not so basic.

\begin{observation}
Let $M \models \ZFC$ and $\Def(M)$ consist of the definable, possibly with parameters, subsets of $M$. Then $(M,\Def(M)) \models \GBc$. If additionally, $\Def(M)$ contains a global well-order of $M$ then $(M,\Def(M)) \models \GBC$. Similarly, if $M \models \ZFCm$ then $(M,\Def(M)) \models \GBcm$, or $\GBCm$ in case $M$ has a definable global well-order of ordertype $\Ord$.
\end{observation}

Of course, $M$ may not have a definable global well-order. So in general it requires more to get a $\GBC$-realization (or $\GBCm$-realization) for $M$.

\begin{definition}
Let $M$ be a model of first-order set theory. Say that $A \subseteq M$ is {\em amenable to $M$} if $A \cap x \in M$ for all $x \in M$. 
\end{definition}

\begin{definition} \label{def1:t-amenable}
Let $T$ be a second-order set theory and $M$ be a model of first-order set theory. Say that $A \subseteq M$ is {\em $T$-amenable to $M$} if there is a $T$-realization $\Xcal$ for $M$ with $A \in \Xcal$.
\end{definition}

A special case of interest is when $T$ is $\GBc$ or $\GBcm$. 

\begin{observation}
$A$ is $\GBcm$-amenable to $M$ if and only if $(M,A)$ satisfies the Separation and Replacement schemata for formulae in the expanded language.
\end{observation}

\begin{proof}
$(\Rightarrow)$ Because $\GBcm$ includes Elementary Comprehension and Class Replacement.

$(\Leftarrow)$ Then $\Def(M;A)$ is a $\GBcm$-realization for $M$.
\end{proof}

It is obvious that if $G$ is a $\GBc$-amenable global well-order of $M$ then $(M,\Def(M;G)) \models \GBC$. Accordingly, to show that $M$ is $\GBC$-realizable we want to find such a global well-order.

\begin{theorem}[Folklore] \label{thm1:gbc-cons-zfc}
Suppose $M \models \ZFC$ is countable. Then $M$ is $\GBC$-realizable. In general, if $(M,\Xcal) \models \GBc$ is countable then there is $\Ycal \supseteq \Xcal$ a $\GBC$-realization for $M$.
\end{theorem}

The tool used in this proof will be Cohen-forcing to add a generic subclass of $\Ord$.

\begin{definition}
The forcing to add a Cohen-generic subclass of $\Ord$, denoted $\Add(\Ord,1)$, consists of all set-sized partial functions from $\Ord$ to $2$, ordered by reverse inclusion. 
\end{definition}

\begin{lemma}
Over $\GBc$, the forcing $\Add(\Ord,1)$ is tame. 
\end{lemma}

\begin{proof}
First let us see that $\Add(\Ord,1)$ is pretame. Because $\Add(\Ord,1)$ is $<\kappa$-distributive for every $\kappa$,\footnote{$\Add(\Ord,1)$ is $<\kappa$-distributive because it is $<\kappa$-closed. Checking this is an easy exercise.}
below any set-sized collection of open dense subclasses of $\Add(\Ord,1)$ we can find a single dense subclass. That is, if $\seq{D_i : i \in a}$ is a set-indexed collection of open dense subclasses then we can find a dense subclass $D$ so that any generic which meets $D$ must meet all the $D_i$. As such, we may assume without loss that that we are dealing with a single dense subclass. That is, the setup is that we have a dense class $D$ and a condition $p$. We want to find a condition $q \le p$ and a set $d \subseteq D$ which is predense $\le q$. But this is trivial: take $q \le p$ which meets $D$ and let $d = \{q\}$.

Finally, $\Add(\Ord,1)$ preserves Powerset because forcing with it does not add any new sets, which in turn is because it is $<\kappa$-closed for every cardinal $\kappa$.
\end{proof}

This appeared in the proof of the above lemma, but it is important enough to be stated on its own.

\begin{observation}
Over $\GBc$, forcing with $\Add(\Ord,1)$ does not add new sets. \qed
\end{observation}

We can now prove theorem \ref{thm1:gbc-cons-zfc}.

\begin{proof}[Proof of theorem \ref{thm1:gbc-cons-zfc}]
It suffices to prove the more general result, since we get the other result by considering $\Xcal = \Def(M)$ to consist of the (first-order) definable classes.

We obtain $\Ycal$ by forcing over $(M,\Xcal)$ with $\Add(\Ord,1)$. That is, let $C \subseteq \Ord^M$ be generic over $(M,\Xcal)$ for $\Add(\Ord,1)$. Then $\Ycal = \Xcal[C]$. So $(M,\Ycal) \models \GBc$ and in particular the class $C$ is $\GBc$-amenable to $M$. We want to see that we can define a global well-order from $C$. To do this we use that  every set of ordinals in $M$ is coded into $C$. This is because sets can be coded as sets of ordinals---to code $x$ take an isomorphic copy of $\mathord\in \rest \tc(\{x\})$ as a set of pairs of ordinals, which can be coded as a set of ordinals via a pairing function---and so by density every set is coded into $C$. Thus, we can define a global well-order $<_C$ as $x <_C y$ if the first place $x$ is coded into $C$ comes before the first place $y$ is coded into $C$.
\end{proof}

\begin{corollary}[Folklore] \label{cor1:gbc-cons-zfc}
$\GBC$ is conservative over $\ZFC$. That is, if $\phi$ is a first-order sentence in the language of set theory then $\GBC \proves \phi$ if and only if $\ZFC \proves \phi$.
\end{corollary}

\begin{proof}
Suppose otherwise. Then there is a countable $M \models \ZFC$ which satisfies $\neg \phi$ while every model of $\GBC$ satisfies $\phi$. But $M$ has a $\GBC$-realization $\Xcal$ and $(M,\Xcal) \models \GBC + \neg \phi$, a contradiction.
\end{proof}

We can also get a version of theorem \ref{thm1:gbc-cons-zfc} that applies to models of $\ZFCm$, but we need a little more from our model. In the $\ZFC$ context we knew that $\Add(\Ord,1)$ did not add sets because it was $<\kappa$-closed for every $\kappa$. Let me quickly sketch the argument so we know what we would like to generalize.

Fix a set name $\dot a$. Without loss of generality we may assume that $\dot a$ gives a set of ordinals. That is, suppose $C \subseteq \Add(\Ord,1)$ is generic and consider $p \in C$ so that $p \forces \dot a \subseteq \check \kappa$. Let us now see that the ground model can interpret $\dot a$. Start with $p_0 = p$. Given $p_\alpha$ for $\alpha < \kappa$ extend $p_\alpha$ to $p_{\alpha+1}$ which decides whether $\check \alpha \in \dot a$. And at limits use $<\kappa$-closure to continue the construction. And because $\Add(\Ord,1)$ is $<\kappa^+$-closed we get $p_\kappa$ below all the $p_\alpha$'s for $\alpha < \kappa$. Moreover, we may make the choices so that $p_\alpha$ is always in $C$, since the classes we are meeting are dense. So $p_\kappa$, which is in the ground model, contains all the information that the generic uses to interpret $\dot a$. So $\dot a^C$ is in the ground model. 

The same argument will work for a model of $\ZFCm$ without a largest cardinal. But if the model does have a largest cardinal, that will not work. For concreteness, suppose we are working over a model of $\ZFCm$ plus ``every set is countable''. To show that $\Add(\Ord,1)$ we want to be able to make countably many choices according to some definable procedure and have that those choices cohere. This would give us the desired $p_\omega$ which has all the information needed to interpret the name $\dot a$. That is, we want $\omega$-Dependent Choice for definable procedures. In general, if our model has larger cardinals, then we want this but with $\omega$ replaced with the largest cardinal in the model.

\begin{definition}
Work in the context of $\ZFCm$ and let $\kappa$ be a cardinal. The principle of {\em Definable $\kappa$-Dependent Choice} asserts the following: if $T$ is a definable tree of sequences of length $< \kappa$ so that for all $\alpha < \kappa$ each node in $T$ of length $\alpha$ has a successor in $T$, then $T$ has a branch.
\end{definition}

\begin{remark}
Observe that the branch is a set. If it were a definable class, then because $\kappa$ is a set and the branch has length $\kappa$, Replacement would imply that the branch is a set.
\end{remark}

Definable Dependent Choice is not a theorem of $\ZFCm$. 
S.\ Friedman and Gitman \cite{friedman-gitman2017} produced a model of $\ZFCm$ $+$ ``every set is countable'' where Definable $\omega$-Dependent Choice fails. 

If our model of $\ZFCm$ with a largest cardinal $\kappa$ satisfies definable $\kappa$-dependent choice then $\Add(\Ord,1)$ over that model will be pretame and not add any new sets. So in that case we can force with $\Add(\Ord,1)$ to add a global well-order without adding any new sets. Indeed, the two are equivalent.

\begin{proposition}
Let $M \models \ZFCmi$ have a largest cardinal $\kappa$. Then the following are equivalent:
\begin{enumerate}
\item $M$ satisfies Definable $\kappa$-Dependent Choice; and
\item $M$ admits a $\GBCm$-amenable global well-order.
\end{enumerate}
\end{proposition}

\begin{proof}[Proof sketch]
We sketched $(1 \impl 2)$ above. For $(2 \impl 1)$, suppose $G : \Ord^M \to V^M$ is $\GBCm$-amenable to $M$. Work in $(M,\Xcal) \models \GBCm$ with $G \in \Xcal$. We can use $G$ to make choices along $\kappa$-trees. Since $\Xcal$ contains every definable class this yields that $M$ satisfies Definable $\kappa$-Dependent Choice.
\end{proof}

For a related result, Gitman, Hamkins, and Johnstone \cite{gitman-hamkins-johnstone2016} showed that, over $\ZFCm$, Definable $\omega$-Dependent Choice is equivalent to the Reflection schema, i.e.\ the schema asserting that for every formula $\phi(x,a)$ and set $a$ that there is a transitive set $t \ni a$ so that $\phi(x,a)$ reflects to $t$.

Finally, let me observe that we cannot get a version of corollary \ref{cor1:gbc-cons-zfc} that works for $\ZFCm$.

\begin{corollary}
The theory $\GBCm$ is not conservative over $\ZFCm$.
\end{corollary}

\begin{proof}
As we just saw, $\GBCm$ proves the Definable $\kappa$-Dependent Choice schema. But Friedman and Gitman showed that this schema is not a theorem of $\ZFCm$.
\end{proof}

Having investigated what happens without powerset, let us now turn to the uncountable. We saw in theorem \ref{thm1:gbc-cons-zfc} that every countable model of $\ZFC$ is $\GBC$-realizable. This does not generalize to the uncountable. Let us see why.

\begin{definition} \label{def1:rather-classless}
A model $M \models \ZFC$ is called {\em rather classless} if every amenable $X \subseteq M$ is definable. 
\end{definition}

Necessarily, every rather classless model has uncountable cofinality. If $X \subseteq \Ord^M$ has ordertype $\omega$ and is cofinal then it is amenable, because its intersection with an initial segment of $\Ord^M$ is finite, but not definable or danger of contradicting Replacement.

\begin{theorem}[Keisler \cite{keisler1974}, Shelah \cite{shelah1978}]
Any countable model of $\ZFC$ has an elementary rank extension\footnote{If $M \subseteq N$ are models of $\ZFC$ then say $N$ is a {\em rank extension} of $M$ if every new set has a higher rank. For example, if $\kappa < \lambda$ are inaccessible, then $V_\lambda$ is a rank extension of $V_\kappa$.} 
to a rather classless model.\footnote{Keisler showed this theorem under the assumption of $\diamond$ and the assumption of $\diamond$ was later eliminated by Shelah.}
\end{theorem}

As a consequence if $M \models \ZFC$ does not have a definable global well-order then any rather classless elementary rank extension of $M$ will not have a $\GBC$-realization. Keisler's theorem also applies to $\ZFCm$, so there are models of $\ZFCm$ which fail to have a $\GBCm$ realization.

One might hope that theorem \ref{thm1:gbc-cons-zfc} could be generalized to stronger second-order set theories. Of course, this could not work for all countable models of $\ZFC$, as these stronger theories are not conservative over $\ZFC$. But one might hope that any countable model of a strong enough theory is, say, $\KM$-realizable. 

One's hopes are in vain.

\begin{proposition}[Folklore] \label{prop1:stram-tr-impl-tower}
Consider $M \models \ZFC$ so that the truth predicate for $M$ is $\GBc$-amenable. Then there is a club of ordinals $\alpha \in \Ord^M$ so that $V_\alpha^M \prec M$.
\end{proposition}

Before moving to the proof, let us recall the Tarskian definition of a truth predicate.

\begin{definition} \label{def1:tarski-truth-pred}
Let $M$ be a model of first-order set theory. A {\em truth predicate} or {\em satisfaction class} for $M$ is a class $\Tr \subseteq M$ of pairs $(\phi,\bar a)$ satisfying the following recursive requirements.
\begin{enumerate}
\item If $\phi$ is atomic then $(\phi, \bar a) \in \Tr$ if and only if $\phi(\bar a)$ gives a true fact about $M$. That is, $\Tr$ should declare $a \in b$ to be true if and only if $M \models a \in b$ and declare $a = b$ to be true if and only if $M \models a = b$.
\item $(\phi \lor \psi,\bar a)$ is in $\Tr$ if and only if $(\phi,\bar a)$ or $(\psi, \bar a)$ are in $\Tr$.
\item $(\neg \phi, \bar a)$ is in $\Tr$ if and only if $(\phi, \bar a)$ is not in $\Tr$.
\item $(\exists x\ \phi(x), \bar a)$ is in $\Tr$ if and only if there is $b \in M$ so that $(\phi, b \cat \bar a)$ is in $\Tr$.
\end{enumerate}
We are interested in adding $\Tr$ as a class to $M$, so in the case that $M$ is $\omega$-nonstandard let us explicitly require that $\Tr$ measures the `truth' of every nonstandard formula.\footnote{In the literature such a class $\Tr$ is called a {\em full satisfaction class}. See also chapter 3 for a discussion of truth predicates over $\omega$-nonstandard models.}
\end{definition}

Observe that it is a first-order property of a class whether it is a truth predicate, so it does not depend upon what classes are in the model. Also observe that $\GBc$ proves the truth predicate is unique. If two different classes both satisfy the definition then there must be a minimal place where they disagree on the truth of $(\phi,\bar a)$. But they agree on every previous stage so they must agree on the truth of $(\phi,\bar a)$, a contradiction.

\begin{proof}[Proof of proposition \ref{prop1:stram-tr-impl-tower}]
This follows from an instance of the Montague reflection principle. Let $\Tr$ be the truth predicate for $M$. We can use $\Tr$ as a parameter in the formula we are reflecting precisely because it is $\GBc$-amenable. Namely, reflect to find a club of $\alpha$ so that $(V_\alpha^M; \Tr \cap V_\alpha^M) \prec_{\Sigma_1} (M; \Tr)$. Then $V_\alpha^M \prec M$ because by elementarity $(V_\alpha^M; \Tr \cap V_\alpha^M) \models (\phi,\bar a) \in \Tr \cap V_\alpha^M$ if and only if $(M; \Tr) \models (\phi,\bar a) \in \Tr$ if and only if $M \models \phi(\bar a)$.\footnote{There is a minor subtlety. Namely, what happens if $M$ is $\omega$-nonstandard? Then $\Tr$ makes assertions about the `truth' of nonstandard formulae, and for such formulae $\phi$ it does not make sense externally to ask whether $M \models \phi$. But this is not an issue because $\Tr$ must be correct about the truth of standard formulae, as can be checked by an easy induction external to the model.}
\end{proof}

\begin{corollary} \label{cor1:truth-non-rlzbl}
Let $T \supseteq \GBcm$ be a second-order set theory which proves the existence of the truth predicate for the first-order part. Then no first-order theory characterizes which countable models are $T$-realizable.
\end{corollary}

In particular this works for $T = \KM$, $T = \GBC + \ETR$, or even $T = \GBC + \ETR_\omega$.

\begin{proof}[Proof sketch]
Let $S$ be any consistent first-order set theory. Then, by standard results about nonstandard models, there is $M \models S$ which is $\omega$-nonstandard but not recursively saturated.\footnote{For a definition of recursive saturation, see chapter 3.}
It follows that $M$ does not admit an amenable truth predicate, so it cannot be $T$-realizable.
\end{proof}

The reader may find this nonstandard trick to be unsatisfactory. But we get a version of corollary \ref{cor1:truth-non-rlzbl} for $\omega$-models or even transitive models.

\begin{corollary}
Let $T \supseteq \GBcm$ be a second-order set theory which proves the existence of the truth predicate for the first-order part. Suppose that $T$ has an $\omega$-model. Then no first-order set theory characterizes which $\omega$-standard models are $T$-realizable. Moreover, suppose $T$ has a transitive model. Then no first-order set theory characters which transitive models are $T$-realizable.
\end{corollary}

For trivial reasons, we need the assumption that $T$ has an $\omega$-model (or transitive model for the moreover). If $T$ has no $\omega$-models then it is easy to get a first-order set theory characterizing which $\omega$-models are $T$-realizable---take your favorite inconsistent theory.

\begin{proof}
Suppose $S$ is some first-order set theory. If $T + S$ has no $\omega$-model (or no transitive model, for that case) then $S$ cannot characterize which $\omega$-models (or transitive models) are $T$-realizable. So assume that $T + S$ has an $\omega$-model $(M,\Xcal)$. (Or, for the transitive case, assume $(M,\Xcal)$ is transitive. We will find an elementary model of $M$ which is not $T$-realizable, establishing that satisfying $S$ cannot ensure a model is $T$-realizable.

The truth predicate for $M$ is in $\Xcal$ so it must be $\GBc$-amenable. Now take the least $\alpha \in \Ord^M$ so that $V_\alpha^M \prec M$. Then $N \models S$ and is $\omega$-standard (or transitive, if $M$ is transitive). I claim that $N = V_\alpha^M$ is not $T$-realizable.
Otherwise, by proposition \ref{prop1:stram-tr-impl-tower} there is a club of ordinals $\beta \in \Ord^N$ so that $V_\beta^N \prec N$. But $V_\beta^N = V_\beta^M$, because $M$ is a rank extension of $N$. So then $V_\beta^M \prec V_\alpha^M \prec M$, contradicting the leastness of $\alpha$. 
\end{proof}

\begin{remark}
This argument uses essentially that $M$ (and hence also $N$) is an $\omega$-model. Suppose $(M,\Xcal) \models T$, where $T$ is as in the corollary, is countable and $\omega$-nonstandard. Let $N = V_\alpha^M$, where $\alpha \in M$ is least in the club of ordinals $\alpha_0$ so that $V_{\alpha_0}^M \prec M$, where elementarity here is according to the truth predicate in $M$. In particular, $N$ is also countable and $\omega$-nonstandard and has the same theory and standard system\footnote{The {\em standard system} of an $\omega$-nonstandard model $M$ is the collection of all reals coded in $M$, i.e.\ the set of all $x \subseteq \omega$ so that there is $y \in \omega^M$ so that $y \cap \omega = x$.}
as $M$ does. So by a back-and-forth argument we can show that in fact $N$ and $M$ are isomorphic. So because $M$ is $T$-realizable, so must $N$ be $T$-realizable. 
\end{remark}

To finish off this section, let us see that see that $T$-amenability for different choices of $T$ can give different notions. We saw that, for countable models, being $\GBc$-amenable is equivalent to being $\GBC$-amenable. But for some $T$, being $T$-amenable to $M$ is stronger than being $\GBC$-amenable.

\begin{theorem}
Let $M \models \ZFC$ be a countable $\omega$-model. Then there is $A \subseteq M$ which is $\GBC$-amenable to $M$ but is not $(\GBC + \ETR_\omega)$-amenable to $M$. That is, there is a $\GBC$-realization $\Xcal$ for $M$ with $A \in \Xcal$ but no $(\GBC + \ETR_\omega)$-realization $\Ycal$ for $M$ can have $A \in \Ycal$.
\end{theorem}

\begin{proof}
I will show a stronger fact, from which this theorem will immediately follow. Namely, I will show that there is $G \subseteq M$ which is $\GBC$-amenable but no $\GBC$-realization for $M$ can contain both $G$ and the truth predicate for $M$.\footnote{This is why I require $M$ to be an $\omega$-model. If $M$ is $\omega$-nonstandard then there will be many different classes which satisfy the definition of being a truth predicate. See chapter 3 for more detail.}
This $G$ will be a carefully constructed Cohen-generic subclass of $\Ord$. I claim that there is a sequence $\seq{D_i : i \in \Ord}$ of definable dense subclasses of $\Add(\Ord,1)$ so that (1) the sequence is definable from the truth predicate and (2) meeting every $D_i$ is sufficient to guarantee that a filter is generic over $(M,\Def(M))$. It is obvious that the truth predicate can define a sequence of all the definable subclasses of $\Add(\Ord,1)$. Namely, from the truth predicate can be defined the sequence indexed by $(\phi,a)$ where the $(\phi,a)$-th dense subclass is the one consisting of all conditions $p$ so that $\phi(p,a)$ holds (or trivial if $\phi(x,a)$ does not define a dense subclass). But this sequence is not of the correct ordertype. However, we can use that $\Add(\Ord,1)$ is $<\kappa$-distributive for every cardinal $\kappa$ to get the ordertype to be $\Ord$. Namely let $D_i$ be below all the open dense subclasses which are definable from parameters in $V_i$. Then $\seq{D_i : i \in \Ord}$ is as desired.

We will now use this sequence to define our $G$. Externally to the model fix an $\omega$-sequence cofinal in $\Ord^M$. Think of this sequence as an $\Ord^M$-length binary sequence $\seq{b_i : i \in \Ord^M}$, consisting mostly of zeros with ones showing up rarely. This sequence is amenable to $M$, since its initial segments have only finitely many ones. On the other hand, it is not $\GBc$-amenable since from this sequence we can define a cofinal map $\omega \to \Ord$, contradicting an instance of Replacement. 

Build $G$ in $\Ord^M$ many steps, with partial piece $g_i$ at the $i$th step. We start with $g_0 = \emptyset$. Given $g_i$, let $g_i' = g_i \cat \seq{b_i}$. Then get $g_{i+1}$ by extending $g_i'$ to meet $D_i$, where we do this in the minimal possible length. (If there is more than one way to meet $D_i$ with minimal length, then pick arbitrarily.) And if $i$ is limit then $g_i = \bigcup_{j < i} g_j$. For each $i \in \Ord$, we have that $g_i \in M$ because $\seq{b_i : i \in \Ord}$ is amenable to $M$ and $g_i$ can be defined from an initial segment of this sequence. Finally, set $G = \bigcup_{i \in \Ord} g_i$. Then $G$ is generic over $(M,\Def(M))$, since it meets every $D_i$. But hidden within $G$ is this bad sequence $\seq{b_i}$. It is well hidden, but with the truth predicate we can snoop it out.

Suppose towards a contradiction that $(M,\Xcal) \models \GBc$ contains both $G$ and the truth predicate for $M$. Then $\seq{D_i : i \in \Ord} \in \Xcal$. We will see that $\seq{b_i} \in \Xcal$ by inductively determining each $b_i$. First, $b_0$ is the first bit of $G$. We then know the minimal length we have to extend $\seq{b_0}$ to meet $D_0$. We can use this to recover $g_1$ and then discover $b_1$ as the first bit in $G$ after $g_1$. We then repeat this process, using $\seq{D_i}$ to recover $g_2$ then get $b_2$, and so on. So we can define $\seq{b_i}$ from $G$ and the truth predicate, so $\seq{b_i} \in \Xcal$. But then $(M,\Xcal)$ cannot satisfy Replacement, a contradiction.
\end{proof}

\section{\texorpdfstring{$\GBC$-realizations of a countable model}{GBC-realizations of a countable model}}

In this section we look at the structure of the $\GBC$-realizations for a fixed model. At the end of the section I will discuss the extent to which the results generalize for theories stronger then $\GBC$.

\begin{definition}
Let $M$ be a model of first-order set theory and let $T$ be a second-order set theory. Set $\rlzn T M = \{ \Xcal \subseteq \powerset(M) : \Xcal$ is a $T$-realization for $M \}$. Then $\rlzn T M$ is a partial order under $\subseteq$.
\end{definition}

Of course, $M$ may fail to be $T$-realizable and thus $\rlzn T M$ may be empty. But theorem \ref{thm1:gbc-cons-zfc} implies that if $M$ is countable then $\rlzn \GBC M$ is not empty. If we move to the uncountable, however, then $\rlzn \GBC M$ may be anemic. If $M$ is rather classless then $\card{\rlzn \GBC M} \le 1$; if such $M$ has a definable global well-order then it has a single $\GBC$-realization---namely its definable classes---otherwise it will have no $\GBC$-realization at all. 

But if $M \models \ZFC$ is countable then it will have continuum many $\GBC$-realizations. To prove in theorem \ref{thm1:gbc-cons-zfc} that countable models of $\ZFC$ have $\GBC$ realizations we added a Cohen-generic subclass of $\Ord$. But there are continuum many different generic Cohen subclasses of $\Ord^M$ for countable $M$. This gives continuum many different $\GBC$-realizations for $M$. 

\begin{proposition}
Let $M \models \ZFC$ be countable. Then there are continuum many different subclasses of $\Ord^M$ which are Cohen-generic over $(M,\Def(M))$.
\end{proposition}

\begin{proof}
Consider the full binary tree $T = {}^{<\omega}2$ of finite binary sequences. We will construct a family of Cohen-generic subclasses of $\Ord^M$ for each branch through $T$. This construction is done along $T$. Order, in ordertype $\omega$, the dense subclasses $D_i$ of $\Add(\Ord,1)^M$ which are in $\Def(M)$. Start with $p_\emptyset = \emptyset$. Assume that we have already defined $p_s$ for $s$ a node in $T$. Then, to get $p_{s \cat i}$ extend $p_s \cat i$ to a condition in $D_\ell$, where $\ell$ is the length of $s$. Then if $B$ is a branch through $T$ we have that $C_B = \bigcup_{s \in B} p_s$ is Cohen-generic over $(M,\Def(M))$, as it met every dense class. And if $B \ne B'$ are distinct branches this is because there is a node $s \in T$ so that $s \cat 0 \in B$ and $s \cat 1 \in B'$. So $C_B \ne C_{B'}$ because they extend $p_s$ in incompatible ways.
\end{proof}

An intriguing question is whether there is anything between these extremes of continuum many $\GBC$-realizations and $\le 1$ $\GBC$-realization.

\begin{question}
Is there (necessarily uncountable) $M \models \ZFC$ so that $\card{\rlzn \GBC M} = 2$? What about $\card{\rlzn \GBC M} = n$ for finite $n$? What about $\card{\rlzn \GBC M} = \omega$? In general, what cardinals $\kappa$ are the cardinality of $\rlzn \GBC M$ for some $M$?
\end{question}

The rest of this section will be confined to looking at countable models of $\ZFC$ as there it can be shown that $\rlzn \GBC M$ has a rich structure.

Let us begin with some basic properties.

\begin{theorem} \label{thm1:gm-omnibus}
Let $M \models \ZFC$ be countable. Then $\rlzn \GBC M$ satisfies the following properties.
\begin{enumerate}
\item If $\{ \Xcal_i : i \in I \} \subseteq \rlzn \GBC M$ has a lower bound in $\rlzn \GBC M$ it has a greatest lower bound.
\item If $\{ \Xcal_i : i \in I \} \subseteq \rlzn \GBC M$ has an upper bound in $\rlzn \GBC M$ it has a least upper bound. 
\item If $M$ does not have a definable global well-order, then there are pairs of elements of $\rlzn \GBC M$ without a lower bound. On the other hand, if $M$ does have a definable global well-order then any $\{ \Xcal_i : i \in I \} \subseteq \rlzn \GBC M$ has a lower bound.
\item There are pairs of elements of $\rlzn \GBC M$ without an upper bound.
\item If $\seq{\Xcal_i : i \in I}$ is an increasing chain then it has a supremum.
\item There are maximal elements of $\rlzn \GBC M$. That is, there is $\Xcal \in \rlzn \GBC M$ so that there is no $\Ycal \in \rlzn \GBC M$ with $\Ycal \supsetneq \Xcal$.
\end{enumerate}
\end{theorem}

\begin{proof}
$(1)$ The infimum is $\Ycal = \bigcap_{i \in I} \Xcal_i$. We want to see that $(M,\Ycal) \models \GBC$. Extensionality is free, as is Replacement since $\Ycal$ is contained inside a $\GBC$-realization. To see that $\Ycal$ satisfies Global Choice, fix any $\Zcal$ a lower bound for the $\Xcal_i$. Then $\Ycal \supseteq \Zcal$ so the global well-order in $\Zcal$ is in $\Ycal$. Finally, we want to see that $\Ycal$ satisfies Elementary Comprehension. But each $\Xcal_i$ is closed under first-order definability, so their intersection must also be closed under first-order definability.

$(2)$ The supremum is the intersection of all the upper bounds. It follows from $(1)$ that this gives an element of $\rlzn \GBC M$.

$(3)$ Suppose $M$ does not have a definable global well-order. Take $C,D$ subclasses of $\Ord^M$ which are mutually Cohen-generic over $(M,\Def(M))$. Then $\Def(M;C)$ and $\Def(M;D)$ are $\GBC$-realizations for $M$. But their intersection is $\Def(M)$, by mutual genericity, which does not have a global well-order. 

For the other case, suppose $M$ does have a definable global well-order. Then $\Def(M)$ is a lower bound for a subset of $\rlzn \GBC M$.

$(4)$ We will construct $C,D$ Cohen-generic over $(M,\Def(M))$ so that no $\GBC$-realization for $M$ can contain both $C$ and $D$. This will establish that $\Def(M;C)$ and $\Def(M;D)$ are elements of $\rlzn \GBC M$ without an upper bound.

Fix $B: \Ord^M \to 2$ so that $\{i : B(i) = 1\}$ has ordertype $\omega$ and is cofinal in $\Ord^M$. Such $B$ exists because $M$ is countable. But no $\GBCm$-realization for $M$ can contain this bad class $B$ because $B$ reveals that $\Ord^M$ has countable cofinality. We will construct $C$ and $D$ so that together they code $B$. This is a construction in $\Ord^M$ many steps, defining $c_i$ and $d_i$ for $i \in \Ord^M$. Order the dense subclasses of $\Add(\Ord,1)$ as $\langle D_i : i \in \Ord^M \rangle$. 

\begin{itemize}
\item Set $c_0 = d_0 = \emptyset$.
\item Given $c_i$ and $d_i$ set $c_{i+1}' = c_i \cat \seq{0\ldots0} \cat \seq{1,B(i)}$, where the sequence of $0$'s has length chosen so that $c_i \cat \seq{0\ldots0}$ has the same length as $d_i$. Then get $c_{i+1}$ by extending $c_{i+1}'$ to meet $D_i$. Next, let $d_{i+1}' = d_i \cat \seq{0\ldots0} \cat \seq{1}$, where the sequence of $0$'s has length chosen so that $d_i \cat \seq{0\ldots0}$ has the same length as $c_{i+1}'$. Then get $d_{i+1}$ by extending $d_{i+1}'$ to meet $D_i$. 
\item If $i$ is limit then $c_i = \bigcup_{j < i} c_j$ and $d_i = \bigcup_{j < i} d_j$.
\end{itemize}

Finally, set $C = \bigcup_i c_i$ and $D = \bigcup_i d_i$. By construction $C$ and $D$ are Cohen-generic over $M$. Suppose towards a contradiction that $\Xcal$ is a $\GBCm$-realization for $M$ with $C,D \in \Xcal$. Let us see that $B \in \Xcal$, a contradiction. Namely, $\Xcal$ can inductively recover $B(i)$, $c_i$, and $d_i$ from $C$ and $D$. First, $c_0 = d_0 = \emptyset$. Now given $c_i$ we find $B(i)$ by looking at the bit in $C$ after the first 1 after the block of $0$s in $C$ starting after the end of $c_i$. We also get $d_i$ by using that block of $0$s to tell us how far in $D$ we need to go to get $d_i$. Next, looking at the block of $0$s in $D$ starting after the end of $d_i$ tells us how long $c_i$ was extended in $C$ to get $c_{i+1}$. We then continue this process, getting $B(i+1)$, $d_{i+1}$, $c_{i+2}$, and so on.

$(5)$ The supremum is $\Ycal = \bigcup_{i \in I} \Xcal_i$. We need to see that $\Ycal$ is a $\GBC$-realization for $M$. We know for free that $(M,\Ycal)$ satisfies Extensionality. It satisfies Global Choice because each $\Xcal_i$ contains a global well-order. To see that it satisfies Replacement, suppose $F \in \Ycal$ witnesses a failure of Class Replacement. But then $F \in \Xcal_i$ for some $i$ so $(M,\Xcal_i)$ fails to satisfy Class Replacement, a contradiction. To check Elementary Comprehension we want to see that $\Ycal$ is closed under first-order definability. Towards this fix $\bar A \in \Ycal$. Then $\bar A \in \Xcal_i$ for some $i$. This uses that the $\Xcal_i$ are linearly ordered by $\subseteq$. So if $B$ is definable from $\bar A$ then $B \in \Xcal_i \subseteq \Ycal$, as desired.

$(6)$ Combine $(5)$ and Zorn's lemma.
\end{proof}

Let me remark that $(3)$ of theorem \ref{thm1:gm-omnibus} gives us a criterion in terms of the theory of $M$ for when $\rlzn \GBC M$ has a least element.

\begin{corollary}
Let $M \models \ZFC$ be countable. Then, whether $\rlzn \GBC M$ has a least element is recognizable from the theory of $M$.
\end{corollary}

\begin{proof}
We saw that $\rlzn \GBC M$ has a least element if and only if $M$ has a definable global well order. This happens if and only if $M \models \exists x\ V = \HOD(\{x\})$.
\end{proof}

Theorem \ref{thm1:gm-omnibus}.$(4)$ previously appeared as (a special case of) lemma 3.1 of \cite{mostowski1976}. Mostowski moreover embeds the full binary tree of height $\omega_1$ into $\rlzn \GBC M$ for countable $M$, thereby concluding that $\card{\rlzn \GBC M} \ge 2^{\omega_1}$. 

What other orders can be embedded into $\rlzn \GBC M$? We will get to this question in time---see theorems \ref{thm1:embedding-into-gm} and \ref{thm1:star-embedding}. But first let us consider some local properties of $\rlzn \GBC M$. It will be useful to single out those $\Xcal \in \rlzn \GBC M$ which are countable. 

\begin{definition}
Let $M \models \ZFC$ be a countable model of set theory. Set 
\[
\GC(M) = \{ \Xcal \in \rlzn \GBC M : \Xcal \text{ is countable}\}.
\]
\end{definition}

With this definition in hand, we can now see that what was proved in $(4)$ of the previous theorem was really the following.

\begin{corollary} \label{cor1:non-amalg}
For any $\Xcal \in \GC(M)$ there are $\Ycal,\Zcal \supseteq \Xcal$ in $\GC(M)$ so that $\Ycal$ and $\Zcal$ lack an upper bound.
\end{corollary}

\begin{proof}
Carry out the same argument, but over $(M,\Xcal)$ instead of $(M,\Def(M))$. Generics can be found because $\Xcal$ is countable.
\end{proof}

Cohen forcing holds the key to establishing other local properties of $\rlzn \GBC M$.

\begin{theorem} \label{thm1:dense-extensions}
Every $\Xcal \in \GC(M)$ has a dense extension. That is, there is $\Ycal \in \GC(M)$ so that there is a dense linear order in $\GC(M)$ between $\Xcal$ and $\Ycal$. Consequently, the real line with its usual order embeds into $\rlzn \GBC M$ between $\Xcal$ and $\Ycal$.
\end{theorem}

\begin{proof}
Let $C$ be generic over $(M,\Xcal)$ for the forcing to add a Cohen-generic subclass of $\Ord$. Then $\Ycal = \Xcal[C]$ is a $\GBC$-realization for $M$ containing both $\Xcal$ and $C$.
Now take $X \subseteq \Ord^M$ in $\Xcal$ which is unbounded. It follows from the homogeneity of Cohen-forcing that
\[
C_X = \{ \alpha : \text{the $\alpha$th element of $X$ is in $C$} \}
\]
is Cohen-generic. (If you think of $C$ as an $\Ord$-length binary sequence then $C_X$ is the bits which appear in $X$, in order.) Clearly, $C_X \in \Ycal$. On the other hand, if $\Ord^M \setminus X$ is unbounded then $C \not \in \Xcal[C_X]$. The theorem now follows from the fact that there is a dense linear order of subclasses of $\Ord^M$ so that both they and their complements are unbounded. Namely, fix your favorite bijection $b$ between $\omega$ and the rationals. For a rational $q$, put $\alpha = \omega \cdot \alpha_0 + n$ into $X_q$ if and only if $b(n) < q$. Then each $X_q$ and its complement is unbounded and $X_q \subseteq X_{q'}$ if and only if $q < q'$. 

 Namely, let $k$ be a positive integer and $0 \le n < 2^k$. Set $X(k,n)$ to be those ordinals which are equivalent to $n$ modulo $2^k$. For example, $X(0,1)$ is the class of even ordinals and $X(1,1)$ is the class of odd ordinals. 

For the consequently, let $e$ be an embedding of the rationals into $\GC(M)$ between $\Xcal$ and $\Ycal$. For $r$ a real set $\Zcal_r = \bigcup_{q < r} e(q)$, which is a $\GBC$-realization for $M$ by theorem \ref{thm1:gm-omnibus}.$(5)$. Then $r \mapsto \Zcal_r$ gives an embedding of the real line into $\rlzn \GBC M$ between $\Xcal$ and $\Ycal$. 
\end{proof}

We saw in theorem \ref{thm1:gm-omnibus} that $\rlzn \GBC M$ has a least element if and only if $M$ has a definable well-order. This result can be improved.

\begin{theorem} \label{thm1:least-gbc}
Let $M$ be a countable model of $\ZFC$ and consider $\Xcal \in \GC(M)$. Then there is $\Ycal \in \GC(M)$ so that: 
\begin{itemize}
\item If $M$ has a definable global well-order then the only lower bound of $\Xcal$ and $\Ycal$ in $\rlzn \GBC M$ is $\Def(M)$.
\item If $M$ does not have a definable global well-order then $\Xcal$ and $\Ycal$ have no lower bound in $\rlzn \GBC M$.
\end{itemize}
\end{theorem}

\begin{proof}
Let $H$ be $(M,\Xcal)$-generic for $\Add(\Ord,1)$. Then $\Def(M;H) \in \GC(M)$, because meeting every dense class in $\Xcal$ implies meeting every dense class in $\Def(M)$. Set $\Ycal = \Def(M;H)$. Let us see that $\Ycal$ is as desired.

Consider $A \in \Xcal$ and assume that $A \in \Ycal = \Def(M;H)$. Then, there is some first-order formula $\phi$, possibly with set parameters and $H$ as a class parameter but with no other class parameters so that $(M,\Xcal[H]) \models \forall x\ x \in A \iff \phi(x,H)$. By the forcing theorem there is some $p \in H$ so that, in $(M,\Xcal)$, we have $p \forces \forall x\ x \in \check A \iff \phi(x,\dot H)$. So for all $x \in M$, we have $x \in A$ if and only if $(M;\Xcal) \models \text{``}p \forces \phi(\check x, \dot H)$''. But this formula does not depend upon $G$, so the same is true in $(M;\Def(M))$. Therefore, $A \in \Def(M)$. So we have seen that if $A \in \Xcal \cap \Ycal$ then $A \in \Def(M)$, from which the conclusion of the theorem immediately follows.
\end{proof}

Next we look at the opposite phenomenon from theorem \ref{thm1:dense-extensions}. Namely, every $\Xcal \in \GC(M)$ extends to some $\Ycal \in \GC(M)$ with nothing in between them. To do so, I will make use of a variant of Sacks forcing for adding new classes of ordinals.

Generalizing Sacks forcing to add generic classes of ordinals has been considered before. Kossak and Schmerl \cite{kossak-schmerl:book} considered what they call perfect generics, an adaptation of Sacks forcing to models of arithmetic. A set theoretic variant of their idea to add a subclass of $\Ord$ was considered by Hamkins, Linetsky, and Reitz \cite{hamkins:math-tea}. In the arithmetic case, perfect generics are used to produce minimally undefinable inductive sets over a model $M$ of arithmetic, i.e.\ inductive $G \subseteq M$ so that for $A \in \Def(M;G)$ either $A \in \Def(M)$ or $G \in \Def(M;A)$. A similar construction works in set theory to produce minimal extensions of countable models of $\GBC$.

\begin{theorem} \label{thm1:min-above}
Take $\Xcal \in \GC(M)$. Then there is a $\GBC$-realization $\Ycal \supsetneq \Xcal$ for $M$ which is minimal above $\Xcal$, in the sense that if $\Zcal \in \GC(M)$ with $\Xcal \subseteq \Zcal \subseteq \Ycal$ then either $\Zcal = \Xcal$ or $\Zcal = \Ycal$.
\end{theorem}

Let me sketch the main idea before giving a proof. The desired $\Ycal$ will be $\Def(M;\Xcal,G)$ where $G$ is a specially chosen subclass of $\Ord^M$. We need to ensure two things. First, we need that $\Def(M;\Xcal,G)$ is a $\GBC$-realization for $M$, so we need to ensure that we satisfy the Separation and Replacement schemata with $G$ and parameters from $\Xcal$. Second, we need to ensure that every new class codes $G$, so that $\Ycal$ is minimal above $\Xcal$. 

To achieve the first of these we will define $G$ to be the intersection of a certain $\subseteq$-descending sequence $\seq{\Pbb_n : n \in \omega}$ of $\Ord$-height perfect binary trees. To ensure that $G$ is a branch we will have that $\Pbb_n$ only splits above $\alpha_n$, where $\seq{\alpha_n : n \in \omega}$ is a fixed sequence cofinal in $\Ord^M$. That adjoining $G$ gives a $\GBC$-realization will be due to genericity properties of $G$. We can think of each of the $\Pbb_n$ as a forcing notion. While $G$ will not be fully generic over any $\Pbb_n$, it will be increasingly generic over each one. This will suffice.

\begin{definition}
Let $\Qbb \subseteq \Pbb$ be $\Ord$-height perfect binary trees. Say that {\em $\Qbb$ has the $\Sigma_n$-branch genericity property over $\Pbb$} if every branch through $\Qbb$ (not necessarily in the model) is $\Sigma_n$-generic over $\Pbb$, meaning that the branch meets every $\Sigma_n$-definable dense subclass of $\Pbb$.
\end{definition}

We will need a refinement of the forcing theorem for partial generics. Let me state what we need here, sans proof. Any $\Ord$-height perfect binary tree $\Pbb$, considered as a forcing notion, is tame. So we have that the forcing relation is definable. The refinement we need is that $\forces_\Pbb$ restricted to the $\Sigma_k$-formulae is $\Sigma_j$ definable for some $j > k$. This is proved via the usual proof of the definability lemma, but with careful bookkeeping of quantifiers. More generally, if we want $\forces_\Pbb$ restricted to the $\Sigma^0_k$-formulae in some class parameter $P$, then this is $\Sigma^0_j$ definable from that same parameter. 

The other half of the forcing theorem is the truth lemma, which asserts that if $G \subseteq \Pbb$ is generic then $(M,\Xcal)[G] \models \phi$ if and only if there is $p \in G$ so that $p \forces \phi$. The refinement is that this works in a restricted fashion for partial generics. That is, for each $k$ there is $j > k$ so that if $G$ is $\Sigma_k$-generic then for any $\Sigma_k$-formula $\phi$ we have $(M,\Xcal)[G] \models \phi$ if and only if there is $p \in G$ so that $p \forces \phi$. Again, this refinement comes from considering the usual proof of the truth lemma, but paying close attention to the quantifiers involved.

We may assume that the $j > k$ for both lemmata is the same, simply by taking the maximum of the two. To have a name for this $j$, call it {\em the $\Sigma_k$-genericity witness}.

\begin{lemma}
Work with $(M,\Xcal) \models \GBC$. Suppose $\seq{\Pbb_n : n \in \omega}$ is a descending sequence of $\Ord$-height perfect binary trees so that $\Pbb_{n+1}$ has the $\Sigma_n$-branch genericity property over $\Pbb_n$. Let $G = \bigcup \bigcap_n \Pbb_n$. Then $(M,\Def(M;\Xcal,G)) \models \GBC$.
\end{lemma}

\begin{proof}
Suppose towards a contradiction that $F \in \Def(M;\Xcal,G)$ witnesses a failure of Replacement. Then $F$ is defined from $G$ and parameters from $\Xcal$ via a $\Sigma_n$ formula. Now let $j > n$ be the $\Sigma_n$-genericity witness. So there is some $p \in G$ so that $p \forces_{\Pbb_j} \text{``}F$ witnesses a failure of Replacement''. Now take $H \in \Xcal$ which is $\Sigma_j$-generic over $\Pbb_j$; such exists because $\Sigma_j$-truth is definable. Carry out the definition of $F$ but use $H$ instead of $G$. We get that $F \in \Def(M;\Xcal,H)$ witnesses a failure of Replacement, because $H$ is sufficiently generic. But $\Def(M;\Xcal,H) = \Xcal$ because $H \in \Xcal$. So $(M,\Xcal)$ is not a model of $\GBC$, a contradiction.
\end{proof}

Before seeing how to ensure that $\Ycal$ is minimal above $\Xcal$, let us see that we can always arrange such a sequence $\seq{\Pbb_n}$ of perfect binary trees.

\begin{lemma}
Let $(M,\Xcal) \models \GBC$ be countable. Then there is $\seq{\Pbb_n}$ of $\Ord$-height perfect binary trees from $\Xcal$ so that $\bigcup \bigcap_n \Pbb_n$ is a class of ordinals and $\Pbb_{n+1}$ has the $\Sigma_n$-branch genericity property over $\Pbb_n$.
\end{lemma}

This sequence will not be (coded) in $\Xcal$, though each tree in the sequence will be in $\Xcal$.

\begin{proof}
Fix $\seq{\alpha_n : n \in \omega}$ cofinal in $\Ord^M$. Start with $\Pbb_0 = {}^{<\Ord}2$ the full binary tree. Assume we have already found $\Pbb_n \in \Xcal$. We define $\Pbb_{n+1}$ by defining a certain embedding $g : {}^{<\Ord}2 \to \Pbb_n$. Closing $\ran g$ downward in $\Pbb_n$ will give $\Pbb_{n+1}$. We define $g$ by a set-like recursion of height $\Ord$. Fix in advance a global well-order. Let $g(0)$ be the first node according to this well-order which has length $\ge \alpha_n$. This will ensure that the intersection of the $\Pbb_n$ gives a branch. At limit stages, take unions. If we have already defined $g(s)$, then let $t$ be the least, according to the global well-order, splitting node in $\Pbb_n$ which extends $g(s)$ and decides the $\length s$-th instance (according to the global well-order) of the universal $\Sigma_n$-formula. Then set $g(s \cat i) = t \cat i$. 

It is clear that $\Pbb_{n+1} \in \Xcal$, because we defined it from parameters from $\Xcal$. It is also clear that $\Pbb_{n+1}$ is a perfect tree. It has the $\Sigma_n$-branch genericity property over $\Pbb_n$ because any $\Sigma_n$-formula is decided by a long enough node in $\Pbb_{n+1}$.
\end{proof}

It remains to see how to ensure the minimality of the extension. This is encapsulated by the following lemma, which is a set theoretic counterpart to a result from section 6.5 of \cite{kossak-schmerl:book}.

\begin{lemma}[Minimality lemma] \label{lem:per-gen-min}
Let $\phi(x)$ be a formula in the forcing language and $\Pbb \in \Xcal$ be a perfect subtree of the full binary tree. Then there is $\Qbb \subseteq \Pbb$ in $\Xcal$ so that one of the two cases holds:
\begin{enumerate}
\item There is an ordinal $\alpha$ so that for all ordinals $\xi$ we have that all $p \in \Qbb$ of length greater than $\alpha$ decide $\phi(\check \xi)$ (in $\Pbb$) the same.

\item For every ordinal $\alpha$ there is $\beta > \alpha$ so that if $p,q \in \Qbb$ both have length $\beta$ and $p \rest \alpha = q \rest \alpha$ then there is an ordinal $\xi$ so that $p$ and $q$ decide $\phi(\xi)$ differently (in $\Pbb$).
\end{enumerate}
\end{lemma}

\begin{proof}
Fix $k$ so that $\phi$ is a $\Sigma_k$ formula. Take $\Pbb' \subseteq \Pbb$ a $k$-deciding subtree for $\Pbb$. We may assume that there is a function $f : \Bbb \to \Pbb'$ which embeds the full binary tree onto the splitting nodes of $\Pbb'$ and that $f(s)$ decides $\phi(\length s)$. There are two cases. The first is that there is some $s \in \Bbb$ so that for every $t, t' >_\Bbb s$ if $\length t = \length t'$ then $f(t)$ and $f(t')$ decide $\phi(\xi)$ the same for all ordinals $\xi$. In this case, set $Q = \Pbb' \rest f(s)$ and get the first conclusion in the lemma.

The second case is that this does not happen for any $s \in \Bbb$. In this case, we can inductively define a $g : \Bbb \to \Pbb'$ as follows:
\begin{itemize}
\item Set $g(0) = f(0)$.
\item Set $g(s \cat 0) = p_0$ and $g(s \cat 1) = p_1$, where $p_0, p_1$ are least (according to a fixed global well-order) so that $\length p_0 = \length p_1$ and there is an ordinal $\xi$ so that $p_0$ and $p_1$ decide $\phi(\xi)$ differently. Such $p_0$ and $p_1$ always exist, as otherwise we would be in the previous case.
\item At limit stages take unions.
\end{itemize}
Set $\Qbb = \{ p \in \Pbb' : \exists s \in \Bbb\ p \le_{\Pbb'} g(s) \}$. This yields the second conclusion in the lemma.
\end{proof}

Observe that we used global choice in an essential manner here. There are possibly many choices for $p_0$ and $p_1$ in the successor stage of the construction of $g$. In order to guarantee that $g \in \Xcal$ and hence that $\Qbb \in \Xcal$, we need to uniquely specify a choice.

\begin{proof}[Proof of theorem \ref{thm1:min-above}]
Work with countable $(M,\Xcal) \models \GBC$. Fix a cofinal sequence $\seq{\alpha_n}$ of ordinals and an enumeration $\seq{\phi_n(x)}$ of formulae in the forcing language. We construct a descending sequence of perfect trees
\[
{}^{<\Ord}2 = \Qbb_0 \supseteq \Pbb_0 \supseteq \Qbb_1 \supseteq \Pbb_1 \supseteq \cdots \supseteq \Qbb_n \supseteq \Pbb_n \supseteq \cdots
\]
so that $\Pbb_n$ has the $\Sigma_n$-branch genericity property over $\Qbb_n$ and does not split below $\alpha_n$ and $\Qbb_{n+1} \subseteq \Pbb_n$ is as in the previous lemma for $\phi_n$. Set $G = \bigcup (\bigcap_n \Pbb_n)$ and $\Ycal = \Def(M;\Xcal,G)$. Then $\Ycal \supseteq \Xcal$ is a $\GBC$-realization for $M$. 

Now suppose $\Zcal$ is a $\GBC$-realization for $M$ with $\Xcal \subseteq \Zcal \subseteq \Ycal$. We want to see that either $\Zcal = \Xcal$ or $\Zcal = \Ycal$. It is enough to see that if $A \in \Ycal$ then either $G$ is definable from $A$ and parameters in $\Xcal$ or else $A \in \Xcal$. Without loss of generality we may assume that $A$ is a class of ordinals. Then it was defined by some formula $\phi_n$ in our enumeration. 

Consider $\Qbb_{n+1} \subseteq \Pbb_n$. If the first case from the minimality lemma holds, then $A \in \Xcal$ because $\xi \in A$ if and only if for every $p \in \Qbb_{n+1}$ the length of $p$ being sufficiently long implies that $p \forces_{\Pbb_n} \phi_n(\xi)$. If the second case of the previous lemma holds, then we can define $G$ from $A$. In this case, $p \in \bigcap_n \Pbb_n$ if and only if for every ordinal $\alpha$ there is $q >_{\Qbb_{n+1}} p$ of length greater than $\alpha$ so that $q \forces_{\Pbb_n} \phi_n(\xi) \iff \xi \in A$ for all ordinals $\xi$. From a definition of $\bigcap_n \Pbb_n$ can easily be produced a definition for $G$.
\end{proof}

As remarked earlier, Global Choice was used essentially in the proof of lemma \ref{lem:per-gen-min}. Proving this lemma without Global Choice would yield a construction for minimal but not least $\GBC$-realizations. Namely, start with a countable $M \models \ZFC$ with no definable global well-order. Let $\Xcal = \Def(M)$. Then $(M,\Xcal)$ is a model of $\GBC$ minus Global Choice. Applying the theorem to $(M,\Xcal)$ would yield a $\GBC$-realization $\Ycal$ for $M$ which is minimal above $\Xcal$. But since any $\GBC$-realization must contain $\Xcal$, this would give that $\Ycal$ is a minimal $\GBC$-realization for $M$. 

Thus, the problem of constructing a minimal but not least $\GBC$-realization can be reduced down to the problem of proving the minimality lemma without using choice. A similar question can be asked for ordinary Sacks forcing.

\begin{question}
Is choice needed to prove the minimality lemma for Sacks forcing? That is, is it consistent that there are $M \models \ZF + \neg \AC$, $s \subseteq \omega^M$ Sacks-generic over $M$, and $A \in M[s]$ so that $M \subsetneq M[A] \subsetneq M[s]$?
\end{question}

This technique can also be applied to study principal models of $\GBC$.

\begin{definition}
Say that $(M,\Xcal)$ is a {\em principal} model if there is $A \in \Xcal$ so that $\Xcal = \Def(M;A)$. Let $\GP(M)$ denote the collection of principal $\GBC$-realizations for $M$. Note that $\GP(M) \subseteq \rlzn \GBC M$ and if $M$ is countable then $\GP(M) \subseteq \GC(M)$. Like $\rlzn \GBC M$ and $\GC(M)$, $\GP(M)$ is ordered by $\subseteq$.
\end{definition}

\begin{theorem}[S. Friedman, Kossak--Schmerl] \label{thm1:principal-extension}
For $M$ countable, $\GP(M)$ is dense in $\GC(M)$. That is, given any countable $\GBC$-realization $\Xcal$ for $M$ there in $\Ycal \supseteq \Xcal$ which is a principal $\GBC$-realization for $M$.
\end{theorem}

This is theorem 15 of \cite{hamkins:math-tea}. Hamkins, Linetsky, and Reitz attribute the result independently to Friedman, via private communication, and Kossak and Schmerl for the finite set theory (equivalently, arithmetic) case. 
I will not give a full argument here, but let me sketch the proof.

\begin{proof}[Proof sketch]
The argument is similar to the proof of theorem \ref{thm1:min-above}. We want to get a sequence
\[
{}^{<\Ord}2 = \Qbb_0 \supseteq \Pbb_0 \supseteq \Qbb_1 \supseteq \Pbb_1 \supseteq \cdots \supseteq \Qbb_n \supseteq \Pbb_n \supseteq \cdots
\]
so that $\Qbb_{n+1}$ has the $\Sigma_n$-branch genericity property over $\Pbb_n$. Before, we used the minimality lemma to produce $\Pbb_n$ from $\Qbb_n$. Here we need a different lemma.

\begin{lemma}
Work over $(M,\Xcal) \models \GBC$. Let $\Qbb$ be an $\Ord$-height perfect binary tree and $A \in \Xcal$ be a class of ordinals. Then, there is $\Pbb \subseteq \Qbb$ in $\Xcal$ so that from $\Qbb$ and any branch through $\Pbb$ we can define $A$.
\end{lemma}

\begin{proof}[Proof sketch]
Using $A$ we thin out $\Pbb$, keeping every other splitting node. This ensures that the tree we get at the end is still perfect. Reaching the $2i$-th splitting node along a branch, we either go left or right. We go left if $i \in A$ and go right if $i \not \in A$. This gives $\Qbb$. From $\Pbb$ we know where the $2i$-th splitting nodes along a branch are. From a branch through $\Qbb$ we know whether we went left or right to define $\Qbb$ and thus whether $i \in A$. So we can define $A$.
\end{proof}

To get theorem \ref{thm1:principal-extension} we line up the classes in $\Xcal$ in ordertype $\omega$---externally to the model. Then, to define $\Pbb_n$ from $\Qbb_n$ we apply the lemma to code the $n$-th class. At the end, the $G$ we get will allow us to define every set in $\Xcal$, and thus $\Def(M;\Xcal,G) = \Def(M;G) \in \GP(M)$.
\end{proof}

We saw above that the rationals embed into $\rlzn \GBC M$. Indeed, given any $\Xcal \in \GC(M)$ we can embed rationals into $\rlzn \GBC M$ above $\Xcal$. This can be generalized to any countable partial order.

\begin{theorem} \label{thm1:embedding-into-gm}
Every countable partial order embeds into $\rlzn \GBC M$, for countable $M \models \ZFC$. More generally, for every finite partial order $P$ and every $\Xcal \in \GC(M)$ there is an embedding $e : P \to \rlzn \GBC M$ which maps $P$ above $\Xcal$---that is, $e(p) \supseteq \Xcal$ for all $p \in P$.
\end{theorem}

This theorem is similar to the result that every countable partial order embeds into the Turing-degrees. See below for further discussion.

\begin{proof}
Let us first see the special case of an atomic boolean algebra. Let $(B,<^B)$ be a countable atomic boolean algebra, with atoms $b_0, b_1, \ldots$. Let $C_0, C_1, \ldots$ be mutually $(M,\Def(M,\Xcal))$-generic Cohen subclasses of $\Ord$. Then the map $b_i \mapsto \Def(M;\Xcal,C_i)$ induces an embedding of $B$ into $\rlzn \GBC M$; given arbitrary $b \in B$ map $b$ to $\Def(M;\Xcal,C_i : b_i \le^B b)$. 

So we are done once we see that every countable partial order embeds into a countable atomic boolean algebra. Let $(P,<^P)$ be a countable partial order. The desired boolean algebra $(B,<^B)$ will be generated by countably many atoms, with an atom $a(p)$ associated to each $p \in P$. For the embedding $e$, map $p$ to the unique $b \in B$ so that $a(q) \le^B b$ if and only if $q \le^P p$. Then $q \le^P p$ if and only if $e(q) \le^B e(p)$.
\end{proof}

The embedding from this argument destroys a lot of information about the partial order. It may be that $p,q \in P$ have no upper bound. But their image under the embedding will have an upper bound, as we first embed $P$ into a boolean algebra---where all pairs of elements have an upper bound---and then embed that boolean algebra into $\rlzn \GBC M$. But recall from corollary \ref{cor1:non-amalg} that there are $\Xcal,\Ycal \in \rlzn \GBC M$ without an upper bound. 

Can we do better and get embeddings that preserve the non-existence of upper  bounds?

\begin{theorem}[Mostowski \cite{mostowski1976}] \label{thm1:star-embedding}
Let $(M,\Xcal) \models \GBc$ be countable and $F$ be a finite family of finite sets, closed under subset.\footnote{Such $F$ are precisely those orders which are initial segments of a finite boolean algebra of sets.}
Then there are $(M,\Xcal)$-generic Cohen subclasses $C_i$ of $\Ord$ for each $i \in \bigcup F$ and an embedding $e$ from $F$ to $\rlzn \GBC M$ satisfying the following.
\begin{enumerate}
\item If $f \in F$ then the $C_i$ for $i \in f$ are mutually generic; and
\item If $f \subseteq \bigcup F$ is not in $F$ then the $C_i$ for $i \in f$ do not amalgamate: there is no $\GBC$-realization $\Ycal$ for $M$ which contains $C_i$ for all $i \in f$.
\end{enumerate}
\end{theorem}

Before giving the proof let me extract a corollary.

\begin{corollary}
Let $M \models \ZFC$ be countable.
\begin{enumerate}
\item Suppose $M$ has a definable global well-order and let $P$ be a finite partial order with a least element $\zero^P$. Then there is an embedding $e$ of $P$ into $\rlzn \GBC M$ so that $e(\zero) = \Def(M)$ and $e$ preserves the existence/nonexistence of upper bounds and nonzero lower bounds.
\item Suppose that $M$ does not have a definable global well-order and let $P$ be a finite partial order without a least element. Then there is an embedding $e$ of $P$ into $\rlzn \GBC M$ which preserves the existence/nonexistence of upper bounds and lower bounds.
\item Let $\Xcal \in \GC(M)$ and let $P$ be a finite partial order with a least element $\zero^P$ Then there is an embedding $e$ of $P$ into $\rlzn \GBC M$ above $\Xcal$ so that $e(\zero) = \Xcal$ and $e$ preserves existence/nonexistence of upper bounds and nonzero lower bounds. That is, $e(p) \supseteq \Xcal$ for all $p \in P$ and $p$ and $p'$ have a nonzero lower bound if and only if $e(p)$ and $e(p')$ have $\Xcal$ as their greatest lower bound.
\end{enumerate}
\end{corollary}

\begin{proof}
In each case, we will first embed $P$ into a finite family $F$ of finite sets closed under subset, in such a way that the embedding preserves the existence/non-existence of upper bounds and nonzero lower bounds. We will apply theorem \ref{thm1:star-embedding} using $F$. For $(1)$ and $(2)$ we will work over $(M,\Def(M))$ while for $(3)$ we will work over $(M,\Xcal)$. The preservation of upper bounds will be ensured by mutual genericity, which will also ensure the condition about lower bounds. Finally, the preservation of nonexistence of upper bounds follows from non-amalgamability.

It remains to see the embedding of $P$ into $F$.  First embed $(P,<^P)$ into a boolean algebra of sets $B$ by the following $e$: the boolean algebra is generated by an atom $a(p)$ for each $p \in P$ and we map $p$ to $\{ a(q) : q \le p \}$. Let $F$ be the downward closure of $e''P$ in $B$. Then $e(p)$ and $e(q)$ have an upper bound in $F$ if and only if there is $r \ge^P p,q$. And two sets $f,g \in F$ have a nonzero lower bound if and only if $f \cap g \ne \emptyset$ which happens if and only if $\{i\} \subseteq f \cap g$ for some $i \in \bigcup F$. So if $f = e(p)$ and $g = e(q)$ then $e(p)$ and $e(q)$ have a nonzero lower bound if and only if they both contain $a(r)$ for some $r \in P$ if and only if $r \le p$ and $r \le q$. Thus, we have seen that $e$ is as desired.
\end{proof}

\begin{proof}[Proof of theorem \ref{thm1:star-embedding}]
Without loss of generality $\bigcup F = n \in \omega$. And if $n \in F$ then the result is trivial---merely add $n$ mutually-generic Cohen subclasses of $\Ord$. So assume we are the case where $n \not \in F$.

As in the non-amalgamability argument for theorem \ref{thm1:gm-omnibus}.$(4)$, fix a bad sequence $B: \Ord^M \to 2$ which witnesses that $\Ord^M$ is countable. We want to construct the Cohen generics $C_i$ for $i \in n$ so that for $f \subseteq n$ the Cohen generics $\{ C_i : i \in f \}$ code $B$ if and only if $f \not \in F$. We construct the $C_i$ in $\Ord^M$ many stages. Externally to the model, fix an $\Ord^M$-sequence of the dense subclasses in $\Xcal$ of $\Add(\Ord,\card f)$ for some $f \in F$. We can arrange this so that each stage $\alpha$ has a corresponding $f_\alpha \in F$ and for each $f \in F$ each dense subclass of $\Add(\Ord,\card f)$ appears at some stage $\alpha$ with $f = f_\alpha$.

Start with $c_i^0 = \emptyset$ for all $i \in f$. At limit stages, we will simply take unions. All the work is in the successor stage. Suppose we have already built $c_i^\alpha$ for all $i \in f$. Let $D \subseteq \Add(\Ord,\card{f_\alpha})$ be the dense class for stage $\alpha$. We can extend $c_i^\alpha$ to $d_i^\alpha$ for $i \in f$ so that $\prod_{i \in f} d_i^\alpha$ meets $D$. By padding out with $0$s if necessary, we may assume without loss that the $d_i^\alpha$'s all have the same length, call it $\gamma$. Now, for $i \not \in f$ extend $c_i^\alpha$ to $d_i^\alpha$ of length $\gamma$ by adding $0$s everywhere new. Finally, set $c_i^{\alpha+1} = d_i^\alpha \cat \seq{1,B(\alpha)}$.

Then $C_i = \bigcup_{\alpha \in \Ord^M} c_i^\alpha$ is Cohen-generic. And it is clear from the construction that $\{ C_i : i \in f \}$ is a family of mutually-generic Cohen subclasses of $\Ord$ for $f in F$. It remains to see that if $f \not \in F$ then $\{ C_i : i \in f \}$ codes $B$. This is because, $\{ C_i : i \in f \}$ can recognize the coding points. They occur just after the rows of all $1$s. That is, $B(\alpha) = C_i(\xi_\alpha+1)$ (any $i \in f$) where $\xi_\alpha$ is the $\alpha$th $\xi$ so that $C_i(\xi) = 1$ for all $i \in f$.
\end{proof}

Many questions remain open about the structure of $\rlzn \GBC M$. Let me mention one open-ended question.

\begin{question}
What can be said about the theory of the structure $(\rlzn \GBC M,\subseteq)$? What if we add in predicates for collection of the the countable $\GBC$-realizations or the collection of the principal $\GBC$-realizations?
\end{question}

This project of studying the order structure of $\rlzn \GBC M$ has similarities to two extant projects in mathematical logic. Let me briefly mention them and the connections.

The first and older of the two is the study of the Turing degrees under the order of Turing-reducibility. The structure of this partial order has been well studied and many results about the Turing degrees have counterparts in the context of $\rlzn \GBC M$. For instance, it is known that every countable poset embeds into the Turing degrees. 

The reader should be warned, however, that there are differences between the two. Let me illustrate this with the exact pair theorem as an example.

\begin{theorem}[Spector \cite{spector1956}]
Let $\seq{\dbf_n : n \in \omega}$ be a sequence of Turing degrees so that $\dbf_n <_\Trm \dbf_{n+1}$ for all $n$. Then there are Turing degrees $\abf$ and $\bbf$ which are above each $\dbf_n$ but if $\cbf <_\Trm \abf, \bbf$ then $\cbf <_\Trm \dbf_n$ for some $n$.  
\end{theorem}

The analogous result is not true for $\GBC$-realizations for a fixed countable $M \models \ZFC$. Let $\seq{\Xcal_n : n \in \omega}$ be an increasing $\subseteq$-chain of $\GBC$-realizations for $M$ and let $\Ycal, \Zcal \in \rlzn \GBC M$ be upper bounds for the sequence. Then $\bigcup \Xcal_n \in \rlzn \GBC M$ is below both $\Ycal$ and $\Zcal$ but above every $\Xcal_n$. The dis-analogy is because $\GBC$-realizations do not have to be generated by a single class. This suggests that the correct analogy is between $\GBC$-realizations and Turing ideals. But even then, the analogy is not perfect. For example, Turing ideals do not have a non-amalgamability phenomenon, as any collection of Turing ideals are all contained in the Turing ideal $\powerset(\omega)$. 

The second connection is to the generic multiverse. Given a countable transitive\footnote{One does not need the assumption of transitivity, but let me leave it in to simplify the discussion.}
model $M$ of set theory the generic multiverse of $M$ is the smallest collection of countable transitive models containing $M$ which is closed under (set) forcing extensions and (set) grounds.\footnote{A {\em ground} of $M$ is a submodel $W$ so that $M = W[g]$ for some $g$ generic for some forcing notion in $W$.}
It follows from work by Usuba \cite{usuba2017} that the generic multiverse of $M$ can be equivalently defined as the collection of all forcing extensions of grounds of $M$.

Let $\Mcal$ be the generic multiverse of $M$. Then $\Mcal$ is partially ordered under inclusion and we can ask about the order-theoretic properties of $\Mcal$. Many of the arguments about Cohen-generic subclasses of $\Ord$ also apply for Cohen-generic subsets of, say, $\omega$ which yields similar results for the generic multiverse as it does for the $\GBC$-realizations. See e.g.\ \cite{hamkins2016}.

Let me conclude by considering to what extent these results generalize from $\GBC$ to other theories.

First, we consider what happens if we drop powerset. As mentioned in section \ref{sec1:preserving-axioms}, forcing with $\Add(\Ord,1)$ may add sets for some models of $\ZFCm$. So most of the techniques of this section fail badly for those models. However, if $M \models \ZFCm$ is such that forcing with $\Add(\Ord,1)$ does not add sets, then the same arguments go through and we get all the same results for $\rlzn \GBCm M$.

Next consider stronger theories. First, let us see that $\rlznp T M$ may be empty.

\begin{proposition}
Let $T \supseteq \GBCm$ prove that for every class $A$ the first-order truth predicate relative to $A$ exists. Then $T$ has no principal models.
\end{proposition}

\begin{proof}
Suppose otherwise that $(M,\Xcal) \models T$ has that every class is definable from $A \in \Xcal$. But then the truth predicate relative to $A$ is definable from $A$, contradicting Tarski's theorem on the undefinability of truth.
\end{proof}

In particular, there are no principal models of theories extending $\GBCm + \ETR$, or even $\GBCm + \ETR_\omega$. So for most theories $T$ of interest, $\rlznp T M$ is trivial. But we can say something about $\rlzn T M$ and $\rlznc T M$.

The main tool used in this section was Cohen-generic subclasses of $\Ord$. We carefully constructed generics to have certain properties with regard to amalgamability/non-amalgamability, and thereby concluded something about the structure of $\rlzn \GBC M$. Some of these results generalize to $\rlzn T M$ and $\rlznc T M$, for $T$ which is preserved by Cohen-forcing.

\begin{theorem}
Let $M \models \ZFC$ be countable and let $T$ be a second-order set theory which is preserved by forcing with $\Add(\Ord,1)$ over $M$.\footnote{For example, $T$ could be $\KM$, $\KMCC$, $\GBC + \PnCA k$, or $\GBC + \PnCAp k$.}
Suppose $M$ is $T$-realizable. Then $\rlzn T M$ satisfies the following.
\begin{enumerate}
\item For any $\Xcal \in \rlznc T M$ there are $\Ycal,\Zcal \supseteq \Xcal$ in $\rlznc T M$ so that $\Ycal$ and $\Zcal$ lack an upper bound.
\item Every $\Xcal \in \rlznc T M$ has a dense extension. That is, there is $\Ycal \in \rlznc T M)$ so that there is a dense linear order in $\rlzn T M$ between $\Xcal$ and $\Ycal$. 
\item For every countable partial order $P$ and every $\Xcal \in \rlznc T M$ there is an embedding $e : P \to \rlzn T M$ which maps $P$ above $\Xcal$.
\item Every finite partial order embeds into $\rlzn T M$ in such a way as to preserve the existence/nonexistence of upper bounds and nonzero lower bounds. 
\end{enumerate}
\end{theorem}

\begin{proof}[Proof sketch]
The same arguments as in the $T = \GBC$ case.
\end{proof}

However, not all arguments generalize. In particular, the argument for when $\rlzn \GBC M$ has a least element will not work for theories stronger than $\GBC$. Indeed, in chapter 4 we will see that the result is false for sufficiently strong theories. If $T \supseteq \GBC + \PCA$ then $\rlzn T M$ never has a least element, for countable $M$. 

But we are not yet ready to prove this. First we need to know more about the structure of models of strong second-order set theories, which we turn to in chapter 2.

\clearpage

\chapter{\texorpdfstring{Many constructions: unrollings, cutting offs, and second-order $L$}{Many constructions: unrollings, cutting offs, and second-order L}}
\chaptermark{Many constructions}

\epigraph{\singlespacing Second-order logic is set theory in disguise.}{W.V.O. Quine}

This chapter contains an exposition of several constructions relating to models of strong second-order set theories. These constructions are not new, dating at least as far back as work by Marek and Mostowski in the 1970s \cite{marek1973,marek-mostowski1975}. The ultimate origin of the constructions is not clear to me. Indeed, in the introduction to \cite{marek1973} Marek claims that Jensen, Mostowski, Solovay, and Tharp all independently proved that $\KM$ is consistent with $V = L$, which will follow from the third construction considered in this chapter, viz.\ the second-order constructible universe.\footnote{Marek does not mention how they proved the result, but it seems likely that it was via a version of this construction.}
More recently, Antos and S.\ Friedman \cite{antos-friedman2015} independently rediscovered these constructions a few years ago.

Previously these constructions have been done in the context of $\KM$ or $\KMCC$. In this chapter I analyze them more finely, generalizing their application from models of $\KM$ or $\KMCC$ to models of weaker theories.

The first construction I call the unrolling construction. This is essentially the same construction as the one used to code hereditarily countable sets as reals, which has seen wide use within set theory. The point is that the same idea works for sets which are not hereditarily countable. We can code `sets' of rank $> \Ord$ as proper classes. Then, given that our universe satisfies a strong enough second-order set theory these codes can be unrolled to produce a model of first-order set theory. The theory of this unrolled model will depend upon the theory of the ground universe.



The reader who is familiar with reverse mathematics may know that a similar construction is used for theories of second-order arithmetic. This yields that they are bi-interpretable with certain (first-order) set theories. See \cite[chapter VII]{simpson:book} for a thorough exposition.

The second construction is the cutting off construction. This construction starts with a model of $\ZFCm$ (or weaker theory) with a largest cardinal $\kappa$ with $\kappa$ regular and $H_\kappa \models \ZFCm$. We then get a second-order model by considering $(H_\kappa,\powerset(H_\kappa))$ in this model, where $\powerset(H_\kappa)$ is necessarily a (definable) proper class in the model. This will yield a model of $\GBCm$, with more strength coming from a stronger theory in the ground model. If $\kappa$ is moreover inaccessible we will get a model of $\GBC$, or more.

The cutting off construction is exactly the inverse of the unrolling construction. Starting with $(M,\Xcal)$ a model of a sufficiently strong second-order set theory, the cutting off of the unrolling of $(M,\Xcal)$ is isomorphic to $(M,\Xcal)$. In the other direction, start with $N$ a model of a strong enough fragment of $\ZFCm$ with a largest cardinal $\kappa$ with $\kappa$ regular and $H_\kappa^N \models \ZFCm$. Then the unrolling of the cutting off of $N$ is isomorphic to $N$.

Together, these two constructions yield that strong enough second-order set theories are bi-interpretable with certain first-order set theories. I summarize these bi-interpretability results below. First, however, we will need names for the first-order set theories we get from unrolling.

\begin{definition} \label{def2:so-many-theories}
The following are the first-order set theories theories which are satisfied by the unrolled models arising from a model of second-order set theory. Each includes the basic axioms of set theory---namely Extensionality, Pairing, Union, Foundation, Choice, and Infinity---and the assertion that there is a largest cardinal $\kappa$.
\begin{itemize}
\item $\ZFCmi$ consists of the basic axioms plus Separation, Collection, and the assertion that $\kappa$ is inaccessible. To be clear, since this theory does not include Powerset, by ``$\kappa$ is inaccessible'' is meant that $\kappa$ is regular and every set in $V_\kappa$ has a powerset which is also in $V_\kappa$. In particular, $\ZFCmi$ proves that $V_\kappa = H_\kappa$ is a model of $\ZFC$.
\item $\ZFCmr$ consists of the basic axioms plus Separation, Collection, the assertion that $\kappa$ is regular, and the assertion that $H_\kappa$ exists. In particular, $\ZFCmr$ proves that $H_\kappa$ is a model of $\ZFCm$.\footnote{$\ZFCmr$ has natural models, for instance $H_{\omega_2}$. In general, if $\kappa$ is regular then $H_{\kappa^+} \models \ZFCmr$.}
\item $\wZFCmi$ consists of the basic axioms plus Separation and the assertion that $\kappa$ is inaccessible. In particular, $\wZFCmi$ proves that $V_\kappa = H_\kappa$ is a model of $\ZFC$.
\item $\wZFCmr$ consists of the basic axioms plus Separation, the assertion that $\kappa$ is regular, and the assertion that $H_\kappa$ exists. In particular, $\wZFCmr$ proves that $H_\kappa$ is a model of $\ZFCm$.
\end{itemize}
Let $k < \omega$.
\begin{itemize}
\item $\ZFCmi(k)$ consists of the basic axioms plus $\Sigma_k$-Separation, $\Sigma_k$-Collection, and the assertion that $\kappa$ is inaccessible. In particular, $\ZFCmi(k)$ proves that $V_\kappa = H_\kappa$ is a model of $\ZFC$.
\item $\ZFCmr(k)$ consists of the basic axioms plus $\Sigma_k$-Separation, $\Sigma_k$-Collection, the assertion that $\kappa$ is regular, and the assertion that $H_\kappa$ exists. In particular, $\ZFCmr(k)$ proves that $H_\kappa$ is a model of $\ZFCm$.
\item $\wZFCmi(k)$ consists of the basic axioms plus $\Sigma_k$-Separation and the assertion that $\kappa$ is inaccessible. In particular, $\wZFCmi(k)$ proves that $V_\kappa = H_\kappa$ satisfies every axiom of $\ZFC$.
\item $\wZFCmr(k)$ consists of the basic axioms plus $\Sigma_k$-Separation, the assertion that $\kappa$ is regular, and the assertion that $H_\kappa$ exists. In particular, $\wZFCmr(k)$ proves that $H_\kappa$ satisfies every axiom of $\ZFCm$.
\end{itemize}

Let me explain the mnemonic behind the names of these theories for the benefit of the reader, to whom I apologize for giving eight theories to remember. The subscripts tell you what is being asserted about $\kappa$, the largest cardinal. `I' reminds you that $\kappa$ is {\bf i}naccessible while `R' tells you $\kappa$ is merely {\bf r}egular. The $\mathsf{w}$ in front stands for {\bf w}eak, {\bf w}impy, and {\bf w}hy would you ever want to work with a theory which does not have even a fragment of Collection?\footnote{For an extensive case study in why one would want Collection, see \cite{mathias2001}.}\textsuperscript{,}\footnote{So under this naming system $\wZFC$, although not used here, would be Zermelo set theory plus Foundation and Choice. Unfortunately the natural name for this theory, $\ZFC$, is already used to refer to Zermelo set theory plus Foundation, Choice, and Collection.}
The parenthetical $k$ tells us what fragment of Separation and Collection---or just Separation in case $\mathsf{w}$ is in front---is in the theory. 
\end{definition}

\begin{theorem} \label{thm2:bi-int}
The following pairs of theories are bi-interpretable. Below, $k \ge 1$.
\begin{itemize}
\item {\em (Marek \cite{marek1973})} $\KMCC$ and $\ZFCmi$.
\item $\KMCCm$ and $\ZFCmr$.
\item $\KM$ and $\wZFCmi + \Sigma_0$-Transfinite Recursion.\footnote{$\Sigma_0$-Transfinite Recursion, which means what you think it means, will be formally defined in subsection \ref{subsec2:how-hard}.}
\item $\KMm$ and $\wZFCmr + \Sigma_0$-Transfinite Recursion.
\item $\GBC + \PnCAp k$ and $\ZFCmi(k)$.
\item $\GBCm + \PnCAp k$ and $\ZFCmr(k)$.
\item $\GBC + \PnCA k$ and $\wZFCmi(k) + \Sigma_0$-Transfinite Recursion.
\item $\GBCm + \PnCA k$ and $\wZFCmr(k) + \Sigma_0$-Transfinite Recursion.
\item $\GBC + \ETR$ and $\wZFCmi(0) + \Sigma_0$-Transfinite Recursion.
\item $\GBCm + \ETR$ and $\wZFCmr(0) + \Sigma_0$-Transfinite Recursion.
\end{itemize}
\end{theorem}

The third construction I look at in this chapter is the construction of G\"odel's constructible universe but extended from the sets to the classes. Given a class well-order $\Gamma$ we can define $L_\Gamma$ and then consider the definable hyperclass $\Lcal$ consisting of all classes which appear in some $L_\Gamma$. We will consider also the construction of $L$ relative to parameters. The main use to which we will put this construction is in showing how to get models satisfying (a fragment of) Class Collection. The classical result here is that if we start with $(M,\Xcal)$ a model of $\KM$ then the $\Lcal$ we build gives a model of $\KMCC$ \cite{marek1973}. That is, any model of $\KM$ contains an $\Ord$-submodel of $\KMCC$, and it is straightforward to tweak the construction to give a $V$-submodel. I will generalize this result from models of $\KM$ to models of $\GBC + \PnCA k$. This gives the following result.

\begin{theorem}
Let $(M,\Xcal) \models \GBCm + \PnCA k$ for $1 \le k \le \omega$ and suppose $N \in \Xcal$ is an inner model (of $\ZFCm$) of $M$. Then there is $\Ycal \subseteq \Xcal$ a definable hyperclass so that $(N,\Ycal) \models \GBCm + \PnCAp k$. In particular, if $N = M$, this implies that every model of $\GBCm + \PnCA k$ contains a second-order definable $V$-submodel of $\GBC + \PnCAp k$.
\end{theorem}

Finally, I will end this chapter with an application of these constructions. Marek and Mostowski showed \cite{marek-mostowski1975} that the shortest height of a transitive model of $\KM$ is less than the shortest height of a $\beta$-model of $\KM$. Moreover, if the former ordinal is $\tau_\omega$ and the latter is $\beta_\omega$ then $L_{\beta_\omega} \models \tau_\omega$ is countable. I will generalize this result to $\GBC + \PnCA n$, showing that $\tau_n$---the least height of a transitive model of $\GBC + \PnCA n$---is less than $\beta_n$---the least height of a beta model of $\GBC + \PnCA n$. Moreover, $L_{\beta_n} \models \tau_n$ is countable.

The results in this chapter will form the bedrock for much of chapter 3 and chapter 4.

\section{The unrolling construction} \label{sec:unrolling}

The structure of this section is as follows. We work throughout in a fixed model $(M,\Xcal)$ of some second-order set theory. I will first lay out the basic definitions for what will be the unrolling of $(M,\Xcal)$. I will then investigate what theory is satisfied by the unrolled structure. This will be set up as a series of propositions, showing that if the ground model satisfies such and such then the unrolled model will satisfy so and so. At the end, these propositions will yield one half of the bi-interpretability theorem above. I will summarize the results before moving on to the next section, about the cutting off construction.

In reading the following the reader would do well to keep in mind what she knows about coding hereditarily countable sets as reals.

\begin{definition}[Over $\GBcm$]
Call a class binary relation $A$ a {\em membership code} if $A$ is the relation for a well-founded, extensional directed graph with a top element. Let $t_A$ denote the top element of $A$. For any $x \in \dom(A)$, let $A \downarrow x = \{ (a,b) \in A : a,b \le_A x \}$, where $\le_A$ is the reflexive transitive closure of $A$, be $A$ restricted below the node $x$.

To suggest the graph theoretic perspective, I will sometimes use $\lol_A$ as a synonym for $A$. If the context is clear, I will just write $\lol$. As a particular example, to express that $x$ is an immediate predecessor of the top element of $A$ I will write $x \lol t_A$ rather than $x \mathbin A t_A$.

Let $\pen A = \{ x \in \dom A : x \lol t_A \}$ denote the penultimate level of $A$. In the unrolled structure, for $x \in \pen A$ we have $A \downarrow x$ is a membership code which represents an element of the set represented by $A$.
\end{definition}

\begin{figure}[ht]
\begin{center}
\begin{tikzpicture}[<-,>=o,auto,node distance=2cm,
  n/.style={circle,draw,fill=white,minimum size=1cm,inner sep=0pt}]
\node[n] (pair) {$n_a$};
\node[n] (2) [below left of=pair] {$n_2$};
\node[n] (1) [below of=2] {$n_1$};
\node[n] (0) [below right of=1] {$n_0$};
\path 
  (pair) edge [bend right] (2)
  (pair) edge [bend left] (0)
  (2) edge [bend right=5] (1)
  (2) edge [bend left] (0)
  (1) edge [bend right] (0);
\end{tikzpicture}
\end{center}
\caption{A membership code representing the set $a = \{0,2\}$.}
\end{figure}
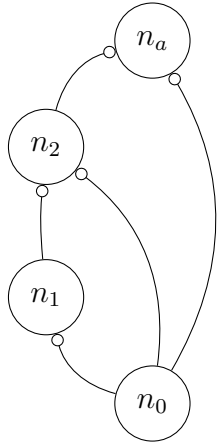

Note that this definition is first-order; it does not require any quantification over classes to check whether $A$ is a membership code. Thus, if $\Xcal$ and $\Ycal$ are $\GBcm$-realizations for the same $M \models \ZFCm$ with $A \in \Xcal \cap \Ycal$ then $\Xcal$ and $\Ycal$ agree on whether $A$ is a membership code.

For insight into the definition, consider the following definition.

\begin{definition}
Let $x$ be a set. The {\em canonical membership code for $x$} is $E_x = \oin \rest \tc(\{x\}($. If $X$ is a class then the {\em canonical membership code for $x$} is $E_X = \oin \rest \tc(\{x\}) \cup \{(x,\star) : x \in X\}$ where $\star$ is a new element.
\end{definition}

Once we have defined the unrolling it will be immediate that the canonical membership code for a set (respectively class) represents that set (respectively class) in the unrolling.

Unlike with the coding of hereditarily countable sets as reals, many membership codes represent virtual objects, `sets' which are too high in rank to be a class in the ground universe. For example, any order of ordertype $\Ord + \Ord + 1$ is a membership code representing an `ordinal' of ordertype $\Ord + \Ord$. But of course there is no literal such ordinal in the ground universe.

Because membership codes can be high in rank, Mostowski's collapse lemma does not apply to them. Indeed, it is precisely those membership codes which are not isomorphic to the restriction of $\in$ to some class that are of interest here. Lacking a way to canonically choose membership codes, we work with all membership codes and will quotient by isomorphism to produce the first-order structure.

Let us check some basic properties about isomorphisms of membership codes. The reader should note that these facts require only a weak background theory to prove. As such, the unrolling construction can be carried out over a model of a rather weak theory to produce some sort of structure. But it will take some strength in order for this unrolled structure to satisfy an appreciable set theory.

First, isomorphisms for membership codes are unique. This is analogous to the rigidity of transitive sets, but applied to membership codes rather than literal sets.

\begin{proposition}[Over $\GBcm$]
Let $A$ and $B$ be membership codes. If $\pi,\sigma : A \cong B$ then $\pi = \sigma$.
\end{proposition}

\begin{proof}
By Elementary Comprehension, the class $X = \{ a \in \dom A : \pi(a) \ne \sigma(a) \}$ exists. Suppose towards a contradiction that $X$ is nonempty. Then, we can pick a minimal $x \in X$, as this is expressible via a first-order property. By minimality, $y \lol_A x$ implies $\pi(y) = \sigma(y)$. But because $\pi$ and $\sigma$ are isomorphisms and $A$ and $B$ are extensional, $\pi(x)$ is determined by $\{\pi(y) : y \lol_A x\}$ and $\sigma(x)$ is determined by $\{\sigma(y) : y \lol_A x\}$. So $\pi(x) = \sigma(x)$, contradicting the choice of $x$. So $X$ is nonempty and thus $\pi = \sigma$.
\end{proof}

Next, nodes in a membership code are determined by the isomorphism type of their class of predecessors.

\begin{proposition}[Over $\GBcm$] \label{prop2:local-uniq}
Let $A$ be a membership code and $x,y \in \dom A$. If $A \downarrow x \cong A \downarrow y$ then $x = y$.
\end{proposition}

\begin{proof}
Let $\sigma : A \downarrow x \cong A \downarrow y$. We want to see that $\sigma$ is the identity. Set $X = \{ a <_A x : \sigma(a) \ne a\}$. Suppose towards a contradiction that $X$ is nonempty. Then, pick $a \in X$ minimal. We have that $\sigma(b) = b$ for all $b \lol_A a$ so it must be that $\sigma(a) = a$. This contradicts the choice of $a$, so it must be that $X$ is nonempty and $\sigma$ is the identity.
\end{proof}

Third, the uniqueness of isomorphism generalizes in an appropriate way to certain partial isomorphisms.

\begin{definition}
Let $A$ and $B$ be membership codes. A partial function $\pi \prtlfn A \to B$ is an {\em initial partial isomorphism} if its domain is downward-closed in $A$,\footnote{That is, if $a \in \dom \pi$ and $a' \lol a$ then $a' \in \dom \pi$} its range is downward closed in $B$, and for all $a,a' \in \dom \pi$ we have $a \lol_A a'$ if and only if $\pi(a) \lol_B \pi(a')$.
\end{definition}

\begin{proposition}[Over $\GBcm$] \label{prop2:prtl-isom-agree}
Let $\pi,\sigma \prtlfn A \to B$ be initial partial isomorphisms.
Then $\pi$ and $\sigma$ agree on the intersection of their domains.
\end{proposition}

\begin{proof}
Let $C = \dom \pi \cap \dom \sigma$. Then $C$ is nonempty, as it must contain the least element of $A$. Now consider $X = \{ x \in C : \pi(x) \ne \sigma(x) \}$. As before if $X$ is nonempty then it has a minimal element $x$, but then by the properties of isomorphism for membership codes it must be that $\pi(x) = \sigma(x)$, contradicting the non-emptiness of $X$.
\end{proof}

Observe that there are always initial partial isomorphisms between membership codes. In particular, the partial map sending the least element of $A$ to the least element of $B$ is an initial partial isomorphism. 

If our ground universe has a stronger theory then we can prove that there is a maximum initial partial isomorphism between membership codes.

\begin{lemma}[Over $\GBcm + \ETR$] \label{lem2:max-prtl-isom}
Let $A$ and $B$ be membership codes. Then there is a maximum initial partial isomorphism $\pi$ between $A$ and $B$. That is, if $\sigma \prtlfn A \to B$ is any initial partial isomorphism then $\sigma \subseteq \pi$.
\end{lemma}

\begin{proof}
This maximum initial partial isomorphism $\pi$ is constructed via an elementary transfinite recursion on $A$. Namely, $\pi$ is constructed via the transfinite recursion to construct an isomorphism between $A$ and $B$, except that we stop constructing higher when we reach a local failure of isomorphism. Formally, $\pi$ is defined via the following recursive requirement:
\begin{itemize}
\item $\pi(a)$ is the unique $b \in B$ so that for all $a' <_A a$ we have $\pi(a') \lol_B b$ if and only if $a' \lol_A a$, if such $b$ exists and $\pi(a')$ is defined for all $a' \lol_A a$; and
\item $\pi(a)$ is undefined, otherwise.
\end{itemize}
Elementary Transfinite Recursion says that this recursion has a solution $\pi$. Manifestly $\pi$ is an initial partial isomorphism. Let us check that $\pi$ is the maximum initial partial isomorphism. Take $\sigma \prtlfn A \to B$ an initial partial isomorphism. Then $\sigma$ and $\pi$ agree on their domain, by proposition \ref{prop2:prtl-isom-agree}. So the recursion to construct $\pi$ will work on all of $\dom \sigma$ and restricted to $\dom \sigma$ will give $\sigma$. So $\sigma \subseteq \pi$.
\end{proof}

Let me remark on this proof. One might attempt to more easily prove the existence of maximum initial partial isomorphisms between membership codes by considering the hyperclass of all initial partial isomorphisms and then taking the union of all of them. The issue with this argument is that it makes a hidden appeal to $\Pi^1_1$-Comprehension: we wish to define $\pi$ by saying that $\pi(a)$ is defined if {\em there exists} some initial partial isomorphism $\sigma \prtlfn A \to B$ with $a \in \dom \sigma$ and that then $\pi(a) = \sigma(a)$. This is a $\Sigma^1_1$ assertion.\footnote{Recall that $\Pi^1_1$-Comprehension is equivalent to $\Sigma^1_1$-Comprehension.}
The more convoluted argument is preferred because it works from a weaker base theory.

Isomorphism will become equality in the unrolled structure. We also must say what will become the membership relation.

\begin{definition}[Over $\GBcm$] \label{def2:vin}
Given membership codes $A$ and $B$ say that $A \vin B$ if there is $a \lol_B t_B$ so that $A \cong B \downarrow a$.
\end{definition}

In particular, if $x \lol_A t_A$ then $(A \downarrow x) \vin A$. 

\begin{proposition}[Over $\GBcm$]
Isomorphism of membership codes is a congruence with respect to $\vin$. That is, if $A$ and $B$ are membership codes so that $A \vin B$, $A \cong A'$, and $B \cong B'$, then $A' \vin B'$.
\end{proposition}

\begin{proof}
Let $e : A \hookrightarrow B$ embed $A$ onto $B \downarrow a$ for some $a \lol_B t_B$. Let $\pi : A' \cong A$ and $\sigma : B \cong B'$. Then $\sigma \circ e \circ \pi : A ' \hookrightarrow B$ embeds $A$ onto $B' \downarrow a'$ for some $a' \lol_{B'} t_{B'}$.
\end{proof}

The remainder of this section is dedicated to working out just what the theory of the ground universe implies about the theory of the unrolling.

Hereon, let $\Ucal$ denote the hyperclass of all membership codes and let $\Ufrak = (\Ucal/\mathord{\cong}, \vin)$ denote the unrolled structure. Note that $\Ucal$ is a definable hyperclass, via a first-order formula.

\begin{theorem}[Over $\GBcm + \ETR$] \label{thm2:unroll-ext}
The unrolled structure $\Ufrak$ satisfies Extensionality.
\end{theorem}

\begin{proof}
Fix membership codes $A$ and $B$. It needs to be shown that $A \cong B$ if and only if $\forall C\ C \vin A \iff C \vin B$. The forward direction of the implication is immediate. It is the other direction which requires work.

Suppose that for any membership code $C$ we have $C \vin A$ if and only if $C \vin B$. In particular, this is true for $C$ of the form $A \downarrow a$ for $a \lol_A t_A$ or the form $B \downarrow b$ for $b \lol_B t_B$. For $a \lol_A t_A$ let $\pi_a$ be the embedding which maps $A \downarrow a$ onto $B \downarrow b$ for some $b$. And in the other direction, for $b \lol_B t_B$ let $\sigma_b$ be the initial partial isomorphism which maps $B \downarrow b$ onto $A \downarrow a$ for some $a \lol_A t_A$. Proposition \ref{prop2:local-uniq} gives that the choice of $b$ is unique and thus $\pi_a$ is well-defined, and similarly for $\sigma_b$. Notice, however, that in the absence of $\Pi^1_1$-Comprehension we have no way to uniformly refer to the $\pi_a$ and the $\sigma_b$. 

By lemma \ref{lem2:max-prtl-isom} let $\pi$ be the maximum initial partial isomorphism from $A$ to $B$. First, note that $\dom \pi$ includes all $a \lol_A t_A$. This is because $\pi \supseteq \pi_a$ for $a \lol_A t_A$, by maximality. Notice also that $B \downarrow b \subseteq \ran \pi$ for all $b \lol_B t_B$ because  $\sigma_b\inv$ is an initial partial isomorphism from $A$ to $B$ which maps onto $B \downarrow b$.  But then $\dom \pi$ must also include $t_A$ and $\ran \pi$ must include $t_B$ as once we have an initial partial isomorphism defined everywhere but the top elements it is obvious how to extend: send $t_A$ to $t_B$. So $\dom \pi = A$ and $\ran \pi = B$, so $\pi$ is a full isomorphism.
\end{proof}

The unrolled structure satisfies other basic axioms of set theory.

\begin{proposition}[Over $\GBcm + \ETR$] \label{prop2:unroll-basic}
The unrolled structure $\Ufrak$ satisfies Union, Pairing, Infinity, and Foundation. 
\end{proposition}

\begin{proof}
(Union) Given a membership code $X$ we need to produce a membership code which represents $\bigcup X = \{ Z : \exists Y\ Z \vin Y \vin X\}$. We define such a $Y$ by cutting out the penultimate level of $X$: 
\[
Y = (X \setminus \{(x,t_{X}) : x \lol t_{X}\}) \cup \{(x,t_{X}) : \exists x'\ x \lol x'\lol t_{X}\}.
\]
This can be constructed by an instance of Comprehension and is easily verified to represent $\bigcup X$. See figure \ref{fig2:union-mem-code}.

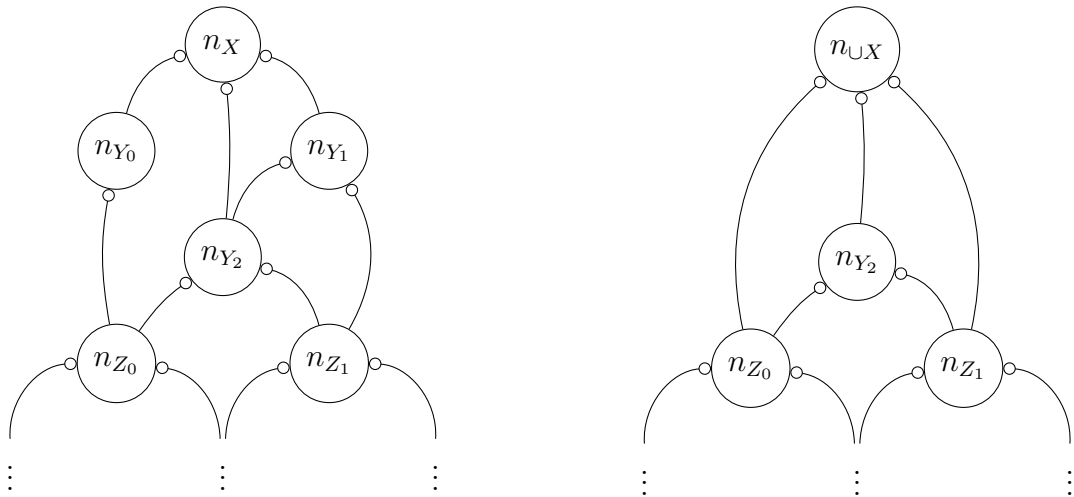
\begin{figure}[ht]
\begin{center}
\begin{multicols}{2}
\begin{tikzpicture}
[->,>=o,auto,node distance=2cm,
  n/.style={circle,draw,fill=white,minimum size=1cm}]
\node[n] (x) {$n_X$};
\node[n] (y0) [below left of=x] {$n_{Y_0}$};
\node[n] (y1) [below right of=x] {$n_{Y_1}$};
\node[n] (y2) [below left of=y1] {$n_{Y_2}$};
\node[n] (z0) [below left of=y2] {$n_{Z_0}$};
\node[n] (z1) [below right of=y2] {$n_{Z_1}$};
\node    (b0) [below left of=z0] {$\vdots$};
\node    (b1) [below right of=z0] {$\vdots$};
\node    (b2) [below right of=z1] {$\vdots$};
\path 
  (y0) edge [bend left] (x)
  (y1) edge [bend right] (x)
  (y2) edge [bend right=5] (x)
  (y2) edge [bend left=30] (y1)
  (z0) edge [bend left=10] (y0)
  (z0) edge [bend left=10] (y2)
  (z1) edge [bend right] (y1)
  (z1) edge [bend right] (y2)
  (b0) edge [bend left=45] (z0)
  (b1) edge [bend right=40] (z0)
  (b1) edge [bend left=40] (z1)
  (b2) edge [bend right=45] (z1);
\end{tikzpicture}

\begin{tikzpicture}
[->,>=o,auto,node distance=2cm,
  n/.style={circle,draw,fill=white,minimum size=1cm}]
\node[n] (x) {$n_{\cup X}$};
\node    (y0) [below left of=x] {};
\node    (y1) [below right of=x] {};
\node[n] (y2) [below left of=y1] {$n_{Y_2}$};
\node[n] (z0) [below left of=y2] {$n_{Z_0}$};
\node[n] (z1) [below right of=y2] {$n_{Z_1}$};
\node    (b0) [below left of=z0] {$\vdots$};
\node    (b1) [below right of=z0] {$\vdots$};
\node    (b2) [below right of=z1] {$\vdots$};
\path 
  (y2) edge [bend right=5] (x)
  (z0) edge [bend left] (x)
  (z0) edge [bend left=10] (y2)
  (z1) edge [bend right] (x)
  (z1) edge [bend right] (y2)
  (b0) edge [bend left=45] (z0)
  (b1) edge [bend right=40] (z0)
  (b1) edge [bend left=40] (z1)
  (b2) edge [bend right=45] (z1);
\end{tikzpicture}
\end{multicols}
\end{center}

\caption{To construct a membership code for $\bigcup X$ we remove all edges ending at the top node and add edges $(n_Z, t_{X})$ for all $n_Z \lol n_Y \lol t_{X}$.}
\label{fig2:union-mem-code}
\end{figure}

(Pairing) Let $A$ and $B$ be membership codes. By lemma \ref{lem2:max-prtl-isom} find $\pi \prtlfn A \to B$ the maximum initial partial isomorphism from $A$ to $B$. Now let 
\begin{align*}
P = &\ \ \ \ A \setminus (A \rest \dom \pi) \\
 &\cup \{ (\pi(a),a') : (a,a') \in A \mand a \in \dom \pi \mand a' \not \in \dom \pi \} \\
 &\cup B \\
 &\cup \{(t'_A,p), (t_B, p)\},
\end{align*}
where $p$ is a new point and $t'_A = t_A$ if $t_A \not \in \dom \pi$ and $t'_A = \pi(t_A)$ otherwise. It is easy to see that $P$ is a membership code which represents the unordered pair consisting of $A$ and $B$.

(Infinity) There is a membership code for $\omega$. It is straightforward to check that it represents an inductive set in $\Ufrak$.

(Foundation) At bottom, Foundation holds in $\Ufrak$ because membership codes are well-founded.

More formally, suppose towards a contradiction that Foundation fails in $\Ufrak$. That is, in the ground universe there is a membership code $A$ (for a nonempty set) so that for every $B \vin A$ there is $C$ such that $C \vin A$ and $C \vin B$. In particular, this holds for $B$ of the form $A \downarrow b$ for $b \lol_A t_A$. Therefore, we get that for all $b \lol_A t_A$ there is $c \lol_A t_A$ so that $c \lol_A t_b$. But then $\pen A$ has no minimal element, contradicting that $A$ is well-founded. 
\end{proof}

We also get that the unrolled structure satisfies Choice, due to having Global Choice in the ground universe.

\begin{proposition}[Over $\GBCm + \ETR$] \label{prop2:unroll-choice}
The unrolled structure $\Ufrak$ satisfies Choice, in the guise of the well-ordering theorem.\footnote{In the absence of Powerset, the various forms of Choice are no longer equivalent \cite{zarach1996}. That every set can be well-ordered is stronger than the existence of choice functions.}
\end{proposition}

\begin{proof}
Fix a membership code $A$. Appealing to Global Choice we may without loss of generality assume that $\dom A \subseteq \Ord$. We want to find a membership code $W$ which codes a well-order of $A$. We do this by modifying $A$ as follows. First, throw out $t_A$ and any edges pointing to $t_A$ to get $A'$. While $A'$ is in general not a membership code we will modify it to get the desired $W$. For each $x, y \in \pen A$ with $x < y$ (under the ordering from the ordinals) we will add the following nodes to $A'$: $\{x\}$, $\{x,y\}$ and $\{\{x\},\{x,y\}\}$. We also add edges to $A'$ in the obvious manner. That is, we add edges from $x$ to $\{x\}$ and $\{x,y\}$, from $y$ to $\{x,y\}$, from $\{x\}$ to $\{\{x\},\{x,y\}\}$, and from $\{x,y\}$ to $\{\{x\},\{x,y\}\}$. Finally, add a new node which will be $t_W$ and add edges from each of the $\{\{x,\},\{x,y\}\}$ nodes to $t_W$. (See figure \ref{fig2:wo-mem-code} for a picture of $W$.) Then $W$ represents a well-order of $A$ of ordertype $\le \Ord$.
\end{proof}

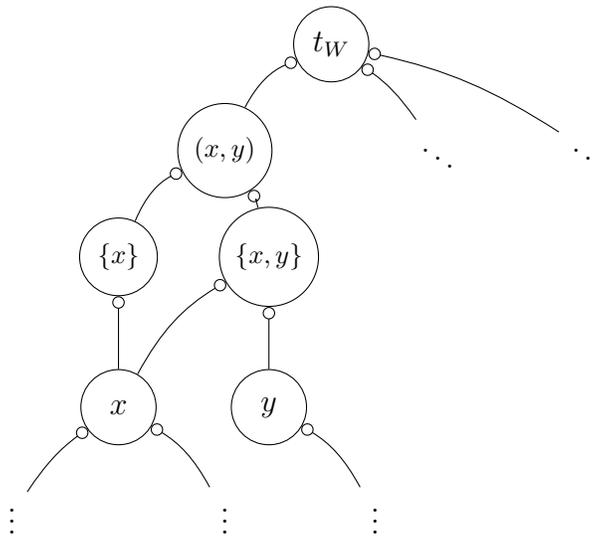
\begin{figure}[ht]
\begin{center}
\begin{tikzpicture}
[->,>=o,auto,node distance=2cm,
  n/.style={circle,draw,fill=white,minimum size=1cm}]
\node[n] (x) {$x$};
\node[n] (y) [right of=x] {$y$};
\node[n] (sx) [above of=x] {\footnotesize$\{x\}$};
\node[n] (xy) [above of=y] {\footnotesize $\{x,y\}$};
\node[n] (pxy) [above right of=sx] {\footnotesize $(x,y)$};
\node[n] (tw)  [above right of=pxy] {$t_W$};
\node (a) [below left of=x] {$\vdots$};
\node (b) [below right of=x] {$\vdots$};
\node (c) [below right of=y] {$\vdots$};
\node (d) [below right of=tw] {$\ddots$};
\node (e) [right of=d] {$\ddots$};

\path 
   (a) edge [bend left=10] (x)
   (b) edge [bend right=15] (x)
   (c) edge [bend right=15] (y)
   (x) edge (sx)
   (x) edge [bend left=15] (xy)
   (y) edge (xy)
  (sx) edge [bend left=20] (pxy)
  (xy) edge [bend right=10] (pxy)
 (pxy) edge [bend left=20] (tw)
   (d) edge [bend right=10] (tw)
   (e) edge [bend right=10] (tw);
\end{tikzpicture}
\end{center}

\caption{A partial picture of $W$. The nodes $x,y$ are from $\pen A$ with $x < y$. Below $x$ and $y$ the membership code $W$ looks like $A$.}
\label{fig2:wo-mem-code}
\end{figure}

The astute reader will note that Elementary Transfinite Recursion was used only so far in two places, namely to get Extensionality and Pairing. In both cases it was used via the lemma that there are maximum initial partial isomorphisms between membership codes. One might wonder whether this use of Elementary Transfinite Recursion can be avoided, whether by a different argument or by a different definition of membership code. See subsection \ref{subsec2:how-hard} below for some discussion.

Let $\Gamma$ be a well-order, possibly class sized. Then $\Gamma + 1 = \{ (g,g') : g \mathbin \Gamma g' \} \cup \{ (g,\star) : g \in \dom \Gamma \} \cup \{(\star,\star)\}$, where $\star \not \in \dom \Gamma$, is a membership code for an ordinal with ordertype $\Gamma$.\footnote{Of course, I assume here that $\dom \Gamma \ne V$. But in case that does happen one can take an isomorphic copy of $\Gamma$ with a smaller domain.}
So every well-order in the ground model corresponds to an ordinal in the unrolling. Of particular interest is the ordinal in the unrolling whose ordertype is the $\Ord$ of the ground universe.

Let $\kappa$ denote the set in $\Ufrak$ represented by the membership code $\Ord + 1$. Then $\Ufrak \models \kappa$ is regular. If the ground universe moreover satisfies Powerset then $\Ufrak \models \kappa$ is inaccessible. To prove these facts---and later results about $\Ufrak$---it will be convenient to be able to translate facts about functions in $\Ufrak$ to facts about class functions on the penultimate level of membership codes, and vice versa.

I will slightly abuse notation by writing e.g.\ $F: A \to B$ to refer to a membership code for a function $F$ in $\Ufrak$ from the set represented by $A$ to the set represented by $B$.

\begin{lemma}[Over $\GBcm + \ETR$] \label{lem2:fns-mem-codes}
There is a correspondence between class functions $F : \pen A \to \pen B$ and membership codes for functions $F^\star : A \to B$.
\end{lemma}

\begin{proof}
Let us do the simpler direction first. For the backward direction, suppose that $G$ is a membership code for a function from membership code $A$ to membership code $B$. That is, if $A' \vin A$ then there is $B' \vin B$ so that (in $\Ufrak$) $G$ maps $A'$ to $B'$. In particular, this holds for $A'$ of the form $A \downarrow a$ for $a \lol_A t_A$. Fix such $a$. Then there is a unique $b_a \lol_B t_B$ so that $G(A \downarrow a) \cong B \downarrow b_a$. Set $G_\star(a) = b_a$. This yields (in the ground universe) a class function $G_\star : \pen A \to \pen B$.

For the forward direction of the correspondence, suppose that $F : \pen A \to \pen B$ is a class function. 
Let $P$ be a membership code for the unordered pair whose elements are $A$ and $B$. Such exists by proposition \ref{prop2:unroll-basic}. Taking isomorphic copies if necessary we may assume without loss that $A, B \subseteq P$. We will now modify $P$ to produce the membership code $F^\star$. This is done similar to the argument in proposition \ref{prop2:unroll-choice} to construct a membership code for a well-ordering of a set. Namely, throw away $t_P$, $t_A$, and $t_B$ to produce $P'$.\footnote{Unless either $t_A \lol_P t_B$ or $t_B \lol_P t_A$, in which case keep, respectively, $t_A$ or $t_B$.} This $P'$ is in general not a membership code, but we will modify it to produce $F^\star$. For $a \in \pen A$ add nodes for $\{a\}$, $\{a,F(a)\}$ and $(a,F(a))$ to $P'$. Then add edges in the obvious way: add an edge from $a$ to $\{a\}$, an edge from $a$ to $\{a,F(a)\}$, an edge from $\{a\}$ to $(a,F(a))$, an edge from $F(a)$ to $\{a,F(a)\}$, and an edge from $\{a,F(a)\}$ to $(a,F(a))$. Finally, add a new top element $t_{F^\star}$ and edges from each $(a,F(a))$ to $t_{F^\star}$. This gives the desired $F^\star$.

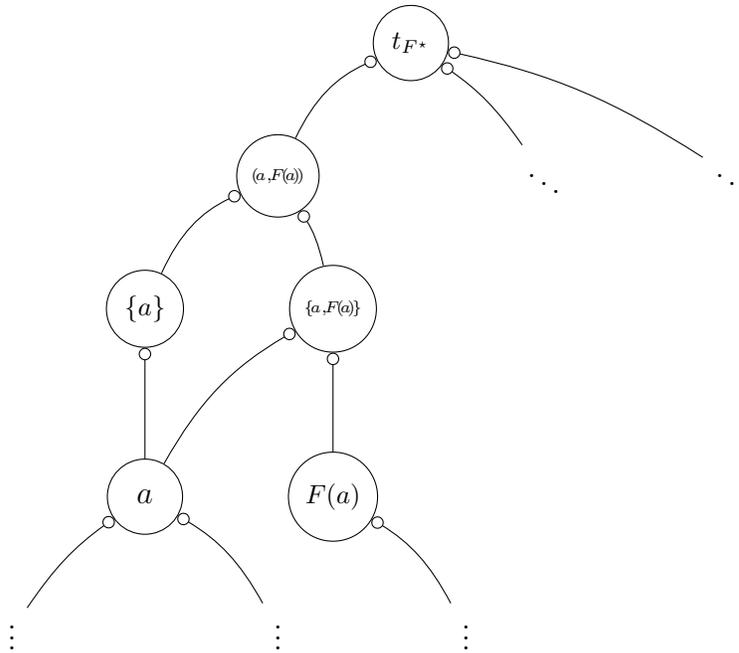
\begin{figure}[ht]
\begin{center}
\begin{tikzpicture}
[->,>=o,auto,node distance=2.5cm,
  n/.style={circle,draw,fill=white,minimum size=1cm}]
\node[n] (x) {$a$};
\node[n] (y) [right of=x] {\footnotesize $F(a)$};
\node[n] (sx) [above of=x] {\footnotesize$\{a\}$};
\node[n] (xy) [above of=y] {\tiny $\{\!a,\hspace{-.08cm}F\!(\!a\!)\!\}$};
\node[n] (pxy) [above right of=sx] {\tiny $(\!a,\hspace{-.08cm}F\!(\!a\!)\!)$};
\node[n] (tw)  [above right of=pxy] {\footnotesize $t_{F^\star}$};
\node (a) [below left of=x] {$\vdots$};
\node (b) [below right of=x] {$\vdots$};
\node (c) [below right of=y] {$\vdots$};
\node (d) [below right of=tw] {$\ddots$};
\node (e) [right of=d] {$\ddots$};

\path 
   (a) edge [bend left=10] (x)
   (b) edge [bend right=15] (x)
   (c) edge [bend right=15] (y)
   (x) edge (sx)
   (x) edge [bend left=15] (xy)
   (y) edge (xy)
  (sx) edge [bend left=20] (pxy)
  (xy) edge [bend right=10] (pxy)
 (pxy) edge [bend left=20] (tw)
   (d) edge [bend right=10] (tw)
   (e) edge [bend right=10] (tw);
\end{tikzpicture}
\end{center}

\caption{A partial picture of $F^\star$.}
\label{fig2:f-star}
\end{figure}

It is now easy to check that $(F^\star)_\star = F$ and $(F_\star)^\star \cong F$.
\end{proof}

The same argument works more generally for relations.

\begin{corollary}[Over $\GBcm + \ETR$] \label{cor2:relns-mem-codes}
There is a correspondence between class relations $R \subseteq \pen A \times \pen B$ and membership codes for relations $R^\star$ between $A$ and $B$. \qed
\end{corollary}

\begin{proposition}[Over $\GBCm + \ETR$] \label{prop2:unroll-kappa}
In $\Ufrak$, $\kappa$ is regular and the largest cardinal. Additionally, if the ground universe satisfies Powerset then $\kappa$ is a strong limit, hence inaccessible.
\end{proposition}

\begin{proof}
($\kappa$ is regular) Let $\Gamma$ be a membership code for an ordinal shorter than $\Ord$. Then $\Gamma$ must be set-sized. Suppose towards a contradiction that there is a membership code for a cofinal function $F^\star : \Gamma \to \Ord$. Let $F : \pen \Gamma \to \pen(\Ord+1)$ be the corresponding function. I claim $F$ must be cofinal in the order on $\pen(\Ord+1) = \Ord$. To see this, pick $\alpha \in \Ord$. Then because $F^\star$ is cofinal there is $\beta > \alpha$ so that for some $g \in \pen \Gamma$ we have $F^\star(\Gamma \downarrow g) = (\Ord+1) \downarrow \beta$. But then $F(g) = \beta > \alpha$, establishing that $F$ is cofinal. This is impossible, however, since $\Ord$ is a class and $\pen \Gamma$ is a set.

($\kappa$ is the largest cardinal) Let $\Gamma$ be a membership code for an ordinal longer than $\Ord$. By Global Choice there is a class bijection $F: \pen(\Ord + 1) \to \pen \Gamma$. Let $F^\star$ be the corresponding membership code for a function $F^\star : \Ord+1 \to \Gamma$. We want to see that $F^\star$ represents a bijection. First, suppose that $x \ne y$ are in $\pen(\Ord + 1)$. Then $F(x) \ne F(y)$ and thus $F^\star((\Ord+1) \downarrow x) \not \cong F^\star((\Ord+1) \downarrow y)$. So $f$ is one-to-one. To see $f$ is onto $\alpha$, take $a \in \pen A$. because $F$ is a bijection, there is $x \in \pen(\Ord +1)$ so that $F(x) = a$. Thus $F^\star((\Ord+1) \downarrow x) \cong A \downarrow a$. 

($\kappa$ is strong limit, if the ground universe satisfies Powerset) Take $\Gamma$ a membership code for a well-order shorter than $\Ord$. Then $\Gamma$ is set-sized. Because the ground universe satisfies powerset this means there is a set-sized membership code for $2^\Gamma$ and thus $2^\Gamma$ represents an ordinal less than $\kappa$. More explicitly, by the Mostowski collapse theorem we may assume without loss that $\Gamma = \oin \rest (\gamma + 1)$ for some ordinal $\gamma$. Then $2^\Gamma$ is represented by the membership code $\oin \rest (2^\gamma + 1)$, where $2^\gamma$ denotes cardinal exponentiation while the addition is ordinal addition. 
\end{proof}



It requires full second-order Comprehension to get full Separation in the unrolled model. However, it goes level by level.

To prove this we will need to translate between first-order formulae to be interpreted over $\Ufrak$ and second-order formulae to be interpreted in the ground universe. That is, given first-order $\phi$ there is a formula $\phi^*$ so that $\Ufrak \models \phi$ if and only if the ground universe satisfies $\phi^*$. This purely syntactic translation is given by the following schema.

\begin{itemize}
\item If $\phi$ is of the form $x = y$ then $\phi^*$ is $X \cong Y$.
\item If $\phi$ is of the form $x \in y$ then $\phi^*$ is $X \vin Y$.
\item If $\phi$ is of the form $\psi \lor \theta$ then $\phi^*$ is $\psi^* \lor \theta^*$, and similarly for conjunctions.
\item If $\phi$ is of the form $\neg \psi$ then $\phi^*$ is $\neg \psi^*$. 
\item If $\phi$ is of the form $\exists x \psi(x)$ then $\phi^*$ is $\exists X\ X$ is a membership code and $\psi^*(X)$, and similarly for unbounded universal quantifiers.
\item If $\phi$ is of the form $\exists x \in y \psi(x)$ then $\phi^*$ is $\exists x \lol_Y t_Y\ \psi^*(Y \downarrow x)$, and similarly for bounded universal quantifiers.
\end{itemize}

Recall that capitals are for second-order variables while lowercase letters are for first-order variables. This translation assumes there is some fixed in advance correspondence between first-order variables and second-order variables so that e.g.\ $x$ can be replaced by $X$. The details of this correspondence are completely uninteresting and will be suppressed.

Given a formula $\phi(\bar x)$ we translate it into a formula $\phi^*(\bar X)$. This is a purely syntactic translation defined via a recursion on the countable set of first-order formulae, which can be carried out in a weak fragment of $\ZFCm$. The translation is transparent to parameters: to handle formulae augmented with parameters $\phi(\bar A)$ simply consider $\phi^*(\bar A)$. 

First we need to see that this translation is coherent, since we gave different translations for $\exists x \in y\ \phi(x)$ and $\exists x\ (x \in y \land \phi(x))$.

\begin{lemmaschema}[Over $\GBcm$]
For all formulae $\phi(x)$ in the language of first-order set theory, $\GBcm$ proves that if $A$ and $B$ are isomorphic membership codes then $\phi^*(A) \iff \phi^*(B)$. 
\end{lemmaschema}

\begin{proof}
This is an easy induction on formulae.
\end{proof}

\begin{lemmaschema}[Over $\GBcm$]
For all first-order formulae $\phi(x)$ the theory $\GBcm$ proves that $(\exists x\ x\in y \land \phi(x))^*$ is equivalent to $(\exists x \in y\ \phi(x))^*$.
\end{lemmaschema}

\begin{proof}
Consider a formula $\phi$. Work in $\GBcm$ and consider an arbitrary membership code $Y$. 

$(\Rightarrow)$ Suppose there is a membership code $X$ so that $X \vin Y$ and $\phi^*(X)$. By the definition of $\vin$ there is $x \lol t_Y$ so that $X \cong Y \downarrow x$. But then $\phi^*(Y \downarrow x)$ holds because of the previous lemma plus the fact that $\phi^*(X)$ holds, so we are done.

$(\Leftarrow)$ Suppose there is $x \lol t_Y$ so that $\phi^*(Y \downarrow x)$ holds. We want to see there is a membership code $X$ so that $X \vin Y$ and $\phi^*(X)$. Take $X = Y \downarrow x$. Done.
\end{proof}

\begin{lemma}[Over $\GBcm + \ETR$]\label{lem2:relns-are-delta11}
Both $X \cong Y$ and $X \vin Y$ are $\Delta^1_1$.
\end{lemma}

\begin{proof}
The definitions of these are both $\Sigma^1_1$. So it remains to see that they are equivalent to $\Pi^1_1$ assertions. Let us consider $A \vin B$; similar reasoning works for $A \cong B$. If $A \vin B$ then this must be witnessed by the maximum initial partial isomorphism $\pi \prtlfn A \to B$. So to say that $A \mathbin{\not\! \vin} B$ it is equivalent to say that the maximum initial partial isomorphism $\pi \prtlfn A \to B$ does not have a range of the form $B \downarrow a$ for some $a \lol_B t_B$. Thus, $A \mathbin{\not\!\vin} B$ is $\Sigma^1_1$-expressible, via the formula asserting that there is an initial partial isomorphism $\pi \prtlfn A \to B$ so that $(1)$ for some $x \in \dom A \setminus \dom \pi$ there is no $y \in \dom B$ so that $\pi \cup \{(x,y)\}$ is an initial partial isomorphism and $(2)$ for all $a \lol_B t_B$ there is $a' \le_B a$ so that $a' \not \in \ran \pi$. Therefore $A \vin B$ is equivalent to a $\Pi^1_1$-formula.
\end{proof}

\begin{lemma}[Over $\GBcm + \ETR$] \label{lem2:star-trans}
If $\phi$ is $\Sigma_k$ for $k \ge 1$ then $\phi^*$ is equivalent to a $\Sigma^1_k$ formula.
\end{lemma}

\begin{proof}
First, recall that being a membership code is a first-order property. So that cannot increase the logical complexity of $\phi^*$. 

Lemma \ref{lem2:relns-are-delta11} allows us to absorb the class quantifiers from $A \cong B$ and $A \vin B$ into the innermost quantifier block of the translated formula. First, $\phi^*$ is equivalent to a formula in prenex normal form. In particular, all unbounded quantifiers are on the outside and $\neg$ only occurs in front of atomic statements. Suppose that the innermost quantifier block consists of existential quantifiers. Then produce an equivalent $\Sigma^1_k$-formula by replacing all instances of $\neg A \cong B$ and $\neg A \vin B$ with the equivalent $\Sigma^1_1$ assertion. Thus, all bounded quantifiers in $\phi$ become $\Sigma^1_1$ assertions and thereby get absorbed into the innermost quantifier block. Similarly, if the innermost quantifier block consists of universal quantifiers then replace all instances of $A \cong B$ and $A \vin B$ with the equivalent $\Pi^1_1$ assertion, thereby absorbing them.
\end{proof}

\begin{proposition}[Over $\GBcm + \ETR$] \label{prop2:sep-in-unroll}
If the ground universe satisfies $\PnCA k$ for $k \ge 1$ then $\Ufrak$ satisfies $\Sigma_k$-Separation. Thus, if the ground universe satisfies $\KM$ then the unrolled structure satisfies Separation.
\end{proposition}

In fact, the backwards implication is also true; see section \ref{sec:cutting-off}.

\begin{proof}
Fix a membership code $A$ and a $\Sigma_k$ formula $\phi(x)$, possibly with (suppressed) parameters. We want to find a membership code $B$ so that $B$ represents the subset of $A$ formed using $\phi$. Because $\phi^*$ is equivalent to a $\Sigma^1_k$ formula we can by $\PnCA k$ form the collection of $x \lol_A t_A$ so that $\phi^*(A \downarrow x)$ holds. Closing this collection downward in $A$ and adding a top element gives a membership code $B$ representing the subset of $A$ formed using $\phi$.
\end{proof}

The case for Collection is similar. We need full Class Collection to get Collection in the unrolled model, but it goes level by level.

\begin{proposition}[Over $\GBCm + \ETR$] \label{prop2:unroll-coll}
If the ground universe satisfies $\Sigma^1_k$-Class Collection then the unrolled structure $\Ufrak$ satisfies $\Sigma_k$-Collection. Thus if the ground universe satisfies $\KMCC$ then the unrolled structure satisfies Collection.
\end{proposition}

In fact, the backwards implication is also true; see section \ref{sec:cutting-off}.

\begin{proof}
We want to show that $\Ufrak$ satisfies every instance of $\Sigma_k$ collection. That is, consider a $\Sigma_k$-formula $\phi$ and suppose there are membership codes $A$ and $P$ so that for all $X \vin A$ there is a membership code $Y$ so that $\phi^*(X,Y,P)$. We want to find a membership code $B$ which collects together all these witnessing $Y$. 

By $\Sigma^1_k$-Class Collection there is a class $C$ of pairs $(x,y)$ so that for each $x \lol_A t_A$ we have $\phi^*(A \downarrow x, (C)_x, P)$. We want to glue together the $(C)_x$s to get the desired membership code $B$. This will be done by means of a certain elementary transfinite recursion. 

There are two layers to the recursion. The outer layer is along $\pen A$, according to some fixed enumeration $\seq{x_\alpha : \alpha \in \Ord}$. Each stage $\alpha$ of the outer layer yields a partial construction of $B$, let us call it $B_\alpha$, along with $P_\alpha$ a partial listing of what will be the penultimate level of $B$. Stage $0$ is merely to take $B_0 = \emptyset$ and $P_0 = \emptyset$. Having done stage $\alpha$, for stage $\alpha+1$ we must do the inner layer of the recursion to produce the maximum initial partial isomorphism $\pi \prtlfn (C)_{x_{\alpha}} \to B_\alpha$. We then use it to define $B_{\alpha+1}$ by adding on everything in $(C)_{x_{\alpha}}$ which we have not yet accounted for. Formally, set
\begin{align*}
B_{\alpha+1} = &\ \ \ \  B_\alpha \\
 &\cup (C)_{x_\alpha} \rest (\dom C \setminus \dom \pi) \\
 &\cup \{ (\pi(c),c') : (c,c') \in C \mand c \in \dom \pi \mand c' \not \in \dom \pi \}.
\end{align*}
We also add either $t_{(C)_{x_\alpha}}$ or $\pi(t_{(C)_{x_\alpha}})$ as appropriate to $P_\alpha$ to get $P_{\alpha+1}$ At limit stages $\lambda$ take unions: set $B_\lambda = \bigcup_{\alpha < \lambda} B_\alpha$ and $P_\lambda = \bigcup_{\alpha < \lambda} P_\alpha$. Finally, after $\Ord$ steps we produce $B$ from $B_\Ord$ by adding a top element connecting each element of $P_\Ord$ to this new top element. Then $B$ is the desired membership code.
\end{proof}

The following theorem summarizes what $\Ufrak$'s theory will be, based upon the theory of the ground universe.

\begin{theorem}
Let $(M,\Xcal)$ be a model of second-order set theory and let $\Ufrak$ be the unrolled model constructed inside $(M,\Xcal)$. 
\begin{itemize}
\item If $(M,\Xcal) \models \KMCC$ then $\Ufrak \models \ZFCmi$.
\item If $(M,\Xcal) \models \KMCCm$ then $\Ufrak \models \ZFCmr$. 
\item If $(M,\Xcal) \models \KM$ then $\Ufrak \models \wZFCmi$.
\item If $(M,\Xcal) \models \KMm$ then $\Ufrak \models \wZFCmr$.
\item If $(M,\Xcal) \models \GBC + \PnCAp k$ then $\Ufrak \models \ZFCmi(k)$, for $k \ge 1$.
\item If $(M,\Xcal) \models \GBCm + \PnCAp k$ then $\Ufrak \models \ZFCmr(k)$, for $k \ge 1$.
\item If $(M,\Xcal) \models \GBC + \PnCA k$ then $\Ufrak \models \wZFCmi(k)$, for $k \ge 1$.
\item If $(M,\Xcal) \models \GBCm + \PnCA k$ then $\Ufrak \models \wZFCmr(k)$, for $k \ge 1$.
\end{itemize}
\end{theorem}

In section \ref{sec:cutting-off} we will get the backward results. But first we must address the conspicuous absence of $\GBC + \ETR$ and $\GBCm + \ETR$ from this last theorem. The following subsection will also establish that for all of the above, the unrolling $\Ufrak$ satisfies $\Sigma_0$-Transfinite Recursion. Of course, $\Sigma_0$-Transfinite Recursion follows from $\Sigma_1$-Collection, so this is only additional content for the theories which lack a fragment of Collection.

\subsection{\texorpdfstring{Unrolling in $\GBC + \ETR$}{Unrolling in GBC + ETR}} \label{subsec2:unroll-in-etr}

Above, to show that $\Ufrak$ satisfies a fragment of Separation and Collection---dependent upon the fragment of Comprehension and Class Collection satisfied in the ground universe---we used a translation $\phi \mapsto \phi^*$. Given a first-order formula $\phi$ in the language of set theory this translation produced a second-order formula $\phi^*$ about membership codes. We used that $A \vin B$ and $A \cong B$ are $\Delta^1_1$ over $\GBCm + \ETR$, so that the class quantifiers arising from the translation of $a = b$ and $a \in b$ could be absorbed by the block of class quantifiers at the front of a prenex-normal form formula equivalent to $\phi^*$.

This of course will not work if we do not have class quantifiers at the front. These came from translating unbounded quantifiers from $\phi$. But, as we will finish seeing later, $\wZFCmr(1)$---which allows Separation for formulae with only a single unbounded quantifier in front---is bi-interpretable with $\GBCm + \PnCA 1$, which is stronger than $\GBCm + \ETR$. So we cannot hope to absorb the class quantifiers arising from $A \vin B$ and $A \cong B$. If we want to calculate the theory of unrollings of models of $\GBC + \ETR$ and $\GBCm + \ETR$ we must use a different translation.

The key idea is that to check the truth of an assertion which only has bounded quantifiers it suffices to look at a single membership code. Before giving the translation let me illustrate this with an example.

We work in $\GBCm + \ETR$. Suppose we are given a membership code $E$ and we wish to know whether $E$ represents an ordered pair in the unrolling, say according to Kuratowski's definition. Formally, we want to know whether 
\begin{align*}
\exists A,B \vin E\ \Big(\forall X \vin E\ &(X \cong A \lor X \cong B) \\ 
 &\land \forall Y,Y' \vin A\ \big(Y \cong Y' \land Y \vin B \\
 &\phantom{\land \forall Y,Y' \vin A\ \big(}
 \land \forall Z,Z' \vin B\ ([Z \not \cong Y \land Z' \not \cong Y] \impl Z \cong Z')\big)\Big)
\end{align*}
is true.
Expanding out the definitions of $\cong$ and $\vin$ this formula has fourteen class quantifiers, which is too many to handle directly in $\GBC + \ETR$. But we do not have to look at, for example, all membership codes $A,B$ so that $A,B \vin E$. It suffices to look just at all membership codes of the form $E \downarrow x$ for $x \in \pen E$. Similar considerations apply for the other instances of $\vin$ or $\cong$ in this formula. So it is equivalent to ask whether
\begin{align*}
\exists a,b \lol_E t_E\ \Big(\forall x \lol_E t_E\ &(x = a \lor x = b) \\
 &\land \forall y,y' \lol_E a\ \big(y = y' \land y \lol_E b \\
 &\phantom{\land \forall y,y' \lol_e a\ \big(}
 \land \forall z,z' \lol_E b\ ([z \ne y \land z' \ne y] \impl z = z')\big)\Big)
\end{align*}
is true.
This formula only has set quantifiers.

\begin{definition}[Over $\GBCm + \ETR$] \label{def2:etr-star-trans}
Let $\phi(A_1, \ldots, A_n)$ be a $\Sigma_0$-formula in the language of first-order set theory with parameters $A_1, \ldots, A_n$, which are membership codes, and possibly with free variables. Then $\phi^\star(\bar A)$ is defined as follows. First, let $P$ be the membership code for the `set' $\{A_1, \ldots, A_n\}$. This can be constructed by an instance of Elementary Transfinite Recursion, similar to the construction of unordered pairs. For $1 \le i \le n$ let $a_i \lol_P t_P$ be the unique member of the penultimate level of $P$ so that $P \downarrow a_i \cong A_i$. Then $\phi^\star(\bar A)$ is defined by the following schema. Here, $t$ and $s$ are either variables $x,y,z,\ldots$ or one of the $a_i$'s.
\begin{itemize}
\item If $\phi$ is $t = s$ then $\phi^\star$ is $t = s$.
\item If $\phi$ is $t \in s$ then $\phi^\star$ is $t \lol_P s$.
\item If $\phi$ is $\psi \land \theta$ then $\phi^\star$ is $\psi^\star \land \theta^\star$, and similarly for disjunctions.
\item If $\phi$ is $\neg \psi$ then $\phi^\star$ is $\neg \psi^\star$.
\item If $\phi$ is $\exists x \in t\ \psi(x)$ then $\phi^\star$ is $\exists x \lol_P t\ \psi^\star(x)$, and similarly for bounded universal quantification.
\end{itemize}
Thus, $\phi^\star$ is a $\Sigma^0_\omega$-formula in the parameter $P$.
\end{definition}

This translation is not purely syntactic, since we needed to know the parameters $\bar A$ to construct $P$. Nevertheless, given $\phi(\bar A)$ we can construct $\phi^\star(P)$ by an elementary recursion. So $\GBCm + \ETR$ lets us carry out the translation.

\begin{proposition}[Over $\GBCm + \ETR$]
The unrolled model $\Ufrak$ satisfies $\Sigma_0$-Separation.
\end{proposition}

\begin{proof}
Fix a membership code $A$ and a $\Sigma_0$-formula $\phi(x,\bar B)$ with parameters $\bar B$. By Elementary Comprehension form the class of all $x \lol_A t_A$ so that $\phi^\star(x,P)$ holds, where $P$ is constructed from $\bar B$ and $A$ as in the translation. Closing this collection downward in $A$ and adding a top element gives a membership code representing the subset of $A$ formed using $\phi$.
\end{proof}

We get more than $\Sigma_0$-Separation. As a warm-up, let us see that the unrolled model has transitive collapses. 

\begin{proposition}[Over $\GBCm + \ETR$]
The unrolled model $\Ufrak$ satisfies Mostowski's collapse lemma that every well-founded extensional binary relation is isomorphic to the restriction of the membership relation to some set.
\end{proposition}

I provide only a sketch of an argument, as we will prove a more general statement later.

\begin{proof}[Proof sketch]
Let $E$ be a membership code for a well-founded extensional relation on $A$. By corollary \ref{cor2:relns-mem-codes} take the corresponding class relation $E_\star \subseteq \pen A \times \pen A$. Then $E_\star$ is itself almost a membership code; all that is missing is that it does not have a top element. Consider the membership code $F = E_\star \cup \{ (e,\dagger) : e \in \dom E_\star \}$ where $\dagger$ is a new element. Then in the unrolling $E$ is isomorphic to the membership relation restricted to the set represented by $F$.
\end{proof}

More generally, the unrolled model $\Ufrak$ starting from a model of $\GBCm + \ETR$ will satisfy a transfinite recursion principle.

\begin{definition} \label{def2:phi-tr}
Let $\Phi$ be a collection of formulae in the language of first-order set theory. Then {\em $\Phi$-Transfinite Recursion} is the axiom schema consisting of the following axiom for each $\phi(x,y,a) \in \Phi$:

Suppose $a$ is a parameter so that $\phi(x,y,a)$ defines a class function $F : V \to V$ and $D$ is a set equipped with a well-ordering $<_D$. Then there is a function $s : D \to V$ so that for all $d \in D$ we have $s(d) = F(s \rest d)$ where $s \rest d$ means $s \rest \{d' \in D : d' <_D d\}$. 
\end{definition}

Before seeing that the unrolled structure satisfies $\Sigma_0$-Transfinite Recursion let us justify the ``more generally'' above and check that $\Sigma_0$-Transfinite Recursion (along with the other axioms we already know to be satisfied by the unrolled model) proves Mostowski's collapse lemma.

\begin{proposition}
The theory $\wZFCmr(0) + \Sigma_0$-Transfinite Recursion\footnote{That is, the theory with axioms axioms Extensionality, Pairing, Union, Infinity, Foundation, Choice (in the guise of the well-ordering theorem), $\Sigma_0$-Separation, and $\Sigma_0$-Transfinite Recursion.}
proves Mostowski's collapse lemma.
\end{proposition}

\begin{proof}
Let $e$ be a binary well-founded, existential relation on a set $D$. We want to see that there is a function $\pi$ with domain $D$ satisfying $\pi(d) = \{ \pi(d') : d' \mathbin e d \}$. This is a recursive requirement on $\pi$, where the property we want to recursively satisfy is $\Sigma_0$. So it exists by $\Sigma_0$-Transfinite Recursion.
\end{proof}

\begin{proposition}[Over $\GBCm + \ETR$] \label{prop2:unroll-etr-to-s0tr}
The unrolled structure $\Ufrak$ satisfies $\Sigma_0$-Transfinite Recursion.
\end{proposition}

\begin{proof}
Consider an instance of $\Sigma_0$-Transfinite Recursion. That is, $F$ is a class function $\Ufrak \to \Ufrak$ which is $\Sigma_0$-definable, possibly using a membership code as a parameter, and $D$ is a membership code equipped with $<_D$ a membership code for a well-ordering of $D$. We want to see that there is a membership code for the desired $s$. First, observe that in the ground universe that $F$ is first-order definable (from parameters), by the translation. 

We will build the the desired membership code $S$ via an instance of Elementary Transfinite Recursion. The idea is to mimic the recursion to produce $s$, but in membership codes. This introduces some extra work, since we have to deal with the picky details of how membership codes work. 

The iteration proceeds as follows, with an outer layer and an inner layer. The outer layer occurs on $\pen D$ according to the well-ordering corresponding to the membership code $<_D$ (see corollary \ref{cor2:relns-mem-codes}). Each step $d$ in the outer layer produces a partial construction of $S$, call it $S_d$. We start with $S_0 = \emptyset$ and take unions at limit stages. The hard work is done in the successor step, where the inner layer of the transfinite recursion occurs. We start with $S_d$ and want to produce $S_{d+1}$. By construction, each $d' \in \pen D$ which comes before $d$ in $<_D$ is in $S_d$. More, there is a corresponding node, call it $f(d')$, which represents $F(S_d \rest d')$ and then nodes for $\{d'\}$, $\{d',f(d')\}$, and $(d',f(d'))$ above, similar to the constructions in proposition \ref{prop2:unroll-choice} and lemma \ref{lem2:fns-mem-codes}. In particular, $S_d$ itself may not be a membership code. Modify $S_d$ to produce a membership code $U$ by adding a top node $t_U$ and edges from each $(d',f(d'))$ node in $S_d$ to $t_U$. Then, we have a membership code $F(U)$ by Elementary Comprehension. Construct by transfinite recursion the maximum initial partial isomorphism between $S_d$ and $F(U)$ and use it to glue a copy of $F(U)$ onto $S_d$, as in the argument for proposition \ref{prop2:unroll-coll}. Then, add $d+1$ to $S_d$ along with nodes for $\{d+1\}$, $\{d+1,t_{F(U)}\}$, and $(d+1,t_{F(U)})$ and the corresponding edges to produce $S_{d+1}$. 
\end{proof}

This completes the last step in the calculation of the theory of the unrolled model starting with a ground universe satisfying $\GBC + \ETR$ or $\GBCm + \ETR$.

\begin{corollary}
Let $(M,\Xcal)$ be a model of second-order set theory and let $\Ufrak$ be the unrolled model constructed inside $(M,\Xcal)$. 
\begin{itemize}
\item If $(M,\Xcal) \models \GBCm + \ETR$ then $\Ufrak \models \wZFCmr(0) + \Sigma_0$-Transfinite Recursion.
\item If $(M,\Xcal) \models \GBC + \ETR$ then $\Ufrak \models \wZFCmi(0) + \Sigma_0$-Transfinite Recursion.
\end{itemize}
\end{corollary}

In section \ref{sec:cutting-off} we will get the backward direction.

\subsection{How hard is it to unroll?} \label{subsec2:how-hard}

I wish now to address the question of what is needed for the unrolling to satisfy a reasonable theory. As the reader who has familiarity with coding hereditarily countable sets as reals will know, there is more than one way to do the coding. Which coding one prefers is, to a significant extent, a matter of taste. However, it means I must also address whether my choice of coding affects how strong a theory is needed in the ground universe to carry out the unrolling.

First, let us consider other codings based on well-founded extensional directed graphs. For example, rather than requiring the graph to have a largest element one could work with pointed graphs with a designated point. If $(A,p_A)$ and $(B,p_B)$ are two such pointed graphs then they represent the same `set' if $A \downarrow p_A \cong B \downarrow p_B$, with a corresponding definition of $\vin$.\footnote{I leave it to the reader to explicitly write down the definition of $\vin$ for this coding.}
It is not hard to see that such a coding will require the existence of maximum initial partial isomorphisms in the same places where they are required by membership codes according to my definition. So this coding is no easier than the one I use. 

The question then is, what is needed to carry out my coding? As remarked in section \ref{sec:unrolling}, $\GBcm$ suffices to define membership codes, isomorphism between them, and their membership relation $\vin$. So it takes very little just to unroll into some structure. But we do not want to unroll to some arbitrary structure, we want to unroll into a model of (first-order) set theory. This lacks a precise definition, but it should be uncontroversial to say that to be a model of set theory a structure must at least satisfy the basic axioms: Extensionality, Pairing, Union, and so forth. Earlier, we used Elementary Transfinite Recursion to show that the unrolling satisfies Extensionality and Pairing. More specifically, $\GBCm + \ETR$ proves the existence of maximal initial partial isomorphisms between membership codes, which is what we used. Indeed, the existence of such is necessary.

\begin{proposition}[Over $\GBcm$] \label{prop2:need-max-ipis}
If $\Ufrak$ satisfies Pairing then if $A$ and $B$ are membership codes there is a maximum initial partial isomorphism between them.
\end{proposition}

\begin{proof}
Fix $A$ and $B$. Then there is a membership code $P$ so that $E \vin P$ if and only if $E \cong A$ or $E \cong B$. Let $\pi_A : A \to P$ and $\pi_B : B \to P$ be the embeddings into $P$. Then $\pi_B\inv \circ \pi_A \rest (\ran A \cap \ran B)$ is the maximum initial partial isomorphism from $A$ to $B$.
\end{proof}

A special case of this is of interest.

\begin{corollary}[Over $\GBCm$]
If the unrolling $\Ufrak$ satisfies Pairing then the ground model satisfies the comparability of class well-orders: given class well-orders $\Gamma$ and $\Delta$ either $\Gamma$ embeds as an initial segment of $\Delta$ or vice versa.
\end{corollary}

\begin{proof}
Let $A$ and $B$ be membership codes for ordinals. That is, $A$ and $B$ are class well-orders of successor length. The initial partial isomorphism $\pi$ between them gives a comparison map. If neither $\dom \pi = A$ nor $\ran \pi = B$ then we can extend $\pi$ to a larger initial partial isomorphism, namely by mapping $\min(A \setminus \dom \pi)$ to $\min(B \setminus \ran \pi)$. Note that this definition can be done using Elementary Comprehension, because it is a first-order property to be the minimum of a well-order. So it must be that either $\dom \pi = A$, in which case $A$ has ordertype $\le$ that of $B$, or $\ran \pi = B$, in which case $B$ has ordertype $\le$ that of $A$.

The general case then follows because $\Gamma$ and $\Delta$ are comparable if and only if $\Gamma + 1$ and $\Delta + 1$ are comparable.
\end{proof}

What does it take to show that class well-orders are always comparable? It is clear that $\GBCm + \ETR$ suffices; we want a map which satisfies the recursive requirement that $e(g)$ is the least element of $\dom \delta$ which is above the range of $e \rest g$. Can we get away with less than Elementary Transfinite Recursion?

In second-order arithmetic, the answer to the analogous question is no. Over $\RCA_0$, the comparability of well-orders is equivalent to $\ATR_0$ (see \cite[chapter V.6]{simpson:book}). But this proof uses that in arithmetic, ``$X$ is a well-order'' is $\Pi^1_1$-universal. This fails badly in set theory, as in this context ``$X$ is a well-order'' is an elementary assertion. So the proof from second-order arithmetic will not generalize.

\begin{question}
Over $\GBCm$, what is the strength of the comparability of class well-orders?
\end{question}

Let us now look at alternative codings which use trees instead of graphs. Perhaps there things may be easier. This is the approach taken by Antos and Friedman \cite{antos-friedman2015}, whom one should consult for precise definitions.
To illustrate the idea, consider a set-sized well-founded tree $(T,<^T)$. By the Mostowski collapse lemma, there is a map $\pi$ which maps $T$ to a transitive set $a$ so that $s <^T t$ if and only if $\pi(s) \in \pi(t)$. (This $\pi$ will not in general be one-to-one.) Rather than representing sets with graphs we can represent them with certain trees. Membership is defined similarly as to my membership codes; say that $S \vin^{\mathrm{tree}} T$ if there is is a child node $t$ of the root node of $T$ so that $S$ is isomorphic to the subtree of $T$ below $t$.

An advantage of this coding is that satisfying Pairing is trivial. To form a code for the unordered pair of $S$ and $T$ just make a tree whose root has two children, which are the roots of $S$ and $T$. However, Extensionality still gives trouble. Similar to the case with membership codes, we want to conclude that $S \cong T$ knowing that $U \vin^{\mathrm{tree}} S$ if and only if $U \vin^{\mathrm{tree}} T$. This amounts to wanting to glue together partial isomorphisms between $S$ and $T$ to get a full isomorphism. Naively, one wants to define an isomorphism $\pi$ between $S$ and $T$ by mapping subtrees from a child of the root of $S$ to the isomorphic copy under the root of $T$ and mapping the root of $S$ to the root of $T$. But this uses an obfuscated appeal to $\Pi^1_1$-Comprehension; cf.\ the remarks after lemma \ref{lem2:max-prtl-isom}. The same idea used to show my membership codes satisfy Extensionality can be used here, namely using Elementary Transfinite Recursion to conclude there is a maximal initial partial isomorphism between the trees and then showing that this must be a full isomorphism. So we look to be in no better condition by coding with trees instead of graphs.

\section{The cutting off construction} \label{sec:cutting-off}

The other direction of the bi-interpretability results is more straightforward. Suppose we have a first-order model of set theory which has a largest cardinal $\kappa$ and $H_\kappa$ exists in the model so that $H_\kappa \models \ZFCm$. We can construct a second-order model whose first-order part is $H_\kappa$ and whose second-order part is the (proper class of) subsets of $H_\kappa$ in the model. The stronger the theory satisfied by the first-order model, the stronger the theory that will be satisfied by the second-order model gotten by this cutting off.

Formally, we will define an interpretation of formulae in the language of second-order set theory in the language of first-order set theory with a constant symbol $\kappa$ for the largest cardinal of the (first-order) model. Of course, $\kappa$ is definable so this use of a constant symbol is only a convenience. This interpretation $\phi \mapsto \phi^I$ is given by the following schema.

\begin{itemize}
\item The interpretation of $x \in y$ is $x \in y$ and the interpretation of $x \in Y$ is $x \in Y$---that is, both membership relations for the second-order model will be the membership relation of the first-order model.
\item The interpretation of $x = y$ is $x = y$ and the interpretation of $X = Y$ is $X = Y$. Unlike the unrolling construction we can directly use equality and not have to quotient out by an equivalence relation.
\item The interpretation of $\phi \land \psi$ is $\phi^I \land \psi^I$, and similarly for disjunction and negation;
\item The interpretation of $\forall x \phi$ is $\forall x\ (x \in H_\kappa \impl \phi^I)$ and similarly for first-order existential quantification; and
\item The interpretation of $\forall X \phi$ is $\forall x\ (x \subseteq H_\kappa \impl \phi^I)$ and similarly for second-order existential quantification.
\end{itemize}

It is immediate that if $\phi$ is $\Sigma^1_k$ (in parameters) then $\phi^I$ is $\Sigma_k$ (in parameters).

As a base theory for this section I will take $\wZFCmr(0) + \Sigma_0$-Transfinite Recursion. Recall that $\wZFCmr(0)$ is the set theory axiomatized by Extensionality, Union, Pairing, Infinity, Foundation, Choice, $\Sigma_0$-Separation, plus the assertions that there is a largest cardinal $\kappa$, that $\kappa$ is regular, and that $H_\kappa$ exists. In particular, $\wZFCmr(0) + \Sigma_0$-Transfinite Recursion proves that $H_\kappa \models \ZFCm$.

\begin{proposition}
Work in $\wZFCmr(0) + \Sigma_0$-Transfinite Recursion and let $\kappa$ be the largest cardinal. Then $H_\kappa \models \ZFCm$.
\end{proposition}

\begin{proof}[Proof sketch]
I will show that $H_\kappa$ satisfies Separation. The rest is an easy exercise for the reader. Consider $a,p \in H_\kappa$ and fix some formula $\phi(x,p)$. We want to see that $b = \{ x \in a : \phi(x,p)^{H_\kappa} \} \in H_\kappa$. Let $T$ be the truth predicate for $H_\kappa$, which exists by an instance of $\Sigma_0$-transfinite recursion. We can then define $b$ by $\Sigma_0$-Separation, namely to consist of those $x \in a$ for which $(\phi, x \cat p) \in T$.
\end{proof}

We can now see that all the axioms of $\GBC$ are satisfied by the cut-off model.

\begin{proposition}
If $\phi \in \GBCm$, then $\wZFCmr(0) + \Sigma_0$-Transfinite Recursion $\proves$ $\phi^I$. Also, if $\phi \in \GBC$ then $\wZFCmi(0) + \Sigma_0$-Transfinite Recursion $\proves$ $\phi^I$.
\end{proposition}

\begin{proof}
That (the interpretation of) $\ZFCm$ holds for the first-order part is because  $H_\kappa \models \ZFCm$. If $\kappa$ is moreover inaccessible then $H_\kappa = V_\kappa$ moreover satisfies Powerset, thus full $\ZFC$. Extensionality for classes holds because of Extensionality in the first-order model. Global Choice holds because Choice holds in the ground model: Any well-order of $H_\kappa$ must have ordertype $\alpha$ for some $\alpha$ of size $\kappa$, since $\kappa$ is the largest cardinal. But then we can use this to get a well-order of ordertype $\kappa$, from which we can extract a bijection $\kappa \to H_\kappa$. Replacement holds because $\kappa$ is regular. If $\phi$ is an instance of Elementary Comprehension, then $\phi^I$ is $\Sigma_0$ and thus holds by Separation applied to $H_\kappa$.
\end{proof}

To get that the cut off model satisfies $\Pi^1_k$-Comprehension requires $\Sigma_k$-Separation from the ground universe.

\begin{proposition}
If $\phi \in \PnCAm k$, for $k \ge 1$, then $\wZFCmr(k) + \Sigma_0$-Transfinite Recursion, i.e.\ $\wZFCmr(0) + \Sigma_0$-Transfinite Recursion $+$ $\Sigma_k$-Separation, proves $\phi^I$. Thus, if $\phi \in \PnCA k$, for $k \ge 1$, then $\wZFCmi(k) + \Sigma_0$-Transfinite Recursion proves $\phi^I$.
\end{proposition}

\begin{proof}
Consider $\phi$ is an instance of $\Pi^1_k$-Comprehension, i.e.\ $\phi$ asserts that there is a class whose members are precisely those sets satisfying $\psi$ where $\psi$ is $\Pi^1_k$, possibly with parameters. By $\Sigma_k$-Comprehension form $A$ the subset of $H_\kappa$ consisting of those sets which satisfy $\psi^I$. This $A$ is the desired class in the cut off model.
\end{proof}

As an immediate corollary we get full Comprehension if the ground universe satisfies $\wZFCmr + \Sigma_0$-Transfinite Recursion.

\begin{corollary}
If $\phi \in \KMm$ then $\wZFCmr + \Sigma_0$-Transfinite Recursion $\proves$ $\phi^I$. Thus, if $\phi \in \KM$ then $\wZFCmi + \Sigma_0$-Transfinite Recursion $\proves$ $\phi^I$. \qed
\end{corollary}

Next we turn to Class Collection.

\begin{proposition}
If $\phi \in \Sigma^1_k$-Class Collection, for $k \ge 1$, then $\ZFCmr(k)$ proves $\phi^I$.\footnote{Note that $\ZFCmr(k)$ includes $\Sigma_0$-Transfinite Recursion.}
\end{proposition}

\begin{proof}
Let $\phi$ be an instance of $\Sigma^1_k$-Class Collection. Then, $\phi^I$ asserts that if for every $x \in H_\kappa$ there is $Y \subseteq H_\kappa$ so that $\psi^I(x,Y,P)$, then there is $Z \subseteq H_\kappa$ so that for every $x \in H_\kappa$ we have $\psi^I(x,(Z)_x,P)$. Because $\psi$ is $\Sigma^1_k$ we have that $\psi^I$ is $\Sigma_k$. Apply $\Sigma_k$-Collection to get $B$ so that for all $x \in H_\kappa$ there is $Y \in B$ so that $Y \subseteq H_\kappa$ and $\psi^I(x,Y,P)$. Define
\[
Z = \left\{(x,b) \in V_\kappa \times \bigcup B : b \in Y \text{ where } Y \text{ is least so that } \psi^I(x,Y,P) \right\}.
\]
Here, $Y$ is least by a fixed well-order of $B$.
Then $Z$ is manifestly a subset of $H_\kappa$ and for all $x \in H_\kappa$ we have $\psi^I(x,(Z)_x,P)$.
\end{proof}

As an immediate corollary, full Collection in the ground model translates to full Class Collection in the cut off model.

\begin{corollary}
If $\phi \in \KMCCm$ then $\ZFCmr \proves \phi^I$. Thus if $\phi \in \KMCC$ then $\ZFCmi \proves \phi^I$. \qed
\end{corollary}

Finally, let us see what we need to get merely $\GBC + \ETR$ or $\GBCm + \ETR$.

\begin{proposition}
If $\phi \in \GBCm + \ETR$ then $\wZFCmr(0) + \Sigma_0$-Transfinite Recursion proves $\phi^I$.
\end{proposition}

\begin{proof}
Consider an instance of Elementary Transfinite Recursion. We want to find the subset of $H_\kappa$ which witnesses that instance holds. This is done in the obvious way by an instance of $\Sigma_0$-Transfinite Recursion.
\end{proof}

Altogether we can summarize the results so far in this section.

\begin{theorem} \label{thm2:cutting-offs}
Let $N$ be an appropriate model for performing the cut off construction and let $(M,\Xcal)$ be the cut off model constructed from $N$.
\begin{itemize}
\item If $N \models \ZFCmi$ then $(M,\Xcal) \models \KMCC$.
\item If $N \models \ZFCmr$ then $(M,\Xcal) \models \KMCCm$.
\item If $N \models \wZFCmi + \Sigma_0$-Transfinite Recursion then $(M,\Xcal) \models \KM$.
\item If $N \models \wZFCmr + \Sigma_0$-Transfinite Recursion then $(M,\Xcal) \models \KMm$.
\item If $N \models \ZFCmi(k)$ then $(M,\Xcal) \models \GBC + \PnCAp k$, for $k \ge 1$.
\item If $N \models \ZFCmr(k)$ then $(M,\Xcal) \models \GBCm + \PnCAp k$, for $k \ge 1$.
\item If $N \models \wZFCmi(k) + \Sigma_0$-Transfinite Recursion then $(M,\Xcal) \models \GBC + \PnCA k$, for $k \ge 1$.
\item If $N \models \wZFCmr(k) + \Sigma_0$-Transfinite Recursion then $(M,\Xcal) \models \GBCm + \PnCA k$, for $k \ge 1$.
\item If $N \models \wZFCmi(0) + \Sigma_0$-Transfinite Recursion then $(M,\Xcal) \models \GBC + \ETR$.
\item If $N \models \wZFCmr(0) + \Sigma_0$-Transfinite Recursion then $(M,\Xcal) \models \GBCm + \ETR$.
\end{itemize}
\end{theorem}

Let us now see that these interpretations, combined with the ones from the previous section, give bi-interpretability results. We need to see that performing one construction then its inverse brings us back to the original model.

\begin{theorem}
Let $(M,\Xcal) \models \GBCm + \ETR$ and let $\Ufrak$ be the unrolled model constructed from $(M,\Xcal)$. Then, if $(M',\Xcal')$ is the cut off model constructed from $\Ufrak$ we get $(M,\Xcal) \cong (M',\Xcal')$. Moreover, this isomorphism is definable over $(M,\Xcal)$. 
\end{theorem}

\begin{proof}
Fix a class $A$. Then in $\Ufrak$ this class $A$ is represented by $E_A$, the canonical membership code representing $A$. So the map $A \mapsto E_A$ gives an isomorphism $(M,\Xcal) \cong (M',\Xcal')$. 
\end{proof}

And in the other direction.

\begin{theorem}
Let $N \models \wZFCmr(0) + \Sigma_0$-Transfinite Recursion and let $(M,\Xcal)$ be the cut off model constructed from $N$. Then if $\Ufrak$ is the unrolled model constructed from $(M,\Xcal)$ we get $N \cong \Ufrak$. Moreover, this isomorphism is definable over $N$.
\end{theorem}

\begin{proof}
Let $\kappa$ be the largest cardinal in $N$. Then for any set $a \in N$ there is a binary relation $E_a$ on $H_\kappa$ so that $(\tc(\{a\}),\in) \cong E_a$. This uses that $\tc(\{a\})$ exists in $N$, which follows from $\Sigma_0$-Transfinite Recursion. So in the cut off model $(M,\Xcal)$ we have that $E_a$ is a membership code representing the set $a$ in the unrolling. This map $a \mapsto E_a$ gives the isomorphism $N \cong \Ufrak$.
\end{proof}

Altogether this finishes the proof of theorem \ref{thm2:bi-int}.

Let me remark on on the role of $\Sigma_0$-Transfinite Recursion here. The reader may have noticed that $\Sigma_0$-Transfinite Recursion played little role in the arguments for calculating the theory of the cut off model. It was used to conclude the cut off model satisfies $\ETR$ given a weak base theory, but it was not used to get (fragments of) Comprehension in the cut off model. However, it plays an in important role for the bi-interpretability results.

To illustrate this, take $\kappa$ an inaccessible cardinal. Then $V_{\kappa + \omega}$ satisfies all the axioms of $\ZFC$ except the Collection schema. Let $W$ consist of all subsets of $H_{\kappa^+}$ which appear in $V_{\kappa + \omega}$. Then $W \models \wZFCmi$. So cutting off $W$ gives $(V_\kappa,\Xcal) \models \KM$. But when we unroll $(V_\kappa,\Xcal)$ we get back more than $W$. For instance, this unrolling contains an ordinal $\gamma$ for each well-order in $\Xcal$, but these will of course often be larger than $\kappa + \omega$. So applying the cutting off construction followed by the unrolling construction starting from $W$ does not produce an isomorphic copy of $W$.

To finish off this section let us consider a basic but useful fact about unrollings.

\begin{proposition} \label{prop2:beta-unroll-trans}
Let $(M,\Xcal) \models \GBCm + \ETR$ and let $\Ufrak$ be its unrolling. Then $\Ufrak$ is well-founded if and only if $(M,\Xcal)$ is a $\beta$-model.
\end{proposition}

\begin{proof}
$(\Rightarrow)$ If $\Ufrak$ is well-founded then every membership code in $(M,\Xcal)$ is well-founded. But every successor class well-order is a membership code and every relation $(M,\Xcal)$ thinks is well-founded has a ranking function into a class well-order, by $\ETR$. Since the class well-orders are actually well-founded, so must every relation the model thinks is well-founded.

$(\Leftarrow)$ If $(M,\Xcal)$ is a $\beta$-model then in particular every membership code in $\Xcal$ is well-founded.
\end{proof}

\section{The constructible universe in the classes} \label{sec2:so-l}

In this section I will exposit a construction of the constructible universe where we iterate longer than $\Ord$. If $(M,\Xcal)$ satisfies a sufficiently strong theory then it can carry out this construction. Let $\Lcal$ denote the hyperclass consisting of constructible classes from $\Xcal$. Then we will see that $(L^M,\Lcal)$ has a nice theory, with the precise details depending upon the theory of $(M,\Xcal)$. We will also carry out the construction of the constructible relative to a class parameter. The main parameter of interest to us will be $(N,G)$ where $N \in \Xcal$ is an inner model of $M$ (possibly $N = M$) and $G$ is a global well-order of $N$. Letting $\Lcal(N,G)$ denote the hyperclass of $(N,G)$-constructible classes we will then get that $(N,\Lcal(N,G))$ satisfies a nice theory. This will imply that for many natural second-order set theories $T$ that being $T$-realizable is closed under taking inner models. 

\begin{theorem} \label{thm2:inner-model-rlzble}
Let $(M,\Xcal) \models \GBCm + \ETR$ and suppose $N \in \Xcal$ is an inner model (of $\ZFCm$ or of $\ZFC$, as appropriate) of $M$. Then there is a hyperclass $\Ycal \subseteq \Xcal \cap \powerset(N)$ second-order definable in $\Xcal$ so that the following.
\begin{itemize}
\item If $(M,\Xcal) \models \KM$ then $(N,\Ycal) \models \KMCC$.
\item If $(M,\Xcal) \models \KMm$ then $(N,\Ycal) \models \KMCCm$.
\item If $(M,\Xcal) \models \GBC + \PnCA k$ then $(N,\Ycal) \models \GBC + \PnCAp k$.
\item If $(M,\Xcal) \models \GBCm + \PnCA k$ then $(N,\Ycal) \models \GBCm + \PnCAp k$.
\item If $(M,\Xcal) \models \GBC + \ETR$ then $(N,\Ycal) \models \GBC + \ETR + \ECC$.
\item If $(M,\Xcal) \models \GBCm + \ETR$ then $(N,\Ycal) \models \GBCm + \ETR + \ECC$.
\end{itemize}
\end{theorem}

The last two, that being $(\GBC + \ETR)$-realizable or $(\GBCm + \ETR)$-realizable is closed under inner models, will be proved in the next chapter. The rest will follow from results in this section.

The special case where $N = M$ yields the following result, which I state separately.

\begin{corollary} \label{cor2:get-a-plus}
Let $(M,\Xcal) \models \GBCm + \ETR$. Then there is a definable hyperclass $\Ycal \subseteq \Xcal$ so that the following.
\begin{itemize}
\item If $(M,\Xcal) \models \KM$ then $(M,\Ycal) \models \KMCC$.
\item If $(M,\Xcal) \models \KMm$ then $(M,\Ycal) \models \KMCCm$.
\item If $(M,\Xcal) \models \GBC + \PnCA k$ then $(M,\Ycal) \models \GBC + \PnCAp k$, for $k \ge 1$.
\item If $(M,\Xcal) \models \GBCm + \PnCA k$ then $(M,\Ycal) \models \GBCm + \PnCAp k$, for $k \ge 1$.
\item If $(M,\Xcal) \models \GBC + \ETR$ then $(M,\Ycal) \models \GBC + \ETR + \ECC$.
\item If $(M,\Xcal) \models \GBCm + \ETR$ then $(M,\Ycal) \models \GBCm + \ETR + \ECC$.
\end{itemize}
\end{corollary}

That is, we can always obtain (a fragment of) Class Collection by moving to a possibly smaller $V$-submodel.

The strategy in this section will be to define $\Lcal$ somewhat indirectly. I will define $\Lfrak$ to consist of the constructible membership codes. Then, $\Lcal$ will be the classes which are coded in $\Lfrak$. (Similar remarks apply when constructing relative to a parameter.) This approach follows the usual construction of $L$. But since we want to iterate the construction beyond $\Ord$ we will be too high in rank to use actual sets and will settle for membership codes for sets. This approach also has the advantage that some of the arguments are trivial modifications of arguments from section \ref{sec:unrolling} of this chapter. As such, the reader is strongly encouraged to read section \ref{sec:unrolling} before reading this section.

Let us move now to the definitions.

\begin{definition}[Over $\GBCm + \ETR$]
Let $\Gamma$ be a well-order. Denote by $L_\Gamma$ the class obtained by iterating the definition of $L$ along $\Gamma$. To clarify, $L_\Gamma$ is the membership code constructed according to the following recursion:
\begin{itemize}
\item $L_0$ is a membership code for $\emptyset$.
\item $L_{\alpha + 1}$ is a membership code for $\Def(L_\alpha)$.
\item $L_\lambda$ is a membership code for $\bigcup_{\alpha < \lambda} L_\alpha$, for $\lambda$ a limit element of $\Gamma$.
\end{itemize}
\end{definition}

This definition cries out for clarification. Recall that $\GBCm + \ETR$ proves the existence of (first-order) truth predicates relative to a class. In particular, if $(A,R_0, R_1, \ldots)$ is a class-sized structure then by Elementary Transfinite Recursion can be constructed the truth predicate for $(A,R_0,R_1, \ldots)$; simply take take the truth predicate for $(V,\in,A,R_0,R_1,\ldots)$ and throw away the irrelevant facts about the background universe. In particular, this works for membership codes. So if $E$ is a membership code which represents a `set' $e$ we can build the truth predicate for $(\dom E, E)$ and from it extract a membership code $D_E$ for the `set' $\Def(e)$. So $\ETR$ lets us construct membership codes that are of the form $L_\Gamma$. We will see in the next chapter that weaker principles will not suffice.

Of course, these membership codes are not unique, as there will be many isomorphic copies of $L_\Gamma$. But if we fix in advance the details of how we build $D_E$ from $E$ then we have fixed a way to build $L_\Gamma$ from $\Gamma$.\footnote{I will not go into the details of the necessary coding, as it is routine and uninteresting.}
So in this sense it is justified to refer to $L_\Gamma$ as {\em the} membership code obtained by iterated the definition of $L$ along $\Gamma$.

\begin{definition}[Over $\GBCm + \ETR$]
Say that a membership code $E$ is {\em constructible} if $E \vin L_\Gamma$ for some $\Gamma$. 
Let $\Lfrak$ consist of the membership codes for all constructible classes. We consider $\Lfrak$ as a first-order model by quotienting out by isomorphism and using $\vin$ (see definition \ref{def2:vin}) for its membership relation. I will call $\Lfrak$ the {\em constructible unrolled model}. Let $\Lcal$ consist of all classes coded in $\Lfrak$.
\end{definition}

In particular, each $L_\Gamma$ is constructible, as $L_\Gamma \vin L_{\Gamma + 1}$. Note that this is true even if $\Gamma$ is very complicated and codes a lot of information. It will nevertheless be isomorphic to $L_{\Gamma + 1} \downarrow e$ for some $e \in \pen(L_{\Gamma + 1})$.

As in the usual first-order setting we can carry out the construction of $L$ relative to some parameter(s). 

\begin{definition}
Let $P$ be a class and $\Gamma$ a well-order. Denote by the $L_\Gamma(P)$ the membership code obtained by iterating the definition of $L(P)$ along $\Gamma$. 
\end{definition}

\begin{definition}
A membership code $E$ is {\em $P$-constructible} if $E \vin L_\Gamma(P)$ for some $\Gamma$.

Similar to before, let $\Lfrak(P)$ consist of the membership codes for all constructible classes. We consider $\Lfrak(P)$ as a first-order model by quotienting out by isomorphism and using $\vin$ as its membership relation. I will call $\Lfrak(P)$ the {\em $P$-constructible unrolled model}. Let $\Lcal(P)$ consist of all classes coded in $\Lfrak(P)$.
\end{definition}

Hereon work in a fixed $(M,\Xcal) \models \GBCm + \ETR$, where we will require more from $(M,\Xcal)$ in various propositions. Fix $N \in \Xcal$ an inner model (of $\ZFCm$ or $\ZFC$, as appropriate) of $M$ and fix $G \in \Xcal$ a $\GBC$-amenable global well-order of $N$. We want to calculate the theory of $(N,\Lcal(N,G))$, based upon how strong the theory of $(M,\Xcal)$ is. Below, $\kappa$ will denote the ordinal in $\Lfrak(N,G)$ which represents the $\Ord$ of $N$, i.e.\ $\kappa$ is represented by the membership code $\Ord+1$.

But first, let us check that there really is such a $G$.

\begin{lemma} \label{lem2:gbc-amen-gwo}
Let $(M,\Xcal) \models \GBCm + \ETR$ and let $N \in \Xcal$ be an inner model of $M$ which satisfies Choice. Then there is a class $G \in \Xcal$ which is a $\GBC$-amenable global well-order of $N$.
\end{lemma}

\begin{proof}
Because $\Xcal$ contains first-order truth predicates relative to any class it has uniform access to the definable subclasses of a definable forcing notion over $N$. In particular, this works for the Cohen forcing $\Add(\Ord,1)^N$ to add a subclass of $\Ord$. Let $C \in \Xcal$ meet every definable dense subclass of $\Add(\Ord,1)^N$. Then $C$ codes a global well-order of $N$, which is necessarily $\GBC$-amenable to $N$.
\end{proof}

Now let us see that $\Lfrak(N,G)$ satisfies some basic axioms.

\begin{proposition}[Over $\GBCm + \ETR$]
The $(N,G)$-constructible unrolled model $\Lfrak(N,G)$ satisfies Extensionality, Pairing, Union, Infinity, Foundation, Choice, and that the cardinal $\kappa = \Ord^N$ is regular.
\end{proposition}

\begin{proof}[Proof sketch]
This is proved similarly to the analogous results from section \ref{sec:unrolling} (namely, theorem \ref{thm2:unroll-ext} and propositions \ref{prop2:unroll-basic}, \ref{prop2:unroll-choice}, and \ref{prop2:unroll-kappa}).
The new content is to check that the arguments never take us outside of the constructible membership codes. Note in particular that $\Lfrak(N,G)$ satisfies Choice precisely because $G$ is coded in $\Lfrak(N,G)$. 
\end{proof}

Note that the previous proposition does not claim that $\Lfrak(N,G) \models \text{``}\kappa$ is the largest cardinal''. Indeed, this will not in general be true. For a counterexample, consider countable $L_\alpha \models \ZFC$ with $\kappa < \alpha$ so that $L_\alpha \models \kappa$ is inaccessible. Now do a class forcing over $L_\alpha$ to collapse all cardinals $> \kappa$. This produces $W \models \ZFCmi$ with $\Ord^M = \alpha$ and $\kappa$ is the largest cardinal in $W$. Let $(M,\Xcal) \models \KMCC$ be the model obtained by the cutting off construction applied to $W$. Then we get that $\Lfrak^{(M,\Xcal)}$ is (isomorphic to) $L_\alpha$. So $\Lfrak$ does not think that $\kappa$ is the largest cardinal. Observe however, that $L^M = L_\kappa$ is still $\KMCC$-realizable, as $V_{\kappa+1}^{L_\alpha}$ is a $\KMCC$-realization for $L^M$.

Back to calculating the theory of $\Lfrak(N,G)$, if the ground universe satisfies Powerset then $\Lfrak(N,G)$ will think that $\kappa$, the cardinal in the unrolling corresponding to the $\Ord$ of the ground universe, is inaccessible.

\begin{corollary}[Over $\GBC + \ETR$]
The unrolled model $\Lfrak(N,G)$ satisfies that $\kappa$ is inaccessible.
\end{corollary}

\begin{proof}
Unrolling with all membership codes, not just the $(N,G)$-constructible ones, gave that $\kappa$ is inaccessible. Using fewer membership codes---and thus possibly having fewer sets in the unrolling---can only make it easier for $\kappa$ to be inaccessible.
\end{proof}

\newcommand\starl{{*L}}

Next, let us see that $\Lfrak(N,G)$ satisfies the level of Separation corresponding to the level of Comprehension in the ground universe $(M,\Xcal)$. Similar to the proof of proposition \ref{prop2:sep-in-unroll} we will need to translate first-order formulae $\phi$ into second-order formulae about membership codes. Here, however, we want to confine to talking only about constructible membership codes. This translation $\phi \mapsto \phi^\starl$ is given by the following schema.
\begin{itemize}
\item If $\phi$ is of the form $x = y$ then $\phi^\starl$ is $X \cong Y$.
\item If $\phi$ is of the form $x \in y$ then $\phi^\starl$ is $X \vin Y$.
\item If $\phi$ is of the form $\psi \lor \theta$ then $\phi^\starl$ is $\psi^\starl \lor \theta^\starl$, and similarly for conjunctions.
\item If $\phi$ is of the form $\neg \psi$ then $\phi^\starl$ is $\neg \psi^\starl$.
\item If $\phi$ is of the form $\exists x\psi(x)$ then $\phi^\starl$ is $\exists X$ $X$ is a constructible membership code and $\psi^\starl(X)$, and similarly for unbounded universal quantifiers.
\item If $\phi$ is of the form $\exists x \in y \psi(x)$ then $\phi^\starl$ is $\exists x \lol_Y t_Y\ \psi^\starl(Y \downarrow x)$, and similarly for bounded universal quantifiers.
\end{itemize} 

Like before, this translation does not increase complexity.

\begin{lemma}
If $\phi$ is $\Sigma_k$ for $k \ge 1$ then $\phi^\starl$ is equivalent to a $\Sigma^1_k$-formula.
\end{lemma}

\begin{proof}[Proof sketch]
First, let us see by induction that $(\phi^L)^*$, using the $*$-translation from lemma \ref{lem2:star-trans}, is equivalent to $\phi^\starl$. The only step to check is the unbounded quantifier step, as that is the only nontrivial step in the definition of the relativization $\phi \mapsto \phi^L$. Suppose $\phi$ is of the form $\exists x\ \psi(x)$. Then $\phi^L$ is of the form $\exists x\in L\ \psi(x)^L$. By inductive hypothesis, $(\psi(x)^L)^*$ is equivalent to $\psi(x)^\starl$. But then we are done, since $(x \in L)^*$ is equivalent to ``$X$ is a constructible membership code''.

So, using lemma \ref{lem2:star-trans}, we are done once we know that $\phi^L$ is equivalent to a $\Sigma_k$ formula when $\phi$ is $\Sigma_k$. This is also seen by induction. The base case, $k = 0$, is trivial because $\phi^L$ is just $\phi$ for a $\Sigma_0$-formula $\phi$. Now suppose $\phi$ is of the form $\exists x\ \psi(x)$ where $\psi$ is $\Pi_k$. By inductive hypothesis $\psi^L$ is equivalent to a $\Pi_k$-formula. So $\phi^L$, which is $\exists x \in L\ \psi(x)^L$ is equivalent to a $\Sigma_{k+1}$-formula, because ``$x \in L''$ is $\Sigma_1$, as desired.
\end{proof}

\begin{proposition}[Over $\GBCm + \ETR$] \label{prop2:sep-in-lfrak}
If the ground universe $(M,\Xcal)$ satisfies $\PnCA k$ then $\Lfrak(N,G)$ satisfies $\Sigma_k$-Separation, for $k \ge 1$.
\end{proposition}

As an immediate corollary we get the following.

\begin{corollary}[Over $\GBCm + \ETR$]
If the ground universe $(X,\Xcal)$ satisfies full Comprehension then $\Lfrak(N,G)$ satisfies Separation. \qed
\end{corollary}

\begin{proof}[Proof of proposition \ref{prop2:sep-in-lfrak}]
Fix a constructible membership code $A$ and a $\Sigma_k$-formula $\phi(x)$, possibly with (suppressed) parameters. Because $\phi^\starl$ is equivalent to a $\Sigma^1_k$-formula we may by $\Pi^1_k$-Comprehension form the collection of $x \lol_A t_A$ so that $\phi^\starl(A \downarrow x)$. Closing this collection downward in $A$ and adding a top element gives a membership code $B$ representing the subset of $A$ formed using $\phi$. 
\end{proof}

Altogether, we have seen the following.

\begin{theorem}
Let $(M,\Xcal) \models \GBCm + \ETR$, $N \in \Xcal$ be an inner model (of $\ZFCm$ or $\ZFC$, as appropriate) of $M$, and $G \in \Xcal$ a $\GBC$-amenable global well-order of $N$. 
\begin{itemize}
\item If $(M,\Xcal) \models \KM$ then $\Lfrak(N,G) \models \wZFCmi$.
\item If $(M,\Xcal) \models \KMm$ then $\Lfrak(N,G) \models \wZFCmr$.
\item If $(M,\Xcal) \models \GBC + \PnCA k$ then $\Lfrak(N,G) \models \wZFCmi(k)$.
\item If $(M,\Xcal) \models \GBCm + \PnCA k$ then $\Lfrak(N,G) \models \wZFCmr(k)$.
\end{itemize}
\end{theorem}

Combined with theorem \ref{thm2:cutting-offs} we immediately get the following.

\begin{theorem}
Let $(M,\Xcal) \models \GBCm + \ETR$, $N \in \Xcal$ be an inner model (of $\ZFCm$ or $\ZFC$, as appropriate) of $M$, and $G \in \Xcal$ a $\GBC$-amenable global well-order of $N$. 
\begin{itemize}
\item If $(M,\Xcal) \models \KM$ then $(N,\Lcal(N,G)) \models \KM$.
\item If $(M,\Xcal) \models \KMm$ then $(N,\Lcal(N,G)) \models \KMm$.
\item If $(M,\Xcal) \models \GBC + \PnCA k$ then $(N,\Lcal(N,G)) \models \GBC + \PnCA k$.
\item If $(M,\Xcal) \models \GBCm + \PnCA k$ then $(N,\Lcal(N,G)) \models \GBCm + \PnCA k$.
\end{itemize}
\end{theorem}

This is most of theorem \ref{thm2:inner-model-rlzble}. We have not yet handled the cases of $\GBC + \ETR$ and $\GBCm + \ETR$, which will happen in the next chapter. It also remains to see that $(N,\Lcal(N,G))$ satisfies a fragment of Class Collection corresponding to the amount of Comprehension satisfied by the ground universe.

\begin{proposition} \label{prop2:coll-in-l}
If the ground universe $(M,\Xcal)$ satisfies $\Pi^1_k$-Comprehension, then $\Lfrak(N,G)$ satisfies $\Sigma_k$-Collection.
\end{proposition}

As an immediate corollary we get the following.

\begin{corollary}[Over $\GBCm + \ETR$]
If the ground universe satisfies full Comprehension, then $\Lfrak(N,G)$ satisfies Collection. \qed
\end{corollary}

\begin{proof}[Proof of proposition \ref{prop2:coll-in-l}]
The key step in the argument for this proof is encapsulated by the following lemma.

\begin{sublemma}[Over $\GBCm + \PnCA k$] \label{lem2:sol-refl}
The $(N,G)$-constructible unrolled model $\Lfrak(N,G)$ satisfies $\Sigma_k$-reflection with respect to the $L(N,G)$-hierarchy: For any $a \in \Lfrak(N,G)$ and any $\Sigma_k$-formula $\phi(x,p)$ there is an ordinal $\gamma$ so that $a \in L_\gamma(N,G)^{\Lfrak(N,G)}$ and for all $x \in L_\gamma(N,G)^{\Lfrak(N,G)}$ we have $L_\gamma(N,G)^{\Lfrak(N,G)} \models \phi(x,a)$ if and only if $\Lfrak(N,G) \models \phi(x,a)$. Moreover, such $\gamma$ are unbounded in the ordinals of $\Lfrak(N,G)$.
\end{sublemma}

Before proving the lemma, let us see how it lets us prove the proposition. This is the familiar argument that reflection implies collection.

Work in $\Lfrak(N,G)$. Fix $a$ and assume that for every $x \in a$ there is $y$ so that $\phi(x,y,p)$, where $p$ is some parameter. We want to find a $b$ so that for all $x \in a$ there is $y \in b$ so that $\phi(x,y,p)$. By the lemma, there is $\gamma$ so that $\phi$ reflects to $L_\gamma(N,G) \ni a$. Hence, for every $x \in a$ there is $y \in L_\gamma(N,G)$ so that $\phi(x,y,p)$. The desired $b$ is then $L_\gamma(N,G)$.

\begin{proof}[Proof of lemma \ref{lem2:sol-refl}] 
I will give a false proof, indicate the error therein, and then explain how to fix the error, giving a proof that works.

The first way one might attempt to prove this is to work in $\Lfrak(N,G)$ and try to find $\gamma$ so that $L_\gamma(N,G)$ is closed under witnesses to existential subformulae of $\phi$. We start with $\gamma_0$ so that $p \in L_{\gamma_0}(N,G)$. Then, define $\gamma_{n+1}$ to be the least ordinal $\eta$ so that $L_\eta(N,G)$ is closed under witnesses to existential subformulae of $\phi$ with parameters in $L_\gamma(N,G)$. Such $\gamma_{n+1}$ exists because the ground universe satisfies $\Pi^1_k$-Comprehension and $\phi^\starl$ is $\Sigma^1_k$. We then define $\gamma$ to be the supremum of the sequence $\seq{\gamma_n}$. 

The problem with this argument is that it is circular; Collection says we can collect the $\gamma_n$, but that is what we are trying to prove! If we look from the perspective of the $\KM$ model, we are appealing to Class Collection to pick out (membership codes for) these $\gamma_n$. If we could find a way to pick these $\gamma_n$ without using Class Collection, then the argument would go through. We would get a sequence of meta-ordinals $\seq{\Gamma_n}$ coded in $\Xcal$ and from that be able to construct a membership code for $L_{\sup_n \Gamma_n}(N,G)$.

Fix a relation $C \subseteq \Ord^2$ which is isomorphic to $\in^N$. Call a membership code {\em compliant} if $E \rest \Ord = C$. Let $\Gamma$ be a meta-ordinal. I claim that in $(M,\Xcal)$ we can uniquely pick an isomorphic copy of $\Gamma$. First, note that if $\Gamma \vin L_\Delta(N,G)$ then in $L_\Delta(N,G)$ there is a least $\Gamma_\Delta \cong \Gamma$. That is, $\Gamma_\Delta = L_\Delta(N,G) \downarrow g$ for some $g \lol t_{L_\Delta(N,G)}$ where this $g$ is picked according to the $L$-order for $L_\Delta(N,G)$. The key observation now is that if $L_\Delta(N,G)$ and $L_{\Delta'}(N,G)$ are compliant, then $\Gamma_\Delta = \Gamma_{\Delta'}$. This is because they agree on how they code $N$ but elements of $\Gamma_\Delta$ and $\Gamma_{\Delta'}$ are pairs of elements of the code of $N$ in $L_\Delta(N,G)$ and $L_{\Delta'}(N,G)$, respectively. This yields a way to uniquely pick an isomorphic copy of $\Gamma$: pick any $\Delta$ so that $\Gamma \vin L_\Delta(N,G)$ and then look at $\Gamma_\Delta$. This is well-defined.

We can use this now to pick the $\Gamma_n$ for the above argument. Since we can do this for all $n$, we can code the sequence $\seq{\Gamma_n}$ by the class $Z = \{ (n,a) : a \in \Gamma_n \}$. From this produce a meta-ordinal $\Gamma$ of ordertype $\sup_n \Gamma_n$ and build $L_\Gamma(N,G)$. Moving to the first-order model, $L_\Gamma(N,G)$ is a representative of the equivalence class for the desired $L_\gamma(N,G)$.

Finally, note that the moreover is immediate as to ensure that $\gamma > \alpha$ for some ordinal $\alpha$ one can run the argument using $a' = (a,\alpha)$ rather than $a$.
\end{proof} \let\qed\relax
\end{proof}

The same argument shows that we can get a fragment of Class Collection starting from $\ETR$, except using the $\star$-translation of definition \ref{def2:etr-star-trans} rather than the $*$-translation.\footnote{See subsection \ref{subsec2:unroll-in-etr} for discussion of why we need the alternate translation.} 
In fact, the argument is a little easier in this case since we only have to worry about $\Sigma_0$-formulae and if $\phi$ is $\Sigma_0$ then $\phi^L$ is just $\phi$. 

\begin{corollary}[Over $\GBCm + \ETR$]
The $(N,G)$-constructible unrolling $\Lfrak(N,G)$ satisfies Elementary Class Collection $\ECC$. \qed
\end{corollary}

This is the final step needed to prove corollary \ref{cor2:get-a-plus}, which I restate below for the convenience of the reader who does not want to flip back a few pages.

\begin{theorem*}
Let $(M,\Xcal) \models \GBCm + \ETR$. Then there is a definable hyperclass $\Ycal \subseteq \Xcal$ so that the following.
\begin{itemize}
\item If $(M,\Xcal) \models \KM$ then $(M,\Ycal) \models \KMCC$.
\item If $(M,\Xcal) \models \KMm$ then $(M,\Ycal) \models \KMCCm$.
\item If $(M,\Xcal) \models \GBC + \PnCA k$ then $(M,\Ycal) \models \GBC + \PnCAp k$, for $k \ge 1$.
\item If $(M,\Xcal) \models \GBCm + \PnCA k$ then $(M,\Ycal) \models \GBCm + \PnCAp k$, for $k \ge 1$.
\item If $(M,\Xcal) \models \GBC + \ETR$ then $(M,\Ycal) \models \GBC + \ETR + \ECC$.
\item If $(M,\Xcal) \models \GBCm + \ETR$ then $(M,\Ycal) \models \GBCm + \ETR + \ECC$.
\end{itemize}
\end{theorem*}




\section{\texorpdfstring{The smallest heights of transitive models and $\beta$-models}{The smallest heights of transitive models and beta-models}}

As an application of the results in this chapter I would like to investigate the smallest heights of transitive and $\beta$-models of $\GBC + \PnCA k$. This builds upon prior work of Marek and Mostowski.

\begin{theorem}[Marek and Mostowski \cite{marek-mostowski1975}]
Assume there is a $\beta$-model of $\KM$. Let $\beta_\omega$ be the least height of a $\beta$-model of $\KM$ and let $\tau_\omega$ be the least height of a transitive model of $\KM$. Then
\begin{itemize}
\item $\tau_\omega < \beta_\omega$; and moreover
\item $L_{\beta_\omega} \models \tau_\omega$ is countable.
\end{itemize}
\end{theorem}

I will generalize their argument to fragments of $\KM$.

\begin{definition}
Let $\beta_\omega$ be the least height of a $\beta$-model of $\KM$ and let $\tau_\omega$ be the least height of a transitive model of $\KM$. For $n \in \omega$ let
\begin{itemize}
\item $\beta_n$ be the least height of a $\beta$-model of $\GBC + \PnCA n$; and
\item $\tau_n$ be the least height of a transitive model of $\GBC + \PnCA n$.
\end{itemize}
\end{definition}

By corollary \ref{cor2:get-a-plus} it is equivalent to define $\beta_n$ as the least height of a $\beta$-model of $\GBC + \PnCAp n$ and $\tau_n$ as the least height of a transitive model of $\GBC + \PnCAp n$. 

\begin{definition}
For $\gamma < \delta$ ordinals countable in $L$ say that $\gamma \ll \delta$ if $L_\delta \models \gamma$ is countable.
\end{definition}

\begin{theorem} \label{thm2:beta-tau-1}
Let $m < n \le \omega$. Assume there is a transitive model of $\GBC + \PnCA n$.  Then, $\beta_m \ll \tau_n$.
\end{theorem}

\begin{theorem} \label{thm2:beta-tau-2}
Let $1 \le n \le \omega$. Assume there is a $\beta$ model of $\GBC + \PnCA n$. Then, $\tau_n \ll \beta_n$.
\end{theorem}

See figure \ref{fig2:heights} for a pictorial representation of these theorems.

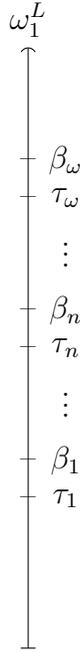
\begin{figure}[ht]
\begin{center}
\begin{tikzpicture}[d/.style={}]

\draw[|-)] (0,0) -> (0,8) node [above] {$\omega_1^L$};

\draw (0,2)   node[d] {--} (.5,2)   node {$\tau_1$}
      (0,2.5) node[d] {--} (.5,2.5) node {$\beta_1$}
                          (.5,3.35) node {$\vdots$}
      (0,4)   node[d] {--} (.5,4)   node {$\tau_n$}
      (0,4.5) node[d] {--} (.5,4.5) node {$\beta_n$}
                          (.5,5.35) node {$\vdots$}
      (0,6)   node[d] {--} (.5,6)   node {$\tau_\omega$}
      (0,6.5) node[d] {--} (.5,6.5) node {$\beta_\omega$};

\end{tikzpicture}
\end{center}

\caption{The least heights of transitive and $\beta$-models of fragments of $\KM$. The ordinals in the picture are ordered by $\ll$.}
\label{fig2:heights}
\end{figure}

\begin{proof}[Proof of theorem \ref{thm2:beta-tau-1}]
Let $m < n \le \omega$. We need to see that $\beta_m \ll \tau_n$.

Consider $(N,\Ycal) \models \GBC + \PnCA n$ with $\Ord^N = \tau_n$. Then, by the existence of $\Sigma^1_m$-truth predicates in $(N,\Ycal)$ we get $C \in \Ycal$ coding $\bar \Xcal \subseteq \Ycal$ so that $(N,\bar \Xcal) \models \GBC + \PnCA m$. Work in $(N,\Ycal)$. Apply reflection to this $C$ to get a club of ordinals $\alpha$ with $C_\alpha \subseteq V_\alpha$ coding $\Xcal_\alpha \subseteq \powerset(V_\alpha)$ so that $(V_\alpha, \Xcal_\alpha) \models \GBC + \PnCA m$. Pick $\alpha$ with uncountable cofinality from this club. Then $N \models (V_\alpha, \Xcal_\alpha)$ is a $\beta$-model because any transitive model with uncountable cofinality is a $\beta$-model.\footnote{See observation \ref{obs1:unc-cof-beta}.}
But $N$ is a transitive model of $\ZFC$ so it is correct about well-foundedness, so $(V_\alpha^N, \Xcal_\alpha)$ really is a $\beta$-model. 

This establishes that $\beta_m < \tau_n$. To further see that $\beta_m \ll \tau_n$ observe that $N$ sees the unrolling $\bar W$ of the model of $\GBC + \PnCAp m$ contained inside $(V_\alpha^N,\Xcal_\alpha)$. But then this $\bar W$ has a countable (from the perspective of $N$) submodel $W$. Cutting off this $W$ gives a countable $\beta$-model of $\GBC + \PnCA m$.
\end{proof}

\begin{proof}[Proof of theorem \ref{thm2:beta-tau-2}]
Fix $n$ with $1 \le n \le \omega$. We want to see that $\tau_n \ll \beta_n$. Let $H_n(\alpha)$ denote $L_\xi$ where $\xi > \alpha$ is least so that $L_\xi \models \ZFCm(n)$.\footnote{Following the naming convention of definition \ref{def2:so-many-theories}, $\ZFCm(n)$ is the theory axiomatized by Extensionality, Pairing, Union, Foundation, Infinity, Choice, $\Sigma_n$-Separation, and $\Sigma_n$-Collection. So, for instance, $\ZFCm(0)$ is $\KP$ plus Choice plus Infinity.}
And let $\Hyp(\alpha)$ be the least admissible set containing $\alpha$, i.e.\ $L_\xi$ where $\xi > \alpha$ is least so that $L_\xi \models \KP$. It is then obvious that $\Hyp(\alpha) \subseteq H_n(\alpha)$. Consider the admissible sets $A = \Hyp(\beta_n)$ and $B = H_n(\beta_n)$. Consider the $\Lcal_A$-theory\footnote{Recall that if $A$ is an admissible set then $\Lcal_A$ is the admissible fragment of $\Lcal_{\Ord,\omega}$ associated with $A$, i.e.\ the infinitary language consisting of formulae in $A$. See \cite{barwise1975} for details.}
$T$ whose axioms consist of
\begin{itemize}
\item $\ZFCmi(n)$;
\item The infinitary $\in$-diagram of $A$;\footnote{That is, the collection of all sentences of the form $\forall x\ x \in a \iff \bigvee_{b \in a} x = b$. Any structure which satisfies all these sentences will have $A$ as a transitive submodel. That is, if $N$ satisfies all these sentences then $A \subseteq N$ and for every $a \in A$ and $b \in N$ if $N \models b \in a$ then $b \in A$.}
and 
\item The assertion that $\beta_n$ is inaccessible.
\end{itemize}
Then $T$ is $\Sigma_1$-definable over $A$ via a formula $\theta$. I claim that $B \models T$ and thus $T$ is consistent. First, it is clear that $A \in B$ and $B \models \ZFCm(n)$. What remains is to see that $B \models \beta_n$ is inaccessible.
To see this, let $(L_{\beta_n},\Xcal)$ be a $\beta$-model of $\GBC + \PnCAp n$, whose classes are all constructible. Then $(L_{\beta_n},\Xcal)$ unrolls to a transitive model $L_\xi \models \ZFCmi(n)$ with $\beta_n \in L_\xi$ inaccessible and the largest cardinal of $L_\xi$. By definition of $B$, we have that $\Ord^B \le \xi$. So $B$ must agree with $L_\xi$ that $\beta_n$ is inaccessible and is the largest cardinal. Thus, $B \models T$.\footnote{If we think in terms of the $\GBC + \PnCAp n$ model $(L_{\beta_n},\Xcal)$, then what we have seen is that this structure contains a membership code for $\Hyp(V)$, the smallest admissible `set' containing its $V$. See chapter 4 for more on $\Hyp(V)$, where it will play an important role.}

It looks like we are setting up to apply the Barwise compactness theorem, but we are not yet ready to do so. (Indeed, we will not apply the Barwise compactness theorem directly to $T$.) First, work in $B$. By condensation there are countable (i.e.\ from the perspective of $B$) ordinals $\gamma,\delta$ so that $j : (L_\gamma, \in, \delta) \to (A, \in, \beta_n)$ is an elementary embedding. Then, $A \models \gamma$ is countable; this is true in $B$ and the bijection from $\gamma$ to $\omega$ must occur earlier than $\beta_n$ in the $L$-hierarchy because $A \models \beta_n$ is inaccessible. 

Now let $T'$ be $L_\gamma$'s version of $T$; formally, $T' = \{ \phi : L_\gamma \models \theta(\phi) \}$. We can explicitly list the axioms of $T'$, namely:
\begin{itemize}
\item $\ZFCmi(n)$;
\item The infinitary $\in$-diagram of $L_\gamma$; and
\item The assertion that $\delta$ is inaccessible.
\end{itemize}
Because $A \models \Con(T)$ we get by elementarity that $L_\gamma \models \Con(T')$. Also by elementarity, $L_\gamma = \Hyp(\delta)$. 

Finally we apply the Barwise compactness theorem, but within $A$ to the theory $T'$. This yields $C \in A$ with $C \models T'$ and $C \supsetend L_\xi$. So the same is true in $V$, by absoluteness. Applying the cutting off construction to $C$ we get that $L_\xi$ is $(\GBC + \PnCAp n)$-realizable. So $\tau_n \le \delta \ll \beta_n$, as desired. 
\end{proof}

Essentially the same argument works for the version of the theories without Powerset, i.e. $\KMm$ and $\GBCm + \PnCA k$. Let $\beta_\omega^-$ be the least height of a $\beta$-model of $\KMm$ and $\tau_\omega^-$ be the least height of a transitive model of $\KMm$. For $k \in \omega$ let $\beta_k^-$ be the least height of a $\beta$-model of $\GBCm + \PnCA k$ and let $\tau_k^-$ be the least height of a transitive model of $\GBCm + \PnCA k$. 

\begin{theorem}
Assume there is a $\beta$-model of $\KMm$. Let $n < m \le \omega$. Then the following hold.
\begin{itemize}
\item $\beta_n^- \ll \tau_m^-$; and
\item $\tau_n^- \ll \beta_n^-$.
\end{itemize}
\end{theorem}

\begin{proof}[Proof sketch]
Just like the proof of theorems \ref{thm2:beta-tau-1} and \ref{thm2:beta-tau-2}. The difference is that when axiomatizing the theory $T$ use $\ZFCmr(n)$ rather than $\ZFCmi(n)$.
\end{proof}

It is quite obvious that $\beta_n^- \ll \tau_m$ for any $n,m \le \omega$. So the above picture could be extended by putting the minus ordinals at the bottom, ordered the same but with superscripts everywhere.

We can ask the same question for other theories. For a second-order set theory $T$, let $\tau(T)$ be the least height of a transitive model of $T$ and let $\beta(T)$ be the least height of a $\beta$-model of $T$.

\begin{question}
Do we have $\tau(\GBC) < \beta(\GBC)?$ Do we have $\tau(\GBC + \ETR) < \beta(\GBC + \ETR)$?
\end{question}

\clearpage

\chapter{Truth and transfinite recursion}
\chaptermark{Truth and transfinite recursion}

\epigraph{\singlespacing Die alte und ber\"uhmte Frage, womit man die Logiker in die Enge zu treiben vermeinte und sie dahin zu bringen suchte, dass sie sich entweder auf einer elenden Diallele mussten betreffen lassen oder ihre Unwissenheit, mithin die Eitelkeit ihrer ganzen Kunst bekennen sollten, ist diese: {\em Was ist Wahrheit?}}{Immanuel Kant}

In this chapter I will explicate the relationship between transfinite recursion and iterated truth predicates. This will be used to separate fragments of $\ETR$ and $\SkTR k$. (See definition \ref{def3:sktr} below, and note that $\Sigma^1_0$-Transfinite Recursion is a synonym for Elementary Transfinite Recursion.) The main result of this chapter is the following.

\begin{theorem} \label{thm3:main-thm}
Consider $(M,\Xcal) \models \GBC$ and $\Gamma \in \Xcal$ a well-order with $\Gamma \ge \omega^\omega$. Fix finite $k \ge 0$. Then, if $(M,\Xcal)$ satisfies the $\Sigma^1_k$-Transfinite Recursion principle for recursions of height $\le \Gamma \cdot \omega$, there is $\Ycal \subseteq \Xcal$ coded in $\Xcal$ so that $(M,\Ycal) \models \GBC$ plus the $\Sigma^1_k$-Transfinite Recursion principle for recursions of height $\le \Gamma$.
\end{theorem}

This theorem gives a separation of fragments of $\Sigma^1_k$-Transfinite Recursion by consistency strength. Combined with the easy facts that $\Pi^1_{k+1}$-Comprehension proves $\Sigma^1_k$-Transfinite Recursion and that $\Sigma^1_{k+1}$-Transfinite Recursion for recursions of finite length proves $\Pi^1_k$-Comprehension, this gives a hierarchy of transfinite recursion principles ranging in strength from $\GBC$ to $\KM$. 
Figure \ref{fig3:tfr-hierarchy} gives a visual representation of the hierarchy of the second-order transfinite induction principles, spanning from $\GBC$ to $\KM$. 

\begin{figure}

\begin{center}
\begin{tikzpicture}[scale=.75]

\draw[|-|] (0,0) -> (0,25);

\draw                   (.5,0)  node[anchor=west] {$\ETR_1$}
                       (-.5,6)  node[anchor=east] {Open Class Determinacy}
                       (-.5,5)  node[anchor=east] {Clopen Class Determinacy}
                       (-.5,3)  node[anchor=east] {Class Forcing Theorem}
                       (-.5,0)  node[anchor=east] {$\GBC$}
      (0,1)   node {--} (.5,1)  node[anchor=west] {$\ETR_{\omega}$}
      (0,2)   node {--} (.5,2)  node[anchor=west] {$\ETR_\gamma$}
      (0,3)   node {--} (.5,3)  node[anchor=west] {$\ETR_\Ord$}
      (0,4)   node {--} (.5,4)  node[anchor=west] {$\ETR_\Gamma$}
      (0,5)   node {--} (.5,5)  node[anchor=west] {$\ETR$}
      (0,7)   node {--} (.5,7)  node[anchor=west] {$\SkTR 1_1$}
                       (-.5,7)  node[anchor=east] {$\PCA$}
      (0,8)   node {--} (.5,8)  node[anchor=west] {$\SkTR 1_\gamma$}
      (0,9)   node {--} (.5,9)  node[anchor=west] {$\SkTR 1_\Ord$}
      (0,10)  node {--} (.5,10) node[anchor=west] {$\SkTR 1_\Gamma$}
      (0,11)  node {--} (.5,11) node[anchor=west] {$\SkTR 1$}
      (0,13)             (1,13) node[anchor=west] {$\vdots$}
      (0,15)  node {--} (.5,15) node[anchor=west] {$\SkTR k_1$}
                       (-.5,15) node[anchor=east] {$\PnCA k$}
      (0,16)  node {--} (.5,16) node[anchor=west] {$\SkTR k_\gamma$}
      (0,17)  node {--} (.5,17) node[anchor=west] {$\SkTR k_\Ord$}
      (0,18)  node {--} (.5,18) node[anchor=west] {$\SkTR k_\Gamma$}
      (0,19)  node {--} (.5,19) node[anchor=west] {$\SkTR k$}
      (0,21)  node {--} (.5,21) node[anchor=west] {$\SkTR{k+1}_1$}
                       (-.5,21) node[anchor=east] {$\PnCA{k+1}$}
                         (1,23) node[anchor=west] {$\vdots$}
                        (.5,25) node[anchor=west] {$\SkTR{\omega}$}
                       (-.5,25) node[anchor=east] {$\KM$}
      (8,25) node {}; 

\end{tikzpicture}
\end{center}

\caption{Transfinite recursion, from $\GBC$ to $\KM$. Ordered by consistency strength.}
\label{fig3:tfr-hierarchy}
\end{figure}
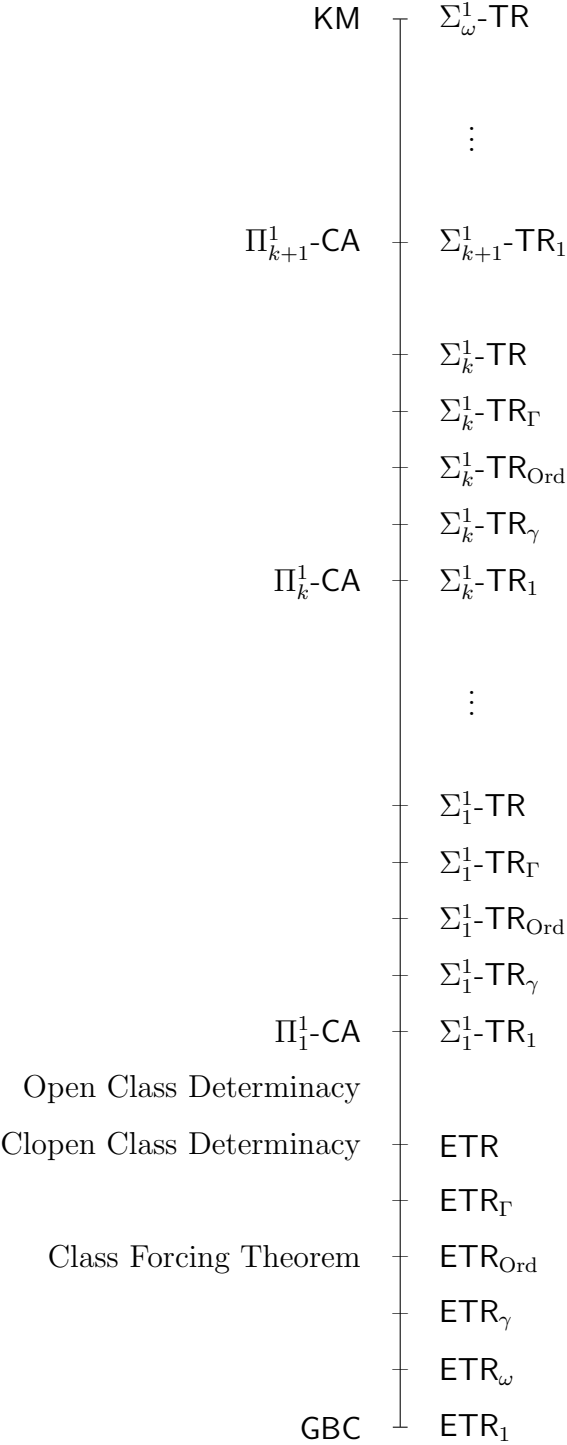

This chapter is organized as follows. First, we consider iterated truth predicates and see that (fragments of) Elementary Transfinite Recursion can be equivalently formulated as asserting the existence of certain iterated truth predicates. Next we look at the connections between iterated truth predicates and the second-order constructible universe. This will show that $\GBC + \ETR$ and $\GBCm + \ETR$ go down to inner models, completing the proof of theorem \ref{thm2:inner-model-rlzble} from chapter 2. We will also see the same for fragments of $\ETR$, though by a different argument. We then settle theorem \ref{thm3:main-thm} for the case $k=0$, separating fragments of $\ETR$. We further will see that fragments of $\ETR$ can be separated in a strong sense, getting transitive models of $\ZFC$ satisfying a strong theory which are $(\GBC + \ETR_\Gamma)$-realizable but not $(\GBC + \ETR_{\Gamma \cdot \omega})$-realizable. So the separation here is due to an inherently second-order property of the model, not its (first-order) theory. This is followed by a detour to arithmetic, where we consider an analogous situation for iterated truth predicates over models of arithmetic. We see that the strong separation results for fragments of $\ETR$ do not have a counterpart in the arithmetic realm. Finally, we turn to non-elementary recursion and settle the $k \ge 1$ case of theorem \ref{thm3:main-thm}.

Much of this chapter will make use of the constructions from chapter 2, so the reader is strongly encouraged to read that chapter first.

\section{Truth and iterated truth}

In chapter 1 we discussed truth predicates over models of set theory. I wish to return to this subject for a deeper look. As we will see in this chapter, certain principles of second-order set theory can be characterized in terms of the existence of truth predicates as classes. The best known case of such is the Tarskian truth predicate for first-order formulae, which we saw in chapter 1. But we can also consider truth predicates relative to a class parameter.

\begin{definition} \label{def3:tarski-relative}
Let $A$ be a class over a model $M$ of first-order set theory. The {\em truth predicate for $M$ relative to $A$}---or, synonymously, the {\em satisfaction class for $M$ relative to $A$}---is the class $T$ satisfying the recursive Tarskian definition of truth for formulae in the language $\Lcal_\in(A)$, i.e.\ the language of first-order set theory with a symbol for $A$. Formally, these recursive requirements are the same as in definition \ref{def1:tarski-truth-pred}, with the following addition:
\begin{itemize}
\item $(x \in A,a)$ is in $T$ if and only if $a \in A$.
\end{itemize}
\end{definition}

\begin{observation} \label{obs3:truth-is-unique}
Let $(M,\Xcal) \models \GBcm$. If $T,T' \in \Xcal$ both satisfy the definition of a truth predicate relative to $A \in \Xcal$ then $T = T'$. 
\end{observation}

\begin{proof}
If $T \ne T'$ then they disagree at a minimal stage. But this would contradict the recursive requirement at that stage.
\end{proof}

As such, we are justified in talking about {\em the} truth predicate relative to $A$.\footnote{The reader who is familiar with satisfaction classes over non-$\omega$-models may want to object here. She should hold her objection. We will discuss that case shortly.}
I will use $\Tr(A)$ to denote the truth predicate relative to $A$.

As a first application of this idea, let us see that there is no principal model of $\GBC + \ETR$. 

\begin{proposition}
Let $(M,\Xcal)$ be a model of second-order set theory. If for every $A \in \Xcal$ we have $\Tr(A) \in \Xcal$ then $(M,\Xcal)$ is not principal.
\end{proposition}

\begin{proof}
Otherwise, if $\Xcal = \Def(M;P)$ for some $P \in M$ then $\Tr(P)$ is definable from $P$, contradicting Tarski's theorem on the undefinability of truth.
\end{proof}

\begin{proposition} \label{prop3:etromega-truth-preds}
The theory $\GBCm + \ETR_\omega$ proves that $\Tr(A)$ exists for every class $A$.
\end{proposition}

\begin{proof}
From $A$ we can define $\Tr(A)$ by means of an elementary transfinite recursion of height $\omega$. See definitions \ref{def3:tarski-relative} and \ref{def1:tarski-truth-pred}.
\end{proof}

\begin{corollary}
No model of any $T \supseteq \GBCm + \ETR_\omega$ is principal. \qed
\end{corollary}

Proposition \ref{prop3:etromega-truth-preds} gives that every model of $\GBCm + \ETR_\omega$ contains truth predicates. This includes non-$\omega$-models, so it would be helpful to discuss what that means in this context.

The trouble is simple. Consider a $\omega$-nonstandard model $M$ and its externally defined satisfaction relation $S = \{ (\phi,a) : M \models \phi(a) \}$. This $S$ will only include standard formulae in its domain, so it cannot be amenable to $M$. 
So this cannot be the class in $\Xcal$ which $(M,\Xcal)$ thinks is the truth predicate if $\Xcal$ is to be a $(\GBCm + \ETR_\omega)$-realization for $M$. Instead, the class in $\Xcal$ which $(M,\Xcal)$ thinks is the truth predicate must measure the `truth' of all formulae in $M$, including the nonstandard formulae. In the literature on nonstandard models, such a class is also known as a {\em full satisfaction class}. 

Let us step away for a moment from the context of second-order set theory. 
Classical results of Krajewski \cite{krajewski1976} show that full satisfaction classes can be non-unique, in contrast to the uniqueness of truth predicates in the second-order context. If $M$ is a countable $\omega$-nonstandard model of, say, $\ZFC$ which admits a full satisfaction class then it admits continuum many different full satisfaction classes. But the disagreement must be confined to the nonstandard realm. An easy induction shows that any full satisfaction class for $M$ agrees with the truth predicate of $M$ (as seen externally from $V$) for standard formulae. 

This implies that for $\omega$-nonstandard models $M$ of set theory we can have different $\GBCm$-realizations for $M$ which contain different truth predicates. On the other hand, if $M$ is $\omega$-standard and two $\GBCm$-realizations for $M$ each contain a truth predicate then their truth predicates are the same.

Let us return now to truth predicates relative to a class parameter. Suppose for the moment that we have a model of, say, $\GBC$ satisfying $\forall A\ \Tr(A)$ exists. In particular this holds when $A$ itself is a truth predicate. So we can talk about truth about truth, truth about truth about truth, and so on. It is easy to see that this gives us $n$th order truth for every standard $n$.

But we are set theorists here. Our inclination is to extend things to the transfinite. We would like to do that for truth about truth about\ldots{}  Formally, this is captured by the following definition of an iterated truth predicate. Informally, an iterated truth predicate is a class of triples $(\gamma, \phi, \bar a)$ so that $\phi(\bar a)$ is true at level $\gamma$, where $\phi$ is allowed to make reference to level $< \gamma$.

\begin{definition}
Let $(M,\Xcal)$ be a model of second-order set theory with $\Gamma \in \Xcal$ a well-order. An {\em iterated truth predicate of length $\Gamma$} (or, synonymously, a {\em $\Gamma$-iterated truth predicate}\footnote{The reader may now see one advantage of using ``truth predicate'' in place of ``full satisfaction class'', even in the $\omega$-nonstandard case. The latter would then lead to talk of iterated full satisfaction classes, which is a mouthful---as can be attested by anyone who has given a talk about iterated truth predicates where he made the mistake of using the wrong terminology.})
 for $M$ is a class $T$ of triples $(\gamma, \phi, \bar a)$ with $\gamma \in \dom \Gamma$ satisfying the following recursive requirements. Here, $\phi$ is in $\Lcal_\in(\Trm)$, the language of set theory augmented with a trinary predicate $\Trm$, and $\bar a$ is a valuation for $\phi$.
\begin{enumerate}
\item $(\gamma, x = y, \bar a)$ is in $T$ if and only if $a_x = a_y$.\footnote{To clarify the notation: if $\bar a$ is a valuation with variable $x$ in its domain, then $a_x$ is the value assigned to $x$.}
\item $(\gamma, x \in y, \bar a)$ is in $T$ if and only if $a_x \in a_y$. 
\item $(\gamma, \Trm(x,y,z), \bar a)$ is in $T$ if and only if
 \begin{itemize}
 \item $a_x <_\Gamma \gamma$;
 \item $a_y$ is an $\Lcal_{\in}(\Trm)$-formula;
 \item $a_z$ is a valuation for $a_y$; and
 \item $(a_x, a_y, a_z)$ is in $T$.
 \end{itemize}
\item $(\gamma, \phi \lor \psi,\bar a)$ is in $T$ if and only if $(\gamma, \phi, \bar a)$ or $(\gamma, \psi, \bar a)$ are in $T$.
\item $(\gamma, \neg \phi, \bar a)$ is in $T$ if and only if $(\gamma, \phi, \bar a)$ is not in $T$.
\item $(\gamma, \exists x\ \phi(x), \bar a)$ is in $T$ if and only if there is $b \in M$ so that $(\gamma, \phi, b \cat \bar a)$ is in $T$.\footnote{Here we of course have the implicit requirement that $x$ be free in $\phi$. 

To clarify the situation with the valuations: By $b \cat \bar a$ I mean the valuation which modifies $\bar a$ by assigning $x$ the value $b$. Note that it could be that $\bar a$ already assigns $x$ a value, as there is no requirement that our valuations only assign values to free variables which appear in the formula. (Indeed, clause $(4)$ of this definition will imply that this always happens; consider e.g.\ the formula $\phi$ given by $x=x \lor y=y$. Then $(\gamma,\phi,\seq{a_x,a_y}) \in T$ if and only if $(\gamma,x=x,\seq{a_x,a_y}) \in T$ or $(\gamma,y=y,\seq{a_x,a_y}) \in T$.) So if $\bar a$ does not assign $x$ a value then $b \cat \bar a$ extends $\bar a$ by assigning $b$ to $x$. Otherwise, if $\bar a$ does assign $x$ a value, then we get $b \cat \bar a$ by dropping that assignment from $\bar a$ and then adding in the assignment of $b$ to $x$.

In the sequel I will avoid repeating this footnote, but the reader should keep these issues in mind.}
\end{enumerate}
We can also have iterated truth predicates relative to a class parameter $A$. This has the same definition, except that the formulae must be in the language $\Lcal_\in(\Trm, A)$ and the following additional criterion must be satisfied:
\begin{itemize}
\item $(\gamma, x \in A, a)$ is in $T$ if and only if $a \in A$.
\end{itemize}
\end{definition}

Note that the property of being the $\Gamma$-iterated truth predicate relative to $A$ is first-order expressible (in parameters $\Gamma$ and $A$). So if $(M,\Xcal)$ and $(M,\Ycal)$ are $\GBc m$ models which both contain $T,\Gamma,A$ then they agree on whether $T$ is the $\Gamma$-iterated truth predicate relative to $A$.

The following observation generalizes observation \ref{obs3:truth-is-unique} and is proved in the same manner. It justifies talk of {\em the} $\Gamma$-iterated truth predicate.

\begin{observation}
Let $(M,\Xcal) \models \GBcm$ with $\Gamma \in \Xcal$ a well-order. Suppose $T,T' \in \Xcal$ both satisfy the definition of a $\Gamma$-iterated truth predicate. Then $T = T'$. The same fact holds for iterated truth predicates relative to a parameter. \qed
\end{observation}

I will use $\Tr_\Gamma(A)$ to refer to the $\Gamma$-iterated truth predicate relative to $A$ and $\Tr_\Gamma$ to refer to the parameter-free $\Gamma$-iterated truth predicate. Observe that $\Tr_1(A)$ and $\Tr(A)$, the ordinary Tarskian truth predicate relative to $A$, are inter-definable. 

Note that if $\Gamma$ is (standard) finite, then the existence of $\Tr_\Gamma(A)$ is equivalent to the existence of certain ordinary truth predicates. That is, for standard finite $n$ we have $\Tr_n(A)$ exists if and only if $T_1 = \Tr(A)$, $T_2 = \Tr(T_1)$, \ldots, $T_n = \Tr(T_{n-1})$ all exist. So the main interest in iterated truth is when the length is transfinite or nonstandard finite.

Finally, let me address the skeptic. She may worry there is danger in allowing the iterated truth predicate to measure the truth of statements that themselves make reference to the iterated truth predicate, thus allowing the liar's paradox to slip in. This worry is unfounded. Clause (3) in the definition of an iterated truth predicate legislates that it be ramified, with truth at level $\gamma$ unable to make reference to truth at level $\ge \gamma$. So any vicious circles are avoided.



We are now ready to begin to see why the word ``recursion'' is in the title of this chapter. Namely, the existence of iterated truth predicates gives an equivalent characterization of Elementary Transfinite Recursion.

\begin{theorem}[Fujimoto \cite{fujimoto2012}] \label{thm3:etr-iff-itr}
The following are equivalent over $\GBC$.
\begin{enumerate}
\item The principle of Elementary Transfinite Recursion; and
\item For all class well-orders $\Gamma$ and all classes $A$ the class $\Tr_\Gamma(A)$ exists.
\end{enumerate}
\end{theorem}

\begin{proof}
$(1 \impl 2)$ Note that $\Tr_\Gamma(A)$ is definable via an elementary transfinite recursion of rank $\omega \cdot \Gamma$. To define truth at level $g \in \dom \Gamma$ requires to have first defined truth at all levels $< g$ and each level is defined via a recursion of rank $\omega$. 

$(2 \impl 1)$ Take an instance of $\ETR$, iterating $\phi(x,i,F,A)$ along a well-order $\Gamma$. I claim that from $\Tr_\Gamma(A)$ can be defined a solution $F$ to this recursion. Specifically, let $F = \{ (i,x) : (i,\psi,x) \in \Tr_\Gamma(A) \}$, where $\psi$ is a formula so that $(M,\in,P,\Tr_{\Gamma \rest i}(A)) \models \psi(x,i)$ if and only if $(M,\in,P, F \rest i) \models \phi(x,i)$. (For the latter, $F$ is defined via $\psi$ as above.) Such $\psi$ exists by an application of the G\"odel fixed-point lemma. It then follows that $F$ satisfies $\phi$ at each stage and is therefore a solution to the full recursion.
\end{proof}

This result can be refined to give equivalences for fragments of $\ETR$. 

\begin{corollary}
Let $\Gamma \ge \omega^\omega$ be a well-order. 
Over $\GBC$ the following are equivalent.
\begin{enumerate}
\item The principle of Elementary Transfinite Recursion for recursions of rank $\le \Gamma$; and
\item For all classes $A$ the class $\Tr_\Gamma(A)$ exists. 
\end{enumerate}
\end{corollary}

\begin{proof}
$(1 \impl 2)$ Again, $\Tr_\Gamma(A)$ is definable by a recursion of rank $\omega \cdot \Gamma$. While it may be that $\Gamma < \omega \cdot \Gamma$, because $\omega^\omega \le \Gamma$ it must be that $\omega \cdot \Gamma < \Gamma + \Gamma$. So we can carry out this recursion as we can get solutions to recursions of rank $\Gamma + \Gamma$ by first getting a solution to the first $\Gamma$ many stages, then doing a second recursion to get the rest. 

$(2 \impl 1)$ The same argument goes through as before.
\end{proof}

This also implies that $\GBC + \ETR$ and $\GBC + \ETR_\Gamma$ are finitely axiomatizable. 

\begin{corollary}
The theories $\GBC + \ETR$ and $\GBC + \ETR_\Gamma$ are finitely axiomatizable, for $\Gamma \ge \omega^\omega$.
\end{corollary}

\begin{proof}
Recall the folklore result that $\GBC$ is finitely axiomatizable. Hence the only work is to see that Elementary Transfinite Recursion (or Elementary Transfinite Recursion restricted to $\Gamma$) is finitely axiomatizable. But as Fujimoto's theorem shows, Elementary Transfinite Recursion is equivalent to $\forall A \forall \Gamma\ \Tr_\Gamma(A)$ exists. Similarly, Elementary Transfinite Recursion restricted to $\Gamma$ is equivalent to $\forall A\ \Tr_\Gamma(A)$ exists.
\end{proof}

\section{Iterated truth and constructibility}

The purpose of this section is to explicate the connection between iterated truth predicates and constructability in the second-order realm. 

As a warm-up, let us first see that $\Tr(A)$ gives a code for $\Def(M;A)$.

\begin{proposition} \label{prop3:truth-impl-coded}
Suppose $(M,\Xcal) \models \GBcm$ is an $\omega$-model with $A \in \Xcal$. Then if $\Tr(A)$ is in $\Xcal$ we get that $\Def(M;A)$ is coded in $\Xcal$. In particular, if $\Xcal$ contains a truth predicate then it contains a coded $V$-submodel. 
\end{proposition}

\begin{proof}
Because $M$ is $\omega$-standard $\Tr(A) = \{ (\phi,\bar a) : (M,A) \models \phi(\bar a) \}$. The following class, which is first-order definable from $\Tr(A)$, is a code for $\Def(M;A)$:
\[
\left\{
\big((\phi, \bar a), b\big) : (\phi, \bar a \cat b) \in \Tr(A) \mand \phi(\bar a) \text{ has a single free variable}
\right\}. \qedhere
\]
\end{proof}

For possibly $\omega$-nonstandard models we can do the same coding to get access to what $(M,\Xcal)$ thinks are the definable classes. I will use $\Def^{(M,\Xcal)}(M;A)$, for $A \in \Xcal$, to refer to the hyperclass coded by
\[
\left\{
\big((\phi, \bar a), b\big) : (\phi, \bar a \cat b) \in \Tr(A) \mand \phi(\bar a) \text{ has a single free variable}
\right\}.
\]
Of course, this can only be done if $(M,\Xcal) \models \Tr(A)$ exists. If $M$ is $\omega$-standard then $\Def^{(M,\Xcal)}(M;A) = \Def(M;A)$. But if $M$ is $\omega$-nonstandard then we get that $\Def^{(M,\Xcal)}(M;A) \supsetneq \Def(M;A)$. 

Indeed, over $\GBC$, the existence of coded $V$-submodels is equivalent to the existence of truth predicates.

\begin{proposition} \label{prop3:truth-iff-coded}
Let $(M,\Xcal) \models \GBC$. Then the following are equivalent.
\begin{enumerate}
\item For all $A \in \Xcal$ we have $\Tr(A) \in \Xcal$; and
\item For all $A \in \Xcal$ there is a coded $V$-submodel $(M,\Ycal) \models \GBC$ of $(M,\Xcal)$ with $A \in \Ycal$.
\end{enumerate}
\end{proposition}

\begin{proof}
$(1 \impl 2)$ We saw earlier that we have a coded $V$-submodel of $\GBc$. Specifically, the model $(M,\Def^{(M,\Xcal)}(M;A))$ is coded in $\Xcal$. We want to see that we can get Global Choice. Because $\Def^{(M,\Xcal)}(M;A)$ is coded in $\Xcal$ we have in $\Xcal$ uniform access to the dense subclasses of $\Add(\Ord,1)$ which appear in $\Def(M;A)$. (We also get dense subclasses which $(M,\Xcal)$ thinks are defined by nonstandard formulae, but the important thing is that we get all the dense subclasses defined by standard formulae.) So inside $(M,\Xcal)$ we can line them up in ordertype $\Ord$ to find $C \in \Xcal$ which meets each of them. Then $(M,\Def^{(M,\Xcal)}(M;A,C)) \models \GBC$ is coded in $\Xcal$.

$(2 \impl 1)$ Fix $A$ and let $C$ be a code for a $\GBC$-realization $\Ycal$ for $M$ which contains $A$. Then, by reflection, there are club many ordinals $\alpha$ so that $(V_\alpha^M,\bar A,\bar C) \prec_{\Sigma_1} (M,A,C)$, where $\bar A = A \cap V_\alpha^M$ and $\bar C = C \cap V_\alpha^M$. I claim that $(V_\alpha^M,\bar A) \prec (M,A)$. We can see this by the Tarski--Vaught test. The setup is we assume by induction that we know $\Sigma_n$-elementarity and we would like to show $\Sigma_{n+1}$-elementarity. We fix a $\Sigma_{n+1}$-formula $\exists x\ \phi(x,p,\dot A)$, where $\dot A$ is a predicate symbol for $A$ or $\bar A$ as appropriate. Assume there is $x \in M$ so that $\phi(x,p,\bar A)$ holds for $p \in V_\alpha^M$. We would like to find such an $x'$ in $V_\alpha^M$. But $\{ x \in M : M \models \phi(x,p,A) \}$ is $(C)_y$ for some $y$. So by $\Sigma_1$-elementarity $W = \{ x \in V_\alpha^M : V_\alpha^M \models \phi(x,p,\bar A) \}$ is $(\bar C)_y$ for some $y$. And by $\Sigma_1$-elementarity it is nonempty. So pick $x' \in W$ and we are done.

This then gives us $\Tr(A)$ as a class. Namely, to decide whether $(\phi,\bar a) \in \Tr(A)$ look at some large enough $\alpha$ from this club so that $V_\alpha^M$ can see all the parameters. Then say that $(\phi,\bar a) \in \Tr(A)$ if and only if $(V_\alpha^M,\bar A) \models \phi(\bar a)$. Because (what $M$ thinks is) the satisfaction relation for $(V_\alpha,\bar A)$ satisfies the recursive Tarskian requirements, so will $\Tr(A)$. 
This definition is first-order in the parameters $A$ and $C$, so by Elementary Comprehension $\Tr(A) \in \Xcal$.
\end{proof}

This lays bare the basic idea behind this section: iterating the $\Def$ operation is essentially the same thing as iterating a truth predicate.

Let us now move to the main result of this section.

\begin{theorem} \label{thm3:etr-iff-l-hier}
The following are equivalent over $\GBCm$.
\begin{enumerate}
\item The principle of elementary transfinite recursion.
\item For any class $A$ and any well-order $\Gamma$ the iterated truth predicate $\Tr_\Gamma(A)$ exists.
\item For any class $A$ and any well-order $\Gamma$ the membership code $L_\Gamma(A)$ exists.\footnote{See chapter 2 for a definition of $L_\Gamma(A)$.}
\end{enumerate}
\end{theorem}

\begin{proof}
We have already seen $(1 \iff 2)$. The new content is $(1 \Rightarrow 3)$ and $(3 \Rightarrow 2)$.

$(1 \Rightarrow 3)$ We saw this back in chapter 2. Briefly: fix $A$ and $\Gamma$. Then $L_\Gamma(A)$ is constructed via an elementary recursion of height $\omega \cdot \Gamma$. It takes $\omega$ many steps to construct (a membership code for) $\Def(X)$ from $X$ and this must be done $\Gamma$ many times to get $L_\Gamma(A)$. 

$(3 \Rightarrow 2)$ This essentially comes down to the fact that $L_{\gamma + 2}$ contains the truth predicate for $L_\gamma$. Fix a class $A$. Work in the $(M,A)$-constructible unrolling $\Lfrak(M,A)$. Inductively see that $\Tr_\gamma^{(M,A)}$---the $\gamma$-iterated truth predicate for the structure $(M,\in^M,A)$---appears in $L_{\gamma + \gamma}(M,A)$; if $\Tr_\delta^{(M,A)}$ is in $L_\beta(M,A)$ then $\Tr(\Tr_\delta^{(M,A)})$ is in $L_{\beta + 2}(M,A)$, from which we can define $\Tr_{\delta + 1}^{(M,A)}$. Now given $\Gamma \in \Xcal$ let $\gamma \in \Lfrak(M,A)$ be isomorphic to $\Gamma$. Then we can transform $\Tr_\gamma^{(M,A)}$ into $\Tr_\Gamma(A) \in \Xcal$, as desired.
\end{proof}

\begin{corollary} \label{cor3:etr-inner-model}
If $(M,\Xcal) \models \GBCm + \ETR$ and $N \in \Xcal$ is an inner model of $M$ of $\ZFC$ then $N$ is $(\GBCm + \ETR)$-realizable. In particular, if $(M,\Xcal) \models \GBC + \ETR$ then $N$ is $(\GBC + \ETR)$-realizable.
\end{corollary}

\begin{proof}
Fix $G \in \Xcal$ a $\GBCm$-amenable global well-order of $N$. Such exists by lemma \ref{lem2:gbc-amen-gwo}. By the theorem, $L_\Gamma(N,G)$ exists in $\Xcal$ for all $\Gamma \in \Xcal$. So we can work in the $(N,G)$-constructible unrolling $\Lfrak(N,G)$, i.e.\ the structure consisting of all membership codes $E$---quotiented out by isomorphism---$E$ so that $E \vin L_\Gamma(N,G)$ for some $\Gamma$.\footnote{See chapter 2 for a definition of $\vin$ the membership relation between membership codes.}
Now let $\Ycal = \Lcal(N,G)$. That is, $\Ycal$ consists of the classes of $N$ which are coded by membership codes in $\Lfrak(N,G)$. It is immediate that $(N,\Ycal) \models \GBCm$. We want to see that it also satisfies $\ETR$. Fix $A,\Gamma \in \Ycal$ where $(N,\Ycal) \models \Gamma$ is a well-order. Then, it must be that $(M,\Xcal)$ agrees that $\Gamma$ is a well-order. Otherwise, there is an ordinal $\alpha \in M$ so that $M \models \Gamma \rest F''\alpha$ is ill-founded, where $F \in \Ycal$ is a bijection between $\Ord$ and $\dom \Gamma$, which exists by Global Choice. But by Replacement $\Gamma \rest F''\alpha$ must be in $N$. And $N$ is a transitive submodel of $M$ and they are both models of $\ZFCm$, so they must agree on what sets are well-founded. So $N \models \Gamma \rest F''\alpha$ is ill-founded, so $(N,\Ycal) \models \Gamma$ is ill-founded, a contradiction.

Then there is some $\Delta \in \Xcal$ so that $A,\Gamma$ are coded in $L_\Delta(N,G)$. More formally, $E_A$ and $E_\Gamma$, the canonical membership codes for $A$ and $\Gamma$, are $\vin$-elements of $L_\Delta(N,G)$. But then $L_\Gamma(A) \vin L_{\Delta + \Gamma + 1}(N,G)$. So $L_\Gamma(A) \in \Ycal$. Since this worked for arbitrary $A$ and $\Gamma$, we get by the theorem that $(N,\Ycal) \models \ETR$, completing the argument.
\end{proof}

Let us now turn to $\ETR_\Gamma$. We get that $\GBC + \ETR_\Gamma$ goes down to inner models, but we need a different argument. The trouble is that without satisfying full $\ETR$ it is not clear that the unrolling process gives a model of a sensible theory, so how are we to build second-order $L$? (Cf. subsection \ref{subsec2:how-hard}.)

\begin{theorem} \label{thm3:etr-gamma-inner-model}
Let $(M,\Xcal) \models \GBCm$ and let $N \in \Xcal$ be an inner model of $M$. Suppose $\Gamma \in \Xcal$ is a well-order $\ge \omega^\omega$ and is a $\GBC$-amenable subclass of $N$ so that $(M,\Xcal) \models \GBCm + \ETR_\Gamma$. Then $N$ is $(\GBCm + \ETR_\Gamma)$-realizable, via some $\Ycal \subseteq \Xcal$.
\end{theorem}

Recall that ``$(N,\Ycal) \models \ETR_\Gamma$'' only makes sense when $\Gamma \in \Ycal$ is a well-order. (You can express it as a theory in first-order logic by using a parameter for $\Gamma$.) So in order to have $N$ be $(\GBC + \ETR_\Gamma)$-realizable it must be that $\Gamma$ can be put into the classes for $N$. The condition that $\Gamma$ be $\GBCm$-amenable for $N$ ensures this can happen. Otherwise, you can run into pathologies. For example, if $N = L^M$ and $\Gamma$ codes $0^\sharp$, then $\Gamma$ is not $\GBC$-amenable to $N$.

And of course, if $M$ is a model of Powerset then $N$ must also be a model of Powerset, so this shows that being $(\GBC + \ETR_\Gamma)$-realizable goes down to inner models.

\begin{proof}
Fix $G \in \Xcal$ a bijection $\Ord^M \to N$ which is $\GBCm$-amenable to $N$. (See lemma \ref{lem2:gbc-amen-gwo}.) We define $\Ycal$ an $\ETR_\Gamma$-realization for $N$ as a certain subset of $\Xcal$, built as a union of an $\omega$-chain of $\GBCm$-realizations for $N$. Start with $\Ycal_0 = \Def^{(M,\Xcal)}(N;G,\Gamma)$. Then $\Ycal_0$ is coded in $\Xcal$ because $\Xcal$ has a truth predicate for $N$ relative to $G$ and $\Gamma$. Also, note that $(N,\Ycal_0) \models \Gamma$ is a well-order; otherwise, $(M,\Xcal)$ would also see the witness that $\Gamma$ is ill-founded, contradicting that $(M,\Xcal) \models \Gamma$ is a well-order. Observe that $(N,\Ycal_0) \models \GBCm$. Now, given $\Ycal_n$ let $\Ycal_{n+1}$ be the smallest extension of $\Ycal_n$ which contains all $\Gamma$-iterated truth predicates relative to parameters from $\Ycal_n$. Formally,
\[
\Ycal_{n+1} = \bigcup \left\{ \Def^{(M,\Xcal)}\left(N;\left(\Tr_\Gamma(A)^N\right)^{(M,\Xcal)}\right) : A \in \Ycal_n \right\}.
\]
Some remarks are in order. First, because $\Ycal_n$ is coded in $\Xcal$, so is $\Ycal_{n+1}$. Second, it must be addressed what this iterated truth predicate is. By way of a transfinite induction of height $\omega \cdot \Gamma$, our model $(M,\Xcal)$ can build what it thinks is the $\Gamma$-iterated truth predicate for $N$, relative to a parameter. This is $(\Tr_\Gamma(A)^N)^{(M,\Xcal)}$.

Finally, set $\Ycal = \bigcup_{n \in \omega} \Ycal_n$. We know that $(N,\Ycal) \models \GBCm$, because  $\Ycal$ is the union of an increasing chain of $\GBCm$-realizations for $N$. We now want to see that $(N,\Ycal) \models \ETR_\Gamma$. Fix $A \in \Ycal$. Then $A \in \Ycal_n$ for some $n$. Thus
\[
\Tr_\Gamma(A)^{(N,\Ycal)} = (\Tr_\Gamma(A)^N)^{(M,\Xcal)} \in \Ycal_{n+1} \subseteq \Ycal.
\]
So $(N,\Ycal)$ contains $\Tr_\Gamma(A)$ for all $A \in \Ycal$, as desired.
\end{proof}

A similar strategy can be used to show that $\GBC + \ETR$ is closed under inner models, rather than going through second-order $L$. For full $\ETR$, however, to build $\Ycal_{n+1}$ we want to include iterated truth predicates of all lengths in $\Ycal_n$, not just those of length $\Gamma$. We will also see a version of this construction reappear in chapter 4.

\section{\texorpdfstring{Separating levels of $\ETR$}{Separating levels of ETR}}

In this section we will see that the levels of $\ETR$ form a hierarchy in consistency strength. 

Let us begin with a lemma that $\GBC + \ETR_\Gamma$ proves well-order comparability for $\Gamma$. That is, if $\Delta$ is any well-order then exactly one of the following holds: $\Delta < \Gamma$, $\Delta = \Gamma$, or $\Delta > \Gamma$.\footnote{If $\Delta$ and $\Gamma$ are class well-orders then $\Delta \le \Gamma$ if there is an embedding of $\Delta$ onto an initial segment of $\Gamma$.}
This is a refinement of the fact that $\ETR$ proves that any two class well-orders are comparable.

\begin{lemma}
The theory $\GBC + \ETR_\Gamma$ proves that $\Gamma$ is comparable to any class well-order. That is, if $\Delta$ is a class well-order then either there is an embedding of $\Gamma$ onto an initial segment of $\Delta$ or else an embedding of $\Delta$ onto an initial segment of $\Gamma$.
\end{lemma}

\begin{proof}
Fix $\Delta$. Consider the transfinite recursion which attempts to construct an embedding of $\Gamma$ onto an initial segment of $\Delta$. That is, this recursion builds such an embedding $\pi$ according to the following rule: for $g \in \dom \Gamma$ set $\pi(g)$ to be the least element of $\dom \Delta \setminus \ran (\pi \rest (\Gamma \rest \mathord<_\Gamma g))$, if such exists, otherwise $\pi(g)$ is undefined. There are two cases. If $\pi(g)$ is always defined then we have embedded $\Gamma$ onto an initial segment of $\Delta$. If $\pi(g)$ is ever undefined at a stage then it will be undefined at every subsequent stage. So we get that $\pi\inv$ embeds $\Delta$ onto an initial segment of $\Gamma$.
\end{proof}

It is clear that if $\Delta < \Gamma$ then $\ETR_\Gamma$ implies $\ETR_\Delta$. But if $\Delta$ and $\Gamma$ are sufficiently close then they are in fact equivalent. For instance, $\ETR_\Delta$ is equivalent (over $\GBC$) to $\ETR_{\Delta + \Delta}$ because to carry out a recursion of height $\Delta + \Delta$ one first carries out a recursion of height $\Delta$, then using the solution of such as a parameter carries out a second recursion of height $\Delta$. In general, $\ETR_\Delta$ is equivalent to $\ETR_{\Delta \cdot n}$ for any standard $n > 0$.

This does not generalize from $n$ to $\omega$.

\begin{theorem} \label{thm3:sep-etr-gamma}
Suppose $(M,\Xcal) \models \GBCm + \ETR_\Gamma$ for $\Gamma \in \Xcal$ a class well-order $\ge \omega^\omega$ and let $\Delta \in \Xcal$ be a class well-order such that $\Delta \cdot \omega \le \Gamma$. Then, there is $(M,\Ycal)$ a coded $V$-submodel of $(M,\Xcal)$ so that $(M,\Ycal) \models \GBCm + \ETR_\Delta$.
\end{theorem}

\begin{proof}
Fix $G \in \Xcal$ a bijection $\Ord^M \to M$. Consider $T_{\Delta \cdot \omega}(G)$ which is in $\Xcal$ because it can be constructed via an elementary recursion of rank $(\omega \cdot \Delta) \cdot \omega \le \Gamma$. Now let $\Ycal$ consist of all sets (internally) definable from an initial segment of $T_{\Delta \cdot \omega}(G)$. That is
\[
\Ycal = \{ Y : Y \in \Def^{(M,\Xcal)}(M; T_\Upsilon(G)) \text{ for some } \Upsilon < \Delta \cdot \omega \}.
\]
Then $\Ycal$ is coded in $\Xcal$ by using $T_{\Delta \cdot \omega}(G)$. 

It is immediate that $(M,\Ycal)$ satisfies Extensionality, Replacement, and Global Choice. It satisfies First-order Comprehension because $\Ycal$ is an increasing union of $\GBCm$-realizations, namely the $\Def(M; T_\Upsilon(G))$ for $\Upsilon < \Delta \cdot \omega$. Finally, it satisfies $\ETR_\Delta$ because if $A \in \Ycal$ then $A \in \Def(M; T_\Upsilon(G))$ for some $\Upsilon < \Delta \cdot \omega$ and thus $\Tr_\Delta(A) \in \Def(M;T_{\Upsilon+\Delta}(G)) \subseteq \Ycal$.
\end{proof}

This establishes theorem \ref{thm3:main-thm} for the $\GBC + \ETR$ case.

As an immediate corollary we get that levels of $\ETR$ can be separated by consistency strength.

\begin{corollary} 
Let $(M,\Xcal) \models \GBC$ and suppose $\Gamma \in \Xcal$ is a well-order $\ge \omega^\omega$. Then if $(M,\Xcal) \models \ETR_{\Gamma \cdot \omega}$ we have that $(M,\Xcal) \models \Con(\GBC + \ETR_\Gamma)$. \qed
\end{corollary}

There is also a version of this corollary for $\GBCm$.

To turn this into a statement about theories in $\Lcal_\in$, i.e.\ the language of set theory without any names for distinguished well-orders, we need that $\Gamma$ is definable. Moreover, in order for $(M,\Xcal)$ to agree with its $V$-submodels as to what $\Gamma$ is we need that $\Gamma$ is defined by a first-order formula (without parameters). So we can say that, for instance, $\GBC + \ETR \proves \Con(\GBC + \ETR_\Ord)$, where both theories are in $\Lcal_\in$. See the discussion in section \ref{sec4:medium} for further details.

The proof for theorem \ref{thm3:sep-etr-gamma} also separates $\ETR_\Gamma$ and $\ETR_{<\Gamma}$ for $\Gamma$ closed under addition.

\begin{definition}
Let $(M,\Xcal) \models \GBCm$ and suppose $\Gamma \in \Xcal$ is a well-order. Then $(M,\Xcal) \models \ETR_{<\Gamma}$ if it satisfies $\ETR_\Delta$ for all $\Delta < \Gamma$. 
\end{definition}

\begin{theorem}
Suppose $(M,\Xcal) \models \GBCm + \ETR_\Gamma$ for $\Gamma \in \Xcal$ a well-order $\ge \omega^\omega$ so that $\Delta + \Delta < \Gamma$ for all $\Delta < \Gamma$. Then, $(M,\Xcal)$ has a coded $V$-submodel $(M,\Ycal) \models \ETR_{<\Gamma}$.
\end{theorem}

\begin{proof}[Proof sketch]
Similar to the proof of theorem \ref{thm3:sep-etr-gamma}, but set
\[
\Ycal = \{ Y : Y \in \Def^{(M,\Xcal)}(M;T_\Upsilon(G)) \text{ for some } \Upsilon < \Gamma \}
\]
where $G \in \Xcal$ is some bijection $\Ord^M \to M$. Then $(M,\Ycal) \models \ETR_{<\Gamma}$.
\end{proof}

Confining one's attention to transitive models (or, more broadly, $\omega$-standard models) this is the end of the story. For any standard $n$ we have that $\ETR_\Gamma$ is equivalent to $\ETR_{\Gamma \cdot n}$ so $\ETR_\Gamma$ is equivalent to $\ETR_{< \Gamma \cdot \omega}$ for models with the correct $\omega$. But if a model has ill-founded $\omega$ then there is a gap. Can an intermediate theory be found in this gap?

The answer is yes.

Fix $(M,\Xcal) \models \GBC + \ETR_{\Gamma \cdot \omega}$ an $\omega$-nonstandard model where $\Gamma \in \Xcal$ so that $(M,\Xcal) \models \Gamma$ is well-founded. Given $\Ycal \subseteq \Xcal$ define the {\em $\Gamma$-recursion cut for $\Ycal$} to be $I_\Gamma(\Ycal) = \{ e \in \omega^M : (M,\Ycal) \models \GBC + \ETR_{\Gamma \cdot e} \}$. Note that $I_\Gamma(\Ycal)$ must be closed under addition, as being closed under multiplication by standard $n$ is equivalent to being closed under addition. This is the only restriction on what $I_\Gamma(\Ycal)$ can be. 

\begin{theorem}
Let $(M,\Xcal) \models \GBC + \ETR_{\Gamma \cdot \omega}$ be $\omega$-nonstandard where $\Gamma \in \Xcal$ is well-founded according to $(M,\Xcal)$. Let $I \subseteq \omega^M$ be a cut closed under addition. Then there is $\Ycal \subseteq \Xcal$ so that $I_\Gamma(\Ycal) = I$.
\end{theorem}

Observe that, unlike for theorem \ref{thm3:sep-etr-gamma}, $\Ycal$ cannot be coded in $\Xcal$ as then $(M,\Xcal)$ could define $I$, which is impossible.

\begin{proof}
Fix $G \in \Xcal$ a global well-order of $M$. Set 
\[
\Ycal = \bigcup_{e \in I} \Def(M; (\Tr_{\Gamma \cdot e}(G))^{(M,\Xcal)}.
\]
A comment is in order. Because $M$ is $\omega$-nonstandard it in general can admit multiple incompatible full satisfaction classes. So $(\Gamma \cdot e)$-iterated full satisfaction classes will not be unique. Nevertheless, $(M,\Xcal)$ will have at most one $\Delta$-iterated truth predicate relative to a given parameter, because iterated truth predicates are unique in a fixed second-order model. So $(\Tr_{\Gamma \cdot e}(G))^{(M,\Xcal)}$ is well-defined. 

Note also that each $\Def(M; (\Tr_{\Gamma \cdot e}(G))^{(M,\Xcal)}$ is a $\GBC$-realization for $M$. Thus, because $\Ycal$ is the increasing union of these $\GBC$-realizations it too must be a $\GBC$-realization. It remains only to check that $(M,\Ycal) \models \ETR_{\Gamma \cdot e}$ if and only if $e \in I$. The backward direction of this implication is immediate from the definition of $\Ycal$.

For the forward direction, take $a > I$. Suppose towards a contradiction that $(M,\Ycal) \models \ETR_{\Gamma \cdot a}$. Then $\Ycal$ has (what it thinks is) $\Tr_{\Gamma \cdot a}(G)$. But then $\Tr_{\Gamma \cdot a}(G)$ is definable from $\Tr_{\Gamma \cdot e}(G)$ for some $e \in I$. In particular, this means that $\Tr_{\Gamma \cdot e + 1}(G)$ is definable from $\Tr_{\Gamma \cdot e}(G)$, contradicting Tarski's theorem on the undefinability of truth. 
\end{proof}

So while the separation between $\ETR_\Gamma$ and $\ETR_{\Gamma \cdot \omega}$ is optimal for $\omega$-standard models, for $\omega$-nonstandard models there are always intermediate levels of $\ETR$.

Let us return now to transitive models. Earlier when we separated $\ETR_\Gamma$ from $\ETR_{\Gamma \cdot \omega}$ we did so via the second-order part of the model. Starting from $(M,\Xcal) \models \GBC + \ETR_{\Gamma \cdot \omega}$ we found $\Ycal \subseteq \Xcal$ so that $(M,\Ycal)$ satisfies $\ETR_\Gamma$ but not $\ETR_{\Gamma \cdot \omega}$. So the separation is entirely due to which classes we allow in each model.

Can we separate $\ETR_\Gamma$ and $\ETR_{\Gamma \cdot \omega}$ via the first-order part of a model? Can we do so with a transitive model? That is, can we find transitive $M \models \ZFC$ which is $(\GBC + \ETR_\Gamma)$-realizable but not $(\GBC + \ETR_{\Gamma \cdot \omega})$-realizable? 

Yes we can, if $\Gamma$ is a set, rather than a proper class. 

\begin{theorem} \label{thm3:strong-sep-etr-frag}
Let $\gamma$ be an ordinal $\ge \omega^\omega$ given by a first-order definition. That is, there is a first-order formula $\phi(x)$ without parameters so that $\ZFC$ proves $\phi(x)$ has a unique witness $\gamma$ and this witness is an ordinal $\ge \omega^\omega$. Suppose there is a transitive model of $\ETR_{\gamma \cdot \omega}$. Then there is a transitive model of $\ZFC$ which is $(\GBC + \ETR_\gamma)$-realizable but not $(\GBC + \ETR_{\gamma \cdot \omega})$-realizable.
\end{theorem}

\begin{proof}
Take $(M,\Xcal) \models \GBC + \ETR_{\gamma \cdot \omega}$ transitive with a definable global well-order. That such $(M,\Xcal)$ exists is a consequence of theorem \ref{cor3:etr-inner-model}. Let $\vec T = \seq{(\Tr_{\delta})^M : \delta < \gamma^{M} \cdot \omega}$ be the sequence of $\delta$-iterated truth predicates for $M$ for $\delta < \gamma \cdot \omega$. Then $\vec T \in \Xcal$ because $\Xcal$ contains $(\gamma \cdot \omega)$-iterated truth predicates. Now let $(N,\gamma^N, \vec S) \prec (M,\gamma^{(M,\Xcal)}, \vec T)$ be the Skolem hull of the empty set, using the global well-order of $M$ to pick witnesses. Then, $(N,\gamma^N, \vec S) \models \vec S$ consists of $\delta$-iterated truth predicates for $\delta < \gamma^N$. I claim that $N$ is $(\GBC + \ETR_\gamma)$-realizable but not $(\GBC + \ETR_{\gamma \cdot \omega})$-realizable.

For the former, let $\Ycal = \Def(N; \gamma^N, \vec S)$ consist of the subsets of $N$ which are definable (with set parameters) over the structure $(N; \gamma^N, \vec S)$. It is immediate that $(N,\Ycal)$ satisfies Extensionality, Global Choice, and First-Order Comprehension. Suppose towards a contradiction that $(N,\Ycal)$ does not satisfy Replacement. Then there is $F \in \Ycal$, defined from $\gamma^N$ and $\vec S$ via a formula $\phi$ possibly with parameters, and a set $a \in N$ so that $F''a \not \in N$. But then by elementary in $(M, \gamma^M, \vec T)$ the formula $\phi$ defines a class function $G$ and there is a set $b \in M$ so that $G'' b \not \in M$. But then $(M,\Xcal)$ fails to satisfy Replacement, contradicting that it is a model of $\GBC$. Altogether we get that $\Ycal$ is a $\GBC$-realization for $N$.

Next let us see that $(N,\Ycal) \models \ETR_\gamma$. But this is immediate; given any class $A \in \Ycal$ there is $\delta < \gamma^N$ so that $A$ is definable from $\Tr_\delta$ so $\Tr_\gamma(A)$ is definable from $\Tr_{\delta + \gamma}$ and hence $\Tr_\gamma(A) \in \Ycal$. 

Finally, $N$ cannot be $(\GBC + \ETR_{\gamma \cdot \omega})$-realizable because if $\Zcal$ were an $(\GBC + \ETR_{\gamma \cdot \omega})$-realization for $N$ then $\Zcal$ would contain $\Tr_{\gamma \cdot \omega}$ but then $\Zcal$ would see that $N$ is a Skolem hull, hence countable. Note that this uses that $N$ is an $\omega$-model, so that there is only one subset of $N$ which can satisfy the definition of a $(\gamma^N \cdot \omega)$-iterated truth predicate, namely the externally constructed one. But no model of $\GBC$ thinks its first-order part is countable, so the existence of such $\Zcal$ is impossible.
\end{proof}

Observe that we can ensure that $\gamma^N = \gamma^M$ by requiring $N$ to be the Skolem hull of, say, $V_\gamma^M$ rather than the Skolem hull of the empty set.

This theorem gives a strong separation for sufficiently weak fragments of $\ETR$. The model $N$ we constructed cannot be made into a model of $\ETR_{\gamma \cdot \omega}$ not because it fails to have a compatible (first-order) theory, but rather due to inherently second-order properties of the model. In the next section we will see that this phenomenon depends essentially upon the transfinite; it does not occur for models of finite set theory, equivalently models of arithmetic.

\section{A detour through the finite realm}

While my analysis has mainly been confined to models of set theory, analogous results are possible for models of arithmetic. I wish to take a brief detour from the infinite world to consider the applications of these ideas to the finite world. 

Let me recall some standard facts about satisfaction classes for nonstandard models of arithmetic. First, we will need a few definitions.

\begin{definition}
A structure $M$ is {\em resplendent} if it realizes any consistent $\Sigma^1_1$-formula. That is, if $\hat X$ is a new predicate symbol, $\bar a$ are elements of $M$, and $\phi(\hat X,\bar a)$ is consistent with $\Th(M,\bar a)$ then there is $X \subseteq M$ so that $(M,X) \models \phi(\hat X,\bar a)$. 

Further say that $M$ is {\em chronically resplendent} if $X$ may be chosen so that $(M,X)$ is resplendent.
\end{definition}

\begin{definition}
A structure $M$ is {\em recursively saturated} if it realizes any consistent computable type. That is, if $p(x,\bar a)$ is a consistent type so that the set of formulae $\phi \in p$ form a computable subset of $\omega$ then there is $t \in M$ so that $M \models \phi(t,\bar a)$ for all $\phi \in p$.
\end{definition}

I am primarily interested in structures which allow an appreciable amount of coding, such as models or arithmetic or models of set theory. For such structures we can write down a formula $\phi(\hat X)$ which asserts that $\hat X$ is a truth predicate. Every countable recursively saturated model admits a full satisfaction class---a theorem of Kotlarski, Krajewski, and Lachlan \cite{KKL1981}.\footnote{But see \cite{enayat-visser2015} for a more elegant proof.}
So because every completion of $\PA$ has a countable recursively saturated model we get that resplendent models admit full satisfaction classes. But note that in general these models will not satisfy induction in the expanded language with a predicate for the full satisfaction class, as induction in the expanded language allows one to prove $\Con(\PA)$.

\begin{theorem}[Lachlan \cite{lachlan1981}]
If $M \models \PA$ admits a full satisfaction class then $M$ is recursively saturated.
\end{theorem}

For countable models all these notions are equivalent, but separations can happen in the uncountable. Kaufmann \cite{kaufmann1977} produced recursively saturated, rather classless\footnote{$M \models \PA$ is rather classless if every class of $M$ is definable, where $A \subseteq M$ is a class if $A \cap [0,x)^M$ is definable for every $x \in M$. Cf.\ definition \ref{def1:rather-classless}.}
models. Combined with a theorem of Smith's \cite{smith1989} that models with a full satisfaction class must have undefinable classes, this gives recursively saturated models which do not admit a full satisfaction class.

\begin{theorem}[Barwise--Schlipf \cite{barwise-schlipf1976}] 
If $M$ is countable and recursively saturated then $M$ is chronically resplendent.
\end{theorem}

As an immediate consequence we get that countable resplendent models are chronically resplendent. It is open whether this is true in general.

\begin{question}
Does resplendency always imply chronic resplendency?
\end{question}

Let us turn now to iterated full satisfaction classes\footnote{In this section and this section alone I will talk about iterated full satisfaction classes instead of iterated truth predicates. We are not working in a second-order context so the uniqueness of iterated truth predicates within a second-order model does not apply here. I wish to emphasize this change of perspective with a change of language. The exception to this choice is when I talk about standard models, in which case the only possible choice of a full satisfaction class is the truth predicate for the model, as seen externally.} 
over models of arithmetic. Essentially the same argument that resplendent models admit full satisfaction classes yields that resplendent models admit iterated full satisfaction classes. But that is not the only way to get models with iterated full satisfaction classes. Another way goes through models of second-order arithmetic.

Recall that $\ATR_0$ is the theory of second-order arithmetic axiomatized by the following: $\PA$ for the first-order part; Extensionality for sets; comprehension for arithmetical (i.e.\ first-order) properties; and arithmetical transfinite recursion, asserting that inductions of arithmetical properties along a well-founded relation have solutions. That is, $\ATR_0$ is the arithmetic counterpart to $\GBC + \ETR$. Many arguments from $\GBC + \ETR$ generalize to $\ATR_0$ and vice versa. One example is the construction of iterated truth predicates. Let me give a definition specialized to this context, for the sake of clarity.

\begin{definition}
Let $M \models \PA$ and $n \in M$.  An {\em iterated full satisfaction class of length $n$} (or, synonymously, an {\em $n$-iterated full satisfaction class}
 for $M$ is a set $T \subseteq M$ of triples\footnote{Recall that $\PA$ allows for the coding of sequences as single numbers, so we can think of $T$ as a subset of $M$, rather than a subset of $M^3$.}
$(i, \phi, \bar a)$ with $i < n$ satisfying the following recursive requirements. Here, $\phi$ is in $\Lcal_\PA(\Trm)$, the language of arithmetic augmented with a trinary predicate $\Trm$, and $\bar a$ is a valuation for $\phi$.
\begin{enumerate}
\item For atomic $\phi$: $(i, \phi, \bar a)$ is in $T$ if and only if $\phi$ given the values from $\bar a$ is true.\footnote{For example, $2 + 3 < 5 + 1$ is declared true at every level $< n$, because $2 + 3$ really is less than $5 + 1$.}
\item $(i, x \in y, \bar a)$ is in $T$ if and only if $a_x \in a_y$. 
\item $(i, \Trm(x,y,z), \bar a)$ is in $T$ if and only if
 \begin{itemize}
 \item $a_x < i$;
 \item $a_y$ is an $\Lcal_{\in}(\Trm)$-formula;
 \item $a_z$ is a valuation for $a_y$; and
 \item $(a_x, a_y, a_z)$ is in $T$.
 \end{itemize}
\item $(i, \phi \lor \psi,\bar a)$ is in $T$ if and only if $(i, \phi, \bar a)$ or $(i, \psi, \bar a)$ are in $T$.
\item $(i, \neg \phi, \bar a)$ is in $T$ if and only if $(i, \phi, \bar a)$ is not in $T$.
\item $(i, \exists x\ \phi(x), \bar a)$ is in $T$ if and only if there is $b \in M$ so that $(i, \phi, b \cat \bar a)$ is in $T$.\footnote{Here we of course have the implicit requirement that $x$ be free in $\phi$.}
\end{enumerate}
We can also consider iterated full satisfaction classes of length $\omega$ (in the sense of the model) by allowing $i$ to be any element of the model, not just those $< n$.
\end{definition}

In this section I will only consider iterated full satisfaction classes of length $\le \omega$. One could consider longer lengths, but the general theory requires additional care. In particular, ``$\Gamma$ is a well-order'' is a second-order assertion in the arithmetical case, so there are subtleties in formulating  $(M,\Gamma,S) \models \text{``}S$ is an iterated full satisfaction class of length $\Gamma$'' for $\Gamma,S \subseteq M$.

Observe that each step in the recursive requirements above is first-order. So $\ATR_0$ proves the existence of iterated full satisfaction classes of all lengths $\le \omega$. Moreover, these will be inductive\footnote{A set $X \subseteq M \models \PA$ is {\em inductive} if $(M,X)$ satisfies the induction schema in the expanded language. Equivalently, $X \subseteq M$ is inductive if there is an $\ACA_0$-realization for $M$ which contains $X$. Compare to $\GBc$-amenability, definition \ref{def1:t-amenable}.}
iterated full satisfaction classes, because any set in a model of $\ATR_0$ must be inductive.

Observe that not every resplendent $M \models \PA$ admits inductive iterated full satisfaction classes. This is a consequence of the following observation.

\begin{observation}
If $M \models \PA$ admits an inductive iterated full satisfaction class of length $n$, possibly nonstandard, then $M \models \Con^n(\PA)$. So if $M$ admits an inductive iterated full satisfaction class of length $\omega$ (in the sense of $M$), then $M \models \forall n\ \Con^n(\PA)$.
\end{observation}

\begin{proof}
An inductive iterated full satisfaction class of length $k+1$ allows one to get a completion of $\PA + \Con^k(\PA)$. But if the model can see a completion of $T$ then it must think $\Con(T)$. 
\end{proof}

In the previous section we considered transitive models of set theory which admitted $\GBC$-amenable iterated truth predicates of length $\eta$ but not ones of  longer lengths.
Moreover, we could find such models satisfying any extension $T$ of $\ZFC$ which has a transitive model and is consistent with the existence of iterated truth predicates of length $\eta$. See theorem \ref{thm3:strong-sep-etr-frag}. In particular, $T$ could be the $\Lcal_\in$-reduct of $\GBC$ $+$ there is an iterated truth predicate of length $\zeta$, for $\zeta$ larger than $\eta$. Can we have the same phenomenon for models of arithmetic? That is, for any $T \supseteq \PA$ can we find a model of $T$ which admits an inductive full satisfaction class of length $\eta$ but not an inductive iterated full satisfaction class of length $\eta + 1$?

For the standard model, there is only one choice for a full satisfaction class, namely the truth predicate for the model. But this will not work because $\nats$ admits iterated truth predicates of any length and any class over $\nats$ is inductive. So we must look at nonstandard models.

\begin{observation}
Let $M \models \PA$ be countable and admit an inductive full satisfaction class. Then $M$ admits an inductive full satisfaction class of any length compatible with the theory of $M$. That is, if ``$S$ is an inductive iterated full satisfaction class of length $n$'' is consistent with $\Th(M)$ then there is $S \subseteq M$ which is an inductive $n$-iterated full satisfaction class for $M$. (And similarly for length $\omega$.)
\end{observation}

\begin{proof}
Because $M$ admits a full satisfaction class it must be recursively saturated. But since $M$ is countable and recursively saturated it is resplendent. So if ``$S$ is an inductive full satisfaction class of length $n$'' is consistent with $\Th(M)$ then there is $S \subseteq M$ realizing that theory. (And similarly for length $\omega$.)
\end{proof}

This observation tells us that for resplendent models the only way to rule out having an inductive iterated full satisfaction class of length $n$ is by the theory of the model. This is unlike the case for transitive models of set theory, where we can find a model of any reasonable $T$ so that the model does not admit a $\GBc$-amenable iterated truth predicate of length $n$. 

If we want to find an inductive iterated full satisfaction class over $M \models \PA$ which cannot be extended to a longer one then there are two ways to do such. One would be if $(M,T)$ is not recursively saturated for $T$ an inductive iterated full satisfaction class. Then $(M,T)$ does not admit a full satisfaction class, inductive or otherwise. But this is a boring case. The more interesting case is when $(M,T)$ is recursively saturated (which, for countable $M$, is equivalent to being resplendent). In this case the only potential obstacle is $\Th(M,T)$. 

As a warm-up let us see that different full satisfaction classes can give different consequences in the extended model.

\begin{proposition}
There is $M \models \PA$ with $U,I \subseteq M$ inductive full satisfaction classes so that $\Th(M,U) \ne \Th(M,I)$.
\end{proposition}

I learned of this argument from \cite{hamkins-yang2014}, where they attribute the argument to Schmerl.

\begin{proof}
Consider the $\Lcal_\PA(S)$ theory $T = \TA + \text{``}S$ is an inductive full satisfaction class, where $\TA = \Th(\nats)$ is true arithmetic. By a well-known fact from computability theory $T$ is Turing-equivalent to $0^{(\omega)}$, the $\omega$-th iterate of the Turing jump of the empty set. However, $\Th(\nats,\TA)$, which extends $T$, is Turing-equivalent to $0^{(\omega+\omega)}$. So if $T$ were complete then we would get $0^{(\omega)} \equiv_T 0^{(\omega+\omega)}$, which is of course impossible. So $T$ must be incomplete. Now take $T_I$ and $T_U$ incompatible extensions of $T$. That is, $T_I \cup T_U$ is inconsistent while each is individually consistent. Notice, though, that $T_I$ and $T_U$ have the same $\Lcal_\PA$ consequences, namely $\TA$. So if $M \models \TA$ is resplendent then there are $U,I \subseteq M$ so that $(M,U) \models T_U$ and $(M,I) \models T_I$. These are the desired $M$, $U$, and $I$.
\end{proof}

Note that the restricting $U$ and $I$ to $\omega$ both yield $\TA$. So $U$ and $I$ agree on standard truth, but not nonstandard truth. On the other hand, for truth about truth disagreement happens on the standard level. 

\begin{definition}
For $n \le \omega$ let $\ITR^n$ be the $\Lcal_\PA(T)$ theory asserting $\PA$ plus that $T$ is an inductive full satisfaction class of length $n$.\footnote{The reader should think {\bf i}terated {\bf tr}uth for $\ITR$.}
Let $\itr^n$ be the $\Lcal_\PA$ consequences of $\ITR^n$. 
\end{definition}

With this definition in mind, the above observation can be phrased as: if $M \models \itr^n$ is countable and recursively saturated then $M$ can be extended to a model of $\ITR^n$. 

\begin{definition}
Let $m < n \le \omega$. Let $\itr^n_m$ be the reduct of $\ITR^n$ to language for a structure with an $m$-iterated full satisfaction class. More formally, we can use an $\Lcal_\PA(T)$-formula to define an $m$-iterated full satisfaction class from an $n$-iterated full satisfaction class by restricting to the first $m$ levels of the iterated full satisfaction class. Then $\itr^n_m$ is what $\ITR^n$ proves about this reduct.
\end{definition}

\begin{theorem} \label{thm3:itr2}
Let $M \models \itr^2$ be nonstandard, countable, and recursively saturated. Then there are $S,S' \subseteq M$ so that the following hold:
\begin{itemize}
\item $(M,S)$ and $(M,S')$ are recursively saturated;
\item $S$ and $S'$ are inductive full satisfaction classes;
\item $S$ can be extended to an inductive iterated full satisfaction class of length $2$; and
\item $S'$ cannot be extended in this way.
\end{itemize}
\end{theorem}

\begin{proof}
It is easy to find $S$. Just take an inductive iterated full satisfaction class of length $2$ over $M$ and restrict it to get $S$. We can ensure $(M,S)$ is recursively saturated because $M$ is chronically resplendent. So the work is in getting $S'$. This reduces to the following claim.

\begin{claim}
The theory $\itr^2_1$ is independent over $\Th(M) + \ITR^1$. 
\end{claim}

Given the claim, find $S'$ by chronic resplendency to get a class over $M$ so that the expansion satisfies $\Th(M) + \ITR^1$ but not $\itr^2_1$. So to finish the proof let us prove the claim. The basic idea is that having an inductive full satisfaction class allows a model to get a handle of the theory of its arithmetic reduct, enabling a diagonalization trick.

Consider the sentence $\Con(\ITR^1 + \Tr)$, where $\Tr$ is a name for the full satisfaction class. This sentence can be expressed in the language of $\ITR^1$, because $\ITR^1$ is computably axiomatizable and the full satisfaction class gives access to $\Tr$. Note that for any standard $\phi$ in the language of arithmetic and any $(N,S) \models \Th(M) + \ITR^1$ we get that $\phi \in S$ if and only if $\phi \in \Th(M)$. Clearly, $\ITR^2 \proves \phi$ so $\itr^2_1 \proves \phi$. Let us see that $\Th(M) + \ITR^1$ does not prove $\Con(\ITR^1 + \Tr)$, which will then yield the claim. By the G\"odel fixed-point lemma there is a sentence $\psi$ so that $\ITR^1$ proves
\[
\psi \iff \underbrace{\forall x\ \Pr_{\ITR^1 + \Tr}(\psi,x) \impl \exists y<x\ \Pr_{\ITR^1 + \Tr}(\neg \psi,y)}_{= \rho(\psi)}
\]
where $\Pr_{\ITR^1 + \Tr}(\theta,x)$ asserts $x$ codes a proof of $\theta$ from the axioms of $\ITR^1 + \Tr$. It is immediate that $\ITR^1$ proves $\rho(\psi) \impl \neg \Con(\ITR^1 + \Tr)$. So $\ITR^1$ proves $\Con(\ITR^1 + \Tr) \impl \neg \psi$. Now suppose towards a contradiction that $\Th(M) + \ITR^1$ proves $\neg \psi$. Then there is a standard natural number which codes this proof. Now work in a model of $\Th(M) + \ITR^1$. Then, this model thinks that $\ITR^1 + \Tr$ proves $\neg \rho(\psi)$, which is equivalent to
\[
\exists x\ \Pr_{\ITR^1 + \Tr}(\psi,x) \land \forall y < x\ \neg \Pr_{\ITR^1 + \Tr}(\neg \psi, x). 
\]
There are two cases to consider, the first being the case that there is a witnessing $x$ which is standard. Then the proof coded by $x$ could only use formulae from the standard part of $\Tr$, which is $\Th(M)$. Thus we would get a standard proof of $\psi$ from $\Th(M) + \ITR^1$, which would be a contradiction. The second case is then that all witnessing $x$ are nonstandard. But then $\forall y < x\ \neg \Pr_{\ITR^1 + \Tr}(\neg \psi, x)$ cannot be a theorem of $\ITR^1 + \Tr$, as there is a standard $y$ which codes a proof of $\neg \psi$ from $\Th(M) + \ITR^1$, and so our model will think $y$ codes a proof of $\neg \psi$ from $\ITR^1 + \Tr$. In either case we get a contradiction, so our original assumption that $\Th(M) + \ITR^1$ proves $\neg \psi$ must be false, and so it cannot prove $\Con(\ITR^1 + \Tr)$. This completes the proof of the claim, which completes the proof of theorem.
\end{proof}

Krajewski's methods give that there are continuum many such $S$ and $S'$. 

There is nothing special about about $1$ and $2$ in theorem \ref{thm3:itr2}. We can get the same result for $m$-iterated full satisfaction classes and $n$-iterated full satisfaction classes for $m \le n$.

\begin{theorem}
Let $M \models \PA$ be countable and recursively saturated and $m \le n \in M$. Assume that $M \models \itr^n$. Then there is $S \subseteq M$ so that
\begin{itemize}
\item $(M,S)$ is recursively saturated;
\item $S$ is an inductive $m$-iterated full satisfaction class for $M$; and
\item $S$ cannot be extended to an inductive $(m+1)$-iterated full satisfaction class.
\end{itemize}
\end{theorem}

Theorem \ref{thm3:itr2} is a direct consequence of this result: apply it with $m = n = 2$ and with $m = 1$ and $n = 2$ to get, respectively, $S$ and $S'$.

\begin{proof}[Proof sketch]
If $n = m$ this is easy. For $n < m$, use a similar G\"odel--Rosser trick to show that $\itr^n_m$ is independent over $\Th(M) + \ITR^m$. 
\end{proof}

Again, Krajewski's work implies that there are continuum many such $S$.

\begin{corollary}
Let $M \models \PA$ be countable and recursively saturated and $k \le m < n \in M$. Assume that $M \models \itr^n$. Then there is $S \subseteq M$ so that
\begin{itemize}
\item $(M,S)$ is recursively saturated;
\item $S$ is an inductive $k$-iterated full satisfaction class for $M$; and
\item $S$ can be extended to an inductive $m$-iterated full satisfaction class but no further.
\end{itemize}
\end{corollary}

\begin{proof}
Apply the theorem, then restrict the $m$-iterated full satisfaction class to get a $k$-iterated full satisfaction class.
\end{proof}

Given countable and recursively saturated $M \models \PA$ we can form a tree consisting of the inductive iterated full satisfaction classes over $M$. Namely, for $m < n \in M$ an $m$-iterated full satisfaction class $S_m$ is before an $n$-iterated full satisfaction class $S_n$ in the tree if $S_m$ is the restriction of $S_n$ to the first $m$ levels. The results in this section tell us that this tree has lots of branches of all possible lengths.

For concreteness, suppose $M \models \itr^\omega$, i.e.\ the $\Lcal_\PA$-consequences of $\PA + \text{``}T$ is an inductive $\omega$-iterated full satisfaction class''. (Here, $\omega$ is in the sense of the model $M$.) The tree of inductive iterated full satisfaction classes for $M$---see figure \ref{fig3:truth-tree}---has continuum many branches, coming from the continuum many inductive $\omega$-iterated full satisfaction classes. The above corollary implies that for any $m < n \in m$ this tree has continuum many nodes of depth $m$ which extend to a node of depth $n$, but no further. Moreover, if a node $s$ of depth $m$ has extensions to a node of depth $n$, then for any $k$ between $m$ and $n$ we have that $s$ has continuum many extensions to a leaf node of depth $k$.

\begin{figure}
\begin{center}
\begin{tikzpicture}[n/.style={draw=black,fill=white, minimum size=.2cm}]
\draw[dotted] (-5,-8) -- (0,0) -- (5,-8);
\draw (5,-2) node {depth $m$};
\draw (5,-4) node {depth $k$};
\draw (5,-6) node {depth $n$};
\draw (0,0) .. controls (-.5,-1) and (-.2,-1.6) .. (-.5,-2) 
   (-.5,-2) .. controls (-2,-4) and (-1.5,-5) .. (-2,-6)
   (-.5,-2) .. controls (.2,-3) .. (0,-4)
   (-.5,-2) .. controls (.4,-3) .. (.5,-4)
   (-.5,-2) .. controls (1.3,-3) .. (1,-4);
\draw (1.5,-4) node {$\cdots$};
\draw[dashed] (-2,-6) -- (-3,-7.5)
              (-2,-6) -- (-1,-7.5);
\draw (-.5,-2) node[n] {}
       (-2,-6) node[n] {}
        (0,-4) node[n] {}
       (.5,-4) node[n] {}
        (1,-4) node[n] {};
              
\end{tikzpicture}
\end{center}
\caption{The tree of inductive iterated full satisfaction classes over a model of arithmetic. The node at depth $m$ has continuum many extensions to leaf nodes of depth $k$.}
\label{fig3:truth-tree}
\end{figure}
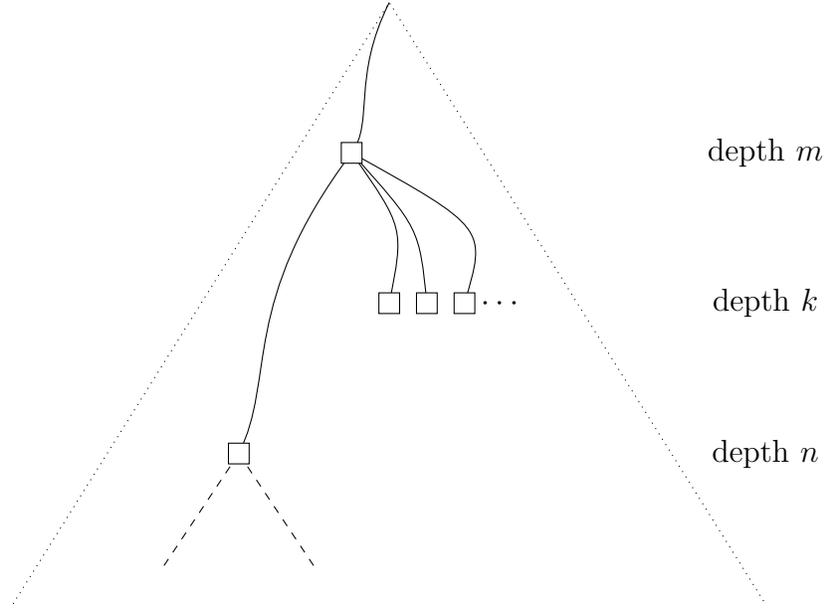

\section{Non-elementary transfinite recursions}

We now return to the infinite. We have investigated transfinite recursion for elementary properties. What about second-order properties?

\begin{definition} \label{def3:sktr}
Let $k$ be a natural number. We define the $\Sigma^1_k$-transfinite Recursion schema $\SkTR k$. Instances of $\Sigma^1_k$-Transfinite Recursion take the following form: let $\phi(x,Y,P)$ be a $\Sigma^1_k$-formula, possibly with a class parameter $P$, and $R$ be a well-founded class relation with transitive closure $<_R$. The instance of recursion for $\phi$ and $R$ asserts that there is a class $S \subseteq \dom R \times V$ which satisfies
\[
(S)_r = \{ x : \phi(x,S \rest r, A) \}
\]
for all $r \in \dom R$. Here, $(S)_r = \{ x : (r,x) \in S \}$ denotes the $r$-th slice of $S$ and
\[
S \rest r = S \cap \left(\{ r' \in \dom R : r' <_R r\} \times V\right)
\]
is the partial solution below $r$.

If $\Gamma$ is a well-order then  $\SkTR{k}_\Gamma$ is the schema obtained by restricting $\SkTR k$ to only ask for solutions to recursions of height $\le \Gamma$.
\end{definition}

Note that $\SkTR 0$ is another name for $\ETR$.

Before wading into a finer analysis, let us put the $\SkTR k$ in the context of the more familiar $\PnCA k$. First, an upper bound.

\begin{proposition}[Over $\GBcm$]
$\PnCA{k+1}$ proves $\SkTR k$. 
\end{proposition}

\begin{proof}
Let $R$ be a well-founded relation and $\phi$ be $\Sigma^1_k$, possibly with parameters. By Comprehension form the class $\{r \in \dom R : $ there is no partial solution of $\phi$ up to $r \}$. This uses $\Pi^1_{k+1}$-Comprehension because it asserts there is no class satisfying the $\Sigma^1_k$-property of being a partial solution to a $\Sigma^1_k$-recursion. We want to see this class is empty, so assume otherwise. Then it has a minimal member $r$. That is, for all $r' \mathbin R r$ we have a partial solution up for $\phi$ up to $r'$. But then there is a partial solution up to $r$, a contradiction. 
\end{proof}

We also get a lower bound.

\begin{observation}[Over $\GBCm$]
$\PnCA k$ is equivalent to $\SkTR k_1$.
\end{observation}

\begin{proof}
Simply observe that a $\Sigma^1_k$ recursion of length $1$ is just asking for a single $\Sigma^1_k$-definable class.
\end{proof}

So, because $\SkTR k$ clearly proves $\SkTR k_1$ we have that $\SkTR k$ is between $\PnCA k$ and $\PnCA{k+1}$. Later---specifically theorem \ref{thm3:pnca-impl-sktr} and a special case of theorem \ref{thm3:main-thm}---we will see that this separation is by consistency strength. 

This observation can also be used to show that $\KM$ can be alternatively axiomatized by a transfinite recursion principle.

\begin{proposition}
Over $\GBCm$, the following are equivalent.
\begin{enumerate}
\item The full second-order Comprehension schema; and
\item $\Sigma^1_\omega$-Transfinite Recursion, the schema asserting that every transfinite recursion of a second-order property has a solution.
\end{enumerate}
\end{proposition}

\begin{proof}
$(1 \impl 2)$ Use $\Pi^1_{k+1}$-Comprehension to get an instance of $\Sigma^1_k$-Transfinite Recursion. $(2 \impl 1)$ Use $\Sigma^1_{k+1}$-Transfinite Recursion to get an instance of $\Pi^1_k$-Comprehension.
\end{proof}

In the sequel we will stratify $\SkTR k$ by the length of recursions, separating them by consistency strength similar to the previous analysis of $\ETR$. Before doing so, however, let us see why the arguments for $\ETR$ do not carry over immediately.

To separate $\ETR_{\Gamma \cdot \omega}$ from $\ETR_\Gamma$ we started with a model of $\ETR_{\Gamma \cdot \omega}$ and looked at the coded $V$-submodel generated from initial segments of $\Tr_{\Gamma \cdot \omega}(G)$, where $G$ was a global well-order. This submodel then satisfied $\ETR_\Gamma$ because iterated first-order truth predicates from the original model continued to be iterated first-order truth predicates in the submodel. This is because the two models have the same first-order part, so they agree on well-orders and they agree on first-order truth. We then inductively get that they agree on iterated truth. And since fragments of Elementary Transfinite Recursion are equivalent to the existence of certain iterated truth predicates, we get $\ETR_\Gamma$ in the submodel.

But that strategy cannot work for the non-elementary transfinite recursion principles $\SkTR k$ for $k \ge 1$. The problem is simple: a $\Sigma^1_k$-truth predicate must be a hyperclass. But this hyperclass contains $(X=X,A)$ for all classes $A$. If we had a class coding this truth predicate we could thus from it define a class coding the hyperclass of all classes. But that is impossible, as seen by an easy diagonalization argument. Because of this, there is no hope for having an iterated $\Sigma^1_k$-truth predicate. So we must take a different strategy. 

Instead, we will work directly with solutions to $\Sigma^1_k$-transfinite recursions. Here we run into a potential obstacle. Unlike with elementary transfinite recursions which depend only upon the sets (and possibly finitely many class parameters), determining whether a class is a solution to a $\Sigma^1_k$-transfinite recursion requires quantifying over all the classes. So if we restrict the classes to a $V$-submodel it may be that the smaller model disagrees about what is a solution to the recursion. We overcome this obstacle by ensuring that our $V$-submodel is sufficiently elementary in the larger model.

It will be convenient to work in the unrolling, so let us see what the theory of the unrolling is. First, recall a special case of definition \ref{def2:phi-tr}.

\begin{definition} 
Let $k \in \omega$. Then the {\em $\Sigma_k$-Transfinite Recursion principle} is the axiom schema consisting of the following axiom for each $\Sigma_k$-formula $\phi(x,y,a)$: 

Suppose $a$ is a parameter so that $\phi(x,y,a)$ defines a class function $F : V \to V$ and $\delta$ is an ordinal. Then there is a function $s : \delta \to V$ so that for all $i \in \delta$ we have $s(i) = F(s \rest i)$.
\end{definition}

We can refine this to the {\em $\Sigma_k$-Transfinite Recursion $\le \gamma$ principle}, which restricts the ordinals $\delta$ allowed to only those $\delta \le \gamma$.

\begin{theorem}
Let $(M,\Xcal) \models \GBC + \PnCA k$ be a model of second-order set theory, with $k \ge 1$, and let $\Ufrak$ be the unrolled model obtained from $(M,\Xcal)$. Then, if $(M,\Xcal) \models \SkTR k$ we have $\Ufrak \models \wZFCmi(k) + \Sigma_k$-Transfinite Recursion.
\end{theorem}

\begin{proof}
We have already seen in chapter 2 that the unrolling of a model of $\GBC + \PnCA k$ must satisfy $\wZFCmi(k)$. The new content is that $\Ufrak$ satisfies $\Sigma_k$-Transfinite Recursion. This is done similar to the argument for the proof of proposition \ref{prop2:unroll-etr-to-s0tr}.

Consider an instance of $\Sigma_k$-transfinite recursion. That is, $F$ is a class function $\Ufrak \to \Ufrak$ which is $\Sigma_k$-definable, possibly using a membership code as a parameter, and $D$ is a membership code equipped with $<_D$ a membership code for a well-ordering of $D$. We want to see that there is a membership code for the desired $s$. First, observe that in the ground universe that $F$ is $\Sigma^1_k$-definable; cf.\ lemma \ref{lem2:star-trans}.

We will build the the desired membership code $S$ via an instance of $\Sigma^1_k$-Transfinite Recursion. The idea is to mimic the recursion to produce $s$, but in membership codes. This introduces some extra work, since we have to deal with the picky details of membership codes. 

The iteration proceeds as follows, with an outer layer and an inner layer. The outer layer occurs on $\pen D$ according to the well-ordering corresponding to the membership code $<_D$ (see corollary \ref{cor2:relns-mem-codes}). Each step $d$ in the outer layer produces a partial construction of $S$, call it $S_d$. We start with $S_0 = \emptyset$ and take unions at limit stages. The hard work is done in the successor step, where the inner layer of the transfinite recursion occurs. We start with $S_d$ and want to produce $S_{d+1}$. By construction, each $d' \in \pen D$ which comes before $d$ in $<_D$ is in $S_d$. More, there is a corresponding node, call it $f(d')$, which represents $F(S_d \rest d')$ and then nodes for $\{d'\}$, $\{d',f(d')\}$, and $(d',f(d'))$ above, similar to the constructions in proposition \ref{prop2:unroll-choice} and lemma \ref{lem2:fns-mem-codes}. In particular, $S_d$ itself may not be a membership code. Modify $S_d$ to produce a membership code $U$ by adding a top node $t_U$ and edges from each $(d',f(d'))$ node in $S_d$ to $t_U$. Then, we have a membership code $F(U)$ by Elementary Comprehension. In the proof for proposition \ref{prop2:unroll-etr-to-s0tr} at this point we build the maximum initial partial isomorphism between $S_d$ and $F(U)$ by elementary transfinite recursion. Here we can do it in a single step, as asserting the existence of such is a $\Sigma^1_1$-assertion. Given this partial isomorphism we can glue a copy of $F(U)$ onto $S_d$, as in the argument for proposition \ref{prop2:unroll-coll}. Then, add $d+1$ to $S_d$ along with nodes for $\{d+1\}$, $\{d+1,t_{F(U)}\}$, and $(d+1,t_{F(U)})$ and the corresponding edges to produce $S_{d+1}$. 
\end{proof}

\begin{corollary} \label{cor3:unroll-sktr-gamma}
Let $(M,\Xcal) \models \GBC + \SkTR k_\Gamma$ where $k \ge 1$ and $\Gamma$ is a well-order, possibly class-sized. Then the unrolled model $\Ufrak$ satisfies $\wZFCmi(k)$ plus the $\Sigma_k$-Transfinite Recursion $\le \Gamma$ principle.
\end{corollary}

\begin{proof}
In the case for $\ETR$, we could not control the height of the recursion to produce the desired membership code $S$ because we did not know how long a recursion was needed to construct partial isomorphisms between membership codes. So although the outer layer of our recursion had $\Gamma$ many steps---where $\Gamma$ is the height of the first-order recursion we were mimicking---each of those steps consisted of an inner layer of recursion which could be very long. But for $\SkTR k$ those inner steps are of finite, bounded length. We produce the partial isomorphism in a single step, because the existence of such a function follows from an instance of $\Pi^1_1$-Comprehension, and then immediately use to define the next approximation to $S$. So this recursion is $n \cdot \Gamma$ for some standard finite $n \ge 2$. Conclude that $n \cdot \Gamma \le \Gamma \cdot n$ by some simple arithmetic of well-orders\footnote{If $\Gamma$ is finite then this is obvious and indeed equality holds. So consider the case where $\Gamma \ge \omega$. Take $\Delta$ and finite $m$ so that $\Gamma = \omega \cdot \Delta + m$. Then, $n \cdot \Gamma = \omega \cdot \Delta + nm = \Gamma + (n-1)m$. This is less than $\Gamma \cdot 2 \le \Gamma \cdot n$ because $\Gamma \ge \omega$.}
and observe that $\SkTR k_\Gamma$ is equivalent to $\SkTR k_{\Gamma \cdot n}$. So $\SkTR k_\Gamma$ suffices to prove there is a solution. 
\end{proof}

And in the other direction.

\begin{proposition} \label{prop3:cut-off-sktr-gamma}
Suppose $N$ is a model of $\wZFCmi(k)$ plus the $\Sigma_k$-Transfinite Recursion $\le \gamma$ principle. Let $\kappa$ be the largest model of $N$ and $(M,\Xcal)$ be the cut off model obtained from $N$, i.e.\ $M = V_\kappa^N$ and $\Xcal$ is the (definable) class in $N$ consisting of all subsets of $M$. Then, $(M,\Xcal) \models \GBC + \SkTR k_\Gamma$, where $\Gamma \in \Xcal$ is such that $N \models \Gamma \cong \gamma$.

Consequently, if $N$ further satisfies the full $\Sigma_k$-Transfinite Recursion principle, then $(M,\Xcal) \models \SkTR k$.
\end{proposition}

\begin{proof}
Consider an instance of $\Sigma_k$-Transfinite Recursion for a recursion along $\Gamma$, possibly using a parameter from $\Xcal$. We want to find the subset of $V_\kappa^N$ which witnesses that this recursion has a solution. This is done in the obvious way in $N$ by means of an instance of $\Sigma_k$-Transfinite Recursion of height $\gamma$, where $\gamma$ is the ordinal isomorphic to $\Gamma$.
\end{proof}

Now let us see that we may assume a fragment of Class Collection without loss.

\begin{theorem} \label{thm3:sktr-goes-to-sol}
Let $(M,\Xcal) \models \GBC + \SkTR k_\Gamma$ for $\Gamma \in \Xcal$ and $k \ge 1$.
Then there is $\Ycal \subseteq \Xcal$ second-order definable over $(M,\Xcal)$ so that $(M,\Ycal) \models \GBC + \SkTR k_\Gamma + \Sigma^1_k$-Class Collection.
\end{theorem}

\begin{proof}
Fix $G \in \Xcal$ a global well-order of $M$. Set $\Ycal = \Lcal(M,G)$, the hyperclass of $(M,G)$-constructible classes. (See section \ref{sec2:so-l} for a definition.) Corollary \ref{cor2:get-a-plus} tells us that $(M,\Ycal) \models \GBC + \PnCAp k$. It remains only to see that $(M,\Ycal) \models \Sigma^1_k$-Transfinite Recursion for recursions of height $\le \Gamma$. By proposition \ref{prop3:cut-off-sktr-gamma} it suffices to show that $\Lfrak(M,G)$, the $(M,G)$-constructible unrolling of $(M,\Xcal)$, satisfies the $\Sigma_k$-Transfinite Recursion $\le \gamma$ principle, where $\gamma \in \Lfrak(M,G)$ is the ordinal isomorphic to $\Gamma$.

Work in the unrolling $\Ufrak$ of $(M,\Xcal)$. Then, $\Lfrak(M,G) = L(M,G)^\Ufrak$. Let $\phi(x,y,a)$ be a $\Sigma_k$-formula with $a \in \Lfrak(M,G)$ a parameter. Then, $\phi(x,y,a)^{L(M,G)}$ is $\Sigma_k$ (in parameters). So applying an instance of $\Sigma_k$-Transfinite Recursion $\le \gamma$ in $\Ufrak$ gives a solution to the recursion of $\phi(x,y,a)^{L(M,G)}$ along $\gamma$. We want to see that this solution is in $\Lfrak(M,G)$. Using $\Sigma_k$-Collection inside $\Lfrak(M,G)$, we can for each $i \in \gamma$ find $\alpha_i$ so that $L_{\alpha_i}(M,G)$ sees a partial solution up to $i$ for this recursion. So we have the sequence $\seq{\alpha_i : i \in \gamma} \in \Lfrak(M,G)$ and thus can get the entire solution in $\Lfrak(M,G)$, as desired.
\end{proof}

Because this $\Ycal$ is a definable hyperclass in $(M,\Xcal)$ any consistency assumptions witnessed in $(M,\Ycal)$ are also visible to $(M,\Xcal)$. Of particular interest to us is the following consequence thereof: If $\SkTR k_{\Gamma} + \Sigma^1_k$-Class Collection proves there is a coded $V$-submodel of $\SkTR k_\Delta$ then so does $\SkTR k_\Gamma$.

\begin{theorem} \label{thm3:sep-sktr-by-levels}
Suppose $(M,\Xcal) \models \GBC + \SkTR k_{\Gamma \cdot \omega} + \Sigma^1_k$-Class Collection, for $k \ge 1$ and $\Gamma \in \Xcal$, has that every class is $(M,G)$-constructible for some fixed global well-order $G$. Then there is $\Ycal \subseteq \Xcal$ a coded $V$-submodel of $(M,\Xcal)$ so that $(M,\Ycal) \models \GBC + \SkTR k_\Gamma$.
\end{theorem}

By the above remarks, this immediately yields the following corollary.

\begin{corollary}
Suppose $(M,\Xcal) \models \SkTR k_{\Gamma \cdot \omega}$ for $k \ge 1$ and $\Gamma \in \Xcal$. Then there is $\Ycal \subseteq \Xcal$ a coded $V$-submodel of $(M,\Xcal)$ so that $(M,\Ycal) \models \SkTR k_\Gamma$. \qed
\end{corollary}

\begin{proof}[Proof of theorem \ref{thm3:sep-sktr-by-levels}]
Unroll $(M,\Xcal)$ to $\Ufrak \models \ZFCmi(k)$ plus $\Sigma_k$-Transfinite Recursion $\le \gamma$ plus $V = L(M,G)$ for some $G \in \Xcal$ a global well-order.

Recall lemma \ref{lem2:sol-refl}, which asserted that our unrolled model satisfies $\Sigma_k$-reflection along the $L_\alpha(M,G)$-hierarchy. Using the instance of this for the universal $\Sigma_k$-formula gives that there are club many $\delta$ so that $L_\delta(M,G) \prec_{\Sigma_n} L(M,G)$. Therefore, we are done if we can can show that there are club many ordinals $\alpha$ so that $L_\alpha(M,G)$ is closed under solutions to $\Sigma_k$-transfinite recursions of height $\le \gamma$, where $\gamma$ is the ordinal isomorphic to $\Gamma$. This is because if $\upsilon$ is in both of these clubs then $L_\upsilon(M,G)$ will satisfy $\ZFCmi(k)$ plus $\Sigma_k$-Transfinite Recursion $\le \gamma$. So if $\Ycal \subseteq \Xcal$ is the hyperclass consisting of classes which appear in $L_\upsilon(M,G)$, then $\Ycal$ is coded in $\Xcal$, since $L_\upsilon(M,G)$ is a set in the unrolling.

Work inside the unrolling $\Ufrak$. Fix an arbitrary ordinal $\alpha_0$. We will use an instance of $\Sigma_k$-Transfinite Recursion $\le \gamma \cdot \omega$ to find $\alpha > \alpha_0$ as in the above paragraph. It is obvious that the class of such $\alpha$ is closed, so this will suffice to show that it is club. We build $\alpha$ by means of a $\Sigma_k$-transfinite recursion of height $\gamma \cdot \omega$. The outer layer of this recursion, which has $\omega$ many steps, builds a sequence $\seq{\alpha_n : n \in \omega}$ so that $L_{\alpha_{n+1}}(M,G)$ has solutions for $\Sigma_k$-transfinite recursions of height $\gamma$ with parameters from $L_{\alpha_n}(M,G)$. The inner layer, building $\alpha_{n+1}$ from $\alpha_n$ has height $\gamma$, so that the whole recursion has height $\gamma \cdot \omega$.

For the inner recursion, we build a grid of ordinals with width $\omega \times L_{\alpha_n}(M,G)$ and height $\gamma$. Each column of this grid is indexed by $(\phi,a)$ where $\phi(x,y,a)$ is some formula and $a \in L_{\alpha_n}(M,G)$.\footnote{Since we have a column for each $(n,a)$, even when $n$ is not (the G\"odel code of) a $\Sigma_k$-formula of appropriate arity, for those bad $n$ just have the ordinals in the column all be $\alpha_n$.}
We then use the universal $\Sigma_k$-formula to simultaneously build these columns upward. Namely, in the $(\phi,a)$-th column at row $i$, the ordinal $\xi^{i}_{(\phi,a)}$ we put in is an ordinal $\xi \ge \sup_{j < i} \xi^j_{(\phi,a)}$ so that $L_\xi(M,G)$ has the length $i$ partial solution to the recursion given by $\phi(x,y,a)$. This a single step because recognizing such a $\xi$ is a $\Sigma_k$ property. So we fill out the grid in $\gamma$ many steps and then set $\alpha_{n+1} = \sup \xi^i_{(\phi,a)}$ to be the supremum of the ordinals in the grid. 

Then at $\alpha = \sup_n \alpha_n$ we have caught our tail and have that $L_\alpha(M,G)$ is closed under solutions to $\Sigma_k$-transfinite recursions of height $\le \gamma$. And since we could get such $\alpha > \alpha_0$ for arbitrary $\alpha_0$, there is a club of such $\alpha$.

Finally, note that both of the class clubs we are looking at are $\Sigma_k$-definable. To say that $L_{\delta}(M,G)$ reflects the universal $\Sigma_k$-formula is a $\Sigma_k$ property. And to say that $L_\alpha(M,G)$ is closed under solutions to $\gamma$-length $\Sigma_k$-transfinite recursions is $\Sigma_k$, since it is a $\Sigma_k$ to check whether something is a solution. So $\ZFCm(k)$ suffices to prove that these class clubs have nonempty intersection. 
\end{proof}

This completes the proof of theorem \ref{thm3:main-thm}.

Finally, let us see by a similar argument that $\PnCA{k+1}$ and $\SkTR k$ can be separated by consistency strength.

\begin{theorem}[Over $\GBCm$] \label{thm3:pnca-impl-sktr}
Let $k$ be a natural number. Suppose $\PnCA{k+1}$ holds. Then there is a coded $V$-submodel of $\GBCm + \SkTR k$.
\end{theorem}

\begin{corollary}
$\GBC + \PnCA{k+1}$ proves $\Con(\GBC + \SkTR k)$ and $\GBCm + \PnCA{k+1}$ proves $\Con(\GBCm + \SkTR k)$.
\end{corollary}

\begin{proof}[Proof of theorem \ref{thm3:pnca-impl-sktr}]
Work over $(M,\Xcal) \models \GBCm + \PnCA{k+1}$. Fix a global well-order $G$ and consider the $(M,G)$-constructible unrolling $W = \Lfrak(M,G)$. (See section \ref{sec2:so-l} for further details of this construction.) Recall that $W$ satisfies $\Sigma_{k+1}$-Separation and $\Sigma_{k+1}$-Collection. Work in $W$.

Let $C \subseteq W$ be the definable club class of ordinals $\alpha$ so that $L_\alpha(M,G) \prec_{\Sigma_k} W$. The existence of such $C$ follows from a reflection argument using the $L(M,G)$-hierarchy. Let $\alpha_0$ be least $> \Ord^M$ in this club $C$. Given $\alpha_n$, pick $\alpha_{n+1} > \alpha_n$ from $C$ which is closed under solutions to $\Sigma_k$-transfinite recursions with parameters and lengths in $L_{\alpha_n}(M,G)$. Such solutions exist because $\Sigma_{k+1}$-Collection implies the existence of solutions to transfinite recursions of $\Sigma_k$ properties. And being the solution to such a recursion is a $\Sigma_k$-expressible property, so $\alpha_{n+1}$ exists by an instance of $\Sigma_k$-Collection. Set $\alpha = \sup_n \alpha_n$, which exists by yet another instance of Collection. Then $L_\alpha(M,G) \models \Sigma_k$-Transfinite Recursion. Let $\Ycal$ be the cutting off of $L_\alpha(M,G)$, i.e.\ the definable class over $W$ consisting of all subsets of $M$ which are in $L_\alpha(M,G)$.  Then $(M,\Ycal) \models \GBCm + \SkTR k$, as desired.
\end{proof}

\clearpage

\chapter{Least models}
\chaptermark{Least models}

\epigraph{\singlespacing Es liegt n\"amlich nahe, das Axiom [der Beschr\"anktheit] in vermeintlich pr\"aziserer Form so zu
fassen, da\ss{} unter allen m\"oglichen Realisierungen des Axiomensystems---wobei isomorphe als nicht verschieden zu betrachten w\"aren---der ``Durchschnitt'', d.h. der kleinste gemeinsame Teilbereich, gew\"ahlt werden soll. Sofern man dieser Fassung nicht \"uberhaupt einen scharfen Sinn abstreiten will, so ist es jedenfalls m\"oglich, da\ss{} die dem Umfang nach verschiedenen m\"oglichen Realisierungen des Axiomensystems nicht einen kleinsten gemeinsamen Teilbereich aufweisen, in dem gleichfalls s\"amtliche Axiome befriedigt w\"urden.}{Abraham Fraenkel}

One desideratum for early axiomatizers of set theory was categoricity, similar to the categoricity results about $\nats$ and $\reals$. Fraenkel \cite{fraenkel1922} and \cite[pp.\ 355--356]{fraenkel1928} wanted a ``Beschr\"anktheitsaxiom'' which would state, essentially, that the only objects that exist are those which are guaranteed to exist by the other axioms. We know now that there can be no such axiom.\footnote{This must be qualified. The L\"owenheim--Skolem theorem implies that there can be no axiomatization of set theory {\em in first-order logic} which admits a unique model. But in different logics we can have categoricity. For instance, it follows from work of Zermelo \cite{zermelo1930} that second-order $\ZFC$---i.e., $\ZFC$ but with Separation and Collection formulated as single axioms in second-order logic---plus ``there are no inaccessible cardinals'' has a unique model, namely $V_\kappa$ where $\kappa$ is the least inaccessible. (It must be noted, however, that this theory has a very ad hoc feel.)

But the set theories considered in this dissertation are all formulated in first-order logic, ruling out any Beschr\"anktheitsaxiom. It would go too far astray to give a defense here of why we would want to restrict to first-order logic, but let me mention \cite{vaananen2001}. See also the epigraph to chapter 2.}
But we can transmute this question of axioms into a model theoretic question. At first approximation, what we would like to know is: What are the objects that must be in every model of $T$? Do they form a model of $T$? A positive answer would give a partial realization of Fraenkel's desire. While we cannot write down an axiom (or a set of axioms) which uniquely picks out this structure, we would know that if we restrict to the bare minimum possible we still get a model of $T$.

As stated, this admits a trivial answer. By the nonstandardness phenomenon, the only objects in every model of set theory are those appearing in $V_n$ for some {\em standard} $n$. These form $V_\omega$, which of course lacks any infinite sets. So in this naive form, the question is not interesting. But we can refine it to a more interesting question by restricting which models we look at. A natural restriction is to only look at transitive models. They hold a special place in set theoretic practice and many set theorists believe we have a determinate notion of well-foundedness and can thereby pick out the transitive models. So the question becomes: is the intersection of all the transitive models of $T$ itself a model of $T$? Equivalently, is there a least transitive model of $T$?

Before moving to the main topic---models of second-order set theories---let me quickly review what is known for models of first-order set theory. Shepherdson \cite{shepherdson1953} and, independently, Cohen \cite{cohen1963} proved that there is a least transitive model of $\ZFC$.\footnote{Of course, their proofs require a consistency assumption, namely that there is some transitive model of $\ZFC$ at all.}
This model is $L_\alpha$ where $\alpha$ is the least ordinal so that there is a transitive model of $\ZFC$ of height $\alpha$. Their argument, which uses that $\ZFC$ is absolute to $L$, generalizes to stronger theories.\footnote{Or more precisely, Shepherdson's argument generalizes to stronger theories. The essence of Shepherdson's argument is the same as the standard contemporary argument that there is a least transitive model of $\ZFC$. But Cohen uses a different argument which goes through what he calls ``strongly constructible'' sets, a strengthening of constructability, which I do not see how to generalize to get results about stronger theories.}
In particular, it generalizes to theories extending $\ZFC$ by asserting the existence of ``small'' large cardinals. Formally, say that a first-order set theory $T$ is {\em absolute to $L$} if $M \models T$ implies that $L^M \models T$. Then if there is a transitive model of $T$ which is absolute to $L$ there is a least transitive model of $T$. So there is a least transitive model of, for example, $\ZFC$ plus there is a proper class of Mahlo cardinals.

But this phenomenon does not extend too far up the large cardinal hierarchy. It fails for large cardinals which give elementary embeddings of the universe into an inner model.

\begin{proposition}
Let $T \supseteq \ZFC$ be a theory which proves there is a measurable cardinal. Then there is not a least transitive model of $T$.
\end{proposition}

\begin{proof}
Suppose otherwise that $N$ is the least transitive model of $T$. Let $M \subseteq N$ be the inner model obtained from taking an ultrapower of $M$ using a measure on a measurable cardinal in $N$. By leastness, $M = N$, a contradiction.
\end{proof}

On the other hand, we can recover something of this phenomenon for measurable cardinals and beyond. Results from inner model theory show that if an ordinal $\kappa$ is measurable in some model then there is a least model in which $\kappa$ is measurable. And this has been extended higher up the large cardinal hierarchy, although it remains open in many cases, most notably for $\kappa$ supercompact.

The lesson to be had is that for strong enough first-order set theories, we do not have least transitive models. However, if we restrict the models we look at in some further (non-first-order expressible) way, then we do get least models. In this chapter we will see that there is a similar phenomenon for second-order set theories, except the reason and the `fix' to get leastness are different.

\begin{definition}
Let $T$ be a second-order set theory. The {\em least transitive model of $T$}---if it exists---is the unique transitive $(M,\Xcal) \models T$ so that $(M,\Xcal)$ is a submodel of any transitive model of $T$. The {\em least $\beta$-model of $T$}---if it exists---is the unique transitive $\beta$-model $(M,\Xcal) \models T$ which is a submodel of any transitive $\beta$-model of $T$.\footnote{In chapter 1, we did not require $\beta$-models to be transitive. But every $\beta$-model is isomorphic to a transitive model so the extra requirement here is harmless. If one prefers to drop it, then one would need to tweak the definition so that the least $\beta$-model of $T$ embeds into every $\beta$-model of $T$, rather than being a literal submodel.}
\end{definition}

The main theorem of this chapter answers which second-order set theories have least transitive models for a broad class of theories. In short, strong theories do not have least transitive models while weaker theories do.

\begin{theorem} \label{thm4:least-trans-models} \ 
\begin{itemize}  
\item There is not a least transitive model of $\KM$ nor of $\KMCC$.
\item For $k \ge 1$ there is not a least transitive model of $\GBC + \PnCA k$ nor of $\GBC + \PnCAp k$. 
\item There is a least transitive model of $\GBC + \ETR_\Gamma$, for $\omega^\omega \le \Gamma \le \Ord$.
\item (Shepherdson \cite{shepherdson1953}) There is a least transitive model of $\GBC$.
\end{itemize}
\end{theorem}

There is some redundancy in the statement of this theorem. In chapter 2 we saw that any model of $\KM$ contains a $V$-submodel of $\KMCC$, so there is a least transitive model of $\KM$ if and only if there is a least transitive model of $\KMCC$. (And similar remarks apply for $\GBC + \PnCA k$ versus $\GBC + \PnCAp k$.)

The argument for the negative part of theorem \ref{thm4:least-trans-models}, that strong second-order set theories do not have least transitive models, actually shows something stronger. Namely, it shows that a given fixed first-order part does not have a least $\KM$-realization (or $(\GBC + \PnCA k)$-realization or\ldots). 

\begin{definition}
Let $M$ be a model of first-order set theory and $T$ be some second-order set theory. The {\em least $T$-realization for $M$}---if it exists---is the $T$-realization $\Xcal$ for $M$ so that for any $T$-realization $\Ycal$ for $M$ we have $\Xcal \subseteq \Ycal$. 
\end{definition}

\begin{theorem} \label{thm4:no-least-rlzn}
No countable $M \models \ZFC$ has a least $\KM$-realization. Moreover, if $k \ge 1$ then no countable $M \models \ZFC$ has a least $(\GBC + \PnCA k)$-realization.
\end{theorem}

The proof of this uses little about the theory of $M$ itself. All the important work takes place above $M$ in the unrolling. As such, a version of this theorem goes through for $\KMm$ and $\GBCm + \PnCA k$. It also yields that least models cannot be recovered by moving to stronger theories.

\begin{theorem} 
No computably axiomatizable extension of $\KM$ (in $\Lcal_\in$) has a least transitive model. More generally, no computably axiomatizable extension of $\GBCm + \PnCA k$, for $k \ge 1$, has a least transitive model.\footnote{Indeed, this is true for more than just computably axiomatizable extensions. What we get is that if $T \supseteq \GBCm + \PnCA 1$ is an element of every transitive model of $T$ then $T$ cannot have a least transitive model. In particular, no arithmetical $T$ or even hyperarithmetical $T \supseteq \GBCm + \PnCA 1$ can have a least transitive model.}
\end{theorem}

Before discussing $\beta$-models, where we get positive results even for strong theories, let me highlight the conspicuous absence of $\GBC + \ETR$ in theorem \ref{thm4:least-trans-models}. 

\begin{question}
Is there a least transitive model of $\GBC + \ETR$?
\end{question}

While this question remains open, something can be said about the structure of $(\GBc + \ETR)$-realizations for a model $M$.

\begin{theorem}
Let $M \models \ZFC$ be $(\GBc + \ETR)$-realizable. Then $M$ has a basis of minimal $(\GBc + \ETR)$-realizations, where amalgamable $(\GBc + \ETR)$-realizations\footnote{Two $T$-realizations $\Xcal$ and $\Ycal$ for $M$ are {\em amalgamable} if there is a $\GBcm$-realization $\Zcal$ for $M$ so that $\Xcal$ and $\Ycal$ are both subsets of $\Zcal$.}
sit above the same basis element. That is, there is a set $\{\Bcal\}$ of $(\GBc + \ETR)$-realizations for $M$ satisfying the following.
\begin{enumerate}
\item Elements of the basis are pairwise non-amalgamable;
\item If $\Ycal$ is any $(\GBc + \ETR)$-realization for $M$ then there is a unique basis element $\Bcal$ so that $\Ycal \supseteq \Bcal$; and
\item If $\Xcal$ and $\Ycal$ are amalgamable $(\GBc + \ETR)$-realizations for $M$ then they sit above the same $\Bcal$.
\end{enumerate}
\end{theorem}

And we get the same result for $\GBC + \ETR$ if $M$ has a definable global well-order.

Let us now turn to $\beta$-models. Weak theories have least transitive models while strong theories do not. For first-order set theories, we could get least models for strong theories by requiring extra from our models, namely by nailing down which ordinals are to be the large cardinals of the model. For second-order set theories we recover leastness by requiring that the model be correct about well-foundedness.

\begin{theorem} \label{thm4:least-beta-models} \ 
\begin{itemize}
\item (Folklore) There is a least $\beta$-model of $\KM$.
\item For $k \ge 1$ there is a least $\beta$-model of $\GBC + \PnCA k$.
\item There is a least $\beta$-model of $\GBC + \ETR$.
\item For any $\Gamma \ge \omega^\omega$ there is a least $\beta$-model of $\GBC + \ETR_\Gamma$.
\item (Folklore) There is a least $\beta$-model of $\GBC$.
\end{itemize}
\end{theorem}

Of course, the above results that there is a least model of such and such theory have consistency requirements, namely that the theory has any model (of appropriate type) at all! For the sake of readability I have suppressed mentioning the exact consistency assumptions here. Every result in this chapter follows from $\ZFC$ $+$ ``there is an inaccessible cardinal''. For the reader who wants something more precise, see the statements of the theorems in the body of this chapter.

This chapter is organized as follows. First I collect some observations about $\beta$-models and well-founded classes that will be used later in the chapter. I then review Barwise's notion of the admissible cover, which will be used to prove the results for models of strong theories. This sets us up to finally get to the results about the existence and non-existence of different kinds of least models. I have organized this chapter by the strength of the theories, as the methods used vary. I first give the results about strong theories. Next comes results for theories of medium strength. Finally, I give results about weak theories, or rather collect some results from the literature and from chapter 1. The chapter concludes with a coda on the analogy between second-order set theory and second-order arithmetic.

The section on models of strong theories will make essential use of the results from chapter 2. As such, the reader is strongly encouraged to look at that chapter before reading section \ref{sec4:strong}.

\section{\texorpdfstring{$\beta$-models and well-founded classes}{Beta-models and well-founded classes}}

In this section I collect several observations about $\beta$-models and well-founded classes. They will be used many places in this chapter, usually without explicit citation.

\begin{observation} \label{obs4:beta-submodel}
Suppose $(M,\Xcal)$ is a $\beta$-model of $\GBcm$ and $(M,\Ycal) \models \GBcm$ is a $V$-submodel of $(M,\Xcal)$. Then $(M,\Ycal)$ is a $\beta$-model.
\end{observation}

\begin{proof}
Suppose towards a contradiction that $(M,\Ycal)$ is not a $\beta$-model. Then there is $R \in \Ycal$ so that $(M,\Ycal)$ thinks that $R$ is ill-founded but in $V$ we can see that $R$ is well-founded. Because $(M,\Ycal)$ thinks $R$ is ill-founded there is a set $r \in M$ which witnesses this; $R$ being ill-founded means there is some infinite descending sequence in $R$, but this sequence is countable and hence must be a set. But then $(M,\Xcal)$ thinks $R$ is ill-founded since it also sees that $r$ witnesses the ill-foundedness of $R$. But then, because $(M,\Xcal)$ is a $\beta$-model, $R$ must really be ill-founded, a contradiction.
\end{proof}

Note that we did not need near the full strength of $\GBcm$ for this argument to go through. All we need is that the theory of the models be strong enough to verify that a class being ill-founded is witnessed by a set. 

We can strengthen this observation.

\begin{observation} \label{obs4:beta-v-submodel}
Suppose $(N,\Ycal) \models \GBcm$ is an $\Ord$-submodel of a $\beta$-model $(M,\Xcal) \models \GBcm$. Then $(N,\Ycal)$ is a $\beta$-model.
\end{observation}

\begin{proof}
Take arbitrary ill-founded $R \in \Ycal$. Then, because $(M,\Xcal)$ is a $\beta$-model, there is a set $r \in M$ which witnesses the ill-foundedness of $R$. This $r$ might not be in $N$, but it has some rank $\alpha$ in $M$. Now consider $r' = R \cap V_{\alpha+1}^N \in N$. This $r'$ is well-founded if and only if $R$ is, since the ill-foundedness in $R$ occurs by rank $\alpha$. But $M$ thinks $r'$ is ill-founded and well-foundedness is absolute between transitive models of $\ZFCm$, so $N$ thinks $r'$ is ill-founded. So $N$ correctly thinks that $R$ is ill-founded and, as $R$ was arbitrary, $N$ is correct about well-foundedness.
\end{proof}

The following observation generalizes observation \ref{obs4:beta-submodel} in a different direction to tell us to what extent non-$\beta$-models with the same first-order part must agree on what is well-founded. In short, they must agree as much as possible.

\begin{observation} \label{obs4:agree-on-wfddness}
Suppose $\Xcal$ and $\Ycal$ are $\GBcm$-realizations for $M$. If $R \in \Xcal \cap \Ycal$ is a class relation then $(M,\Xcal)$ and $(M,\Ycal)$ agree on whether $R$ is well-founded. 
\end{observation}

\begin{proof}
Because the two models have the same first-order part.
\end{proof}

The following observation also appeared in the proof of corollary \ref{cor3:etr-inner-model}, but for the sake of clarity I reproduce it with a proof here.

\begin{observation}
Suppose $(N,\Ycal)$ is an $\Ord$-submodel of $(M,\Xcal)$, where both are models of $\GBCm$. Suppose $(N,\Ycal) \models R$ is well-founded. Then $(M,\Xcal) \models R$ is well-founded.
\end{observation}

\begin{proof}
Suppose otherwise. Then there is an ordinal $\alpha \in M$ so that $M \models R \rest F''\alpha$ is ill-founded, where $F \in \Ycal$ is a bijection between $\Ord$ and $\dom R$, which exists by Global Choice. But by Replacement $R \rest F''\alpha$ must be in $N$. And $N$ is a transitive submodel of $M$ and they are both models of $\ZFCm$, so they must agree on what sets are well-founded. So $N \models R \rest F''\alpha$ is ill-founded, so $(N,\Ycal) \models R$ is ill-founded, a contradiction.
\end{proof}

\section{Admissible covers and nonstandard compactness arguments} \label{sec4:admissible}

To get theorems \ref{thm4:least-trans-models} and \ref{thm4:no-least-rlzn} will require more than the elementary tools of admissible set theory. I will review the necessary material in this section.

Speaking roughly, the strategy to prove that strong second-order set theories do not have least transitive models will be to work with their unrollings. Given an unrolling $W$ of a model $(M,\Xcal)$, we want to produce a new model $N$ so that $N$ and $W$ agree up to their largest cardinal, but we introduced ill-foundedness in $N$ above that. Then, cutting off $N$ we get $(M,\Ycal)$. But the ill-foundedness will mean that $\Ycal$ cannot be contained in $\Xcal$, so $(M,\Xcal)$ cannot be the least transitive model of our theory.

It has previously been studied by H.\ Friedman where we can introduce ill-foundedness in a model of set theory. He proved the following very general theorem.

First though, let me set up some terminology. If $N$ is a model of set theory with $A$ as a transitive submodel\footnote{$A$ is a {\em transitive submodel} of $N$ if $A \subseteq N$ and for all $a \in A$ and all $b \in N$ if $N \models b \in a$ then $b \in A$.}
then say that $N$ is {\em ill-founded at $A$} if there is an omega sequence of ordinals in $N \setminus A$ which is co-initial in $\Ord^N \setminus \Ord^A$.\footnote{This is equivalent to asking that $A$ be {\em topless} in $N$, meaning that there is no smallest ordinal in $N$ above the ordinals of $A$.}
If $A$ is an admissible set then $\Lcal_A$ denotes the associated admissible fragment of $\Lcal_{\Ord,\omega}$, i.e.\ the infinitary language consisting of the formulae in $A$. If $A$ is countable then the Barwise compactness theorem applies to $\Lcal_A$-theories.

\begin{theorem}[{\cite[theorem 2.2]{friedman1973}}] \label{thm4:trans-friedman}
Let $A$ be a countable admissible set and let $T \subseteq A$ be an $\Lcal_A$ theory which is $\Sigma_1$-definable in $A$. 
If there is a model of $T$ which contains $A$ then there is an ill-founded model $N$ of $T$ so that the following hold.
\begin{itemize}
\item $N$ contains $A$ as a transitive submodel and $N$ is ill-founded at $A$. In particular: 
\item $\wfp(N) \supseteq A$;\footnote{$\wfp(N)$ is the well-founded part of $N$, that is the subset of $N$ consisting of elements $a$ so that the membership relation of $N$ below $a$ is well-founded.} and 
\item $\Ord^{\wfp(N)} = \Ord^A$.
\end{itemize}
\end{theorem}

Figure \ref{fig4:trans-friedman} illustrates the theorem. 

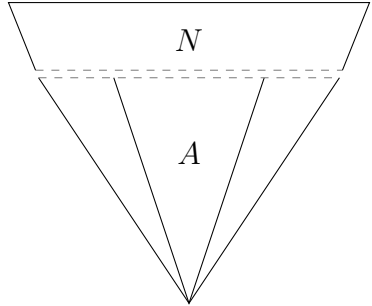
\begin{figure}
\begin{center}
\begin{tikzpicture}
\draw (1,3) -- (0,0) -- (-1,3);
\draw (-2,3) -- (0,0) -- (2,3);
\draw[dashed,color=gray] (-2,3) -- (2,3);
\draw[dashed,color=gray] (-2.04,3.1) -- (2.04,3.1);
\draw (-2.04,3.1) -- (-2.4,4) -- (2.4,4) -- (2.04,3.1);
\draw (0,2) node {$A$};
\draw (0,3.5) node {$N$};
\end{tikzpicture}
\end{center}

\caption{An illustration of Friedman's theorem. The dashed line represents that $N$ is ill-founded at $A$.}
\label{fig4:trans-friedman}
\end{figure}

Observe that in general we cannot get $\wfp(N) = A$. For a counterexample, suppose that $\alpha$ is a countable ordinal so that $L_\alpha \models \ZFCm$ $+$ ``every set is countable''. Then by Friedman's theorem there is $N \supsetend L_\alpha$ satisfying $\KP + V=L$ $+$ ``every set is countable'' so that $N$ is ill-founded at $L_\alpha$. Consider an ordinal $\gamma \in N$ which is not in $\wfp(N)$. Then, $N$ thinks $\gamma$ is countable so there is $G \subseteq \omega^2$ so that $(\omega,G) \cong (\gamma,\in)$. In particular, $G$ is in $\wfp(N)$. I claim that this $G$ cannot be in $L_\alpha$. To see this, note that $L_\alpha$ satisfies enough Replacement that it can compute $\wfp(G)$ from $G$. But $\wfp(G) \cong \alpha$. So if $G \in L_\alpha$ then by Mostowski's collapse lemma we must have $\alpha \in L_\alpha$, a contradiction.

The reader may find it useful to see a simple application of this theorem before moving on.

\begin{proposition}
Consider Zermelo set theory $\Zsf$, the first-order set theory axiomatized by Extensionality, Pairing, Union, Infinity, Powerset, and the axiom schema of Separation. There are transitive models of $\Zsf$ which are wrong about well-foundedness. That is, there is $M \models \Zsf$ transitive with $R \in M$ so that $M \models R$ is well-founded but externally we can see that $R$ is ill-founded. 
\end{proposition}

Since we are concerned with transitive models, they will automatically satisfy Foundation. As well, we can easily arrange so that $M$ also satisfies Choice. So this cute little proposition can be seen as yet more evidence for the importance of Collection in the axioms for set theory. If we include no fragment of Collection, then we do not get the incredibly useful fact that well-foundedness is absolute for transitive models.

\begin{proof}
Let $A = L_{\omegaoneck}$, where $\omegaoneck$ is the Church--Kleene ordinal, i.e.\ the least admissible ordinal $> \omega$. It is easy to see that there are countable models of $\KP + \text{``}V_{\omega+\omega}$ exists'' which contain $A$. So by Friedman's theorem there is $N \models \KP + \text{``}V_{\omega+\omega}$ exists'' which is ill-founded at $A$. Take $\gamma \in N$ a countable ordinal which is not in the well-founded part. Then, there is $G \subseteq \omega^2$ in $N$ so that $(\omega,G) \cong (\gamma,\in)$.

Now let $M = V_{\omega+\omega}^N$. It is well-known that $V_{\omega+\omega} \models \Zsf$. Moreover, $\KP$ is enough to verify this so that indeed $M \models \Zsf$. Observe also that $M$ is transitive, because all of its elements have rank $< \omega + \omega < \omegaoneck$. But $M$ is not correct about well-foundedness, because $M$ thinks $G$ is well-founded.
\end{proof}

While Friedman's theorem is fantastic, it is not quite general enough for my purposes. To illustrate the difficulty, start with a countable $\beta$-model $(M,\Xcal) \models \KMCC$. Consider its unrolling $W$. Then, $\Hyp(M) \in W$ (see proposition \ref{prop4:hyp-v-in-kmp} below). So by Friedman's theorem we can find $N \models \ZFCmi$ which is ill-founded at $\Hyp(M)$ and so that $M$ is a rank-initial segment of $N$. Then, $N$ thinks that $\Hyp(M)$ exists. However, $\Hyp(M)^N$ cannot be in the well-founded part of $N$ by construction. So Friedman's theorem cannot be used to produce a model of $\ZFCmi$ which is ill-founded at $\Hyp(M)^N$. 

The problem is that Friedman's theorem only applies to well-founded models of $\KP$, whereas I need to be able to handle ill-founded models. Fortunately, Barwise developed machinery for compactness arguments over an ill-founded domain.

The important notion here is that of the admissible cover of a model of set theory, which I review here. For further details see the appendix to \cite{barwise1975}. Briefly, the admissible cover of $U \models \KP$ is a certain admissible structure with $U$ as its urelements. The admissible cover of $U$ allows us to then apply the tools of admissible set theory to $U$, even though $U$ itself is ill-founded.

First we must discuss set theory with urelements. An {\em urelement} is an object that is neither a set nor a class but can be an element of sets.\footnote{The reader (or maybe just me) may find it amusing to think of urelements as the opposite of classes. A class can have elements but cannot be an element whereas an urelement cannot have elements but can be an element. A set, of course, can both have elements and be an element.}
In contemporary set theory we usually formulate things just in terms of pure sets, i.e.\ those sets with no urelements in their transitive closures. This is harmless, as any structures of interest can be simulated in the pure sets. For instance, we do not need the natural numbers as urelements since we can instead work with the finite ordinals. But for this context we will want models with urelements. These structures, of course, can be simulated with pure sets. So despite the use of set theories with urelements, all of the below can be formalized in ordinary $\ZFC$.

Formally, the theory we will work with is $\KPU$, {\bf K}ripke--{\bf P}latek set theory with {\bf u}relements. I give an axiomatization here, both for the benefit of the reader unfamiliar with $\KP$---that is, $\KPU$ sans urelements---and to highlight the fact that $\KPU$ does not prove there is a set of all urelements.

\begin{definition}
The theory $\KPU$ is a first-order set theory formulated with urelements. That is, $\KPU$ is a two-sorted theory whose objects are {\em sets} and {\em urelements}. In addition to basic axioms asserting that sets and urelements are distinct, nothing is a member of an urelement, etc., $\KPU$ is axiomatized by the following axioms: Extensionality for sets, Pairing, Union, $\Delta_0$-Separation, $\Delta_0$-Collection, and Foundation. Here, Foundation is the schema whose instances are of the form
\[
\exists x\ \phi(x) \impl (\exists x\ \phi(x) \land \forall y \in x\ \neg \phi(y))
\]
for each $\Lcal_\in$-formula $\phi$.\footnote{The models of $\KPU$ we will consider will all be well-founded, so they will automatically satisfy the strongest form of Foundation.}
\end{definition}

I will write $(M,U)$ for the model of $\KPU$ with sets $M$ and urelements $U$, suppressing writing the set-set and urelement-set membership relations. Often I will give a single name to this structure, usually in the fraktur font, such $\Cfrak$ or $\Rfrak$. We will also consider structures with additional functions and relations, which I will denote by $(M,U;R_0, R_1, \ldots)$. These structures will satisfy the schemata of $\Delta_0$-Separation and $\Delta_0$-Collection in the expanded language.

For an example of a model of $\KPU$, let us consider a model with the reals as its urelements. That is, consider $(\reals, +, \times, <)$ the set of real numbers with its arithmetic operations. Let $M$ consist of all hereditarily finite sets above the reals. Formally, let $M_0 = \emptyset$ and $M_{n+1}$ be the set of all finite subsets of $\reals \cup M_n$. Then $M = \bigcup_{n \in \omega} M_n$. The reader can easily check that $\Rfrak = (M,\reals;+,\times,<)$ is a model of $\KPU$. But notice that $M$ does not contain $\reals$, as all sets in $M$ are finite. So $\Rfrak$ gives an example of a model of $\KPU$ without a set of all urelements. It also gives a model of $\KPU$ which does not satisfy the axiom of Infinity.

Let us consider another example of a model of $\KPU$, this one more relevant to the present discussion. Let $U$ be an $\omega$-nonstandard first-order model of set theory with membership relation $E$. We will treat $U$ as the urelements for a model of $\KPU$. Similar to the previous example we can define the hereditarily finite sets above $U$. If $M$ is the collection of such then $\Ufrak = (M,U;E)$ is a model of $\KPU$. 

We would like to use a structure like this to mimic internal talk in $U$ with talk of honest-to-$V$ well-founded sets. But this $\Ufrak$ does not have enough sets to do so. For $u \in U$ let $u_E = \{ v \in U : u \mathbin E v \}$ be the set of what $U$ thinks are the elements of $u$. Because $M$ only consists of finite sets the only $u \in U$ for which $u_E \in M$ are those which really are finite. In particular, $(\omega^U)_E \not \in M$ and any $\alpha \in \omega^U \setminus \omega$ will have $\alpha_E \not \in M$. So this $\Ufrak$ cannot directly talk about all the `sets' in $U$. To do so, we need more.

\begin{definition}
Let $U$ be a (possibly ill-founded) model of first-order set theory with membership relation $E$. Then $\Mfrak = (M,U;E,F) \models \KPU$, a model of $\KPU$ with $U$ as urelements, {\em covers} $U$ if $F$ is a function from $U$ to $M$ so that $F(u) = u_E$. 
\end{definition}

Any $\Mfrak$ which covers $U$ can mimic $\Lcal_U$ talk. One useful fact about $\Lcal_A$ for admissible $A$ is that every element of $A$ is definable by a single $\Lcal_A$-formula. Namely, $x = a$ is defined by the formula
\[
\forall y\ y \in x \iff \bigvee_{b \in a} y = b
\]
where ``$y = b$'' is an abbreviation for the formula defining $b$. Because $A$ is well-founded this recursive definition unwraps into a single $\Lcal_A$-formula. A similar idea allows us to define elements of $U$ by $\Lcal_\Mfrak$-formulae: define $x = a$ for $a \in U$ by 
\[
\forall y\text{ an urelement } y \mathbin E x \iff \bigvee_{b \in F(a)} y = b.
\]
More, $F$ lets us translate bounded quantification for $U$ to bounded quantification for $\Mfrak$: replace $\exists x \mathbin E y$ with $\exists x \in F(y)$ and similarly for universal quantification. So corresponding to each (of what $U$ thinks is a) $\Delta_0$ $\Lcal_U$-formula is a $\Delta_0$ $\Lcal_\Mfrak$-formula, and similarly for $\Sigma_1$ or $\Pi_1$ formulae.

Barwise proved that there is a smallest admissible structure which covers $U$. This structure, the {\em admissible cover of $U$}, is the intersection of all admissible structures which cover $U$ and enjoys many nice properties. I summarize them here.

\begin{theorem}[Barwise {\cite[appendix]{barwise1975}}]
Let $U \models \KP$ be a possibly ill-founded model of set theory with membership relation $E$ and let $\Cfrak = (C,U;E,F)$ be the admissible cover of $U$. 
\begin{itemize}
\item If $U$ is countable then $\Cfrak$ is countable.
\item The pure sets of $\Cfrak$ are isomorphic to the well-founded part of $U$. 
\item For any $A \subseteq U$ we have $A \in \Cfrak$ if and only if there is $a \in U$ so that $A = a_E$. 
\item The infinitary $\in$-diagram of $U$,\footnote{That is, the collection of all sentences of the form $\forall x\ x \in a \iff \bigwedge_{b \in a} x = b$.}
considered as a set of $\Lcal_\Cfrak$-sentences, is $\Sigma_1$-definable over $\Cfrak$. 
\end{itemize}
\end{theorem}

With this notion in hand we are now ready to generalize Friedman's theorem to the ill-founded. We want, when starting with a possibly ill-founded model $A$ to produce the picture in figure \ref{fig4:cis-friedman}, a variation of the picture in figure \ref{fig4:trans-friedman} for Friedman's theorem. That is, given a theory $T$ satisfying an appropriate consistency assumption, we want $N \supseteq A$ a model of $T$ which is ill-founded at $A$.

\begin{figure}
\begin{center}
\begin{tikzpicture}
\draw (-.5,1.5) -- (0,0) -- (.5,1.5);
\draw  (.54,1.6) --  (1,3);
\draw (-.54,1.6) -- (-1,3);
\draw (-1,1.5) -- (0,0) -- (1,1.5);
\draw  (1.07,1.6) --  (2,3);
\draw (-1.07,1.6) -- (-2,3);
\draw[dashed,color=gray] (-1,1.5) -- (1,1.5);
\draw[dashed,color=gray] (-1.07,1.6) -- (1.07,1.6);
\draw[dashed,color=gray] (-2,3) -- (2,3);
\draw[dashed,color=gray] (-2.04,3.1) -- (2.04,3.1);
\draw (-2.04,3.1) -- (-2.4,4) -- (2.4,4) -- (2.04,3.1);
\draw (0,2) node {$A$};
\draw (0,3.5) node {$N$};
\draw (1.2,0) -- (1.3,0) -- (1.3,1.5) -- (1.2,1.5);
\draw (2.1,.75) node {$\wfp(N)$};
\end{tikzpicture}
\end{center}
\caption{A picture of the desired generalization of Friedman's theorem to the ill-founded realm.} \label{fig4:cis-friedman}
\end{figure}
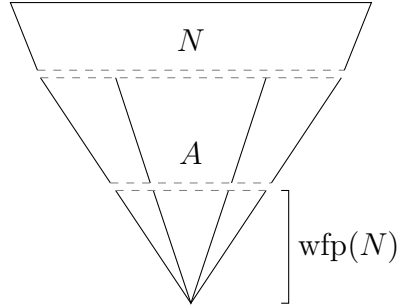

\begin{theorem} \label{thm4:cis-friedman}
Let $(A,E^A) \models \KP$ be countable and $\Cfrak = \Cov_A$. Suppose that $T$ is an $L_\Cfrak$ theory which is $\Sigma_1$-definable over $\Cfrak$.
If there is a model of $T$ which contains $A$ then there is $(N,E^N) \models T$ so that:
\begin{itemize}
\item $A$ is a transitive submodel of $N$;
\item $\Ord^A$ is a proper initial segment of $\Ord^N$;
\item There is an $\omega$-sequence coinitial in $\Ord^N \setminus \Ord^A$.
\end{itemize}
\end{theorem}

\begin{proof}
Friedman's proof can be adapted to this context, using the technology of the admissible cover. 

Extend $T$, if necessary, to include the infinitary $\in$-diagram of $A$. This extension is consistent as there is a model of $T$ containing $A$. The goal is now to construct a further extension $T'$ in a language with countably many new constants $c_n$ so that the following conditions hold:
\begin{enumerate}
\item Each $\phi \in T'$ is consistent;
\item For $\phi \in \Lcal_\Cfrak$, either $\phi \in T'$ or $\text{``}\neg\phi\text{''} \in T'$;
\item For $\bigwedge \Phi \in \Lcal_\Cfrak$ if $\Phi \subseteq T'$ then $\bigwedge \Phi \in T'$;
\item For $\text{``}\forall x \phi(x)\text{''} \in \Lcal_\Cfrak$ if for each $c_n$ we have $\phi(c_n) \in T'$ then $\text{``}\forall x \phi(x)\text{''} \in T'$;
\item For each $a \in A$, $\text{``}a \in c_0\text{''} \in T'$; and
\item If $\text{``}a \in c_n\text{''} \in T'$ for each $a \in A$ then there is $m > n$ so that $\text{``}c_m \in c_n\text{''} \in T'$ and $\text{``}a \in c_m\text{''} \in T'$ for each $a \in A$.
\end{enumerate}
Conditions $(1\text{--}3)$ ensure there is a model of $T'$. Condition $(5)$ forces any model of $T'$ to contain new ordinals. Conditions $(4\text{--}6)$ force that the model of $T'$ is ill-founded above $A$. To see this note, that if $\beta < \gamma$ are ordinals above $\Ord^A$ but below $\inf\{\rank(c_n) : c_n \text{ above } A\}$ then for every $c_n$ we get $\text{``}\rank(c_n) \ge \beta \impl \rank(c_n) > \gamma\text{''} \in T'$. So by condition $(4)$ we can conclude $\beta > \gamma$, a contradiction.

$T'$ is constructed from $T$ in $\omega$ many stages. We continually add new formulae to ensure properties $(1\text{--}6)$ hold at the end. Fix an enumeration $\seq{\phi_n}$ of the $\Lcal_\Cfrak(c_n : n \in \omega)$-sentences\footnote{To be clear, by $\Lcal_\Cfrak(c_n : n \in \omega)$ I mean the infinitary language consisting of formulae in $\Cfrak$ in the language of $\Cfrak$ with additional symbols $c_n$ for $n \in \omega$.}
so that $c_n$ first appears after $\phi_n$ and before the first appearance of $c_{n+1}$. 
\begin{itemize}
\item Define $T'_0 = T \cup \{ a \in c_0 : a \in A \}$. This theory is consistent by Barwise compactness. This ensures property $(5)$. Set $m_0 = 1$ and $m_{-1} = 0$.

\item For $n \ge 0$, define $T'_{3n+1}$ to be $T'_{3n} \cup \{\psi\}$, where $\psi$ is chosen from $\phi_n$ and $\neg \phi_n$ so as to be consistent with $T'_{3n}$. This step will ensure properties $(1)$ and $(2)$.

\item For $n \ge 0$, define $T'_{3n+2}$ as follows, according to which of three cases we fall into.
\begin{itemize}
\item If $\phi_n$ is of the form $\bigwedge \Phi$ and $\neg \phi_n$ is in $T'_{3n+1}$, then take $T'_{3n+2} = T'_{3n+1} \cup \{\neg \phi\}$ where $\phi \in \Phi$ and this is consistent. This step ensures property $(3)$.

\item If $\phi_n$ is of the form $\forall x \psi(x)$ and $\neg \phi_n$ is in $T'_{3n+1}$, then take $T'_{3n+2} = T'_{3n+1} \cup \{ \neg \psi(c_m) \}$, where $m$ is the index of the least unused $c_m$. This step ensures property $(4)$.

\item Otherwise, just take $T'_{3n+2} = T'_{3n+1}$.
\end{itemize}

\item For $n \ge 0$, define $T'_{3n+3}$ as follows, according to which of two cases we fall into.
\begin{itemize}
\item If there is $a \in A$ so that $T'_{3n+2} \cup \{ a \not \in c_{m_n} \}$ is consistent, take this to be $T'_{3n+3}$. Set $m_{n+1} = m_n$.

\item Otherwise, the theory $T'_{3n+2} \cup \{ a \in c_{m_n} : a \in M \}$ is consistent. By Barwise compactness, so is the the theory $T'_{3n+2} \cup \{ a \in c_{m_n} : a \in A \} \cup \{ c_{m_n} \in c_{m_{n-1}} \}$. Take this to be $T'_{3n+3}$ and set $m_{n+1}$ to be the index of the least unused $c_m$.
\end{itemize}

\item Set $T' = \bigcup_n T'_n$.

\end{itemize}

By the construction for $T'_{3n+3}$, for every $n \ge -1$ we have that $c_{m_{n+1}} \in c_{m_{n}}$ and $a \in c_{m_n}$, for all $a \in A$ are in $T'$. This gives property $(6)$. 
\end{proof}

\section{Strong theories} \label{sec4:strong}

In this section we will get results about least models of strong theories, those of strength $\GBC + \PCA$ and above. The main results, that strong theories do not have least transitive models, will be derived from the following master lemma.

\begin{masterlemma} \label{lem4:master-lemma}
Let $T \supseteq \GBCm + \ETR$ be a second-order set theory which proves the existence of $\Hyp(V)$. Suppose $(M,\Xcal) \models T$ is countable and $T$ is in $M$. Then there is $\Ycal \subseteq \powerset(M)$ so that $(M,\Ycal) \models T$ but $\Xcal \not \subseteq \Ycal$. 
\end{masterlemma}

Before proving this master lemma I must clarify what it means for a second-order set theory to prove that $\Hyp(V)$ exists. Recall that for a set $a$ that $\Hyp(a)$ is the smallest admissible set $h$ with $a \in h$. Always, $\Hyp(a) = L_\alpha(a)$ where $\alpha$ is the least ordinal $\xi$ so that $L_\xi(a) \models \KP$. Of course, if $A$ is a proper class then there can be no class, admissible or otherwise, with $A$ as an element. So it does not make literal sense to talk of $\Hyp(V)$ inside a model of second-order set theory. But recall from chapter 2 that models of $\ETR$ can reach higher than $\Ord$, coding `sets' of high rank by class-sized relations. (Indeed, this is why the master lemma asks that $T \supseteq \GBCm + \ETR$.) In particular, there are codes for `sets' which look like $L_\Gamma(A)$, for $\Gamma$ a class well-order and $A$ a class.
In $\GBCm + \ETR$, we can talk about the theory of a coded transitive `set', so it makes sense to ask whether $L_\Gamma(A)$ satisfies $\KP$. If there is a class well-order $\Gamma$ so that $L_\Gamma(A) \models \KP$ then we say that $\Hyp(A)$ exists. Given such a $\Gamma$ there is a least initial segment $\Gamma_0$ of $\Gamma$ so that $L_{\Gamma_0}(A) \models \KP$. This is $\Hyp(A)$. 

From the perspective of the unrolling, if $(M,\Xcal) \models \Hyp(V)$ exists then the unrolling has a set which (it thinks) is $\Hyp(M)$. 

\begin{proof}[Proof of master lemma \ref{lem4:master-lemma}]
Because $T \supseteq \GBCm + \ETR$ we can unroll $(M,\Xcal)$ into $W$. Taking isomorphic copies if necessary we may assume without loss that $M = H_\kappa^W$ where $\kappa$ is the largest cardinal in $W$. And because $T$ proves the existence of $\Hyp(V)$ we get that $\Hyp(M)^W \in W$. In general, $A = \Hyp(M)^W$ may be ill-founded, for example if $M$ is ill-founded. Let $\Cfrak$ be the admissible cover of $A$. Now consider the $\Lcal_\Cfrak$ theory $S$ axiomatized by the following.
\begin{itemize}
\item Every theorem $T$ proves about the unrolling;\footnote{If you think of the special case where $T$ is $\KMCC$, then this theory is $\ZFCmi$. In general, this theory is in $M$ because $T \in M$ and it is computable from $T$.}
\item $M = H_\kappa^W$ is an $H_\alpha$-initial segment of the universe. That is, this statement asserts that if $x$ is hereditarily of cardinality $<\kappa$ then $x \in M$; and\footnote{This can be expressed as a single $\Lcal_\Cfrak$-sentence because $M \in A$.}
\item $\kappa$ is the largest cardinal.
\end{itemize}
This $S$ can be expressed as a conjunction of a countable set (in $A$) of $\Lcal_{\omega,\omega}$-formulae with two $\Lcal_{\Cfrak}$-formulae. So it is a single $\Lcal_\Cfrak$-sentence and hence is $\Sigma_1$-definable over $\Cfrak$. This puts us in a position to apply the generalization of Friedman's theorem, since $A$ is countable. That is, there is $N \models S$ which is ill-founded at $A$. Put differently, there is a descending sequence of ordinals in $N$ co-initial in $\Ord^N \setminus \Ord^{A}$. Let $\Ycal = \{Y \in N : N \models Y \subseteq M\}$.\footnote{If one wants to be picky, since we officially only work with models whose second-order part consists of subsets of the first-order part, we actually take an isomorphic copy so that elements of $\Ycal$ are subsets of $M$.}
Then $(M,\Ycal) \models T$, by construction.

\begin{figure}
\begin{center}
\begin{tikzpicture}

\draw (0,0) -- (-1.5,3) -- (1.5,3) -- (0,0);
\draw (0,2) node {$M$};

\draw (-1.5,3) -- (-2,4) -- (2,4) -- (1.5,3);
\draw (0,4.2) node {$W$};

\draw (-1.5,3) -- (-1.55,3.5) -- (1.55,3.5) -- (1.5,3);
\draw (0,3.2) node {$A$};

\filldraw (0,3.5) circle (1.5pt);
\draw (.3,3.7) node {\tiny $\Ord^A$};

\draw (4,2.5) node[scale=2.5] {$\Longrightarrow$};

\draw (8,0) -- (6.5,3) -- (9.5,3) -- (8,0);
\draw (8,2) node {$M$};

\draw (6.5,3) -- (6.45,3.5);
\draw (9.5,3) -- (9.55,3.5);
\draw (8,3.2) node {$A$};

\draw (6.5,3) --  (6,3.5);
\draw (9.5,3) -- (10,3.5);
\draw[dashed] (6,3.5) -- (10,3.5);
\draw[dashed] (5.9,3.6) -- (10.1,3.6);
\draw (5.9,3.6) -- (5.5,4) -- (10.5,4) -- (10.1,3.6);
\draw (8,4.25) node {$N$};
\end{tikzpicture}
\end{center}
\caption{Friedman's theorem gives $N \models S$ with $V_\kappa^N = M = V_\kappa^W$ and $N$ ill-founded at $A$.}
\end{figure}
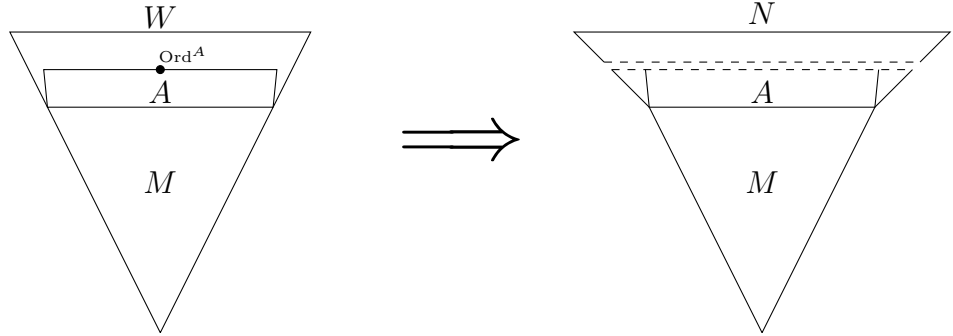

To motivate the following, suppose for a moment that $(M,\Xcal)$ is a $\beta$-model. Then $\Ord^A$ really is an ordinal. Thus, $\Ycal$ has no element with ordertype $\Ord^A$ because otherwise $\Ord^A$ would be in the well-founded part of $N$, contrary to the construction. So $\Xcal \not \subseteq \Ycal$, as desired.

But in general $(M,\Xcal)$ need not be a $\beta$-model, and it may even be that $\Hyp(V)^{(M,\Xcal)}$ is ill-founded. So the above argument cannot work. Nevertheless, it provides the right idea. Fix a membership code $\Upsilon \in \Xcal$ which represents $\Ord^A$ in the unrolling.
Suppose towards a contradiction that $\Xcal \subseteq \Ycal$.

In particular this implies that $\Upsilon \in \Ycal$. Because $\Xcal$ and $\Ycal$ must agree about which of their common classes are well-founded, $(M,\Ycal)$ thinks that $\Upsilon$ is a membership code for an ordinal. Let $\upsilon \in N$ be that ordinal. 
We also have that every initial segment of $\Upsilon$ is in $\Ycal$.\footnote{Recall from chapter 2 that a membership code for an ordinal $\gamma$ is a class well-order of ordertype $\gamma+1$, so it makes sense to talk about initial segments of $\Upsilon$.}
In $W$ we have that every initial segment of $\Upsilon$ is isomorphic to an ordinal in $A$. Because of the assumption that $\Xcal \subseteq \Ycal$ the same isomorphisms exist in $N$. So for every ordinal $\gamma \in A$ we have $N \models \gamma < \upsilon$.

If $N \models \gamma < \upsilon$ then $\gamma$ is isomorphic to an initial segment of $\Upsilon$. But $W$ sees that initial segments of $\Upsilon$ represent ordinals in $A$. So $N$ must see the same and thus $\gamma \in A$. The upshot of all this is that $\Ord^A$ is topped in $N$, namely by $\upsilon$. This contradicts the construction of $N$. So our assumption that $\Xcal \subseteq \Ycal$ must be false, completing the argument. We have found a $T$-realization for $M$ which does not contain $\Xcal$.
\end{proof}

\begin{remark} \label{rmk4:omega-nonstd}
The attentive reader may worry about what happens if $M$ is $\omega$-nonstandard. She is right to worry! There is a subtlety that must be addressed. Namely, if $M$ is $\omega$-nonstandard then no infinite first-order theory $T$ can be in $M$; otherwise, the standard cut would be definable as the supremum of the ranks of elements of $T$. So as written the master lemma does not apply to $\omega$-nonstandard models. 

Nevertheless, there is a version which does apply to $\omega$-nonstandard models. While it does not make sense to ask for $T$ to be an element of $M$ it is sensible to ask that $T$ is coded in $M$, that is whether there is $t \in V_\omega^M \setminus \omega$ so that $t \cap V_\omega = T$. In particular, this always happens if $T$ is computable; run the Turing machine which enumerates $T$ in the nonstandard model of arithmetic coming from $M$ and cut it off at some nonstandard level to get $t$. 

In the $\omega$-nonstandard case replace the assumption that $T \in M$ with the assumption that $T$ is coded in $M$. Then the above proof works. In particular, if (strong enough) $T$ is computable then no countable $T$-realizable $M$ has a least $T$-realization.
\end{remark}

It remains to see that the master lemma yields the nonexistence of least transitive models of strong second-order set theories. Because these theories are all computably axiomatizable it is immediate that they appear as elements of any transitive model. As such, the only thing we need to show is that these theories prove the existence of $\Hyp(V)$. As a warm-up let us prove that $\KMCC$ proves $\Hyp(V)$ exists. This follows from a stronger statement.

\begin{proposition} \label{prop4:hyp-v-in-kmp}
Let $\phi(x, \bar y)$ be a first-order formula in the language of set theory. Then $\ZFCm$ proves that for every $\bar b$ there is an ordinal $\alpha$ so that for all $a \in L_\alpha(\bar b)$, we have $L_\alpha(\bar b) \models \phi(a,\bar b)$ if and only if $\phi(a,\bar b)^{L(\bar b)}$. 
\end{proposition}

\begin{proof}
This is the standard argument for reflection along the $L$-hierarchy. Namely, pick an ordinal $\alpha_0$. Given $\alpha_n$ let $\alpha_{n+1}$ be the least ordinal so that $L_{\alpha_{n+1}}(\bar b)$ is closed under witnesses for existential subformulae of $\phi$ with parameters from $L_{\alpha_n}(\bar b)$. Note that we use Collection to find $\alpha_{n+1}$, since we need to collect witnesses into a single set. Then, if $\alpha = \sup_n \alpha_n$ we have caught our tail and $L_\alpha(\bar b) \models \phi(a,\bar b)$ if and only if $\phi(a,\bar b)$ is true in $L(\bar b)$. 
\end{proof}

\begin{corollary}
$\KMCC$ proves that $\Hyp(V)$ exists.
\end{corollary}

\begin{proof}
Work in the unrolling, which satisfies $\ZFCmi$. Then, since the axioms of $\KP$ are of bounded complexity, there is $\alpha > \kappa$, where $\kappa$ is the largest cardinal, so that $L_\alpha(V_\kappa) \models \KP$. The least such $\alpha$ gives $\Hyp(V_\kappa)$. But $V_\kappa$ is the $V$ of the original $\KMCC$ model. So there is a membership code for $\Hyp(V)$.
\end{proof}

But we need much less than $\KMCC$ to get the existence of $\Hyp(V)$. 

\begin{lemma}
The theory $\GBCm + \PnCA 1$ proves the existence of $\Hyp(V)$. Consequently any $T \supseteq \GBCm + \PnCA 1$ proves the existence of $\Hyp(V)$.
\end{lemma}

See also \cite[theorem 64]{antos-barton-friedman2018} where the same result appears. (They state the result in terms of $\GBC + \PnCA 1$ instead of $\GBCm + \PnCA 1$, but nowhere does their proof use Powerset in the first-order part.)

\begin{proof}
Work with $(M,\Xcal) \models \GBCm + \PnCA 1$ and consider $W = \Lfrak(M,G)$  the $(M,G)$-constructible unrolling of $(M,\Xcal)$, for some $G \in \Xcal$. Then, by results of section \ref{sec2:so-l} we have that $W = \Lfrak \models \ZFCm(1) + V=L(M,G)$.\footnote{Recall that $\ZFCm(1)$ is axiomatized by the axioms of $\ZFCm$ but with Separation and Collection restricted to $\Sigma_1$-formulae.}
Moreover, there is $\kappa \in W$ so that $M = (H_\kappa)^W$. We want to find $\alpha \in W$ so that $L_\alpha(M) \models \KP$. This reduces down to proving an instance of reflection along the $L(M)$-hierarchy, which I give here so the reader can see it can be carried out in the weak theory in which we are currently working.

Every $L_\alpha(M)$ satisfies $\Sigma_0$-Separation, so the work is in getting $\Sigma_0$-Collection. We will see that there are unboundedly many $\alpha$ so that $L_\alpha(M) \models \Sigma_0$-Collection. 
Let $\upsilon$ be the formula giving truth for $\Sigma_0$-formulae. Then $\upsilon$ is $\Sigma_1$. It is convenient here to assume (without loss) that $\upsilon$ has four free variables, so that $\upsilon(\phi,x,y,p)$ asserts that $\phi(x,y,p)$ holds for a $\Sigma_0$-formula $\phi$. To show that $L_\alpha(M) \models \Sigma_0$-Collection it suffices to prove the instance of $\Sigma_0$-Collection for $\upsilon$.

Fix arbitrary $\alpha_0$. By $\Sigma_1$-Collection find $\alpha_1$ the least ordinal $> \alpha_0$ so that if $x,p \in L_{\alpha_0}(M)$ and $\phi$ is a formula then there is $y \in L_{\alpha_1}(M)$ so that $\upsilon(\phi,x,y,p)$. Now repeat the process: given $\alpha_n$ let $\alpha_{n+1}$ be the least ordinal $> \alpha_n$ so that if $x,p \in L_{\alpha_n}(M)$ and $\phi$ is a formula then there is $y \in L_{\alpha_{n+1}}(M)$ so that $\upsilon(\phi,x,y,p)$. Finally, set $\alpha = \sup \alpha_n$, again using an instance of $\Sigma_1$-Collection. Then $L_\alpha(M) \models \Sigma_0$-Collection. Since $\alpha_0$ was arbitrary, this proves there are unboundedly many such $\alpha$.

Now take $\alpha > \kappa$ least so that $L_\alpha(M) \models \KP$. Then $L_\alpha(M) = \Hyp(M)$, so $(M,\Xcal) \models \Hyp(V)$ exists, as desired.
\end{proof}

As a corollary we get the negative part of theorem \ref{thm4:least-trans-models}.

\begin{corollary}
There is not a least transitive model of $\KM$, nor of $\KMm$. For $k \ge 1$ there is not a least transitive model of $\GBC + \PnCA k$, nor of $\GBCm + \PnCA k$. Moreover, the same holds for any computably axiomatizable extensions of these theories.
\end{corollary}

\begin{proof}
I will state the proof in terms of $\KM$. The same argument goes through, {\em mutatis mutandis}, for the other theories.

Suppose for sake of a contradiction that there is a least transitive model of $\KM$. Then it must be some countable $(M,\Xcal) \models \KMCC$. By the master lemma there is a $\KMCC$-realization $\Ycal$ for $M$ so that $\Xcal \not \subseteq \Ycal$. So $(M,\Xcal)$ is not actually least, a contradiction.
\end{proof}

This result holds for more than just computably axiomatizable extensions. What is needed about the extension $T$ is that $T$ is an element of any transitive model of $T$. In particular, the first-order part of $T$ can be complete. If $\Hyp(V)$ exists then the truth predicate for the first-order part must exist, since $\Hyp(V) \models \KP$ and $\KP$ proves the existence of truth predicates for set-sized structures and so any model of a second-order set theory $T \subseteq \GBcm$ which proves the existence of $\Hyp(V)$ must contain its (first-order) theory as an element, since it can be obtained by restricting the truth predicate to sentences.

We also get that countable models do not have least realizations for strong theories.

\begin{corollary}
Let $M \models \ZFCm$ be countable and $T$ an extension of $\GBCm + \PnCA 1$ with $T \in M$.\footnote{Or, if $M$ is $\omega$-nonstandard, with $T$ coded in $M$. Cf.\ remark \ref{rmk4:omega-nonstd}.} 
Then $M$ does not have a least $T$-realization.
\end{corollary}

\begin{proof}
Suppose towards a contradiction that $\Xcal$ is the least $T$-realization for $M$. By the master lemma there is $\Ycal \not \subseteq \Xcal$ a $T$-realization for $M$. But this contradicts the leastness of $\Xcal$.
\end{proof}

In personal communication, Ali Enayat pointed out to me an alternative argument that countable models do not have least $\KM$-realizations. I reproduce his argument, which goes by way of an old theorem by Barwise, here.

\begin{theorem}
Let $M \models \ZFC$ be a countable transitive model. Then $M$ does not have a least $\KM$-realization.
\end{theorem}

From this it immediately follows that $\KM$ does not have a least transitive model.

\begin{proof}[Proof (Enayat):]
If $M$ is not $\KM$-realizable, then the conclusion is trivial. So work in the case where $M$ is $\KM$-realizable. Recall the following theorem.

\begin{theorem*}[Barwise, theorem IV.1.1 of \cite{barwise1975}]
Let $U$ be a countable structure. Let $A = \Hyp(U)$. Let $T$ be an $\Lcal_A$-theory which is $\Sigma_1$-definable over $A$ and which has a model of the form $\Bfrak = (B,U;E, \ldots)$, where $E$ is a binary relation.\footnote{Recall from section \ref{sec4:admissible} that $(B,U;R_0, R_1, \ldots)$ has sets $B$, urelements $U$, and additional relations $R_0, R_1, \ldots$.}
Suppose $S \subseteq U$ has the property that for every such model there is $b \in B$ so that $S = b_E = \{ x \in \Bfrak : x \mathbin E b\}$. Then $S \in \Hyp(U)$. 
\end{theorem*}

Suppose we have a model of the form $(B,M;E)$. Consider the theory $T$ asserting that $M$ forms the first-order part of and $B$ forms the second-order part of a model of $\KM$ with membership relation $E$. This is a computable $\Lcal_{\omega,\omega}$-theory, so in particular it is $\Sigma_1$-definable over $\Hyp(M)$. Because $M$ is $\KM$-realizable there is a model of $T$ of form $\Bfrak = (B,M;E)$. Now suppose that $S \subseteq M$ is in every $\KM$-realization for $M$. Then if $\Bfrak = (B,M;E) \models T$ we can find $b \in B$ so that $S = b_E$. So by Barwise's theorem we get that $S \in \Hyp(M)$. 

This yields the following lemma, which is of independent interest.

\begin{lemma}
Let $M \models \ZFC$ be countable, transitive, and $\KM$-realizable. Then the intersection of all the $\KM$-realizations for $M$ is $\Hyp(M) \cap \powerset(M)$.
\end{lemma}

\begin{proof}
Let $\Xcal$ be the intersection of all the $\KM$-realizations of $M$. We have seen that $\Xcal \subseteq \Hyp(M) \cap \powerset(M)$. For the other direction, take $A \in \Hyp(M) \cap \powerset(M)$. Then there is $\gamma < \Ord^{\Hyp(M)}$ so that $A \in L_\gamma(M)$. To conclude that $A \in \Xcal$ it is enough to know that every $\KM$-realization for $M$ unrolls to a structure which is well-founded up to $\gamma$. To see that, take $(M,\Ycal) \models \KM$. By throwing out classes if necessary, assume without loss that $(M,\Ycal) \models \KMCC$. Let $W \models \ZFCmi$ be the unrolling of $(M,\Ycal)$. It follows from a result of H.\ Friedman \cite[theorem 3.1]{friedman1973} that $W$, being well-founded up to at least $\Ord^M$, must have that $\Ord^{\wfp(W)}$ at least as large as the next admissible ordinal above $\Ord^M$. But it could be that $\Hyp(M)$ is taller than the next admissible ordinal above $\Ord^M$, so we need a small argument.

Take $\gamma \in \Hyp(M)$ an ordinal. Then there is $\Gamma \subseteq M$ in $\Hyp(M)$ which is isomorphic to $(\gamma,\in)$. Because $\Gamma \in \Hyp(M)$ it is $\Delta^1_1$-definable over $M$---see \cite[corollory IV.3.4]{barwise1975}. So $\Gamma \in \Ycal$, because $(M,\Ycal)$ satisfies $\Delta^1_1$-Comprehension. And since $\Gamma$ is seen to be well-founded from the external universe, it must be that $(M,\Ycal)$ agrees that $\Gamma$ is a well-order. So the unrolling of $(M,\Ycal)$, namely $W$, must contain $\gamma$. So $\Ord^{\wfp(W)} > \gamma$, completing the argument.
\end{proof}

Now suppose that $\Xcal = \powerset(M) \cap \Hyp(M)$ is $\KM$-realization for $M$. Then, because $\Xcal$ unrolls to $\Hyp(M)$ and $\KM$ proves that $\Hyp(V)$ exists we get that $\Hyp(M) \in \Hyp(M)$, a contradiction. So $M$ cannot have a least $\KM$-realization.
\end{proof}

The full strength of $\KM$ is not needed here. We used two facts about $\KM$: first, that $\KM$ proves that $\Hyp(V)$ exists; and second, that $\KM$ proves $\Delta^1_1$-Comprehension. So the same argument goes through for weaker theories theories, in particular any theory extending $\GBCm + \PCA$. 

Let me turn now to $\beta$-models, in which context we do get least models.

\begin{theorem}[Folklore]
There is a least $\beta$-model of $\KM$, if there is any $\beta$-model of $\KM$.
\end{theorem}

\begin{proof}
First, note that it is equivalent to ask for a least $\beta$-model of $\KMCC$, by theorem \ref{thm2:inner-model-rlzble} from chapter 2. Next, observe that if $(M,\Xcal) \subseteq (N,\Ycal)$ are $\beta$-models of $\KMCC$, then their unrollings are transitive and the unrolling of $(M,\Xcal)$ must be contained in the unrolling of $(N,\Ycal)$. So we just have to see that there is a least transitive model of $\ZFCmi$. Because there is a $\beta$-model of $\KM$, there is some transitive model of $\ZFCmi$.

Take $M \models \ZFCmi$ transitive with largest cardinal $\kappa$. Then $L^M \models \ZFCm + \kappa$ is inaccessible. However, it could be that $L^M$ satisfies Powerset. (Imagine if $M$ were obtained by class forcing over a model of $\ZFC + V=L$ $+$ ``there is an inaccessible cardinal'' to collapse all cardinals above the first inaccessible.) Nevertheless, there is some ordinal $\alpha \in M$ so that $L_\alpha \models \ZFCmi$. So while $\ZFCmi$ is not absolute to $L$, every transitive model of $\ZFCmi$ contains an $L_\alpha$ which is a model of $\ZFCmi$.

Therefore, if $\alpha$ is least so that $L_\alpha \models \ZFCmi$ then $L_\alpha$ is the least transitive model of $\ZFCmi$, as desired.
\end{proof}

This proof generalizes to a fixed first-order part. First we need a definition.

\begin{definition} \label{def4:beta-rlzn}
Let $T$ be a second-order set theory and $M \models \ZFCm$. Then $\Xcal \subseteq \powerset(M)$ is a {\em $\beta$-$T$-realization for $M$} if $(M,\Xcal) \models T$ is a $\beta$-model. If such $\Xcal$ exists, then $M$ is {\em $\beta$-$T$-realizable}. The {\em least $\beta$-$T$-realization for $M$}, if it exists, is the unique $\beta$-$T$-realization for $M$ which is contained in every $\beta$-$T$-realization.
\end{definition}

\begin{corollary} \label{cor4:least-beta-km-rlzn}
Let $M$ be a $\beta$-$\KM$-realizable model of set theory with a definable global well-order. Then $M$ has a least $\beta$-$\KM$-realization.
\end{corollary}

\begin{proof}
Because $M$ is $\beta$-$\KM$-realizable, there is a model $W$ of $\ZFCmi$ with largest cardinal $\kappa$ so that $M = V_\kappa^W$. Consider $\alpha$ least so that $L_\alpha(M) \models \ZFCmi$ and $M = V_\kappa^{L_\alpha(M)}$. Such exists by an argument as in the proof of the previous theorem, using the fact that $M$ has a definable global well-order to get that $L_\alpha(M)$ satisfies Choice. Then the cutting off $(M,\Xcal)$ for $L_\alpha(M)$ gives the least $\beta$-$\KM$-realization for $M$.
\end{proof}

Essentially the same argument, using tools from chapter 2, gives that there is a least $\beta$-model of $\GBC + \PnCA k$.

\begin{theorem}
Let $k \ge 1$. There is a least $\beta$-model of $\GBC + \PnCA k$, if there is any $\beta$-model of $\GBC + \PnCA k$.
\end{theorem}

\begin{proof}
Again, this reduces to showing that there is a least transitive model of $\ZFCmi(k)$, using that $\beta$-models of $\GBC + \PnCAp k$ unroll to transitive models of $\ZFCmi(k)$. 

If $M \models \ZFCmi(k)$ has largest cardinal $\kappa$ then $L^M \models \ZFCmi(k) + \kappa$ is inaccessible. So there is $\alpha \in M$ so that $L_\alpha \models \ZFCmi(k)$ has largest cardinal $\kappa$. Thus, if $\alpha$ is least so that $L_\alpha \models \ZFCmi(k)$ then $L_\alpha$ is the least transitive model of $\ZFCmi(k)$.
\end{proof}

\begin{corollary} \label{cor4:least-beta-pnca-rlzn}
Fix $k \ge 1$ and let $M$ be a $\beta$-$(\GBC + \PnCA k)$-realizable model of set theory with a definable global well-order. Then $M$ has a least $\beta$-$(\GBC + \PnCA k)$-realization.
\end{corollary}

\begin{proof}
Essentially the same as the proof of corollary \ref{cor4:least-beta-km-rlzn}.
\end{proof}

This is not the end of the story. Does corollary \ref{cor4:least-beta-km-rlzn} give an exact characterization of when $M$ has a least $\beta$-$\KM$-realization? (And in light of corollary \ref{cor4:least-beta-pnca-rlzn} the same question can be asked about $\GBC + \PnCA k$, for $k \ge 1$, instead of $\KM$.)

\begin{question}
Suppose $M \models \ZFC$ is $\beta$-$\KM$-realizable but does not have a definable global well-order. Can we conclude that $M$ does not have a least $\beta$-$\KM$-realization?
\end{question}

It is Global Choice that is the possible culprit here. If we drop that from the axioms then we do always get least realizations.

\begin{proposition}
Let $\KM^{\neg \axiom{GC}}$ denote $\KM$ with Choice for sets but without Global Choice. Suppose that $M \models \ZFC$ is $\beta$-$\KM^{\neg \axiom{GC}}$-realizable. Then $M$ has a least $\beta$-$\KM^{\neg \axiom{GC}}$-realization. 
\end{proposition}

\begin{proof}[Proof sketch]
Similar to the proof of corollary \ref{cor4:least-beta-km-rlzn}. The least $\beta$-$\KM^{\neg \axiom{GC}}$-realization for $M$ is the cut off model obtained from $L_\alpha(M)$ where $\alpha$ is least so that $L_\alpha(M) \models \ZFmi$ and the largest cardinal of $L_\alpha(M)$ is $\Ord^M$.
\end{proof}

A similar fact holds for $\GBc + \PnCA k$, i.e.\ $\GBC + \PnCA k$ but where we drop Global Choice.

To finish off this section, let us see that there are minimal but non-least $\KM$-realizations. That is, there are $M \models \ZFC$ which have a $\KM$-realization $\Xcal$ so that there is no $\KM$-realization $\Ycal$ for $M$ which is strictly contained inside $\Xcal$. But by theorem \ref{thm4:no-least-rlzn} $\Xcal$ cannot be least.

\begin{observation}
Suppose there is a $\beta$-model of $\KM$. Then, there are $M \models \ZFC$ which have minimal but non-least $\KM$-realizations.
\end{observation}

\begin{proof}
Let countable $M \models \ZFC$ have a definable global well-order and be $\beta$-$\KM$-realizable. By corollary \ref{cor4:least-beta-km-rlzn} it has a least $\beta$-$\KM$-realization, call it $\Xcal$. Because any $\KM$-realization $\Ycal \subseteq \Xcal$ must also give a $\beta$-model and thus $\Ycal = \Xcal$, this $\Xcal$ is a minimal $\KM$-realization. But as we saw earlier in this section, $M$ does not have a least $\KM$-realization.
\end{proof}

Left open is the question of when, if ever, other countable $M \models \ZFC$ have minimal but non-least $\KM$-realizations. 

\begin{question}
Is there countable $M \models \ZFC$ which does not have a minimal $\KM$-realization? 
\end{question}

The same question can be asked for $\GBC + \PnCA k$ instead of $\KM$.

\section{Medium theories} \label{sec4:medium}

I turn now to the theories of medium strength, namely $\GBC + \ETR$ and its variants. The main results of this section are that $\GBC + \ETR$ has a least $\beta$-model and that for nice enough choice of $\Gamma$ that $\GBC + \ETR_\Gamma$ has a least transitive and a least $\beta$-model. (See the discussion around the relevant theorems in this section for what ``nice enough'' means.) The major question left open is whether there is a least transitive model of $\ETR$.

As a starting-off point, let us see that some models have least $\beta$-$(\GBC + \ETR)$-realizations.

\begin{theorem} \label{thm4:least-beta-etr-rlzn}
Let $M \models \ZFC$ be a transitive model with a definable global well-order. Then if $M$ has a $\beta$-$(\GBC + \ETR)$-realization it has a least $\beta$-$(\GBC + \ETR)$-realization $\Xcal$. Moreover, $\Xcal$ is also the least $(\GBC+\ETR)$-realization of $M$.
\end{theorem}

\begin{proof}
Fix $\Ycal$ a $\beta$-$(\GBC + \ETR)$-realization for $M$. The strategy is to define $\Xcal$, which will be the least $\beta$-$(\GBC + \ETR)$-realization for $M$ contained inside $\Ycal$. We will then see that in fact $\Xcal$ is contained inside any $(\GBC+\ETR)$-realization for $M$. In particular, if we started with a different $\beta$-realization we would define the same $\Xcal$. 

We define $\Xcal$ in $\omega$ many steps. First, let $\Xcal_0 = \Def(M)$. Then $(M,\Xcal_0) \models \GBC$ because $M$ has a definable global well-order. Now given $\Xcal_n \subseteq \Ycal$ define $\Xcal_{n+1}$ to consist of all classes in $\Ycal$ which are definable from $\Tr_\Gamma(A)$ for some $\Gamma,A \in \Xcal_n$. Formally,
\[
\Xcal_{n+1} = \bigcup \left\{
\Def\left(M;\Tr_\Gamma(A)\right) : A,\Gamma \in \Xcal_n \mand \Gamma \text{ is a well-order} 
\right\}.
\]
Because $(M,\Ycal)$ is a $\beta$-model, it is correct about which $\Gamma$'s are well-orders. So we could equivalently ask in the definition of $\Xcal_{n+1}$ that $(M,\Ycal) \models \Gamma$ is a well-order.
Then $\Xcal_{n+1} \subseteq \Ycal$ is a $\GBC$-realization for $M$. Finally, set $\Xcal = \bigcup_n \Xcal_n$. It is clear that $\Xcal \subseteq \Ycal$. 

Let us check that $(M,\Xcal) \models \GBC + \ETR$. It satisfies $\GBC$ because $\Xcal$ is the union of an increasing chain of $\GBC$-realizations for $M$. To see that it satisfies Elementary Transfinite Recursion, pick $A,\Gamma \in \Xcal$ where $(M,\Xcal) \models \Gamma$ is a well-order. Observe that $\Ycal$ agrees with $\Xcal$ that $\Gamma$ is a well-order. Since $A,\Gamma \in \Xcal_n$ for some $n$ this means that $\Tr_\Gamma(A) \in \Xcal_{n+1} \subseteq \Xcal$. 

Finally, let us see that $\Xcal$ is contained in any $(\GBC+\ETR)$-realization $\Zcal$, which will establish that $\Xcal$ is both the least $\beta$-$(\GBC+\ETR)$-realization for $M$ and the least $(\GBC+\ETR)$-realization for $M$. Clearly, $\Xcal_0 = \Def(M)$ is contained inside $\Zcal$. We continue upward inductively. Having already seen that $\Xcal_n \subseteq \Zcal$, consider $\Gamma,A \in \Xcal_n$. Then all three of $\Ycal$, $\Xcal_n$, and $\Zcal$ must agree on which classes are well-founded and which classes are iterated truth predicates. And since $\Gamma$ really is well-founded, externally we can see that there is only one option for what class is the $\Gamma$-iterated truth predicate relative to $A$. So since $(M,\Zcal) \models \ETR$ we get that $(\Tr_\Gamma(A))^{(M,\Zcal)} = \Tr_\Gamma(A) \in \Xcal_{n+1} \cap \Zcal$. And since $\Zcal$ is closed under first-order definability, any class definable from $\Tr_\Gamma(A)$ must be in $\Zcal$. So $\Xcal_{n+1} \subseteq \Zcal$. This holds for all $n$, so $\Xcal \subseteq \Zcal$, as desired.
\end{proof}

The proof did not use Powerset. So we get a version for theories without Powerset. And the only place we used that $M$ has a definable global well-order was to get Global Choice in $\Xcal$. So we also get a version for theories without Global Choice. The following corollary encapsulates both results.

\begin{corollary}
Let $M \models \ZFCm$ be a transitive model. Suppose $M$ is $\beta$-($\GBcm + \ETR)$-realizable. Then $M$ has a least $\beta$-$(\GBcm + \ETR)$-realization. If $M$ moreover has a definable global well-order then $M$ has a least $\beta$-$(\GBCm + \ETR)$-realization. \qed
\end{corollary}

Observe that although $\Xcal_n$ is always a coded $V$-submodel of $\Ycal$, in general $\Xcal$ need not be coded in $\Ycal$. In particular, this will happen when $\Ycal = \Xcal$. 

We get a version of this result for non-$\beta$-models. In this broader context we cannot ensure that different $\Ycal$'s will define the same $\Xcal$. But any $V$-submodel of $(M,\Ycal)$ will have to agree with $(M,\Ycal)$ as to what is a well-order and whether a class is $\Tr_\Gamma(A)$. So a similar argument yields a local leastness result.

\begin{theorem} \label{thm4:etr-rlzn-basis}
Let $M \models \ZFC$ be $(\GBc + \ETR)$-realizable. Then $M$ has a basis of minimal $(\GBc + \ETR)$-realizations, where amalgamable $(\GBc + \ETR)$-realizations\footnote{Two $T$-realizations $\Xcal$ and $\Ycal$ for $M$ are {\em amalgamable} if there is a $\GBcm$-realization $\Zcal$ for $M$ so that $\Xcal$ and $\Ycal$ are both subsets of $\Zcal$.}
sit above the same basis element. That is, there is a set $\{\Bcal_i : i \in I\}$ of $(\GBc + \ETR)$-realizations for $M$ satisfying the following.
\begin{enumerate}
\item Elements of the basis are pairwise non-amalgamable;
\item If $\Ycal$ is any $(\GBc + \ETR)$-realization for $M$ then there is a unique basis element $\Bcal$ so that $\Ycal \supseteq \Bcal$; and
\item If $\Xcal$ and $\Ycal$ are amalgamable $(\GBc + \ETR)$-realizations for $M$ then they sit above the same $\Bcal$.
\end{enumerate}
\end{theorem}

See figure \ref{fig4:etr-rlzn-basis} for a picture of the $(\GBc + \ETR)$-realizations for $M$.

In case $M$ has a definable global well-order we get a basis of minimal $(\GBC + \ETR)$-realizations, since every $\GBc$-realization for $M$ must contain the definable global well-order and thereby satisfy Global Choice.

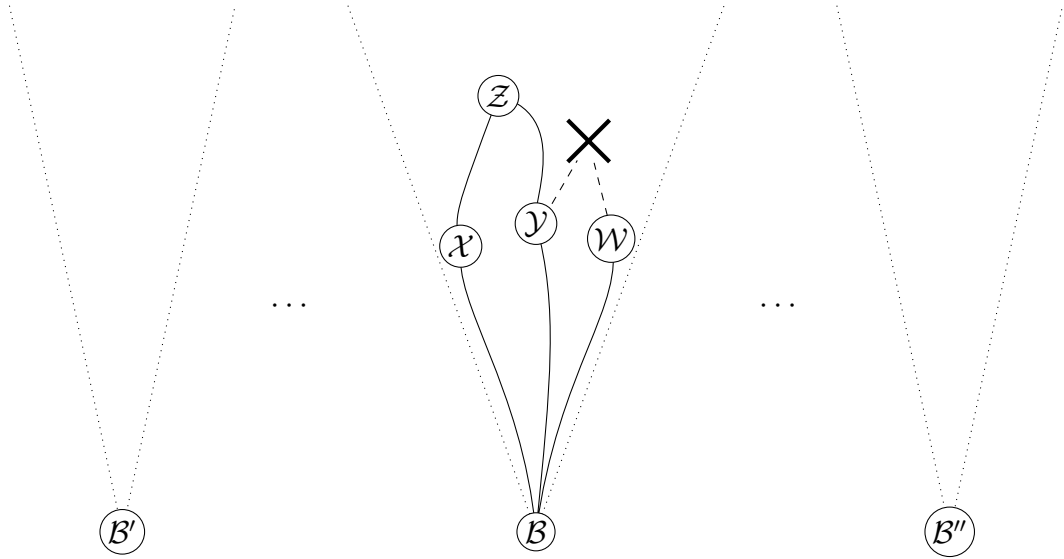
\begin{figure}
\begin{center}
\begin{tikzpicture}[n/.style={circle,draw,fill=white,minimum size=.5cm,inner sep=1pt}, 
n1/.style={cross out,draw,ultra thick,fill=white,minimum size=.5cm,inner sep=1pt},
n2/.style={circle,fill=white,minimum size=.6cm,inner sep=1pt}]
\draw[dotted] (-2.5,7) -- (0,0) -- (2.5,7);
\draw (0,0) .. controls (-.2,2) and (-1.1,3) .. (-1,3.8)
      (0,0) .. controls (.2,2) and (.3,3) ..  (0,4.1)
      (0,0) .. controls (.3,2.2) and (1.2,3) ..  (1,3.9)
     (-1,3.8) .. controls (-1.1,4.2) .. (-.5,5.8)
      (0,3.9) .. controls (-.2,4.1) and (.6,5.5) .. (-.5,5.8);
\draw[dashed] (0,4) -- (.7,5.2)
      (1,4) -- (.7,5.2);
\draw  (0,0) node[n] {$\Bcal$}
      (-1,3.8) node[n] {$\Xcal$}
       (0,4.1) node[n] {$\Ycal$}
       (1,3.9) node[n] {$\Wcal$}
     (-.5,5.8) node[n] {$\Zcal$}
      (.7,5.2) node[n2] {}
      (.7,5.2) node[n1] {};

\draw[dotted] (-7,7) -- (-5.5,0) -- (-4,7);
\draw (-5.5,0) node[n] {$\Bcal'$};
\draw (-3.25,3) node {$\cdots$};

\draw[dotted] (4,7) -- (5.5,0) -- (7,7);
\draw (5.5,0) node[n] {$\Bcal''$};
\draw (3.25,3) node {$\cdots$};
\end{tikzpicture}
\end{center}
\caption{The $(\GBc+\ETR)$-realizations for $M$ form a disjoint collection of cones, each one with a basis element at the bottom.}
\label{fig4:etr-rlzn-basis}
\end{figure}

\begin{proof}
Fix $\Ycal$ a $(\GBC + \ETR)$-realization for $M$. We define the basis element $\Bcal$ below $\Ycal$ similar to how we defined $\Xcal$ in the proof of theorem \ref{thm4:least-beta-etr-rlzn}. Start with $\Bcal_0 = \Def(M)$. Clearly $\Bcal_0 \subseteq \Ycal$ is a $\GBc$-realization for $M$. Given $\Bcal_n \subseteq \Ycal$ a $\GBc$-realization for $M$ we set
\[
\Bcal_{n+1} = \bigcup \left\{
\Def\left(M;\left(\Tr_\Gamma(A)\right)^{(M,\Ycal)}\right) : A,\Gamma \in \Bcal_n \mand (M,\Ycal) \models \Gamma \text{ is a well-order} 
\right\}.
\]
Some illuminating remarks are in order. First, because $\Bcal_n \subseteq \Ycal$ and $(M,\Ycal) \models \ETR$ we get that for $A,\Gamma \in \Bcal_n$ there is a unique class in $\Ycal$ which $(M,\Ycal)$ thinks is $\Tr_\Gamma(A)$. So $\Bcal_{n+1}$ is well-defined. And since $\Ycal$ is closed under first-order definability this moreover shows that $\Bcal_{n+1} \subseteq \Ycal$.

Next, let me emphasize that $\Def$ here is the {\em external} $\Def$ operator. This only makes a difference in case $M$ is an $\omega$-model. In this case, none of the $\Bcal_n$ will be coded in $\Ycal$. Nevertheless, we still get that $\Bcal_n \subseteq \Ycal$, so that $\Bcal_{n+1}$ is well-defined.

Third, let us check that $(M,\Bcal_{n+1}) \models \GBc$. I will be more detailed than in the proof of theorem \ref{thm4:least-beta-etr-rlzn} to reassure the reader who is worried things may go wrong in an $\omega$-nonstandard model. Both Class Extensionality and Class Replacement are immediate. To see Elementary Comprehension we want to see that $\Bcal_{n+1}$ is closed under first-order definability. It suffices to check the case where we define a class from two class parameters, so consider $X,Y \in \Bcal_{n+1}$. Then, by construction, there are $\Gamma,\Delta,A,B \in \Bcal_n$ so that $X$ is definable from $(\Tr_\Gamma(A))^{(M,\Ycal)}$ and $Y$ is definable from $(\Tr_\Delta(B))^{(M,\Ycal)}$. Then any class definable from $X$ and $Y$ must be definable from $(\Tr_{\max\{\Gamma,\Delta\}}(A \oplus B))^{(M,\Ycal)} \in \Bcal_n$, where $A \oplus B = A \times \{0\} \cup B \times \{1\}$. So any class definable from $X$ and $Y$ is in $\Bcal_{n+1}$.

Finally, set $\Bcal = \bigcup_n \Bcal_n$. It is immediate that $\Bcal \subseteq \Ycal$. Let us see that $(M,\Bcal) \models \GBc + \ETR$. Fix $\Gamma,A \in \Bcal$ so that $(M,\Bcal) \models \Gamma$ is a well-order. Then $\Gamma,A \in \Bcal_n$ for some $n$. Because $(M,\Bcal_n)$ is a $V$-submodel of $(M,\Bcal)$ which in turn is a $V$-submodel of $(M,\Ycal)$ they all agree as to whether $\Gamma$ is a well-order and whether a class is $\Tr_\Gamma(A)$. So
\[
(\Tr_\Gamma(A))^{(M,\Bcal)} = (\Tr_\Gamma(A))^{(M,\Ycal)} \in \Bcal_{n+1} \subseteq \Bcal.
\]
So $(M,\Bcal) \models \ETR$.

The proof will be finished once we see that any $\Xcal$ which is amalgamable with $\Ycal$ defines the same $\Bcal$. To see this, take $\Zcal$ a $\GBc$-realization for $M$ which contains both $\Xcal$ and $\Ycal$. Then $(M,\Ycal)$ and $(M,\Zcal)$ must agree whether a class is a well-order and whether a class is $\Tr_\Gamma(A)$ and the same holds for $(M,\Xcal)$ and $(M,\Zcal)$, so in fact all three agree. Clearly $\Bcal_0$ is the same whether defined used $\Ycal$ or $\Xcal$. And inductively upward they must agree on $\Bcal_n$ because they agree as to whether a class is $\Tr_\Gamma(A)$. So no matter whether we start with $\Xcal$ or $\Ycal$ we define the same basis element $\Bcal$.
\end{proof}

It follows from theorem \ref{thm4:least-beta-etr-rlzn} that if $M$ has a $\beta$-$(\GBC+\ETR)$-realization then there is only one basis element in the poset of $(\GBC+\ETR)$-realizations for $M$. Does this hold in general?

\begin{question}
Is there a $(\GBc + \ETR)$-realizable model $M$ so that the basis for the $(\GBc+\ETR)$-realizations for $M$ has more than one element? Or, asked in the negative, is it true that every $(\GBc + \ETR)$-realizable model has a least $(\GBc + \ETR)$-realization?
\end{question}

Let me detour to discuss fragments of $\ETR$. Essentially the same argument as in theorem \ref{thm4:least-beta-etr-rlzn} gives that $\GBC + \ETR_\Gamma$ will have least $\beta$-realizations. For example, to show that a $\beta$-$(\GBC + \ETR_\Gamma)$-realizable model $M$ with a definable global well-order has a least $\beta$-$(\GBC + \ETR_\Gamma)$-realization we define $\Xcal$ in a similar manner. Set $\Xcal_0 = \Def(M;\Gamma)$. The definition for $\Xcal_{n+1}$ is then
\[
\Xcal_{n+1} = \bigcup \left\{ \Def\left(M; \Tr_\Gamma(A)\right) : A \in \Xcal_n \right\}
\]
where $\Ycal$ is some fixed-in-advance $\beta$-$(\GBC + \ETR_\Gamma)$-realization for $M$. Observe that $\Gamma$ here really is well-founded, because a $\beta$-model thinks it is well-founded, so $\Tr_\Gamma(A)$ is externally seen to be unique. Then $\Xcal = \bigcup_n \Xcal_n$ will be the least $\beta$-$(\GBC + \ETR_\Gamma)$-realization for $M$. This gives us the following results, analogous to the above results about $\ETR$.

\begin{theorem} \label{thm4:beta-etr-gamma}
Suppose $(M,\Ycal) \models \GBC + \ETR_\Gamma$ is an $\beta$-model with a definable global well-order and $\omega^\omega \le \Gamma \in \Ycal$. Then $M$ has a least $\beta$-$(\GBC + \ETR_\Gamma)$-realization $\Xcal$. Moreover, $\Xcal$ is also the least $(\GBC+\ETR_\Gamma)$-realization for $M$. \qed
\end{theorem}

\begin{remark}
The purpose of requiring $\Gamma \ge \omega^\omega$ is that this ensures $\ETR_\Gamma$ is equivalent to the existence of $\Gamma$-iterated truth predicates relative to any class. The same applies to later results about $\ETR_\Gamma$, but I will suppress making this comment every time.
\end{remark}

Since we are concerned only with a fixed $\Gamma$, we can get a least $(\GBC + \ETR_\Gamma)$-realization even if our $M \models \ZFC$ has realizations which are wrong about well-foundedness. All that matters is whether they are correct about $\Gamma$ being well-founded.

\begin{theorem} \label{thm4:least-etr-rlzns}
Consider an $\omega$-model $(M,\Ycal) \models \GBC$ and $\Gamma \in \Ycal$ so that $\Gamma \ge \omega^\omega$ really is well-founded, as seen externally.
Suppose $M$ has a definable global well-order and $(M,\Ycal) \models \ETR_\Gamma$. Then $M$ has a least $(\GBC + \ETR_\Gamma)$-realization. \qed
\end{theorem}

If $M$ is ill-founded then this works for $\Gamma$ of length in the well-founded part of $M$.\footnote{This explains why I restricted the statement of the theorem to $\omega$-models. If $M$ is $\omega$-nonstandard and $(M,\Ycal) \models \Gamma \ge \omega^\omega$ then $\Gamma$ must seen from outside to be ill-founded. Nevertheless, the conclusion of the theorem is still true if $\Gamma$ is in the well-founded part of $\omega$-nonstandard $M$, because in such a case $\Gamma$ is standard finite and $\ETR_n$ for standard finite $n$ is equivalent to Elementary Comprehension. So in this case it is just asking for $M$ to have a least $\GBC$-realization, which indeed does happen if $M$ has a definable global well-order.}
If $M$ is transitive then we can go up to $\Ord^M$, and even longer. For transitive $M$ we always get that $\Ord^M + \Ord^M$, $\Ord^M \cdot \Ord^M$, and so on are well-founded.

And like before, for non-$\beta$-models we get a local leastness result, even if $(M,\Ycal)$ is wrong about $\Gamma$ being well-founded.

\begin{corollary}
Suppose $(M,\Ycal) \models \GBC + \ETR_\Gamma$ has a definable global well-order and $(M,\Ycal) \models \omega^\omega \le \Gamma$. Then $M$ has an $(\GBC + \ETR_\Gamma)$-realization which is least below $\Ycal$. \qed
\end{corollary}

We are now ready to see that $\GBC + \ETR$ has a least $\beta$-model.

\begin{theorem} \label{thm4:least-beta-etr}
There is a least $\beta$-model of $\GBC + \ETR$, if there is any $\beta$-model of $\GBC + \ETR$.
\end{theorem}

\begin{proof}
The least $\beta$-model of $\GBC + \ETR$ will be the least $\beta$-$(\GBC + \ETR)$-realizable $L_\alpha$ along with its least $\beta$-$(\GBC + \ETR)$-realization. First though we have to know that if $M$ is $\beta$-$(\GBC + \ETR)$-realizable then so is $L^M$. We saw in chapter 3 that $M$ being $(\GBC + \ETR)$-realizable implies that $L^M$ is also $(\GBC + \ETR)$-realizable. More specifically, if $(M,\Xcal) \models \GBC + \ETR$ then there is $\Ycal \subseteq \Xcal \cap \powerset(L^M)$ so that $(L^M,\Ycal) \models \GBC + \ETR$. If $(M,\Xcal)$ is a $\beta$-model then so is $(L^M, \Ycal)$, by observation \ref{obs4:beta-v-submodel}. Therefore, if there is a $\beta$-model of $\GBC + \ETR$ then $L_\alpha$ is $\beta$-$(\GBC + \ETR)$-realizable where $\alpha$ is the least height of a $\beta$-model of $\GBC + \ETR$. 

Now let $\Xcal$ be the least $\beta$-$(\GBC + \ETR)$-realization for $L_\alpha$, which exists by theorem \ref{thm4:least-beta-etr-rlzn}. We want to see that $(L_\alpha,\Xcal)$ is contained inside every $\beta$-model of $\GBC + \ETR$. Fix $(N,\Ycal) \models \GBC + ETR$ a $\beta$-model. If $\Ord^N = \alpha$, then theorem \ref{thm4:least-beta-etr-rlzn} yields that $\Xcal \subseteq \Ycal$. If $\Ord^N > \alpha$ then $L_\alpha \in N$ and thus $N$ can construct $\Xcal$ as in theorem \ref{thm4:least-beta-etr-rlzn} by ordinary transfinite recursion on sets. That is, $N$ starts with $\Xcal_0 = \Def(L_\alpha)$. Then given $\Xcal_n$ add in all the classes definable from $\Tr_\Gamma(A)$ for $\Gamma,A \in \Xcal_n$ to get $\Xcal_{n+1}$. This makes sense, because each $(L_\alpha,\Xcal_n)$ is a $\beta$-model, as can be seen externally from $V$, and thus in $N$ as $N$ is correct about well-foundedness. Then $\Xcal = \bigcup_n \Xcal_n$ must be in $N$, as otherwise would imply that $N$ does not satisfy an instance of Replacement. Thus, $(L_\alpha,\Xcal) \subseteq (N,\Ycal)$
\end{proof}

Essentially the same argument gives least $\beta$-models for $\GBC + \ETR_\Gamma$. But first a subtlety needs to be cleared up. When dealing with a fixed model with a fixed class well-order $\Gamma$ it was sensical to ask whether it satisfies $\ETR_\Gamma$. However, this will not work if we do not have a fixed model in mind. How are we even to express $\ETR_\Gamma$ as an $\Lcal_\in$-theory?


What we can do is ask that $\Gamma$ be given some definition which evaluated in a model of $\GBC$ always gives a well-order. For instance, $\Gamma$ could be $\Ord$ or $\omega_1$. Then, although different models may disagree on what $\Gamma$ is, $\ETR_\Gamma$ can be expressed as an $\Lcal_\in$ theory. To distinguish this case from when $\Gamma$ is a literal class in a model and $\ETR_\Gamma$ is expressed as an $\Lcal_\in(\Gamma)$-theory I will talk of $\Gamma$ being given by a definition. For example, if I say that $\Gamma$ is given by a first-order definition I mean that there is a certain first-order $\Lcal_\in$-formula $\phi(x,y)$ so that $\GBC$ proves that the class defined by $\phi(x,y)$ is a well-order.
 If $(M,\Xcal) \models \GBC$ then I will write $\Gamma^{(M,\Xcal)}$ for the well-order in $\Xcal$ given by applying the definition of $\Gamma$ inside $(M,\Xcal)$. In case $\Gamma$ is given by a first-order definition I will simply write $\Gamma^M$, as the evaluation depends only upon $M$.

\begin{theorem} \label{thm4:least-beta-etr-gamma}
Let $\Gamma \ge \omega^\omega$ be given by a first-order definition without parameters.\footnote{That is, $\GBC$ proves that $\Gamma \ge \omega^\omega$.}
Assume that $\Gamma$ is necessarily absolute to $L$, meaning that if any $(M,\Xcal) \models \GBC$ then $\Gamma^{L^M} = \Gamma^M$.
If there is a $\beta$-model of $\GBC + \ETR_\Gamma$ then there is a least $\beta$-model of $\GBC + \ETR_\Gamma$.
\end{theorem}

I will explain after the proof why we need the absoluteness condition on $\Gamma$. For now, observe that the condition can be ensured for many ordertypes of interest, e.g.\ $\omega^\omega$, $\Ord$, $\Ord + \Ord$, and so on. One can give a definition for a well-order of ordertype, e.g., $\Ord + \Ord$ which is not absolute to $L$. For example: ``If $0^\sharp$ exists then $\Gamma$ is $\Ord$ followed by $\Ord \times \{0^\sharp\}$ (with the obvious order) and if $0^\sharp$ does not exist then $\Gamma$ is the even ordinals followed by the odd ordinals.'' But there is a perfectly good definition of $\Gamma$ with ordertype $\Ord + \Ord$ which is absolute to $L$. 

\begin{proof}
Let $\alpha$ be least such that there is a $\beta$-$(\GBC + \ETR_\Gamma)$-realizable $M$ of height $\alpha$. By theorem \ref{thm3:etr-gamma-inner-model} we can conclude $L_\alpha$ is $(\GBC + \ETR_\Gamma)$-realizable. This uses that $\Gamma$ is absolute to $L$, as the result in chapter 3 was for $\Gamma$ being a fixed well-order, rather than being given by a definition. Absoluteness to $L$ ensures that $M$ and $L_\alpha$ have the same $\Gamma$. By theorem \ref{thm4:beta-etr-gamma} we have $\Xcal$ the least $\beta$-$(\GBC + \ETR_\Gamma)$-realization for $L_\alpha$. This uses that $\Gamma$ is defined by a first-order formula, so we get the same well-order regardless of what collection of classes we put on $L_\alpha$. 

It remains only to see that $(L_\alpha,\Xcal)$ is the least $\beta$-model of $\GBC + \ETR_\Gamma$. Take $(N,\Ycal)$ a $\beta$-model of $\GBC + \ETR_\Gamma$. There are two cases. First, consider the case $\Ord^N = \alpha$. Then $\Gamma^N = \Gamma^{L_\alpha}$. We saw in theorem \ref{thm3:etr-gamma-inner-model} that there is $\bar \Ycal \subseteq \Ycal$ so that $(L_\alpha,\bar \Ycal) \subseteq (N,\Ycal)$. So, by the leastness of $\Xcal$ we get $(L_\alpha,\Xcal) \subseteq (L_\alpha,\bar \Ycal) \subseteq (N,\Ycal)$. Second, consider the case $\Ord^N > \alpha$. Then $L_\alpha \in N$ so $\Gamma^{L_\alpha} \in N$ and $N$ can build $\Xcal$ by ordinary transfinite recursion. So $\Xcal \in N$ and thus $(L_\alpha,\Xcal) \subseteq (N,\Ycal)$. 
\end{proof}

Let me now give an example to explain why we want to require $\Gamma$ to be absolute to $L$. Consider $\Gamma$ defined by: ``If $V = L$ then $\Gamma = \Ord$ and if $V \ne L$ then $\Gamma = \omega^\omega$.'' Let us see that $\ETR_\Gamma$ does not have a least $\beta$-model. Take $\alpha$ least so that there is a $\beta$-model of $\GBC + \ETR_{\omega^\omega}$ of height $\alpha$. Then $\alpha$ must be countable. Let $x$ and $y$ be mutually generic Cohen-reals over $L_\alpha$. Then $L_\alpha[x]$ and $L_\alpha[y]$ have least $\beta$-$\ETR_{\Gamma}$-realizations, call them $\Xcal$ and $\Ycal$ respectively. But there is no $\beta$-model of $\ETR_\Gamma$ which is contained in both $(L_\alpha[x],\Xcal)$ and $(L_\alpha[y],\Ycal)$. By leastness of $\alpha$, such a model would have to have height $\alpha$. But $L_\alpha[x] \cap L_\alpha[y] = L_\alpha$ by mutual genericity, so the first-order part of such a model would have to be $L_\alpha$. However $L_\alpha$ is not $\beta$-$(\GBC + \ETR_\Gamma)$-realizable because $\Gamma^{L_\alpha} = \Ord^{L_\alpha}$ and $\GBC + \ETR_\Ord$ proves there is a set-sized $\beta$-model of $\ETR_{\omega^\omega}$.\footnote{This is because $\GBC + \ETR_\Ord$ proves there is a coded $V$-submodel of $\ETR_{\omega^\omega}$ and this can be reflected down to get a set-sized $\beta$-model of $\ETR_{\omega^\omega}$.}
So if $L_\alpha$ were $\beta$-$(\GBC + \ETR_\Gamma)$-realizable that would contradict the leastness of $\alpha$.

Requiring $\Gamma$ to be absolute to $L$ rules out definitions like this one.

This argument only needs that the models are correct about their $\Gamma$ being well-founded. So if $\Gamma$ is given by a definition which interpreted in any transitive model gives a relation which really is a well-order, then we can get a least transitive model of $\GBC + \ETR_\Gamma$.

\begin{theorem} \label{thm4:least-etr-gamma}
Let $\Gamma \ge \omega^\omega$ be given by a first-order definition without parameters. Assume the following.
\begin{enumerate}
\item $\Gamma$ is absolute to $L$, meaning that if $(M,\Xcal) \models \GBC$ then $\Gamma^{L^M} = \Gamma^M$; and
\item If $(M,\Xcal) \models \GBC$ is transitive then $\Gamma^M$ really is well-founded, as seen externally.
\end{enumerate}
Then, if there is a $\beta$-model of $\GBC + \ETR_\Gamma$ there is a least transitive model of $\GBC + \ETR_\Gamma$. \qed
\end{theorem}

As particular cases of interest, this works when $\Gamma = \Ord$ or $\Gamma$ is a given by a definition for a (set-sized) ordinal.

Left open is how high up this can be pushed. Can it be pushed all the way up to $\ETR$?

\begin{question}[Open]
Is there a least transitive model of $\GBC + \ETR$?
\end{question}

Let me note this question is not immediately settled by master lemma \ref{lem4:master-lemma}.

\begin{proposition}
The theory $\GBC + \ETR$ does not prove that $\Hyp(V)$ exists, assuming that it is consistent.
\end{proposition}

\begin{proof}
Suppose $(M,\Xcal) \models \GBC + \ETR + \Hyp(V)$ exists. (If there is no such model, then we are already done.) By shrinking down to an inner model if necessary, assume without loss that $M$ has a definable global well-order. Unroll $(M,\Xcal)$ to $W$. From the results in chapter 2, we get that $W$ satisfies $\Sigma_0$-Transfinite Recursion. Now let $\Ycal = \{ A \in W : W \models A \in \Hyp(M) \mand A \subseteq M \}$ consist of the classes of $M$ which are in $\Hyp(M)^W$. Then $\Ycal$ is closed under first-order definability, so it satisfies Elementary Comprehension. It satisfies Class Replacement because it is a $V$-submodel of a model of $\GBC$. Class Extensionality is obvious and Global Choice holds because $M$ has a definable global well-order. Altogether, we have seen that $(M,\Ycal) \models \GBC$. But it also must satisfy $\ETR$, because $\Hyp(M)^W \models \KP$ and $\KP$ proves $\Sigma_0$-Transfinite Recursion. But $(M,\Ycal) \not \models \Hyp(V)$ exists, by construction. so $\GBC + \ETR$ does not prove that $\Hyp(V)$ exists.
\end{proof}

Let me mention a couple related questions. The first question was also asked at the end of chapter 2.

\begin{question}
Let $\tau(\GBC + \ETR)$ be the least height of a transitive model of $\GBC + \ETR$ and let $\beta(\GBC + \ETR)$ be the least height of a $\beta$-model of $\GBC + \ETR$. Can we conclude that $\tau(\GBC + \ETR) < \beta(\GBC + \ETR)$?
\end{question}

If the answer is no, then there is a least transitive model of $\GBC + \ETR$, which would be the same as the least $\beta$-model of $\GBC + \ETR$. 

\begin{question}
Can there be $M \models \ZFC$ which is $(\GBC + \ETR)$-realizable with a class well-order $\Gamma \in \Def(M)$ so that there are two different $(\GBC+\ETR)$-realizations $\Xcal$ and $\Ycal$ for $M$ so that $\Tr_\Gamma^{(M,\Xcal)} \ne \Tr_\Gamma^{(M,\Ycal)}$? What if we restrict to countable $M$?
\end{question}

It follows from theorem \ref{thm4:least-beta-etr-rlzn} that this cannot happen if $M$ has a $\beta$-$(\GBC+\ETR)$-realization.

\section{Weak theories}

The results in this section are either old or appeared in chapter 1. I state them here to round out the presentation in this chapter.

Let us start with a classical result.

\begin{theorem}[Shepherdson \cite{shepherdson1953}]
There is a least transitive model of $\GBC$, if there is any transitive model of $\GBC$.
\end{theorem}

\begin{proof}
Let $L_\alpha$ be the least transitive model of $\ZFC$. Then $(L_\alpha, \Def(L_\alpha))$ is the least transitive model of $\GBC$.
\end{proof}

A slight modification also gives a least $\beta$-model.

\begin{corollary}
There is a least $\beta$-model of $\GBC$, if there is any $\beta$-model of $\GBC$.
\end{corollary}

\begin{proof}
First, observe that if $(M,\Xcal)$ is a $\beta$-model of $\GBC$ then $(L^M,\Def(L^M))$ is also a $\beta$-model of $\GBC$. This is because $(L_\alpha, \Def(L_\alpha))$ is always a model of $\GBC$ and because $\Ord$-submodels of $\beta$-models are $\beta$-models. 

Now let $\alpha$ be least such that there is a $\beta$-model of $\GBC$ of height $\alpha$. Then $(L_\alpha,\Def(L_\alpha))$ is the least $\beta$-model of $\GBC$.
\end{proof}

As with the case for $\ETR$, it is not clear to me that these are actually two different results.

\begin{question}
Is the height of the least transitive model of $\GBC$ less than the height of the least $\beta$-model of $\GBC$? Phrased differently, is the least transitive model of $\GBC$ a $\beta$-model?
\end{question}

We can also characterize when a countable model of $\ZFC$ has a least $\GBC$-realization.

\begin{theorem}
Let $M \models \ZFC$ be countable. Then the following are equivalent.
\begin{enumerate}
\item $M$ has a least $\GBC$-realization.
\item $M$ has a definable global well-order.
\item $M \models \exists x\ V = \HOD(\{x\})$.
\end{enumerate}
\end{theorem}

\begin{proof}
$(1 \iff 2)$ was theorem \ref{thm1:gm-omnibus}.$(3)$ from chapter 1. $(2 \iff 3)$ is a well-known fact.\footnote{The reason for using $\HOD(\{x\})$ instead of $\HOD$ is that we want to possibly allow parameters for the definition of the global well-order.}
\end{proof}

\section{Coda: the analogy to second-order arithmetic}

Several times throughout the course of this dissertation we have touched upon the analogy between second-order set theory and second-order arithmetic. Now that we are done with the major results I would like to flesh this analogy out more fully. Let us begin by seeing how the theories line up.

\begin{figure}[h]
\begin{center}
\begin{tabular}{c c}
Arithmetic & Set theory \\
\hline
$\Zsf_2$ & $\KM$ \\
$\PCA_0$ & $\PCA$ \\
$\ATR_0$ & $\ETR$ \\
$\ACA_0$ & $\GBC$ \\
$\WKL_0$ & \\
$\RCA_0$ & 
\end{tabular}
\end{center}
\caption{Arithmetic versus set theory.}
\end{figure}

It's well-known that $\PA$ is bi-interpretable with finite set theory $\ZFC^{\neg\infty}$, i.e.\ $\ZFC$ with Infinity replaced with its negation.\footnote{Note that, however, this is sensitive to how Foundation is formulated, as formulations equivalent over infinitary set theory are not equivalent over finite set theory. So finite set theory should be understood as formulated with the right version of Foundation. See \cite{kaye-wong2007} for a discussion.}
This bi-interpretability carries over for second-order set theories so that e.g. $\ACA_0$ and $\GBC^{\neg\infty}$ are bi-interpretable. So the analogy here is really between the finite and the transfinite. 

Let me briefly address the gap in the table. First, $\WKL_0$. Enayat and Hamkins \cite{enayat-hamkins2016} proved that (in $\ZFC$) there is a definable $\Ord$-tree whose levels are all set-sized with no definable branch. Consequently, $\GBC$ does not prove the analog of K\H{o}nig's lemma for $\Ord$-trees instead of $\omega$-trees. So it is not clear how tree properties on $\omega$---such as weak K\H{o}nig's lemma or K\H{o}nig's lemma---could be generalized to this context. For $\RCA_0$, different ways of thinking of computability suggests different generalizations to set theory. First, consider the view that computable $= \Delta_1$. With this view in mind, the set theoretic counterpart to $\RCA_0$ would be based upon $\Delta^0_1$-Comprehension. Second, take the view that computable means verifiable and refutable by only looking at a bounded segment of the universe. By this view, the set theoretic counterpart to $\RCA_0$ would be based upon $\Delta^0_2$-Comprehension, as the $\Delta_2$ properties are precisely those which are verifiable and refutable by looking at a rank-initial segment of the universe. Which of these two counterparts is the `correct' one would depend upon the mathematics one can do with them. If, say, $\Delta^0_2$-Comprehension allows us to do a lot of interesting mathematics but $\Delta^0_1$-Comprehension does not, then we would have reason to prefer one over the other. But we cannot make such a call before actually doing that mathematics. For now, it is not clear what the set theoretic counterpart to $\RCA_0$ should be.

That aside, let us return to where there are no gaps. Some results about models of arithmetic have direct generalizations to results about models of set theory. For example, it is well-known that every model of arithmetic has a least $\ACA_0$-realization. The set theoretic counterpart to this fact is that every model of $\ZFC$ with a definable global well-order has a least $\GBC$-realization. We also get counterparts to results about strong theories. That $\KM$ has a least $\beta$-model but no least transitive model is the transfinite analog of the folklore fact that $\Zsf_2$ has a least $\beta$-model and H.\ Friedman's theorem \cite{friedman1973} that $\Zsf_2$ has no least $\omega$-model.

But disanalogies appear at the level of $\GBC + \ETR$ versus $\ATR_0$. We have seen that $\GBC + \ETR$ has a least $\beta$-model. On the other hand, $\ATR_0$ has neither a least $\omega$-model nor a least $\beta$-model---see \cite{simpson:book}. Let us explore this disanalogy further. The main culprit here is the property ``$X$ is a well-order''. In set theory, this is a first-order assertion. Whether $X$ is a well-order is determined by a countable piece of information, so is witnessed by the (non)existence of certain sets, not proper classes. On the other hand, in arithmetic it is a second-order assertion to say that $X$ is a well-order, since infinite sequences are second-order objects in the arithmetic context. Indeed, ``$X$ is a well-order'' is $\Pi^1_1$-universal in arithmetic, so which classes are well-ordered is very much caught up in the second-order part of the model.

Another disanalogy concerns the existence of $\Hyp(V)$. It is well-known that $\Hyp(V_\omega) = L_\omegaoneck$, where $\omegaoneck$ is the least non-computable ordinal. So $\ATR_0$ suffices to prove that $\Hyp(V)$ is coded\footnote{Remember that in arithmetic, $V$ is $V_\omega$!}
because $\omegaoneck$ is arithmetical and so we can do a transfinite recursion along it to produce the $L$-hierarchy. Indeed, the intersection of all the $\beta$-models of $\ATR_0$ is the collection of hyperarithmetical sets, those reals appearing in $L_\omegaoneck$---see \cite{simpson:book}. (And the same is true if we consider $\omega$-models instead of $\beta$-models.)
But in the set theoretic context, we have seen that $\GBC + \ETR$ does not suffice to prove that $\Hyp(V)$ is coded.

More disanalogies are known. To pick one last example, in arithmetic Clopen Determinacy and Open Determinacy are both equivalent to $\ATR_0$, a result originating in Steel's dissertation \cite{steel:diss}. In set theory, Clopen Determinacy (for class games) is equivalent (over $\GBC$) to $\ETR$ \cite{gitman-hamkins2016}. But Open Determinacy is strictly stronger \cite{hachtman2016}.

These suggest that the emerging field of reverse mathematics of second-order set theory should reveal a landscape with some differences from that of reverse mathematics of arithmetic. As further evidence pointing toward this, in the set theoretic context a fragment of $\ETR$ captures natural mathematical principles. Gitman, Hamkins, Holy, Schlicht, and myself showed \cite{GHHSW2017} that $\ETR_\Ord$ is equivalent (over $\GBC$) to the class forcing theorem and several other  natural statements. On the other hand, the same does not happen for fragments of $\ATR_0$. 

More speculatively, work in second-order set theory might shed some light, even if only privatively, on the project of reverse mathematics in second-order arithmetic. Each of the Big Five subsystems of second-order arithmetic roughly corresponds to a philosophical position. (Cf.\ chapter I of \cite{simpson:book}.) For example, $\ATR_0$ corresponds to predicative reductionism, where ``predicative'' here means ``predicative given $\omega$''. We could instead ask about predicativism given $V$, or given $H_{\omega_1}$, or some other object. This project has already been taken up, e.g.\ by Sato \cite{sato2014}. The disanalogies between $\ETR$ and $\ATR_0$ suggest that starting from the finite realm has a large impact on predicativism and related projects. If we take the transfinite as our starting point then we end up with a very different theory. On the other hand, where we do seen an analogy between second-order arithmetic and second-order set theory it suggests that there the role of the finite versus the transfinite is not so vital.

In short, work in second-order set theory may help clarify what in reverse mathematics (of arithmetic) relies essentially upon the first-order domain consisting of finite objects.

\clearpage

\backmatter

\singlespacing

\printbibliography[heading=bibintoc,title={Bibliography}]

\end{document}